\documentclass[10pt, reqno]{amsart}
\usepackage{amsfonts,latexsym,enumerate}
\usepackage{amsmath}
\usepackage{amscd}
\usepackage{float,amsmath,amssymb,mathrsfs,bm,multirow,graphics, mathtools}
\usepackage[dvips]{graphicx}
\usepackage[percent]{overpic}
\usepackage[pdftex]{color}
\usepackage{amsaddr}
\usepackage[numbers,sort&compress]{natbib}

\newcommand{\R}{{\Bbb R}}

\newcommand{\C}{{\Bbb C}}
\newcommand{\Z}{{\Bbb Z}}

\newcommand{\proofbegin}{\noindent{\it Proof. }}
\newcommand{\proofendcontinue}{\hfill \raisebox{.8mm}[0cm][0cm]{$\bigtriangledown$}\bigskip}

\newcommand{\res}{\text{\upshape Res\,}}
\newcommand{\diag}{\text{\upshape diag\,}}
\newcommand{\re}{\text{\upshape Re\,}}

\newcommand{\im}{\text{\upshape Im\,}}

\newcommand{\ntlim}{\lim^\angle}

\DeclareMathOperator{\sgn}{sgn}
\DeclareMathOperator{\dist}{dist}

\def\XXint#1#2#3{{\setbox0=\hbox{$#1{#2#3}{\int}$}
\vcenter{\hbox{$#2#3$}}\kern-.5\wd0}}

\newtheorem{theorem}{Theorem}[section]

\newtheorem{lemma}[theorem]{Lemma}
\newtheorem{definition}[theorem]{Definition}
\newtheorem{assumption}[theorem]{Assumption}
\newtheorem{remark}[theorem]{Remark}

\newtheorem{figuretext}[theorem]{Figure}

\numberwithin{equation}{section}

\input epsf
\date{\today}
\title[Construction of solutions and asymptotics]
{Construction of solutions and asymptotics for the sine-Gordon equation in the quarter plane}

\author{Lin Huang and Jonatan Lenells}
\address{Department of Mathematics, KTH Royal Institute of Technology, \\ 100 44 Stockholm, Sweden.}
\email{linhuang@kth.se, jlenells@kth.se}

\begin{document}
\begin{abstract} 
\noindent
We consider the sine-Gordon equation in laboratory coordinates in the quarter plane.
The first part of the paper considers the construction of solutions via Riemann-Hilbert techniques. In addition to constructing solutions starting from given initial and boundary values, we also construct solutions starting from an independent set of spectral (scattering) data. 
The second part of the paper establishes asymptotic formulas for the quarter-plane solution $u(x,t)$ as $(x,t) \to \infty$. Assuming that $u(x,0)$ and $u(0,t)$ approach integer multiples of $2\pi$ as $x \to \infty$ and $t \to \infty$, respectively, we show that the asymptotic behavior is described by four asymptotic sectors. In the first sector (characterized by $x/t \geq 1$), the solution approaches a multiple of $2\pi$ as $x \to \infty$. In the third sector (characterized by $0 \leq x/t \leq 1$ and $t|x-t| \to \infty$), the solution asymptotes to a train of solitons superimposed on a radiation background. The second sector (characterized by $0 \leq x/t \leq 1$ and $x/t \to 1$) is a transition region and the fourth sector (characterized by $x/t \to 0$) is a boundary region. We derive precise asymptotic formulas in all sectors. In particular, we describe the interaction between the asymptotic solitons and the radiation background, and derive a formula for the solution's topological charge.
\end{abstract}

\maketitle



\setcounter{tocdepth}{1}
\tableofcontents

\section{Introduction}
One of the most interesting aspects of the sine-Gordon equation is its integrability, which manifests itself in the presence of a Lax pair and an infinite number of conservation laws. The Lax pair can be used to define a nonlinear Fourier transform (the inverse scattering transform) which linearizes the time evolution of the equation. Considering the sine-Gordon equation in laboratory coordinates
\begin{align}\label{sg}
u_{tt}- u_{xx} + \sin u=0,
\end{align}
and assuming appropriate smoothness and decay, the inverse scattering transform implies that the solution of the {\it initial-value problem} on the line endowed with the boundary conditions $\lim_{|x| \to \infty} u(x,t) = 0$ (mod $2\pi$)
can be expressed in terms of the solution of a $2 \times 2$-matrix Riemann-Hilbert (RH) problem \cite{K1975, ZTF1974}; see also \cite{FT1986} and Appendix A of \cite{BM2008}. 
By applying the nonlinear steepest descent approach of Deift and Zhou \cite{DZ1993} to this RH problem, it is possible to analyze the long-time asymptotics of the solution. For pure radiation solutions (i.e. in the absence of solitons), asymptotic formulas were derived in this way in \cite{CVZ1999}. More precisely, it was shown in \cite{CVZ1999} that if the initial data $u_0(x) := u(x,0)$ and $u_1(x) := u_t(x,0)$ are in the Schwartz class $\mathcal{S}(\R)$ and the corresponding solution $u(x,t)$ of the initial-value problem for (\ref{sg}) is solitonless, then 
\begin{enumerate}[$(i)$]
\item $\sin u \to 0$ rapidly as $t \to \infty$ in the sector $|x/t| \geq 1$,
\item $\sin u$ approaches an explicitly given modulated sinusoidal traveling wave with amplitude of order $O(t^{-1/2})$ in the sector $|x/t| < 1$ provided that $t|x-t| \to \infty$.
\end{enumerate}

In this paper, we study the {\it initial-boundary value (IBV) problem} for equation (\ref{sg}) posed in the quarter-plane domain
\begin{align}\label{quarterplane}
\{(x,t) \in \R^2 \, | \, x \geq 0, t \geq 0\}.
\end{align}
We denote the initial data by $u_0, u_1$ and the Dirichlet and Neumann boundary values by $g_0$ and $g_1$, respectively:
\begin{align}\label{uugg}
\begin{cases} u_0(x) = u(x,0), \\ u_1(x) = u_t(x,0), \end{cases}  x \geq 0; \qquad\qquad
\begin{cases} g_0(t) = u(0,t), \\ g_1(t) = u_x(0,t), \end{cases}  t \geq 0.
\end{align} 

Studies of the quarter-plane IBV problem for (\ref{sg}) were initiated by Fokas and Its \cite{FI1992} under the assumption that $u_0(x)$ and $g_0(t)$ approach integer multiples of $2\pi$ as $x \to \infty$ and as $t \to \infty$, respectively.
The presence of the boundary at $x = 0$ means that the inverse scattering transform formalism cannot be applied in its standard form. However, it was shown in \cite{FI1992} that the solution $u(x,t)$ can be related to the solution of a RH problem whose formulation involves both the $x$- and $t$-parts of the Lax pair. 
This approach was later refined and developed into a unified approach to IBV problems for integrable PDEs \cite{F1997,F2002}. In this unified framework, which can be viewed as a generalization of the inverse scattering transform to the setting of IBV problems, the $x$- and $t$-parts of the Lax pair are combined into a single differential form and thereby put on an equal footing. The RH problem of \cite{F2002} associated with equation (\ref{sg}) is formulated in terms of four spectral functions $a,b,A,B$, where $a(k)$ and $b(k)$ are defined in terms of the initial data $\{u_j(x)\}_{0}^1$, whereas $A(k)$ and $B(k)$ are defined in terms of the boundary values $\{g_j(t)\}_{0}^1$. Boundary value problems for the elliptic version of (\ref{sg}) have been studied in \cite{FP2012, FLP2013, PP2010, P2009}.

Equation (\ref{sg}) is second order in both space and time.
This has the implication that both $u_0(x)$ and $u_1(x)$, but only one of the functions $g_0(t)$ and $g_1(t)$  (or a linear combination of these two functions) can be independently prescribed for a well-posed problem; the second boundary condition lies at $x = \infty$. Thus, in order to correspond to a well-posed problem, the four functions $u_0, u_1, g_0, g_1$ need to satisfy a consistency condition. 
When expressed at the level of the spectral functions, this consistency condition is called the {\it global relation}.

The purpose of the present paper is twofold. The first objective is to construct quarter-plane solutions of (\ref{sg}) via RH techniques. This includes a rigorous implementation of the approach of \cite{F2002} under limited regularity and decay assumptions. However, in addition to constructing solutions starting from the initial and boundary values $u_0, u_1, g_0, g_1$ as in \cite{F2002}, we also consider the construction of solutions starting from an independent collection of spectral data. 

The second objective is to establish asymptotic formulas for the quarter-plane solution $u(x,t)$ as $(x,t) \to \infty$. Assuming that $u_0(x)$ and $g_0(t)$ rapidly approach integer multiples of $2\pi$ as $x \to \infty$ and as $t \to \infty$, respectively, 
\begin{align}\label{NxNtdef}
\lim_{x \to \infty} u_0(x) = 2\pi N_x,  \qquad
\lim_{t \to \infty} g_0(t) = 2\pi N_t, \qquad N_x, N_t \in \Z,
\end{align}
we show that the asymptotic behavior is described by four asymptotic sectors which we refer to as Sectors I-IV (see Figure \ref{sectors.pdf}):

\begin{enumerate}[$-$]
\item In Sector I (characterized by $x/t \geq 1$), the solution tends to $2\pi N_x$ as $x \to \infty$. 

\item In Sector II (the transition sector, characterized by $0 \leq x/t \leq 1$ and $x/t \to 1$), the solution satisfies $u(x,t) = 2\pi N_x + O\big((1-x/t)^N + t^{-N}\big)$ for any integer $N \geq 1$. 

\item In Sector III (characterized by $0 \leq x/t \leq 1$ and $t|x-t| \to \infty$), the solution asymptotes, to leading order, to a train of solitons traveling at speeds $v_j \in (0,1)$. These solitons are generated by the poles of the RH problem; in general, they cause the solution to carry a nonzero topological charge (also known as winding number). The subleading contribution to the asymptotics is a modulated sinusoidal traveling wave whose amplitude is $O(t^{-1/2})$ in the bulk of the sector. The amplitude of the sinusoidal wave vanishes as $(x,t)$ approaches the boundary at $x = 0$. 

\item In Sector IV (the boundary sector, characterized by $0 \leq x/t \leq 1$ and $x/t \to 0$), the solution satisfies $u(x,t) = 2\pi N_t + O\big((x/t)^N + t^{-N}\big)$ for any integer $N \geq 1$. 
\end{enumerate}

\begin{figure}
\bigskip\begin{center}
\begin{overpic}[width=.65\textwidth]{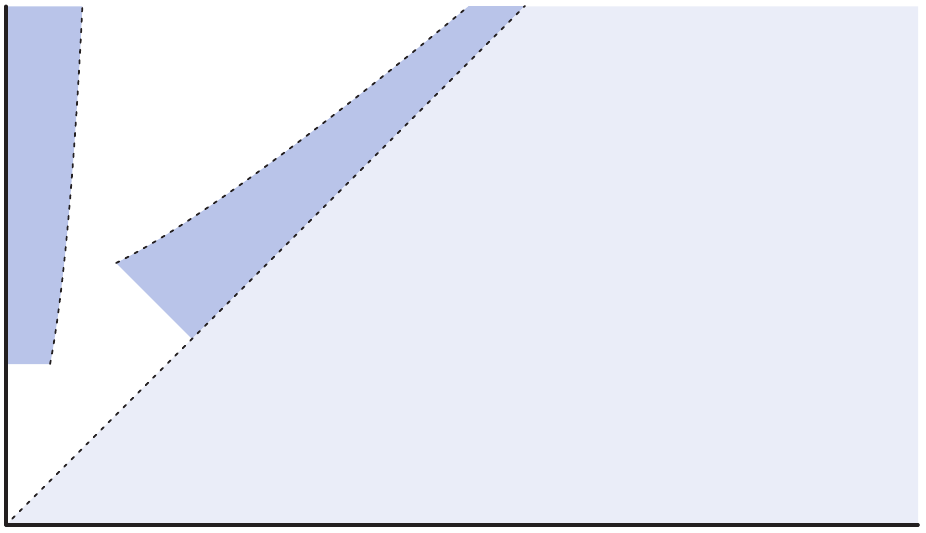}
      \put(102,0){\small $x$}
      \put(0.3,59){\small $t$}
      \put(60,26){\small Sector I}
      \put(54.5,21){\small (Rapid decay)}
      \put(31,36){\makebox(0,0){\rotatebox{42}{{\small Sector II (Transition)}}}}
      \put(4.2,38.3){\makebox(0,0){\rotatebox{88}{{\small Sector IV}}}}
      \put(16,54){\small Sector III}
      \put(13,49.5){\small (Solitons and}
      \put(16,45){\small radiation)}
    \end{overpic}
     \begin{figuretext}\label{sectors.pdf}
       The four asymptotic sectors in the $xt$-plane.
     \end{figuretext}
     \end{center}
\end{figure}

Our approach relies on RH techniques and nonlinear steepest descent ideas.
For didactic purposes, we first consider pure radiation solutions. Our main results for such solutions are summarized in four theorems as follows:
\begin{itemize}
\item Theorem \ref{existenceth} and Theorem \ref{existenceth2} are concerned with the {\it construction} of quarter-plane solutions of (\ref{sg}). They show how solutions can be constructed starting from some given spectral functions, or from some given initial and boundary values satisfying the global relation, respectively.

\item Theorem \ref{asymptoticsth} and Theorem \ref{asymptoticsth2} are concerned with the {\it asymptotics} of quarter-plane solutions of (\ref{sg}). They provide asymptotic formulas for the solution $u(x,t)$ constructed from some given spectral functions, or from some given initial and boundary values satisfying the global relation, respectively.

\end{itemize}

Later, in Section \ref{solitonsec}, we extend the above theorems to the setting when solitons are present. We describe how kinks, antikinks, and breathers emerge asymptotically for large $t$ in the presence of a point spectrum. By demonstrating how the train of asymptotic solitons is superimposed on the radiation background, we complete a large part of the program for the sine-Gordon equation envisioned in \cite{CVZ1999}, which sought for the calculation of the solution and its winding number in the presence of both radiation and solitons. The calculation is somewhat involved because each of the continuous and discrete spectra influences both the soliton and radiation terms; thus there are in total four different effects to take into account. 
The computation of the subleading radiation term is the most difficult. However, by consistently keeping track of the terms of subleading order in $t$ in each of the algebraic systems generated by the discrete spectrum, and by appealing to the results of Section \ref{mainsec}, we are able to compute this term in detail.

For the problem on the line, the topological charge (or winding number) of the solution is defined as the integer
$$\frac{1}{2\pi}\Big(\lim_{x \to \infty} u(x,t) - \lim_{x \to -\infty} u(x,t)\Big).$$
The simplest class of solutions of (\ref{sg}) exhibiting a nonzero topological charge is the class of traveling waves
\begin{align}\label{kink}
u(x,t) = 4 \arctan\Big(e^{\sigma \frac{x - vt - x_0}{\sqrt{1 - v^2}}}\Big),
\end{align}
where $x_0 \in \R$,  $\sigma = \pm 1$ is the topological charge, and $v \in (-1,1)$ denotes the wave velocity. The solutions in (\ref{kink}) are one-solitons and are referred to either as kinks (for $\sigma = 1$) or as antikinks (for $\sigma = -1$). 

For the quarter-plane problem endowed with the boundary conditions (\ref{NxNtdef}), we refer to the integer $N_x - N_t$ as the topological charge. By patching the asymptotic formulas in the four different asymptotic sectors together, we are able to compute the overall topological charge of the solution. For this purpose, the asymptotic analysis in the transition region, although being a relatively narrow sector, is critical. One conclusion of our analysis is that only solitons can generate a nonzero topological charge. In other words, the topological charge of a pure radiation solution always vanishes.

Although all theorems in this paper are derived and stated for the half-line problem, the analogous results for the problem on the line can be obtained essentially as special cases by ignoring the contributions stemming from the boundary. For example, it follows in this way that the asymptotics in the transition sector takes the same form on the line as on the half-line. Since the asymptotics in the transition sector appears not to have been considered before even on the line, this yields an interesting result also for the pure initial-value problem. 
A similar remark applies to the interaction between the asymptotic solitons and the radiation background: Both on the line and half-line, the effect of the asymptotic solitons on the subleading radiation term can be expressed as in Theorem \ref{solitonasymptoticsth}.

Apart from the limit of large time, another limit of great interest for equation (\ref{sg}) is the semiclassical limit. This limit has been carefully studied in the context of the initial-value problem \cite{BM2008, BM2013}. At the expense of changing the initial and boundary values, the semiclassical limit can be related to the large $(x,t)$ limit via a rescaling of the variables $x$ and $t$. Indeed, if the function $v(x,t; \epsilon)$ satisfies the equation
\begin{align}\label{semiclassicalsg}
\epsilon^2 v_{tt} - \epsilon^2 v_{xx} + \sin v = 0,
\end{align}
then $u(x,t) := v(\epsilon x, \epsilon t; \epsilon)$ satisfies (\ref{sg}). Hence, for a fixed $(x,t)$, the semiclassical limit $\epsilon \to 0$ of $v(x,t; \epsilon) = u(x/\epsilon,t/\epsilon)$ translates into the large $(x,t)$ limit of $u(x,t)$. However, unless the initial and boundary values are invariant with respect to the rescaling, the two limits demand separate treatment.
It would be of interest to consider the semiclassical limit also for the IBV problem.

Our results are constructive and independent of any assumption of wellposedness. Nevertheless, let us briefly comment on the wellposedness of equation (\ref{sg}) in the quarter-plane domain (\ref{quarterplane}). A number of similar nonlinear evolution PDEs have been shown to be wellposed in appropriate Sobolev spaces when posed in the quarter plane with a Dirichlet or Neumann boundary condition imposed at $x = 0$; examples include the KdV equation \cite{BSZ2002, H2006}, the generalised KdV equation \cite{CK2002}, the nonlinear Schr\"odinger equation \cite{H2005}, and the good Boussinesq equation \cite{HM2015}.
The proofs in the listed references involve solving the linearized version of the IBV problem with an additional forcing term. Replacing the forcing in the linear solution formula by the nonlinear term and using appropriate linear estimates, a contraction mapping argument leads to the wellposedness. It seems likely that the Dirichlet and Neumann problems for (\ref{sg}) can be analyzed in a similar way, but we know of no reference.

\subsection{Organization of the paper}
In Section \ref{specsec}, after recalling the Lax pair formulation of (\ref{sg}), we define the spectral functions relevant for the quarter-plane problem and recall some properties of these functions. 

The four main theorems are stated (for the pure radiation case) in Section \ref{mainsec}: Theorem \ref{existenceth} and Theorem \ref{existenceth2} are concerned with the construction of quarter-plane solutions (\ref{sg}), whereas Theorem \ref{asymptoticsth} and Theorem \ref{asymptoticsth2} provide asymptotic formulas for such solutions as $(x,t)$ goes to infinity.

The proofs of Theorems \ref{existenceth} and \ref{existenceth2} rely on RH techniques and are presented in Sections \ref{existenceproofsec} and \ref{existence2proofsec}, respectively. 

Sections \ref{overviewsec}-\ref{sectorIVsec} are devoted to the proof of Theorem \ref{asymptoticsth}. The proof is based on the nonlinear steepest descent approach of Deift and Zhou \cite{DZ1993} and is rather involved. An overview of the proof is provided in Section \ref{overviewsec}. Then, in Sections \ref{sectorIsec}-\ref{sectorIVsec}, we consider the derivation of the asymptotics in each of the four asymptotic sectors I-IV, respectively. Given Theorem \ref{asymptoticsth}, it is straightforward to establish Theorem \ref{asymptoticsth2}; the required argument is presented already in Section \ref{mainsec}.

Section  \ref{solitonsec} presents versions of the main theorems which are valid also in the presence of solitons. 

\subsection{Notation}
The following notation will be used throughout the paper.
\begin{enumerate}[$-$]
\item For a $2 \times 2$ matrix $A$, we let $[A]_1$ and $[A]_2$ denote the first and second columns of $A$.

\item If $A$ is an $n \times m$ matrix, we define $|A| \geq 0$ by
$|A|^2 = \sum_{i,j} |A_{ij}|^2$. Then $|A + B| \leq |A| + |B|$ and $|AB| \leq |A| |B|$.

\item For a (piecewise smooth) contour $\gamma \subset \C$ and $1 \leq p \leq \infty$, we write $A \in L^p(\gamma)$  if $|A|$ belongs to $L^p(\gamma)$. Then $A \in L^p(\gamma)$ iff each entry $A_{ij}$ belongs to $L^p(\gamma)$. We write $\|A\|_{L^p(\gamma)} := \| |A|\|_{L^p(\gamma)}$. 

\item For a complex-valued function $f(k)$ of $k \in \C$, we let $f^*(k) := \overline{f(\bar{k})}$ denote the Schwartz conjugate of $f(k)$.

\item We let $\{\sigma_j\}_1^3$ denote the three Pauli matrices defined by
\begin{align*}
\sigma_1=\begin{pmatrix}
           0 & 1 \\
           1 & 0
         \end{pmatrix},\quad \sigma_2=\begin{pmatrix}
                                        0 & -i \\
                                        i & 0
                                      \end{pmatrix},\quad \sigma_3=\begin{pmatrix}
                                                                     1 & 0 \\
                                                                     0 & -1
                                                                   \end{pmatrix}.
\end{align*}

\item The operator $\hat{\sigma}_3$ acts on a $2\times 2$ matrix $A$ by $\hat{\sigma}_3A = [\sigma_3, A]$, i.e., $e^{\hat{\sigma}_3} A = e^{\sigma_3} A e^{-\sigma_3}$.

\item We write $\C_\pm = \{k \in \C \, | \, \im k \gtrless 0\}$ for the open upper and lower half-planes. The closed half-planes are denoted by $\bar{\C}_+ = \{k \in \C \, | \, \im k \geq 0\}$ and $\bar{\C}_- = \{k \in \C \, | \, \im k \leq 0\}$.

\item The open domains $\{D_j\}_1^4$ of the complex $k$-plane are defined by (see Figure \ref{Gamma.pdf})
\begin{align}\nonumber
D_1 = \{\im k > 0\} \cap \{|k|>1\},  \qquad
D_2 = \{\im k > 0\} \cap \{|k|<1\},
	\\ \nonumber
D_3 = \{\im k < 0\} \cap \{|k|<1\},  \qquad
D_4 = \{\im k < 0\} \cap \{|k|>1\}.
\end{align}

\item We let $\Gamma = \R \cup \{|k| = 1\}$ denote the contour which separates the domains $D_j$, oriented so that $D_1 \cup D_3$ lies to the left of the contour as in Figure \ref{Gamma.pdf}. 

\item The notation $k \in (D_i, D_j)$ indicates that the first and second columns of the equation under consideration hold for $k \in D_i$ and $k \in D_j$, respectively.

\item If $D$ is an open connected subset of $\C$ bounded by a piecewise smooth curve $\gamma \subset \hat{\C} := \C \cup \{\infty\}$ and $z_0 \in \C \setminus \bar{D}$, then we write $\dot{E}^p(D)$, $1 \leq p < \infty$, for the space of all analytic functions $f:D \to \C$ with the property that there exist piecewise smooth curves $\{C_n\}_1^\infty$ in $D$ tending to $\Gamma$ in the sense that $C_n$ eventually surrounds each compact subset of $D$ and such that
$$\sup_{n \geq 1} \int_{C_n} |z - z_0|^{p-2} |f(z)|^p |dz| < \infty.$$
If $D = D_1 \cup \cdots \cup D_n$ is a finite union of such open subsets, then $\dot{E}^p(D)$ denotes the space of analytic functions $f:D\to \C$ such that $f|_{D_j} \in \dot{E}^p(D_j)$ for each $j$.

\item For a (piecewise smooth) contour $\Gamma \subset \hat{\C}$ and a function $h$ defined on $\Gamma$, the Cauchy transform $\mathcal{C}h$ is defined by
\begin{align}\label{Cauchytransform}
(\mathcal{C}h)(z) = \frac{1}{2\pi i} \int_\Gamma \frac{h(z')dz'}{z' - z}, \qquad z \in \C \setminus \Gamma,
\end{align}
whenever the integral converges. 
If $h \in L^2(\Gamma)$, then the left and right nontangential boundary values of $\mathcal{C}h$, which we denote by $\mathcal{C}_+ h$ and $\mathcal{C}_- h$ respectively, exist a.e. on $\Gamma$ and lie in $L^2(\Gamma)$; moreover, $\mathcal{C}_\pm \in \mathcal{B}(L^2(\Gamma))$ and $\mathcal{C}_+ - \mathcal{C}_- = I$, where $\mathcal{B}(L^2(\Gamma))$ denotes the space of bounded linear operators on $L^2(\Gamma)$.

\item The function $\theta := \theta(x,t,k)$ is defined by $\theta = \theta_1x + \theta_2 t$, where
\begin{align}\label{thetadef}
\theta_1 := \theta_1(k) = \frac{1}{4}\bigg(k - \frac{1}{k}\bigg) \quad \text{and} \quad \theta_2 := \theta_2(k) = \frac{1}{4}\bigg(k + \frac{1}{k}\bigg).
\end{align}

\item We use $C > 0$ (and sometimes $c > 0$) to denote generic constants which may change within a computation.
\end{enumerate}

\begin{figure}
\begin{center}
\bigskip\bigskip\bigskip
\begin{overpic}[width=.65\textwidth]{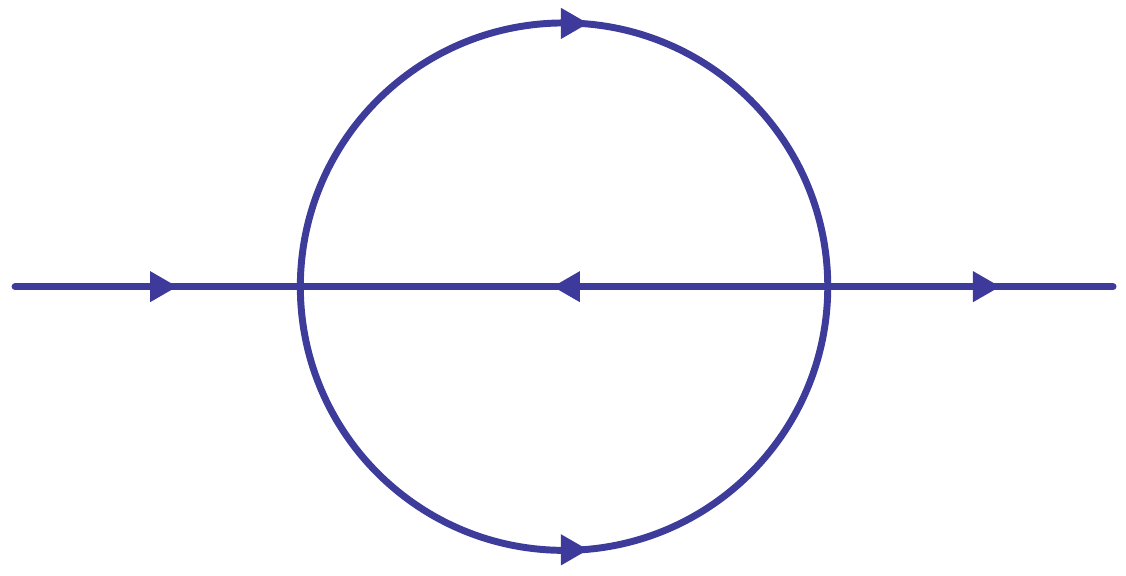}
      \put(74.5,22){\small $1$}
      \put(21,22){\small $-1$}
      \put(49,57){\small $D_1$}
      \put(49,35){\small $D_2$}
      \put(49,13){\small $D_3$}
      \put(49,-7){\small $D_4$}
      \put(102,24.5){\small $\Gamma$}
      \end{overpic}\bigskip\bigskip\bigskip
     \begin{figuretext}\label{Gamma.pdf}
       The contour $\Gamma$ and the domains $\{D_j\}_1^4$ in the complex $k$-plane.
     \end{figuretext}
     \end{center}
\end{figure}

\section{Spectral functions}\label{specsec}
This section introduces several spectral functions and reviews how these functions can be combined to set up a RH problem suitable for solving equation (\ref{sg}) in the quarter plane. 

\subsection{Lax pair}
Equation (\ref{sg}) is the compatibility condition of the Lax pair
\begin{align}\label{lax}
\begin{cases}
\mu_x +\frac{i}{4} (k-\frac{1}{k})[\sigma_3, \mu] = Q(x,t,k) \mu,
	\\
\mu_t + \frac{i}{4} (k+\frac{1}{k} )[\sigma_3, \mu] = Q(x,t,-k)\mu,
\end{cases} 
\end{align}
where $k \in \C$ is the spectral parameter, $\mu(x,t,k)$ is a $2\times 2$-matrix valued eigenfunction, and $Q$ is defined by
\begin{align}\label{Qdef}
  Q(x,t,k) = Q_0(x,t) + \frac{Q_1(x,t)}{k},
  \end{align}
with
$$  Q_0(x,t) = -\frac{i(u_x+u_t)}{4}\sigma_2, \qquad
 Q_1(x,t) = \frac{i\sin u}{4}\sigma_1 + \frac{i(\cos u-1)}{4}\sigma_3.$$
  
\subsection{Spectral functions}
Let $m \geq 1$, $n \geq 1$, $N_x \in \Z$, and $N_t \in \Z$ be integers.
Let $u_0(x)$, $u_1(x)$, $g_0(t)$, and $g_1(t)$ be functions satisfying the following regularity and decay assumptions (cf. (\ref{NxNtdef})):
\begin{subequations}\label{ujgjassump}
\begin{align}\label{ujassump}
\begin{cases}
\text{$(1+x)^{n} (u_0(x) - 2\pi N_x) \in L^1([0,\infty))$},
	\\
(1+x)^{n}\partial^i u_0(x) \in L^1([0,\infty)), \qquad  i = 1, \dots, m+2,
	\\
(1+x)^{n}\partial^i u_1(x) \in L^1([0,\infty)), \qquad  i = 0,1, \dots, m+1,
\end{cases}
\end{align}
and
\begin{align}\label{gjassump}
\begin{cases}
\text{$(1+t)^{n}(g_0(t) - 2\pi N_t) \in L^1([0,\infty))$},
	\\
(1+t)^{n}\partial^i g_0(t) \in L^1([0,\infty)), \qquad i = 1, \dots, m+2,
	\\
(1+t)^{n}\partial^i g_1(t) \in L^1([0,\infty)), \qquad i = 0,1, \dots, m+1.
\end{cases}
\end{align}
\end{subequations}
We define two spectral functions $\{a(k),b(k)\}$ in terms of the initial data $\{u_j(x)\}_0^1$ and two spectral functions $\{A(k),B(k)\}$ in terms of the boundary values $\{g_j(t)\}_0^1$ as follows \cite{F2002}:
Let $\mathsf{U}(x,k)$ and $\mathsf{V}(t,k)$ be given by (cf. definition (\ref{Qdef}) of $Q(x,t,k)$)
\begin{align}\label{mathsfUdef}
& \mathsf{U}(x,k) = -\frac{i(u_{0x}+u_1)}{4}\sigma_2 + \frac{i\sin u_0}{4k}\sigma_1 + \frac{i(\cos u_0-1)}{4k}\sigma_3,
	\\ \label{mathsfVdef}
& \mathsf{V}(t,k)= -\frac{i(g_1+g_{0t})}{4}\sigma_2 - \frac{i\sin g_0}{4k}\sigma_1 - \frac{i(\cos g_0-1)}{4k}\sigma_3,
\end{align}
and define the $2 \times 2$-matrix valued functions $X(x,k)$ and $T(t,k)$ as the unique solutions of the linear Volterra integral equations
\begin{align}\label{Xvolterra}
& X(x,k) = I + \int_{\infty}^x e^{i\theta_1(x'-x)\hat{\sigma}_3} (\mathsf{U}X)(x',k) dx',
	\\ \label{Tvolterra}
&  T(t,k) = I + \int_{\infty}^t e^{i\theta_2(t'-t)\hat{\sigma}_3} (\mathsf{V}T)(t',k) dt',
\end{align}
where $\theta_1$, $\theta_2$ are the functions in (\ref{thetadef}).
The symmetries
$$X(x,k) =\sigma_2 \overline{X(x, \bar{k})}\sigma_2, \qquad  T(t,k)=\sigma_2 \overline{T(t, \bar{k})}\sigma_2,$$
imply that we can define $a,b,A,B$ by
\begin{align}\label{abABdef}
s(k) = \begin{pmatrix}
\overline{a(\bar{k})} 	&	b(k)	\\
-\overline{b(\bar{k})}	&	a(k)
\end{pmatrix}, \qquad
S(k) = \begin{pmatrix}
\overline{A(\bar{k})} 	&	B(k)	\\
-\overline{B(\bar{k})}	&	A(k)
\end{pmatrix},
\end{align}
where $s(k) := X(0,k)$ and $S(k) := T(0,k)$. We will need several properties of the spectral functions $a,b,A,B$.
Proofs of the following lemmas can be found in \cite{HLNonlinearFourier}.

\begin{lemma}[Properties of $a(k)$ and $b(k)$]\label{ablemma}
Suppose $\{u_0(x), u_1(x)\}$ satisfy (\ref{ujassump}) for some integers $m \geq 1$, $n \geq 1$, and $N_x \in \Z$.
Then the spectral functions  $a(k)$ and $b(k)$ defined in (\ref{abABdef}) have the following properties:
\begin{enumerate}[$(a)$]
\item $a(k)$ and $b(k)$ are continuous for $\im k \geq 0$ and analytic for $\im k > 0$.

\item For $j = 1, \dots n$, the derivatives $a^{(j)}(k)$ and $b^{(j)}(k)$ are well-defined and continuous on $\bar{\C}_+ \setminus \{0\}$.

\item There exist complex constants $\{a_i, b_i, \hat{a}_i, \hat{b}_i\}_{i=1}^m$ such that the following expansions hold uniformly for each $j = 0, 1, \dots n$:
\begin{align*}\nonumber
\begin{pmatrix} b^{(j)}(k) \\ a^{(j)}(k) \end{pmatrix}
= \begin{pmatrix} \frac{d^j}{dk^j}\big(\frac{b_1}{k} + \cdots + \frac{b_m}{k^m}\big) \\
\frac{d^j}{dk^j}\big(1 + \frac{a_1}{k} + \cdots + \frac{a_m}{k^m}\big)
\end{pmatrix} + O\bigg(\frac{1}{k^{m+1}}\bigg),\quad k \to \infty, \ \im k \geq 0,
\end{align*}
and
\begin{align*}
\begin{pmatrix} b^{(j)}(k)  \\ a^{(j)}(k) \end{pmatrix}
= G_0(0) \begin{pmatrix} \frac{d^j}{dk^j}\big(\hat{b}_1k + \cdots + \hat{b}_mk^m\big) \\
\frac{d^j}{dk^j}\big(1 + \hat{a}_1k + \cdots + \hat{a}_mk^m\big)
\end{pmatrix} + O\big(k^{m+1-2j}\big),\quad k \to 0, \ \im k \geq 0,
\end{align*}
where the matrix $G_0(x)$ is defined by
\begin{align}\label{G0xdef}
G_0(x)= (-1)^{N_x}\begin{pmatrix}
  \cos \frac{u_0(x)}{2} & -\sin \frac{u_0(x)}{2} \\
  \sin \frac{u_0(x)}{2}  & \cos \frac{u_0(x)}{2}
 \end{pmatrix}.
\end{align}
 
\item $a(k) = \overline{a(-\bar{k})}$ and $b(k) = \overline{b(-\bar{k})}$ for $\im k \geq 0$.

\item $|a(k)|^2 +|b(k)|^2 = 1$ for $k \in \R$.

\end{enumerate}\end{lemma}

\begin{lemma}[Properties of $A(k)$ and $B(k)$]\label{ABlemma}
Suppose $\{g_0(t), g_1(t)\}$ satisfy (\ref{gjassump}) for some integers $m \geq 1$, $n \geq 1$, and $N_t \in \Z$.
Then the spectral functions  $A(k)$ and $B(k)$ defined in (\ref{abABdef}) have the following properties:
\begin{enumerate}[$(a)$]
\item $A(k)$ and $B(k)$ are continuous for $k \in \bar{D}_1 \cup \bar{D}_3$ and analytic for $k \in D_1 \cup D_3$.

\item For $j = 1, \dots n$, the derivatives $A^{(j)}(k)$ and $B^{(j)}(k)$ are well-defined and continuous on $(\bar{D}_1 \cup \bar{D}_3) \setminus \{0\}$.

\item There exist complex constants $\{A_i, B_i, \hat{A}_i, \hat{B}_i\}_{i=1}^m$ such that the following expansions hold uniformly for each $j = 0, 1, \dots n$:
\begin{align*}\nonumber
\begin{pmatrix} B^{(j)}(k) \\ A^{(j)}(k) \end{pmatrix}
= \begin{pmatrix} \frac{d^j}{dk^j}\big(\frac{B_1}{k} + \cdots + \frac{B_m}{k^m}\big)
\\
\frac{d^j}{dk^j}\big(1 + \frac{A_1}{k} + \cdots + \frac{A_m}{k^m}\big)
\end{pmatrix} + O\bigg(\frac{1}{k^{m+1}}\bigg),\quad k \to \infty, \ k \in \bar{D}_1 \cup \bar{D}_3,
\end{align*}
and
\begin{align*}
\begin{pmatrix} B^{(j)}(k) \\ A^{(j)}(k) \end{pmatrix}
= \mathcal{G}_0(0) \begin{pmatrix} \frac{d^j}{dk^j}\big(\hat{B}_1k + \cdots + \bar{B}_mk^m\big) \\
\frac{d^j}{dk^j}\big(1 + \hat{A}_1k + \cdots + \hat{A}_mk^m\big)
\end{pmatrix} + O\big(k^{m+1-2j}\big),\quad k \to 0, \ k \in \bar{D}_1 \cup \bar{D}_3,
\end{align*}
where the matrix $\mathcal{G}_0(t)$ is defined by
\begin{align}\label{G0tdef}
\mathcal{G}_0(t)= (-1)^{N_t} \begin{pmatrix}
  \cos \frac{g_0(t)}{2} & -\sin \frac{g_0(t)}{2} \\
  \sin \frac{g_0(t)}{2}  & \cos \frac{g_0(t)}{2}
 \end{pmatrix}.
\end{align}

\item $A(k) = \overline{A(-\bar{k})}$ and $B(k) = \overline{B(-\bar{k})}$ for $k \in \bar{D}_1 \cup \bar{D}_3$.

\item $A(k)\overline{A(\bar k)}+B(k)\overline{B(\bar k)}= 1$ for $k \in \R \cup \{|k| = 1\}$.

\end{enumerate}
\end{lemma}

We will also need the spectral functions $c(k)$ and $d(k)$ defined by
\begin{align}\label{cddef}
& c(k) = b(k)A(k) - a(k)B(k), \qquad k \in \bar{D}_1 \cup \R,
	\\
& d(k) = a(k)\overline{A(\bar{k})} +  b(k) \overline{B(\bar{k})}, \qquad k \in \bar{D}_2 \cup \R.
\end{align}
These definitions are motivated by the fact that
$$S(k)^{-1}s(k) = \begin{pmatrix}
\overline{d(\bar{k})} 	&	c(k)	\\
-\overline{c(\bar{k})}	&	d(k)
\end{pmatrix}, \qquad k \in \R.$$

\begin{definition}\upshape
We say that the initial and boundary values $u_0, u_1, g_0, g_1$ are {\it compatible with equation (\ref{sg}) to order $m$ at $x=t=0$} if all of the following relations which involve derivatives of at most $m$th order of $u_0, g_0$ and at most $(m-1)$th order of $u_1, g_1$ hold:
\begin{align*}
& g_0(0) = u_0(0), \quad g_0'(0) = u_1(0), \quad g_1(0) = u_0'(0), \quad g_1'(0) = u_1'(0), 
	\\
&g_0''(0) - u_0''(0) + \sin u_0(0) = 0, \quad g_1''(0) - u_0'''(0) + (\sin u_0)'(0) = 0, \quad \text{etc.}
\end{align*}
\end{definition}

\begin{lemma}[Properties of $c(k)$ and $d(k)$]\label{cdlemma}
Suppose $\{u_j(x), g_j(t)\}_0^1$ satisfy (\ref{ujgjassump}) for some integers $m \geq 1$, $n \geq 1$, $N_x \in \Z$, and $N_t \in \Z$. 
Suppose $\{u_j(x), g_j(t)\}_0^1$ are compatible with equation (\ref{sg}) to order $m+1$ at $x=t=0$.
Then the spectral functions  $c(k)$ and $d(k)$ defined in (\ref{cddef}) have the following properties:
\begin{enumerate}[$(a)$]
\item $c(k)$ is continuous for $k \in \bar{D}_1 \cup \R$ and analytic for $k \in D_1$.

\item $d(k)$ is continuous for $k \in \bar{D}_2 \cup \R$ and analytic for $k \in D_2$.

\item For $j = 1, \dots n$, the derivatives $c^{(j)}(k)$ and $d^{(j)}(k)$ are well-defined and continuous on $(\bar{D}_1 \cup \R) \setminus \{0\}$ and $(\bar{D}_2 \cup \R) \setminus \{0\}$, respectively.

\item There exist complex constants $\{d_i, \hat{d}_i\}_{i=1}^m$ such that the following expansions hold uniformly for each  $j = 0, 1, \dots, n$:
\begin{subequations}\label{cdexpansions}
\begin{align}\label{cdexpansionsa}
& c^{(j)}(k) = O\big(k^{-m-1}\big), \qquad k \to \infty, \ k \in \bar{D}_1,
	\\\label{cdexpansionsb}
& d^{(j)}(k) = \frac{d^j}{dk^j}\bigg(1+ \frac{d_1}{k} + \cdots + \frac{d_m}{k^m}\bigg) + O\big(k^{-m-1}\big), \qquad |k| \to \infty, \ k \in \R,	
\end{align}
and
\begin{align}\label{cdexpansionsc}
& c^{(j)}(k) = O\big(k^{m+1-2j}\big), \qquad |k| \to 0, \ k \in \R,
	\\\label{cdexpansionsd}
& d^{(j)}(k) = (-1)^{N_x - N_t}\frac{d^j}{dk^j}\big(1 + \hat{d}_1k + \cdots + \hat{d}_mk^m\big) + O\big(k^{m+1-2j}\big), \qquad k \to 0, \ k \in \bar{D}_2.	
\end{align}
\end{subequations}

\item $c(k) = \overline{c(-\bar{k})}$ for $k \in \bar{D}_1 \cup \R$ and $d(k) = \overline{d(-\bar{k})}$ for $k \in \bar{D}_2 \cup \R$.

\item $|c(k)|^2 + |d(k)|^2 = 1$ for $k \in \R$.

\end{enumerate}
\end{lemma}
\begin{proof}
See \cite{HLNonlinearFourier}.
\end{proof}

\subsection{Riemann-Hilbert problem}
As motivation for what follows, we briefly review the formulation of \cite{F2002} of a RH problem suitable for analyzing equation (\ref{sg}) in the quarter plane. 

Assume $u(x,t)$ is a sufficiently smooth\footnote{Since the discussion in the present subsection serves only as motivation for what follows, we will not be precise with regularity and decay assumptions.} solution of (\ref{sg}) for $x \geq 0$ and $t \geq 0$ corresponding to some initial and boundary values $u_0, u_1, g_0, g_1$ satisfying (\ref{ujgjassump}) with $m \geq 1$ and $n \geq 1$.
Define three eigenfunctions $\{\mu_j(x,t,k)\}_1^3$ of (\ref{lax}) as the solutions of the Volterra integral equations
\begin{align}\label{muVolterra}
\mu_j(x,t,k) = I + \int_{(x_j, t_j)}^{(x,t)} e^{i(\theta_1(x' - x) + \theta_2(t' - t))\hat{\sigma}_3}(Q\mu_j dx' + \tilde{Q}\mu_jdt')
\end{align}
where $\tilde{Q}(x,t,k) = Q(x,t,-k)$, $(x_1, t_1) = (0,\infty)$, $(x_2, t_2) = (0,0)$, and $(x_3,t_3) = (\infty, t)$.
Define the spectral functions $r_1(k)$ and $h(k)$ by
\begin{align}\label{r1hdef}
& r_1(k) = \frac{\overline{b(\bar{k})}}{a(k)}, \quad k \in \R; \qquad\quad
 h(k) = -\frac{ \overline{B(\bar{k})}}{a(k) d(k)}, \quad k \in \bar{D}_2 \cup \R,
\end{align}
and let $r(k)$ denote their sum given by
\begin{align}\label{rdef}
& r(k) = r_1(k) + h(k) = \frac{\overline{c(\bar{k})}}{d(k)}, \qquad k \in \R.
\end{align}

Then it can be shown \cite{F2002} that the sectionally meromorphic function $M(x,t,k)$ defined by
\begin{align}\label{originalMdef}
  M(x,t,k) = \begin{cases}
  \begin{pmatrix} \frac{[\mu_2(x,t,k)]_1}{a(k)} & [\mu_3(x,t,k)]_2 \end{pmatrix}, & k \in D_1,
   \vspace{.1cm}	 \\ 
  \begin{pmatrix} \frac{[\mu_1(x,t,k)]_1}{d(k)} & [\mu_3(x,t,k)]_2 \end{pmatrix}, & k \in D_2,
   \vspace{.1cm}	\\ 
  \begin{pmatrix} [\mu_3(x,t,k)]_1 & \frac{[\mu_1(x,t,k)]_2}{d^*(k)} \end{pmatrix}, & k \in D_3,
   \vspace{.1cm}	\\ 
  \begin{pmatrix} [\mu_3(x,t,k)]_1 & \frac{[\mu_2(x,t,k)]_2}{a^*(k)} \end{pmatrix}, & k \in D_4,
\end{cases}	
\end{align}
satisfies the normalization condition  $\lim_{k \to \infty} M(x,t,k) = I$ as well as the jump condition $M_+ = M_-J$ across the contour $\Gamma = \R \cup \{|k| = 1\}$, where the jump matrix $J$  is given by
\begin{align}\label{Jdef}
&J(x,t,k) = \begin{cases} \begin{pmatrix} 1 & 0 \\  -h(k) e^{2i\theta} & 1 \end{pmatrix}, & k \in \bar{D}_1 \cap \bar{D}_2,	\\
\begin{pmatrix} 1 & - \overline{r(\bar{k})} e^{-2i\theta} \\
- r(k)e^{2i\theta}& 1 + |r(k)|^2\end{pmatrix}, & k \in \bar{D}_2 \cap \bar{D}_3,	\\
\begin{pmatrix} 1 & - \overline{h(\bar{k})} e^{-2i\theta} \\ 0 & 1 \end{pmatrix}, & k \in \bar{D}_3 \cap \bar{D}_4, \\
\begin{pmatrix} 1 + |r_1(k)|^2 & \overline{r_1(\bar{k})} e^{-2i\theta} \\
 r_1(k)e^{2i\theta} & 1 \end{pmatrix}, & k \in \bar{D}_4 \cap \bar{D}_1,
 \end{cases}
\end{align}
with $\theta := \theta_1x + \theta_2 t$. In particular, if $a(k)$  and $d(k)$ have no zeros, then $M$ satisfies the RH problem
\begin{align}\label{RHM}
\begin{cases}
M(x, t, \cdot) \in I + \dot{E}^2(\C \setminus \Gamma),\\
M_+(x,t,k) = M_-(x, t, k) J(x, t, k) \quad \text{for a.e.} \ k \in \Gamma.
\end{cases}
\end{align}
Roughly speaking, the functions $r_1(k)$ and $r(k)$ play the roles of `reflection coefficients' for the initial half-line $\{x \geq 0, t = 0\}$ and for the union $\{x \geq 0, t = 0\} \cup \{x = 0, t \geq 0\}$ of the initial half-line and the boundary, respectively.

\section{Main results}\label{mainsec}
This section presents the four main theorems of the paper in the pure radiation case. The theorems are extended to the case when solitons are present in Section \ref{solitonsec}. 

\subsection{Construction of solutions}

\begin{assumption} \label{r1hassumption}
Suppose $r_1:\R \to \C$ and $h:\partial D_2 \to \C$ are continuous functions with the following properties:
\begin{enumerate}[$(a)$]
  \item There exist complex constants $\{r_{1,j}\}_{j=1}^2$ such that 
\begin{align}\label{r1expansion}
& r_1(k) = \frac{r_{1,1}}{k} + \frac{r_{1,2}}{k^2} + O\Big(\frac{1}{k^{3}}\Big), \qquad |k| \to \infty, \ k \in \R.
\end{align}

\item $r_1(k)$ and $h(k)$ obey the symmetries
\begin{align}\label{r1hsymm}
r_1(k) = \overline{r_1(-\bar{k})}, \qquad h(k) = \overline{h(-\bar{k})}.
\end{align}

 \item The function $r:[-1,1] \to \C$  defined by $r(k) = r_1(k) + h(k)$ satisfies $r(\pm1) = 0$ and $r(k) = O(k^3)$ as $k \to 0$.
\end{enumerate} 
\end{assumption}

\begin{theorem}[Construction of quarter-plane solutions]\label{existenceth}
Let $r_1:\R \to \C$ and $h:\partial D_2 \to \C$ be functions satisfying Assumption \ref{r1hassumption}. Define the jump matrix $J(x,t,k)$ by (\ref{Jdef}).
Then the RH problem (\ref{RHM})
has a unique solution for each $(x,t) \in [0,\infty) \times [0, \infty)$. Moreover, the nontangential limit
\begin{align}\label{m0def}
\hat{M}(x,t) := \ntlim_{k \to 0} M(x,t,k)
\end{align}
exists for each $(x,t) \in [0,\infty) \times [0, \infty)$ and the function $u(x,t)$ defined by 
\begin{align}\label{ulim}
  u(x,t) = 2 \arg\big(\hat{M}_{11}(x,t) + i \hat{M}_{21}(x,t)\big)
\end{align}
belongs to $C^2([0,\infty) \times [0, \infty), \R)$ and satisfies the sine-Gordon equation \eqref{sg} for $x \geq 0$ and $t \geq 0$. 
\end{theorem}
\begin{proof}
See Section \ref{existenceproofsec}.
\end{proof}

The definition (\ref{ulim}) of $u(x,t)$ deserves some comments. The function $\arg:\C \setminus \{0\} \to \R$ is multiple-valued, so in order to define $u(x,t)$ by (\ref{ulim}), we need to pick a branch for this function. The proof of Theorem \ref{existenceth} will show that the function $\hat{M}(x,t)$ defined in (\ref{m0def}) is continuous and that $\hat{M}_{11} + i \hat{M}_{21}$ takes values on the unit circle. 
When writing equation (\ref{ulim}), we assume that the branch of $\arg$ is chosen continuously so that $u(x,t)$ also becomes continuous. Then (\ref{ulim}) defines $u(x,t)$ uniquely up to the addition of an integer multiple of $4\pi$. This nonuniqueness of $u(x,t)$ is consistent with the fact that the sine-Gordon equation (\ref{sg}) is invariant under the shifts $u \to u + 2\pi j$, $j \in \Z$. 
For the purpose of dealing with boundary conditions of the form (\ref{NxNtdef}), it will be convenient to define $u(x,t)$ by
\begin{align}\label{ulim2}
  u(x,t) = 2 \arg\big(\hat{M}_{11}(x,t) + i \hat{M}_{21}(x,t)\big) + 2\pi j,
\end{align}
where the integer $j \in \Z$ is fixed by the requirement that $u(x,0) \to 2\pi N_x$ as $x \to \infty$.

\begin{assumption}\label{assumption1}
Assume the following conditions are satisfied:
\begin{itemize}
\item The spectral functions $a(k)$, $b(k)$, $A(k)$, $B(k)$ defined in (\ref{abABdef}) satisfy the global relation 
\begin{align}\label{GR}
& A(k)b(k) - B(k)a(k) = 0, \qquad k \in \bar{D}_1.
\end{align}

\item $a(k)$ and $d(k)$ have no zeros in $\bar{D}_1 \cup \bar{D}_2$ and $\bar{D}_2$, respectively.

\end{itemize}
\end{assumption}

\begin{theorem}[Construction of quarter-plane solutions with given initial and boundary values]\label{existenceth2}
Let $u_0, u_1, g_0, g_1$ be functions satisfying (\ref{ujgjassump}) for $n = 1$, $m = 2$, and some integers $N_x, N_t \in \Z$. Suppose also that $u_0, u_1, g_0, g_1$ are compatible with equation (\ref{sg}) to third order at $x=t=0$. Define the spectral functions $a(k)$, $b(k)$, $A(k)$, $B(k)$ by (\ref{abABdef}) and suppose Assumption \ref{assumption1} holds. 

Then the spectral functions $r_1(k)$  and $h(k)$ defined by (\ref{r1hdef}) satisfy Assumption \ref{r1hassumption}. Moreover, the associated solution $u(x,t)$ of sine-Gordon defined by (\ref{ulim2}) has, for an appropriate choice of the integer $j \in \Z$, initial and boundary values given by $u_0, u_1, g_0, g_1$, i.e.,
\begin{itemize}
\item $u(x,0) = u_0(x)$ and $u_t(x,0)=u_1(x)$ for $x \geq 0$,

\item $u(0,t) = g_0(t)$ and $u_x(0,t) = g_1(t)$ for $t \geq 0$.
\end{itemize}
\end{theorem}
\begin{proof}
See Section \ref{existence2proofsec}.
\end{proof}

\subsection{Asymptotics}
If $r_1:\R \to \C$ and $h:\partial D_2 \to \C$ are two functions fulfilling the assumptions of Theorem \ref{existenceth}, then the function $u(x,t)$ defined in (\ref{ulim}) is a quarter-plane solution of sine-Gordon. Our next result provides asymptotic formulas for this solution as $(x,t) \to \infty$. As displayed in Figure \ref{sectors.pdf}, there are four different asymptotic sectors. The proof requires slightly stronger assumptions on the functions $r_1$ and $h$ than before. For simplicity, we will work in the smooth setting and assume the following.

\begin{assumption} \label{r1hassumption2}
Suppose $r_1:\R \to \C$ and $h:\bar{D}_2 \to \C$ have the following properties:
\begin{enumerate}[$(1)$]
 \item $r_1 \in C^\infty(\R)$. 
 
 \item There exist complex constants $\{r_{1,j}\}_{j=1}^\infty$ such that, for each integer $N \geq 1$, 
\begin{subequations}\label{r1expansion2}
\begin{align}
& r_1(k) = \sum_{j=1}^N \frac{r_{1,j}}{k^j} + O\Big(\frac{1}{k^{N+1}}\Big), \qquad |k| \to \infty, \ k \in \R,
\end{align}
and this expansion can be differentiated termwise any number of times in the sense that
\begin{align}
\frac{d^nr_1}{dk^n}(k) = \frac{d^n}{dk^n}\bigg(\sum_{j=1}^N \frac{r_{1,j}}{k^j}\bigg) + O\bigg(\frac{1}{k^{N+1}}\bigg), \qquad |k| \to \infty, \ k \in \R,
\end{align}
\end{subequations}
for any integers $n, N \geq 1$.

 \item $h \in C^\infty(\bar{D}_2)$ and $h(k)$ is analytic for $k \in D_2$.

\item $r_1(k)$ and $h(k)$ obey the symmetries (\ref{r1hsymm}).

\item The function $r:[-1,1] \to \C$  defined by $r(k) = r_1(k) + h(k)$ vanishes at $k = \pm 1$ and vanishes to all orders at $k = 0$, i.e., $r(1) = r(-1) = 0$ and $r^{(j)}(0) = 0$ for every integer $j \geq 0$.
\end{enumerate} 
\end{assumption}

\begin{theorem}[Asymptotics of quarter-plane solutions]\label{asymptoticsth}
  Let $r_1:\R \to \C$ and $h:\bar{D}_2 \to \C$ be two functions satisfying Assumption \ref{r1hassumption2} and let $u(x,t)$ be the associated sine-Gordon quarter-plane solution defined in (\ref{ulim}). 
  
Then $u \in C^\infty([0,\infty) \times [0,\infty), \R)$ and there exists a choice of the branch of $\arg$ in (\ref{ulim}) such that $u(x,t) \to 0$ uniformly as $(x,t) \to \infty$. 
In fact, for this choice of branch, the following asymptotic formulas hold for any $N \geq 1$:
\begin{subequations}\label{uasymptotics}
\begin{align} \label{uasymptoticsI}
\text{\upshape Sector I:} \quad & u(x,t) = O\big(x^{-N}\big), \qquad 1 \leq \zeta < \infty, \ x > 1,	
	\\\label{uasymptoticsII}
\text{\upshape Sector II:} \quad & u(x,t) = O\big((1-\zeta)^N + t^{-N}\big), \qquad 1/2 \leq \zeta \leq 1, \ t > 1,
	\\\label{uasymptoticsIII}
\text{\upshape Sector III:} \quad & u(x,t) = -2 \sqrt{\frac{2(1 + k_0^2)\nu}{k_0t}}   \sin \alpha+ O\bigg(\frac{1}{k_0 t} + \frac{\ln{t}}{t}\bigg), \qquad 0 \leq \zeta < 1, \ t > 2,
	\\\label{uasymptoticsIV}
\text{\upshape Sector IV:} \quad & u(x,t) = O\big(\zeta^2 + t^{-1}\big), \qquad 0 \leq \zeta \leq 1/2, \ t > 1,
\end{align}
\end{subequations}
where the error terms are uniform in the given ranges, the constant $c> 0$ is independent of $(x,t)$, the variable $\zeta \geq 0$ is defined by $\zeta := x/t$, and the functions $k_0 := k_0(\zeta)$, $\nu := \nu(\zeta)$, and $\alpha := \alpha(\zeta,t)$ are defined by
\begin{align}\nonumber
& k_0 = \sqrt{\frac{1-\zeta}{1+\zeta}}, \qquad \nu = \frac{1}{2\pi} \ln(1 + |r(k_0)|^2),
	\\ \nonumber
& \alpha(\zeta, t) = \frac{2tk_0}{1 + k_0^2} - \nu\ln\bigg(\frac{8tk_0}{1 + k_0^2}\bigg) 
 - \frac{\pi}{4} + \arg r(k_0) + \arg\Gamma(i\nu)
 	\\ \label{alphadef}
& \qquad\qquad - \frac{1}{\pi} \int_{-k_0}^{k_0} \ln (k_0-s) \frac{d\ln(1 + |r(s)|^2)}{ds} ds.
\end{align}

If $r(k)$ vanishes to some higher order $Z \geq 1$ at $k = 1$, i.e., $r^{(j)}(1) = 0$ for $j = 0, \dots, Z$, then (\ref{uasymptoticsIV}) can be replaced with
\begin{align}\label{uasymptoticsIV2}
\text{\upshape Sector IV:} \quad  u(x,t) = O\big(\zeta^{Z+2} + t^{-\frac{Z}{2}-1}\big), \qquad 0 \leq \zeta \leq 1/2, \ t > 1.
\end{align}
\end{theorem}
\begin{proof}
See Sections \ref{overviewsec}-\ref{sectorIVsec}.
\end{proof}

Our last theorem provides the asymptotics as $(x,t) \to \infty$ of the solution $u(x,t)$ of the sine-Gordon equation in the quarter-plane for consistently given initial and boundary values such that $u_0 - 2\pi N_x$, $u_1$, $g_0 - 2\pi N_t$, and $g_1$ belong to the Schwartz class $\mathcal{S}([0,\infty))$.

\begin{theorem}[Asymptotics of quarter-plane solutions with given initial and boundary values]\label{asymptoticsth2}
Let $u_0, u_1, g_0, g_1$ be functions satisfying 
\begin{align}\label{ujgjschwartz}
u_0 - 2\pi N_x, u_1, g_0 - 2\pi N_t, g_1 \in \mathcal{S}([0,\infty)).
\end{align}
for some integers $N_x, N_t \in \Z$. Suppose $u_0, u_1, g_0, g_1$ are compatible with equation (\ref{sg}) to all orders at $x=t=0$. Define $a(k)$, $b(k)$, $A(k)$, $B(k)$ by (\ref{abABdef}) and suppose Assumption \ref{assumption1} holds. Define $r_1(k)$ and $h(k)$ by (\ref{r1hdef}) and let $M(x,t,k)$ be the associated solution of the RH problem (\ref{RHM}). 
 
Then, for an appropriate choice of the integer $j \in \Z$, the function $u(x,t)$ defined in (\ref{ulim2}) is a smooth solution of (\ref{sg}) in the quarter-plane $\{x \geq 0, t \geq 0\}$ such that  
\begin{align}\label{initialboundaryvalues}
\begin{cases}
u(x,0)=u_0(x), \quad u_t(x,0)=u_1(x),  & x \geq 0, 
	\\
u(0,t)=g_0(t), \quad u_x(0,t)=g_1(t), & t \geq 0,
\end{cases}
\end{align}
and $u(x,t)$ satisfies the asymptotic formulas (\ref{uasymptotics}) and (\ref{uasymptoticsIV2}) for any $N \geq 1$ and $Z \geq 1$.
\end{theorem}
\begin{proof}
Suppose $u_0, u_1, g_0, g_1$ lie in the Schwartz class $\mathcal{S}([0,\infty))$. Then Lemmas \ref{ablemma} and \ref{ABlemma} together with Assumption \ref{assumption1} imply that $r_1(k)$ and $h(k)$ satisfy Assumption \ref{r1hassumption2}.
The global relation (\ref{GR}) implies that $c(k)$, and hence also $r(k)$, vanishes to all orders at $k = \pm1$.   
Thus, by Theorem \ref{asymptoticsth}, $u(x,t)$ is a smooth solution which satisfies the asymptotic formulas (\ref{uasymptotics}) and (\ref{uasymptoticsIV2}) for any $N \geq 1$ and $Z \geq 1$. By Theorem \ref{existenceth2}, $u(x,t)$ satisfies the initial and boundary values (\ref{initialboundaryvalues}).
\end{proof}

\begin{remark}[Matching]\upshape
  The asymptotic formulas (\ref{uasymptoticsI}) and (\ref{uasymptoticsII}) overlap on the line $\zeta = 1$ (i.e. $x = t$) and they both state that $u(x,t) = O(t^{-N})$ on this line. 
  
  The two formulas (\ref{uasymptoticsII}) and (\ref{uasymptoticsIII}) are both uniformly valid for $(x,t)$ such that $0 \leq \zeta < 1$ and $t > 2$. 
Formula (\ref{uasymptoticsII}) is useful if $\zeta \to 1$ while formula (\ref{uasymptoticsIII}) is useful if $k_0^{3/2}t \to \infty$ (because then the errors are small). The two formulas are consistent in the overlap since the amplitude of the modulated sine wave in (\ref{uasymptoticsIII}) vanishes as $\zeta \to 1$. Indeed, the function $\nu$ vanishes to all orders as $\zeta \to 1$ as a consequence of the fact that $r(k)$ vanishes to all orders at $k = 0$.

Thanks to the vanishing of $r(k)$  at $k=1$, the amplitude of the sine wave in (\ref{uasymptoticsIII}) vanishes also as $\zeta \to 0$. This shows that the asymptotics in (\ref{uasymptoticsIII}) is consistent with the estimate in (\ref{uasymptoticsIV}), which in turn is consistent with the decay of the boundary function $g_0(t)$ as $t \to \infty$.
\end{remark}

\section{Proof of Theorem \ref{existenceth}}\label{existenceproofsec}
Suppose $r_1:\R \to \C$ and $h:\partial D_2 \to \C$ are continuous functions satisfying Assumption \ref{r1hassumption} and let $J(x,t,k)$ be the jump matrix defined in (\ref{Jdef}). 
The matrix $J$ could have discontinuities at the points $k = \pm 1$. Moreover, in general, $J(x,t,k)$ only decays like $O(k^{-1})$ as $k \in \Gamma$ tends to infinity. Our first step is therefore to transform the RH problem (\ref{RHM}) in such a way that the new jump matrix is continuous on $\Gamma$ and of order $O(k^{-3})$ as $k \to \infty$. 

Define $m(x,t,k)$ by
\begin{align}\label{mxtkdef}
m(x,t,k)
= \begin{cases} M(x,t,k)  \begin{pmatrix} 1 & 0 \\ -r_{1,a}(k) e^{2i\theta} & 1 \end{pmatrix}, & k \in D_1,
	\\
M(x,t,k)  \begin{pmatrix} 1 & \overline{r_{1,a}(\bar{k})} e^{-2i\theta} \\ 0 & 1 \end{pmatrix}, & k \in D_4,
	\\
M(x,t,k), & \text{otherwise},		
\end{cases}
\end{align}
where $r_{1,a}(k)$ is a rational function with the following properties:
\begin{itemize}
\item $r_{1,a}(k)$ has no poles in $\bar{D}_1$,

\item $r_{1,a}(k) = \frac{r_{1,1}}{k} + \frac{r_{1,2}}{k^2} + O(k^{-3})$ as $k \to \infty$, 

\item $r_{1,a}(1) = r_1(1)$,

\item $r_{1,a}(k) = \overline{r_{1,a}(-\bar{k})}$ for $k \in \bar{D}_1$.
\end{itemize}
It is easy to see that a function with the above properties exists.

Note that $\im \theta \geq 0$ for $k \in D_1$ and $\im \theta \leq 0$ for $k \in D_4$ whenever $x \geq 0$ and $t \geq 0$. Together with the properties of $r_{1,a}$, this implies that $r_{1,a} e^{2i\theta} \in (\dot{E}^2 \cap E^\infty)(D_1)$ and $r_{1,a}^* e^{-2i\theta} \in (\dot{E}^2 \cap E^\infty)(D_4)$.
It follows that (see Lemma A.5 in \cite{LNonlinearFourier}) that $M$ satisfies the RH problem (\ref{RHM}) iff $m$ satisfies the RH problem
\begin{align}\label{RHm}
\begin{cases}
m(x, t, \cdot) \in I + \dot{E}^2(\C \setminus \Gamma),\\
m_+(x,t,k) = m_-(x, t, k) v(x, t, k) \quad \text{for a.e.} \ k \in \Gamma,
\end{cases}
\end{align}
where
\begin{align}\label{vdef}
&v(x,t,k) = \begin{cases} \begin{pmatrix} 1 & 0 \\  -(h + r_{1,a}) e^{2i\theta} & 1 \end{pmatrix}, & k \in \bar{D}_1 \cap \bar{D}_2,	\\
\begin{pmatrix} 1 & - r^* e^{-2i\theta} \\
- r e^{2i\theta}& 1 + |r|^2\end{pmatrix}, & k \in \bar{D}_2 \cap \bar{D}_3,	\\
\begin{pmatrix} 1 & - (h^* + r_{1,a}^*) e^{-2i\theta} \\ 0 & 1 \end{pmatrix}, & k \in \bar{D}_3 \cap \bar{D}_4, \\
\begin{pmatrix} 1 + |r_1 - r_{1,a}|^2 & (r_1^* - r_{1,a}^*) e^{-2i\theta} \\
 (r_1 - r_{1,a})e^{2i\theta} & 1 \end{pmatrix}, & k \in \bar{D}_4 \cap \bar{D}_1,
 \end{cases}
\end{align}
and we recall that $r:[-1,1] \to \C$  is defined by $r(k) = r_1(k) + h(k)$.

The jump matrix $v(x,t,k)$ is continuous at $k = \pm 1$ with $v(x,t,\pm 1) = I$ as a consequence of the relations $r_{1,a}(\pm 1) = r_1(\pm 1)$ and the assumption $r(\pm 1) = 0$.
Moreover,
\begin{align}\label{r1r1ainfty}
r_1(k) - r_{1,a}(k) = O(k^{-3}), \qquad |k| \to \infty, \ k \in \R.
\end{align}
The symmetries (\ref{r1hsymm}) imply that $v$ satisfies
\begin{align}\label{vsymm}
v(x,t,k)&= \overline{v(x, t, \bar{k})}^T = v(x,t,-k)^T, \qquad k \in \Gamma.
\end{align}

\begin{lemma}[Vanishing Lemma]\label{vanishinglemma}
Fix $(x,t) \in [0, \infty) \times [0,\infty)$. If $\mathcal{N}(k)$ is a solution of the homogeneous RH problem determined by $(\Gamma, v(x,t,\cdot))$, i.e., if
\begin{align}\label{homogenousRH}
\begin{cases}
\mathcal{N} \in \dot{E}^2(\C \setminus \Gamma),\\
\mathcal{N}_+(k) = \mathcal{N}_-(k) v(x, t, k) \quad \text{for a.e.} \ k \in \Gamma,
\end{cases}
\end{align}
then $\mathcal{N}(k)$ vanishes identically.
\end{lemma}
\begin{proof}
Define the sectionally analytic function $\mathcal{G}(k)$ by
$$\mathcal{G}(k) = \mathcal{N}(k)\overline{\mathcal{N}(\bar{k})}^T, \qquad k \in \C \setminus \Gamma.$$
Since $\mathcal{N} \in \dot{E}^2(\C \setminus \Gamma)$, we have $\mathcal{G} \in \dot{E}^1(\C \setminus \Gamma)$.
The first symmetry in (\ref{vsymm}) together with the identities
\begin{align*}
  \mathcal{G}_+(k)&=\mathcal{N}_-(k)v(x,t,k) \overline{\mathcal{N}_-(\bar{k})}^T \quad \text{for a.e.} \ k \in \Gamma, \\
  \mathcal{G}_-(k)&=\mathcal{N}_-(k)\overline{v(x,t,\bar k)}^T \overline{\mathcal{N}_-(\bar{k})}^T \quad \text{for a.e.} \ k \in \Gamma,
\end{align*}
shows that $\mathcal{G}_+(k)=\mathcal{G}_-(k)$ for a.e. $k \in \Gamma$. 
But the fact that $\mathcal{G} \in \dot{E}^1(\C \setminus \Gamma)$ implies (see Theorem 4.1 in \cite{LCarleson})
$$\mathcal{G}(k) = [\mathcal{C}(\mathcal{G}_+ - \mathcal{G}_-)](k), \qquad k \in \C \setminus \Gamma.$$
Hence $\mathcal{G}$ vanishes identically. In particular, by evaluating $\mathcal{G}_+$ on the open interval $(-1,1) \subset \R$, we find
$$\mathcal{N}_-(k)\overline{v(x,t,k)}^T \overline{\mathcal{N}_-(k)}^T = 0 \quad \text{for a.e.} \ k \in (-1,1).$$
But the matrix $v(x,t,k)$ is positive definite for each $k \in (-1,1)$, because it is a Hermitian matrix with unit determinant and strictly positive $(11)$ entry. We deduce that $\mathcal{N}_-(k)=0$ for a.e. $k \in (-1,1)$, which means that $\mathcal{N}_+ = \mathcal{N}_- v$ also vanishes a.e. on $(-1,1)$.
Let $B_r$ be an open disk centered at some point in $(-1,1)$ and contained in the open unit disk. The function $\tilde{\mathcal{N}}(k)$ defined by
$$\tilde{\mathcal{N}}(k) = \frac{1}{2\pi i} \int_{\partial B_r} \frac{\mathcal{N}(s)}{s-k}ds$$
is analytic in $B_r$ and equals $\mathcal{N}(k)$ for $k \in B_r\setminus \R$. Indeed, if $k \in B_r \cap \C_+$, 
$$\tilde{\mathcal{N}}(k) = \frac{1}{2\pi i} \int_{\partial (B_r \cap \C_+)} \frac{\mathcal{N}_+(s)}{s-k}ds + \frac{1}{2\pi i} \int_{\partial (B_r \cap \C_-)} \frac{\mathcal{N}_-(s)}{s-k}ds
= \mathcal{N}(k) + 0,$$
and a similar argument applies if $k \in B_r \cap \C_-$. We infer that $\mathcal{N}(k)$ is analytic in $B_r$. Since $\mathcal{N} = 0$ on $B_r \cap \R$, it follows by analytic continuation that $\mathcal{N}$ vanishes identically for $|k| < 1$. But this means that $\mathcal{N}_\pm = 0$ on the unit circle, so by a similar argument we find that $\mathcal{N}$ vanishes also for $|k| \geq 1$.
\end{proof}

\begin{remark}\upshape
  We could have used either of the intervals $(-\infty, -1)$ and $(1,\infty)$ instead of $(-1,1)$ in the proof of Lemma \ref{vanishinglemma}. Indeed, the jump matrix $v(x,t,k)$ is positive definite also on these intervals. 
\end{remark}

The vanishing result of Lemma \ref{vanishinglemma} together with Fredholm theory can be used to infer the existence of a unique solution of the RH problem (\ref{RHm}). 
To see this, we define the matrix-valued functions $w^\pm(x,t,k)$ by
$$w^- = \begin{cases}
\begin{pmatrix} 0 & 0 \\ -(h + r_{1,a}) e^{2i\theta} & 0 \end{pmatrix}, & k \in \bar{D}_1 \cap \bar{D}_2,
	\\
\begin{pmatrix} 0 & 0 \\
- r(k)e^{2i\theta} & 0\end{pmatrix}, & k \in \bar{D}_2 \cap \bar{D}_3,
	\\
0, & k \in \bar{D}_3 \cap \bar{D}_4,
	\\
\begin{pmatrix} 0 & (r_1^* - r_{1,a}^*) e^{-2i\theta}  \\ 0 & 0 \end{pmatrix},  & k \in \bar{D}_4 \cap \bar{D}_1,
\end{cases}
$$
and
$$w^+ = \begin{cases} 0, & k \in \bar{D}_1 \cap \bar{D}_2,
	\\
\begin{pmatrix} 0 & - r^* e^{-2i\theta}\\
0& 0\end{pmatrix}, & k \in \bar{D}_2 \cap \bar{D}_3,
	\\
\begin{pmatrix} 0 & - (h^* + r_{1,a}^*) e^{-2i\theta} \\ 0 & 0 \end{pmatrix}, & k \in \bar{D}_3 \cap \bar{D}_4,
	\\
\begin{pmatrix} 0 & 0 \\ (r_1-r_{1,a}) e^{2i\theta} & 0 \end{pmatrix}, & k \in \bar{D}_4 \cap \bar{D}_1.
\end{cases}
$$
We also define the operator $\mathcal{C}_{w}: L^2(\Gamma) + L^\infty(\Gamma) \to L^2(\Gamma)$ by
\begin{align}\label{Cwdef}
\mathcal{C}_{w}(f) = \mathcal{C}_+(f w^-) + \mathcal{C}_-(f w^+).
\end{align}

\begin{lemma}[Existence of solution of RH problem]\label{existencemlemma}
For each  $(x,t) \in [0,\infty) \times [0,\infty)$, the RH problem (\ref{RHm}) has a unique solution $m(x,t,\cdot) \in I + \dot{E}^2(\C \setminus \Gamma)$. Moreover, this solution has unit determinant and admits the representation $m = I + \mathcal{C}(\mu (w^+ + w^-))$, that is,
\begin{align}\label{m1representation}
m(x,t,k) = I + \frac{1}{2\pi i}\int_\Gamma \frac{(\mu (w^+ + w^-))(x,t,s)}{s-k}ds, \qquad k \in \C \setminus \Gamma,
\end{align}
where the function $\mu(x,t,k) \in I + L^2(\Gamma)$  is defined by
\begin{align}\label{mudef}
  \mu = I + (I - \mathcal{C}_w)^{-1}\mathcal{C}_w I.
\end{align}
\end{lemma}
\begin{proof}
Note that $v = (v^-)^{-1}v^+$, where $v^+ = I + w^+$, $v^- = I - w^-$, and the matrices $w^\pm$ are nilpotent. For each $(x,t) \in [0,\infty) \times [0,\infty)$, we have $v^\pm \in C(\Gamma)$ (because $r(0) = r(\pm1) = 0$) and $v^{\pm}, (v^{\pm})^{-1} \in I + L^2(\Gamma) \cap L^\infty(\Gamma)$. 
Hence Lemma \ref{vanishinglemma} together with a Fredholm argument (see e.g. Lemmas A.1 and A.2 in \cite{LNonlinearFourier}) implies that the operator $I - \mathcal{C}_w \in \mathcal{B}(L^2(\Gamma))$ is bijective and that the RH problem (\ref{RHm}) has a unique solution $m(x,t,\cdot)$ given by the stated formula. Since $\det v = 1$, we have $\det m = 1$.
\end{proof}

\begin{lemma}\label{C2lemma}
The following maps are $C^2$:
\begin{subequations}\label{xtww}
\begin{align}\label{xtwwa}
& (x,t) \mapsto (w^+(x,t,\cdot), w^-(x,t,\cdot)):[0, \infty) \times [0, \infty) \to L^2(\Gamma) \times L^2(\Gamma),
	\\ \label{xtwwb}
& (x,t) \mapsto (w^+(x,t,\cdot), w^-(x,t,\cdot)):[0, \infty) \times [0, \infty) \to L^\infty(\Gamma) \times L^\infty(\Gamma),
	\\ \label{xttomuI}
& (x,t) \mapsto \mu(x,t, \cdot) - I: [0,\infty) \times [0, \infty) \to L^2(\Gamma).
\end{align}
\end{subequations}
\end{lemma}
\begin{proof}
Since $2i\theta = \frac{i}{2}(k - \frac{1}{k})x + \frac{i}{2}(k + \frac{1}{k})t$, differentiation of $e^{\pm 2i\theta}$ with respect to $x$ and $t$ generates factors of $k \pm k^{-1}$, which are singular at $k = 0$ and $k = \infty$. However, by assumption $r(k) = O(k^3)$ as $k \to 0$ and, by (\ref{r1r1ainfty}), $r_1 - r_{1,a}$ is of $O(k^{-3})$ as $k \to \infty$. Thus, using the dominated convergence theorem, we find that the maps (\ref{xtwwa}) and (\ref{xtwwb}) are $C^2$.
 Furthermore, the map
\begin{align}
(w^+, w^-) \mapsto I - \mathcal{C}_w: L^\infty(\Gamma) \times L^\infty(\Gamma) \to \mathcal{B}(L^2(\Gamma))
\end{align}
is smooth by the estimate
\begin{align*}
\|\mathcal{C}_w\|_{\mathcal{B}(L^2(\Gamma))} \leq C \max\big\{\|w^+\|_{L^\infty(\Gamma)}, \|w^-\|_{L^\infty(\Gamma)} \big\},
\end{align*}
and the bilinear map
\begin{align}\label{wwCwI}
(w^+, w^-) \mapsto \mathcal{C}_wI:  L^2(\Gamma) \times L^2(\Gamma) \to L^2(\Gamma)
\end{align}
is smooth by the estimate
$$\|\mathcal{C}_wI\|_{L^2(\Gamma)} \leq C \max\big\{\|w^+\|_{L^2(\Gamma)}, \|w^-\|_{L^2(\Gamma)} \big\}.$$
By the open mapping theorem, $(I - \mathcal{C}_w)^{-1} \in \mathcal{B}(L^2(\Gamma))$ for each  $(x,t)$.
Since (\ref{xttomuI}) can be viewed as a composition of maps of the form (\ref{xtww})-(\ref{wwCwI}) together with the smooth inversion map $I - \mathcal{C}_w \mapsto (I - \mathcal{C}_w)^{-1}$, it follows that the map (\ref{xttomuI}) also is $C^2$.
\end{proof}

\begin{lemma}[Definition of $\hat{M}(x,t)$]\label{m0lemma}
The nontangential limit in (\ref{m0def}) exists for each $(x,t) \in [0,\infty) \times [0, \infty)$. Moreover, the entries of the function $\hat{M}(x,t)$ defined by (\ref{m0def}) are $C^2$ and satisfy the relations
\begin{align}\label{m0entries}
\hat{M}_{11} = \hat{M}_{22}, \qquad  \hat{M}_{21} = -\hat{M}_{12}, \qquad |\hat{M}_{11} + i \hat{M}_{21}| = 1,
\end{align}
for all $(x,t)$.
\end{lemma}
\begin{proof}
Let $W_0$ be a nontangential sector at $0$ with respect to $\Gamma$, that is, $W_0 = \{k \in \C \, | \, \alpha \leq \arg k \leq \beta, 0 < |k| < \epsilon\}$, where $\alpha, \beta, \epsilon$ are such that $W_0 \subset \C \setminus \Gamma$. 
Fixing $(x,t) \in [0, \infty) \times [0,\infty)$, we write (\ref{m1representation}) as
\begin{align*}
m(x,t,k)
 = &\; I + \frac{1}{2\pi i}\int_\Gamma (\mu(w^+ + w^-))(x,t,s) \bigg(\frac{1}{s} + \frac{k}{s^2} + \frac{k^2}{s^2(s - k)}\bigg) ds
 	\\
= &\; \hat{M}(x,t) + k E(x,t,k), \qquad k \in \C \setminus \Gamma,
\end{align*}
where $\hat{M}$ and $E$ are given by
\begin{align}\label{m0representation}
\hat{M}(x,t) = I + \int_\Gamma \frac{(\mu(w^+ + w^-))(x,t,s)}{s}ds
\end{align}
and
\begin{align*}
E(x,t,k) = &\; \frac{1}{2\pi i}\int_\Gamma ((\mu-I)(w^+ + w^-))(x,t,s) \bigg(\frac{1}{s^2} + \frac{k}{s^2(s - k)}\bigg) ds
	\\
& + \frac{1}{2\pi i}\int_\Gamma (w^+ + w^-)(x,t,s) \bigg(\frac{1}{s^2} + \frac{k}{s(s - k)}\bigg) ds.
\end{align*}
By Lemma \ref{C2lemma}, we have $\mu(x,t, \cdot) - I \in L^2(\Gamma)$. 
Moreover, since $r(k) = O(k^3)$ as $k \to 0$, the $L^1(\Gamma)$ and $L^2(\Gamma)$ norms of $s^{-j}(w^+ + w^-)(x,t,s)$, $j = 1,2$, and of $\frac{k}{s(s - k)}(w^+ + w^-)(x,t,s)$ are bounded (the latter uniformly with respect to all small $k \in W_0$). This shows that the integral defining $\hat{M}$ converges and that $|E(x,t,k)| < C$ for all small $k \in W_0$. Consequently,
\begin{align}\label{matzero}
m(x,t,k) = \hat{M}(x,t) + O(k), \qquad k \to 0, \ k \in W_0,
\end{align}
where the error term is uniform with respect to $k \in W_0$. Since $\ntlim_{k \to 0} M(x,t,k) = \ntlim_{k \to 0} m(x,t,k)$, this shows that the nontangential limit in (\ref{m0def}) exists and is given by (\ref{m0representation}). 

We next show that $\hat{M}$ is $C^2$. 
Using again that $r(k) = O(k^3)$ as $k \to 0$, the same type of argument that proved (\ref{xtwwa}) shows that the maps
\begin{align*}
& (x,t) \mapsto \bigg(s \mapsto \frac{w^\pm(x,t,s)}{s}\bigg):[0, \infty) \times [0, \infty) \to L^p(\Gamma)\end{align*}
are $C^2$ for $p =1,2$. Since the map (\ref{xttomuI}) is $C^2$, we conclude that $\hat{M}$ also is $C^2$.

It remains to prove (\ref{m0entries}). 
The symmetries (\ref{vsymm}) show that the functions
$$\sigma_2 \overline{m(x,t,\bar{k})} \sigma_2 = (\overline{m(x,t,\bar{k})}^T)^{-1} \quad \text{and}\quad \sigma_2 m(x,t,-k) \sigma_2 = (m(x,t,-k)^T)^{-1}$$ 
satisfy the same RH problem as $m(x,t,k)$. By uniqueness, it follows that $m$ obeys the symmetries
\begin{align}\label{msymm}
m(x,t,k) = \sigma_2 \overline{m(x,t,\bar{k})} \sigma_2 = \sigma_2 m(x,t,-k) \sigma_2, \qquad k \in \C \setminus \Gamma.
\end{align} 
Since $m_+(x,t,0) = m_-(x,t,0) = \hat{M}(x,t)$, taking the limit $k \to 0$ in the symmetries (\ref{msymm}) yields
$$\hat{M}(x,t) =  \sigma_2 \overline{\hat{M}(x,t)} \sigma_2 =  \sigma_2 \hat{M}(x,t) \sigma_2.$$
That is, the entries $\hat{M}_{ij}$, $i,j = 1,2$, of $\hat{M}$ are real and satisfy
$$\hat{M}_{11} = \hat{M}_{22}, \qquad  \hat{M}_{21} = -\hat{M}_{12}.$$
The relation $\det \hat{M}(x,t) = 1$ then gives $\hat{M}_{11}^2 + \hat{M}_{21}^2 = 1$. This proves (\ref{m0entries}).
\end{proof}

We define $u(x,t)$ by equation (\ref{ulim}), where the branch of $\arg$ is chosen so that $u(x,t)$ is continuous. This defines $u(x,t)$ uniquely up to the addition of a constant $4\pi n$, $n \in \Z$. In what follows, we assume that a choice for this irrelevant constant has been made. 

\begin{lemma}[Properties of $u(x,t)$]\label{propertiesulemma}
The function $u(x,t)$ is $C^2$ from $[0,\infty) \times [0,\infty)$ to $\R$. Moreover,
\begin{align}\label{hatmuxt}
\hat{M}(x,t) = \begin{pmatrix} \cos\frac{u(x,t)}{2} & -\sin \frac{u(x,t)}{2} \\
\sin \frac{u(x,t)}{2} & \cos \frac{u(x,t)}{2} 
\end{pmatrix}, \qquad (x,t) \in [0,\infty) \times [0,\infty).
\end{align}
\end{lemma}
\begin{proof}
This follows immediately from Lemma \ref{m0lemma} noting that $e^{\frac{iu}{2}} = \hat{M}_{11} + i \hat{M}_{21}$.
\end{proof}

\begin{lemma}[Lax pair for $m$]\label{laxlemma}
For each $k \in \C \setminus \Gamma$, the map $(x,t) \mapsto m(x,t, k)$ lies in $C^2([0,\infty) \times [0,\infty),\C^{2\times 2})$ and satisfies the Lax pair equations
\begin{align}\label{mlax}
\begin{cases}
m_x +\frac{i}{4} (k-\frac{1}{k})[\sigma_3, m] = Q(x,t,k) m,
	\\
m_t + \frac{i}{4} (k+\frac{1}{k} )[\sigma_3, m] = Q(x,t,-k) m,
\end{cases} \quad (x,t) \in [0,\infty) \times [0, \infty), \ k \in \C \setminus \Gamma,
\end{align}
where $Q(x,t,k)$ is defined in terms of $u(x,t)$ by (\ref{Qdef}).
\end{lemma}
\begin{proof}
For each $k \in \C \setminus \Gamma$, the linear map
$$f \mapsto (\mathcal{C}f)(k) = \frac{1}{2\pi i} \int_\Gamma \frac{f(s)ds}{s - k}: L^2(\Gamma) \to \C$$
is bounded. Also, using that the three maps in (\ref{xtww}) are $C^2$, we find that
$$(x,t) \mapsto \mu(w^+ + w^-) = (\mu- I)(w^+ + w^-) + (w^+ + w^-): [0,\infty) \times [0,\infty) \to L^2(\Gamma)$$
is $C^2$.
Hence, for each $k \in \C \setminus \Gamma$, the map
\begin{align*}
(x,t) \mapsto m(x,t,k) = &\; I + \frac{1}{2\pi i}\int_\Gamma \frac{\mu(x,t,s)(w^+ + w^-)(x,t,s)}{s-k} ds
\end{align*}
is $C^2$ on $[0,\infty) \times [0,\infty)$.
Since the derivatives $\{\partial_x^j(\mu(w^+ + w^-))\}_{j=1}^2$ and $\{\partial_t^j(\mu(w^+ + w^-))\}_{j=1}^2$ exist in $L^2(\Gamma)$, we obtain
\begin{subequations}\label{MxMt}
\begin{align}\label{MxMta}
\begin{cases}
(\partial_x^jm)(x,t,\cdot) = \mathcal{C}(\partial_x^j(\mu(w^+ + w^-))) \in \dot{E}^2(\C \setminus \Gamma), 
	\\
(\partial_t^jm)(x,t,\cdot) = \mathcal{C}(\partial_t^j(\mu(w^+ + w^-))) \in \dot{E}^2(\C \setminus \Gamma),
\end{cases} \quad j = 1,2,
\end{align}
\end{subequations}
for each  $(x,t) \in [0,\infty) \times [0,\infty)$.
Introducing the function $\psi$ by 
$$\psi(x,t,k) = m(x,t,k) e^{-\frac{i}{4}((k-\frac{1}{k})x+ (k+\frac{1}{k})t)\sigma_3},$$ 
we see that
$$\psi_x \psi^{-1} = \bigg(m_x-\frac{i}{4}\bigg(k-\frac{1}{k}\bigg) m\sigma_3\bigg)m^{-1}.$$
Pick a point $p \in \C \setminus \Gamma$. Since $m \in I + \dot{E}^2(\C \setminus \Gamma)$ and $m_x \in \dot{E}^2(\C \setminus \Gamma)$, we find (cf. Lemmas A.4 and A.6 in \cite{LNonlinearFourier}) that there exist functions $\{f_j(x,t)\}_1^3$ such that the function $f$ defined by
\begin{align*}
f(x,t,k)  =\frac{k}{(k - p)^3}\bigg(m_x -\frac{i}{4}\bigg(k-\frac{1}{k}\bigg) m\sigma_3\bigg)m^{-1}
- \frac{f_3(x,t)}{(k-p)^3} - \frac{f_2(x,t)}{(k-p)^2} - \frac{f_1(x,t)}{k-p}
\end{align*}
lies in $\dot{E}^1(\C \setminus \Gamma)$. Here the factor $\frac{k}{(k -p)^3}$ has been included to move the poles at zero and infinity to the finite point $p \notin \Gamma$. The jump matrix $v$ satisfies $v_x = -i\theta_1[\sigma_3, v]$. Hence the jump condition (\ref{RHm}) for $m$ implies that $f_+ = f_-$ a.e. on $\Gamma$ and hence that $f$ vanishes identically:
$$f(x,t,\cdot) = \mathcal{C}(f_+ - f_-) = 0.$$
We conclude that there exist functions $F_j(x,t)$, $j=-1,0,1$, such that
\begin{subequations}\label{mxmtFFF}
\begin{align}\label{mxFFF}
m_x-\frac{i}{4}\bigg(k-\frac{1}{k}\bigg) m\sigma_3 = \bigg(k F_1(x,t)+ F_0(x,t) + \frac{F_{-1}(x,t)}{k}\bigg)m, \qquad k \in \C \setminus \Gamma.
\end{align}
A similar argument starting with the equation
$$\psi_t \psi^{-1} = \bigg(m_t - \frac{i}{4}\bigg(k+\frac{1}{k}\bigg) m\sigma_3\bigg)m^{-1}$$
shows that there exist functions $G_j(x,t)$, $j=-1,0,1$, such that
\begin{align}\label{mtFFF}
m_t-\frac{i}{4}\bigg(k+ \frac{1}{k}\bigg) m\sigma_3 = \bigg(k G_1(x,t)+ G_0(x,t) + \frac{G_{-1}(x,t)}{k}\bigg)m, \qquad k \in \C \setminus \Gamma.
\end{align}
\end{subequations}
We will perform an asymptotic analysis of the equations (\ref{mxFFF}) and (\ref{mtFFF}) as $k \to \infty$ and as $k \to 0$ to conclude that
\begin{subequations}\label{FGQidentities}
\begin{align}
& k F_1(x,t)+ F_0(x,t) + \frac{F_{-1}(x,t)}{k} = Q(x,t,k) - \frac{i}{4}\bigg(k - \frac{1}{k}\bigg) \sigma_3, 	
	\\
& k G_1(x,t)+ G_0(x,t) + \frac{G_{-1}(x,t)}{k} = Q(x,t,-k) - \frac{i}{4}\bigg(k + \frac{1}{k}\bigg) \sigma_3.
\end{align}
\end{subequations}
This will complete the proof of the lemma. 

We first consider the limit as $k$ approaches infinity. Let $W_\infty$ be a nontangential sector at $\infty$ with respect to $\Gamma$. 
We write (\ref{m1representation}) as
\begin{align*}
m(x,t,k)
 = I - \frac{1}{2\pi i}\int_\Gamma (\mu(w^+ + w^-))(x,t,s) \bigg(\frac{1}{k} + \frac{s}{k^2} + \frac{s^2}{k^2(k - s)}\bigg) ds.
\end{align*}
By Lemma \ref{C2lemma}, we have $\mu(x,t, \cdot) - I \in L^2(\Gamma)$. 
Since $r_1(k) - r_{1,a}(k) = O(k^{-3})$ as $k \to \infty$, the $L^1$ and $L^2$ norms of $s(w^+ + w^-)(x,t,s)$ and $\frac{s^2}{k - s} (w^+ + w^-)(x,t,s)$ are bounded (the latter uniformly with respect to all large $k \in W_\infty$). Hence,
\begin{align}\label{Mexpansiona}
m(x,t,k) = I + \frac{n_1(x,t)}{k} + O(k^{-2}), \qquad k \to \infty, \ k \in W_\infty,
\end{align}
where the error term is uniform with respect to $k \in W_\infty$ and the coefficient $n_1$ is given by
$$n_1(x,t) = - \frac{1}{2\pi i}\int_\Gamma (\mu(w^+ + w^-))(x,t,s) ds.$$

Applying a similar argument to the derivatives $m_x$  and $m_t$, we find
\begin{align}\label{Mexpansionb}
\begin{cases}
 m_x = O(k^{-1}),
	\\
 m_t = O(k^{-1}),
\end{cases}	
\quad k \to \infty, \ k \in W_\infty.
\end{align}
Indeed, the $L^2$ norms of $\mu_x$ and $\mu_t$ are bounded by Lemma \ref{C2lemma}. Furthermore, the definition of $w^\pm$ implies that the $L^1$ and $L^2$ norms of $(1 + \frac{s}{k-s})(w^+ + w^-)_x$ and $(1 + \frac{s}{k-s})(w^+ + w^-)_t$ are bounded. This leads to (\ref{Mexpansionb}).

Substituting (\ref{Mexpansiona}) and (\ref{Mexpansionb}) into (\ref{mxmtFFF}) the terms of $O(k)$ and $O(1)$ yield
\begin{align}\label{F1G1}
 F_1 = G_1 =& -\frac{i}{4}\sigma_3 \quad  \text{and} \quad F_0 = G_0 = \frac{i}{4} [\sigma_3, n_1],
\end{align}
respectively.
Let us now write $n_1 = a_1 \sigma_1 + b_1 \sigma_2 + c_1 \sigma_3 + d_1I$, where $a_1,b_1,c_1,d_1$ are complex-valued functions of $(x,t)$. The symmetry $m(x,t,k) = \sigma_2 m(x,t,-k) \sigma_2$ (see (\ref{msymm})) yields $b_1 = d_1 = 0$, so that $n_1 = a_1 \sigma_1 + c_1 \sigma_3$. Moreover, the symmetry $m(x,t,k) = \sigma_2 \overline{m(x,t,\bar{k})} \sigma_2$ shows that the coefficients $a_1$ and $c_1$ are pure imaginary. In particular,
\begin{align}\label{F0G0a1}
F_0 = G_0 = -\frac{1}{2}a_1(x,t)\sigma_2,
\end{align}
where $a_1(x,t) \in i \R$.

We now consider the limit as $k$ approaches zero. In this limit, $m$ obeys the asymptotics (\ref{matzero}). Using the fact that $r(k) = O(k^3)$ as $k \to 0$, an argument similar to the one leading to (\ref{Mexpansionb}) shows that
\begin{align}\label{mxmtatzero}
\begin{cases}
 m_x = \hat{M}_{x}(x,t) + O(k),
	\\
m_t = \hat{M}_{x}(x,t) + O(k),
\end{cases}	
\quad k \to 0, \ k \in W_0.
\end{align}
Substituting (\ref{matzero}) and (\ref{mxmtatzero}) into (\ref{mxFFF}) and (\ref{mtFFF}) the terms of $O(1/k)$ yield
\begin{align*}
\frac{i}{4} \hat{M} \sigma_3 = F_{-1}\hat{M}
\quad \text{and} \quad - \frac{i}{4}\hat{M}\sigma_3 = G_{-1}\hat{M},
\end{align*}
respectively. Hence 
\begin{align}\label{FGminus1}
F_{-1} = -G_{-1} = \frac{i}{4} \hat{M} \sigma_3 \hat{M}^{-1} 
= \frac{i}{4} \begin{pmatrix} \cos u(x,t) & \sin u(x,t) \\
\sin u(x,t) & -\cos u(x,t) \end{pmatrix}.
\end{align}
Taking the sum of (\ref{mxFFF}) and (\ref{mtFFF}) and using (\ref{FGminus1}), we find
\begin{align*}
& m_x + m_t -\frac{i}{2} k m\sigma_3 = 2(k F_1+ F_0)m.
\end{align*}
The terms of $O(1)$ of this equation yield $\hat{M}_{x} + \hat{M}_{t} = 2F_0\hat{M}$, which by Lemma \ref{propertiesulemma} and (\ref{F0G0a1}) gives
\begin{align}\label{a1uxut}
a_1(x,t) = \frac{i}{2}(u_x + u_t).
\end{align}
The identities (\ref{FGQidentities}) now follow from (\ref{F1G1}), (\ref{F0G0a1}), (\ref{FGminus1}), and (\ref{a1uxut}). 
\end{proof}

In order to complete the proof of Theorem \ref{existenceth}, it only remains to show the following lemma which is an easy consequence of the compatibility of the Lax pair equations (\ref{mlax}).

\begin{lemma}
 $u(x,t)$ satisfies the sine-Gordon equation \eqref{sg} for $x \geq 0$ and $t \geq 0$.
\end{lemma}
\begin{proof}
Fix $k \in \C \setminus \Gamma$.  By Lemma \ref{laxlemma}, $m(x,t,k)$ is $C^2$ as a function of $x$ and $t$. Hence the mixed partials $m_{xt}$ and $m_{tx}$ exist and are equal for $x \geq0$ and $t \geq 0$. This implies that
$$A_t - B_x + [A,B] = 0, \qquad x \geq 0, \ t \geq 0,$$
where 
$$A(x,t,k) = Q(x,t,k) - \frac{i}{4}\bigg(k - \frac{1}{k}\bigg) \sigma_3, \qquad
B(x,t,k) = Q(x,t,-k) - \frac{i}{4}\bigg(k + \frac{1}{k}\bigg) \sigma_3.$$
Since, by a long but straightforward computation,
$$A_t - B_x + [A,B] = \frac{i}{4}(u_{xx} - u_{tt} - \sin u)\sigma_2,$$
it follows that $u(x,t)$ satisfies (\ref{sg}).
\end{proof}

\begin{remark}\upshape
Equation (\ref{a1uxut}) shows that the solution $u(x,t)$ defined in (\ref{ulim}) satisfies the following identity which will prove useful later in order to recover the derivatives of $u(x,t)$ from $M$:
\begin{align}\label{uxutrecover}
& u_x(x,t) + u_t(x,t) = -2i\ntlim_{k\to\infty} \big(kM(x,t,k)\big)_{12}.
\end{align}
\end{remark}

\section{Proof of Theorem \ref{existenceth2}}\label{existence2proofsec}
Suppose $u_0, u_1, g_0, g_1$ are functions which satisfy (\ref{ujgjassump}) with $n = 1$, $m = 2$, and which are compatible with (\ref{sg}) to third order at $x=t=0$.
Let $r_1(k)$ and $h(k)$  be the associated spectral functions defined in (\ref{r1hdef}). Suppose Assumption \ref{assumption1} holds.

\begin{lemma}\label{r1hlemma}
The functions $r_1(k)$ and $h(k)$ satisfy Assumption \ref{r1hassumption}. Moreover, $h(k)$ is an analytic function of $k \in D_2$.
\end{lemma}
\begin{proof}
It follows immediately from Lemmas \ref{ablemma} and \ref{ABlemma} together with Assumption \ref{assumption1} that $r_1$ and $h$ are continuous functions satisfying $(a)$ and $(b)$ of Assumption \ref{r1hassumption} and that $h(k)$ is an analytic function of $k \in D_2$. It remains to prove property $(c)$ of Assumption \ref{r1hassumption}, i.e., that $r(k) = r_1(k) + h(k)$ satisfies $r(\pm 1) = 0$ and $r(k) = O(k^3)$ as $k \to 0$. 
By (\ref{rdef}), we have $r = r_1 + h = c^*/d$ for $k \in \R$. It follows from (\ref{cdexpansionsc}) that $c(k) = O(k^{m+1}) = O(k^3)$ as $k \to 0$. Moreover, the global relation (\ref{GR}) implies that $c(k) =0$ for $k \in \R \setminus (-1,1)$, so we find $r(k) = 0$ for $k \in \R \setminus (-1,1)$; in particular, $r(\pm 1) = 0$.
\end{proof}

By Lemma \ref{r1hlemma} and Theorem \ref{existenceth}, there exists a unique solution $M(x,t,k)$ of the RH problem (\ref{RHM}) formulated in terms of $r_1$ and $h$. Define $u(x,t)$ in terms of $M(x,t,k)$ by (\ref{ulim2}). By Theorem \ref{existenceth}, $u:[0,\infty) \times [0,\infty) \to \R$ is a $C^2$-solution of  equation (\ref{sg}).
It only remains to verify that $u(x,t)$ satisfies the appropriate initial and boundary conditions for an appropriate choice of the integer $j$ in (\ref{ulim2}). We begin with the initial condition. 

\begin{lemma}\label{initiallemma}
For an appropriate choice of the integer $j$, the solution $u(x,t)$ defined by (\ref{ulim2}) satisfies $u(x,0) = u_0(x)$ and $u_t(x,0)=u_1(x)$ for $x \geq 0$.
\end{lemma}
\begin{proof}
Define $X(x,k)$ by (\ref{Xvolterra}) and let $Y(x,k)$ be the eigenfunction defined by
$$Y(x,k) = I + \int_{0}^x e^{i\theta_1(x'-x)\hat{\sigma}_3} (\mathsf{U}Y)(x',k) dx'.$$
Detailed asymptotic formulas for $X$ and  $Y$ as $k \to \infty$ and $k \to 0$ were derived in \cite{HLNonlinearFourier} and we refer to this reference for proofs of some of the facts used below. 
The properties of $X$ and $Y$ imply that the function $m^{(x)}(x,k)$ defined by
\begin{align}
m^{(x)}(x,k) =
\begin{cases}
 \left(\frac{[Y(x,k)]_1}{a(k)}, [X(x,k)]_2\right), \qquad \im k > 0,
	\\
 \left([X(x,k)]_1, \frac{[Y(x,k)]_2}{\overline{a(\bar{k})}}\right), \qquad \im k < 0,
\end{cases}
\end{align}
satisfies the RH problem
\begin{align}\label{mxRH}
\begin{cases}
m^{(x)}(x, \cdot) \in I + \dot{E}^2(\C \setminus \R),\\
m_+^{(x)}(x,k) = m_-^{(x)}(x, k) \begin{pmatrix} 1 + |r_1(k)|^2 & r_1^*(k) e^{-2i\theta_1x} \\
r_1(k) e^{2i\theta_1x} & 1 \end{pmatrix}  \quad \text{for a.e.} \ k \in \R.
\end{cases}
\end{align}
As $k \to 0$, $X$ satisfies
$$\ntlim_{k \to 0} X(x,k) = G_0(x), \qquad x \geq 0, \ k \in (\C_-, \C_+),$$
where $G_0(x)$ is the matrix defined in (\ref{G0xdef}) and the notation indicates that the limit of the first (second) column is taken from the lower (upper) half-plane.
The behavior of $Y$ as $k \to 0$ depends on whether $x > 0$ or $x = 0$. 
For $x > 0$, we have
$$\ntlim_{k \to 0} Y(x,k) = (-1)^{N_x}\cos \frac{u_0(0)}{2}G_0(x), \qquad x > 0, \ k \in (\C_+, \C_-).$$
For $x = 0$, we have $Y(0,k) = I$ for all $k \in \C \setminus \{0\}$. 

If $x > 0$, the asymptotics of $X$ and $Y$ together with the fact that $a(0) = (-1)^{N_x}\cos \frac{u_0(0)}{2}$ shows that the function $\hat{M}^{(x)}(x) := \ntlim_{k\to 0} m^{(x)}(x,k)$ is well-defined (i.e., the nontangential limits of $m^{(x)}(x,k)$  as $k \to 0$ from the upper and lower half-planes exist and coincide) and is given by $\hat{M}^{(x)}(x) = G_0(x)$. This gives the first of the following equations; the second follows from the asymptotics of $X$ and  $Y$ as $k \to \infty$: 
\begin{subequations}\label{ulim0}
\begin{align}\label{ulim0a}
&  u_0(x) = 2 \arg\big[(-1)^{N_x}(\hat{M}_{11}^{(x)}(x) + i \hat{M}_{21}^{(x)}(x))\big], \qquad x > 0.
	\\\label{ulim0b}
& u_1(x)= - u_{0}'(x)-2i\ntlim_{k\to\infty}\big(km^{(x)}(x,k)\big)_{12}, \qquad x \geq 0.
\end{align}
\end{subequations}

On the other hand, since
$$h(k) e^{2i\theta_1x} \in E^\infty(D_2), \qquad \overline{h(\bar{k})} e^{-2i\theta_1x} \in E^\infty(D_3),$$
it follows that the function $M^{(x)}(x,k)$ defined by
\begin{align}\label{Mxdef}
M^{(x)}(x,k) = \begin{cases} M(x, 0, k), & k \in D_1 \cup D_4,
	\\
M(x,0,k) \begin{pmatrix} 1 & 0 \\- h(k) e^{2i\theta_1x} & 1 \end{pmatrix}, & k \in D_2,
	\\
M(x,0,k) \begin{pmatrix} 1 & \overline{h(\bar{k})} e^{-2i\theta_1x} \\ 0 & 1 \end{pmatrix}, & k \in D_3,
\end{cases}
\end{align}
also satisfies (\ref{mxRH}). By uniqueness, $M^{(x)} = m^{(x)}$. 
If $x > 0$, we have $\ntlim_{k \to 0} M^{(x)}(x,k) = \ntlim_{k \to 0} M(x,0,k) = \hat{M}(x,0)$, because
$$\ntlim_{\underset{\im k > 0}{k \to 0}} e^{2i\theta_1x} = 0, \qquad \ntlim_{\underset{\im k < 0}{k \to 0}} e^{-2i\theta_1x} = 0.$$
Comparing the definition (\ref{ulim}) of $u(x,t)$ with (\ref{ulim0a}) and fixing the integer $j$ in (\ref{ulim2}) appropriately, we conclude that $u(x,0) = u_0(x)$ for $x > 0$. By continuity of $u_0(x)$  and $u(x,0)$, we have also $u_0(0) = u(0,0)$. Finally, comparing the expression (\ref{uxutrecover}) for $u_x + u_t$ with (\ref{ulim0b}) gives $u_t(x,0)=u_1(x)$ for $x \geq 0$.
\end{proof}

\begin{remark}\upshape
If $x = 0$, the limits of $m^{(x)}$ and $M^{(x)}$ as $k \to 0$ are more subtle. In fact, for $x = 0$, we have
\begin{align*}
\begin{cases}
\displaystyle{\ntlim_{\underset{\im k > 0}{k \to 0}}} m^{(x)}(0,k) = (-1)^{N_x}\begin{psmallmatrix} \frac{1}{\cos \frac{u_0(0)}{2}} & - \sin \frac{u_0(0)}{2} \\
0 & \cos \frac{u_0(0)}{2} \end{psmallmatrix}, 
	\\
\displaystyle{\ntlim_{\underset{\im k < 0}{k \to 0}}} m^{(x)}(0,k) = (-1)^{N_x}\begin{psmallmatrix} \cos \frac{u_0(0)}{2} & 0 
\\ \sin \frac{u_0(0)}{2} & \frac{1}{\cos \frac{u_0(0)}{2}} \end{psmallmatrix},
\end{cases}
\end{align*}
and
$$\ntlim_{\underset{\im k > 0}{k \to 0}} M^{(x)}(0,k) = \hat{M}(0,0) \begin{pmatrix} 1 & 0 \\ - h(0) & 1 \end{pmatrix}, \qquad
\ntlim_{\underset{\im k < 0}{k \to 0}} M^{(x)}(0,k) = \hat{M}(0,0) \begin{pmatrix} 1 & \overline{h(0)} \\ 0 & 1 \end{pmatrix}.$$
These limits are consistent with the equality $u_0(0) = u(0,0)$, because $h(0) = \tan \frac{u(0,0)}{2}$ and the integer $j$ in (\ref{ulim2}) has the same parity as $N_x$.
\end{remark}

The following lemma completes the proof of Theorem \ref{existenceth2}. 

\begin{lemma} \label{boundarylemma}
The solution $u(x,t)$ defined by (\ref{ulim2}) satisfies $u(0,t) = g_0(t)$ and $u_x(0,t) = g_1(t)$ for $t \geq 0$.
\end{lemma}
\begin{proof}
The proof is similar to that of Lemma \ref{initiallemma}; however, since the function $A(k)$  could have zeros in $D_3$, we introduce deformations $\mathcal{D}_j$ of the domains $D_j$ before formulating the RH problem. 

The assumption that $a(k)$  is nonzero for $\im k \geq 0$ implies that 
$$|A(0)| = \Big|\cos\frac{u_0(0)}{2}\Big| = |a(0)| \neq 0.$$
Thus, there exists an $\epsilon > 0$ such that $A(k)$ is nonzero for $k \in \bar{\mathcal{D}}_3$, where $\{\mathcal{D}_j\}_1^4$ are defined by (see Figure \ref{deformedDjs.pdf})
\begin{align}\nonumber
\mathcal{D}_1 = \{\im k > 0\} \cap \{|k|>\epsilon\},  \qquad
\mathcal{D}_2 = \{\im k > 0\} \cap \{|k|<\epsilon\},
	\\ \nonumber
\mathcal{D}_3 = \{\im k < 0\} \cap \{|k|<\epsilon\},  \qquad
\mathcal{D}_4 = \{\im k < 0\} \cap \{|k|>\epsilon\}.
\end{align}
Let $\Gamma_\epsilon = \R \cup \{|k| = \epsilon\}$ denote the contour separating the $\mathcal{D}_j$ oriented as in Figure \ref{deformedDjs.pdf}.
\begin{figure}
\begin{center}
\bigskip
\begin{overpic}[width=.65\textwidth]{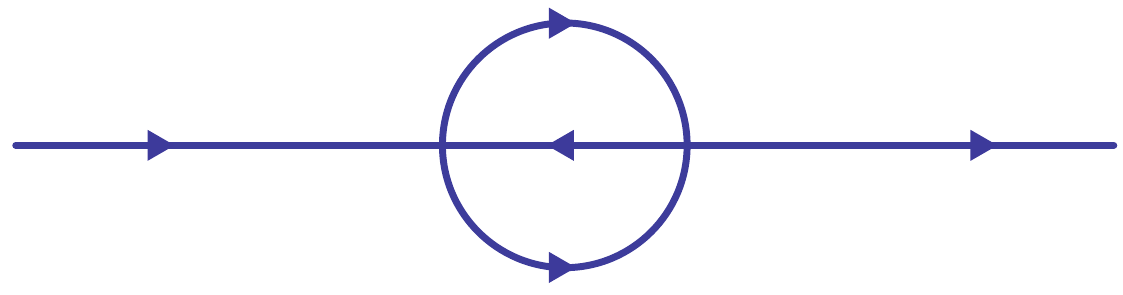}
      \put(62,10.3){\small $\epsilon$}
      \put(33.8,10.3){\small $-\epsilon$}
      \put(48,29){\small $\mathcal{D}_1$}
      \put(48,17.5){\small $\mathcal{D}_2$}
      \put(48,6.5){\small $\mathcal{D}_3$}
      \put(48,-5){\small $\mathcal{D}_4$}
      \put(102,12){\small $\Gamma_\epsilon$}
      \end{overpic}\bigskip\bigskip
     \begin{figuretext}\label{deformedDjs.pdf}
       The contour $\Gamma_\epsilon$ and the deformed domains $\{\mathcal{D}_j\}_1^4$ in the complex $k$-plane.
     \end{figuretext}
     \end{center}
\end{figure}
Define $T(t,k)$ by (\ref{Tvolterra}) and let $U(t,k)$ be the eigenfunction defined by
$$U(t,k) = I + \int_{0}^t e^{i\theta_2(t'-t)\hat{\sigma}_3} (\mathsf{V}U)(t',k) dt'.$$
The global relation (\ref{GR}) and the condition $\det S(k) = 1$ gives
$$d(k) = a(k)\overline{A(\bar{k})} + \frac{B(k) a(k)}{A(k)} \overline{B(\bar{k})}
= \frac{a(k)}{A(k)}, \qquad k \in \partial D_1.$$
Thus $A(k)$ admits an analytic continuation to $D_2$. By (\ref{GR}), $B = bA/a$ in $\bar{D}_1$; hence $B(k)$ also admits an analytic continuation to $D_2$.
The relations (\ref{GR}) and $T = U e^{-i\theta_2t\hat{\sigma}_3}S(k)$ show that
$$\frac{[T(t,k)]_2}{A(k)} = \frac{b(k)}{a(k)} e^{-2i\theta_2t} [U(t,k)]_1 + [U(t,k)]_2, \qquad t \geq 0, \ k \in \bar{D}_1.$$
Since $U(t, \cdot)$ is analytic in $\C \setminus \{0\}$, this shows that $[T]_2/A$ also admits an analytic continuation to $D_2$. 
Hence we may define $m^{(t)}$ by
\begin{align}
m^{(t)}(t,k) =
\begin{cases}
 \left([U(t,k)]_1, \frac{[T(t,k)]_2}{A(k)} \right), \qquad k \in \mathcal{D}_1 \cup \mathcal{D}_3,
	\\
 \left(\frac{[T(t,k)]_1}{\overline{A(\bar{k})}}, [U(t,k)]_2 \right), \qquad k \in \mathcal{D}_2 \cup \mathcal{D}_4.
\end{cases}
\end{align}
The properties of $T$ and $U$ imply that $m^{(t)}(t,k)$ satisfies the RH problem
\begin{align}\label{mtRH}
\begin{cases}
m^{(t)}(t, \cdot) \in I + \dot{E}^2(\C \setminus \Gamma_\epsilon),\\
m^{(t)}_+(t,k) = m^{(t)}_-(t, k) \begin{pmatrix} 1 & \frac{B_+(k)}{A_+(k)} e^{-2i\theta_2t} \\
\frac{\overline{B_+(\bar{k})}}{\overline{A_+(\bar{k})}} e^{2i\theta_2t} & \frac{1}{A_+(k)\overline{A_+(\bar{k})}} \end{pmatrix}  \quad \text{for a.e.} \ k \in \Gamma_\epsilon.
\end{cases}
\end{align}
Moreover, from the asymptotics of $T$ and $U$ (see \cite{HLNonlinearFourier})
\begin{subequations}\label{gjlim}
\begin{align}\label{gjlima}
&  g_0(t) = 2 \ntlim_{k\to 0} \arg\big[m_{11}^{(t)}(t,k) + i m_{21}^{(t)}(t,k)\big] \quad \text{(mod $2\pi$)}, \qquad t > 0,
	\\\label{gjlimb}
& g_1(t)= -g_{0}'(t)-2i\ntlim_{k\to\infty}\big(km^{(t)}(t,k)\big)_{12}, \qquad t \geq 0.
\end{align}
\end{subequations}

On the other hand, the function $\mathcal{M}(x,t,k)$ defined by
$$\mathcal{M}(x,t,k) = \begin{cases} M \begin{pmatrix} 1 & 0 \\ -h(k) e^{2i\theta} & 1 \end{pmatrix}, & k \in \mathcal{D}_1 \cap D_2,
	\\
M \begin{pmatrix} 1 & \overline{h(\bar{k})} e^{-2i\theta} \\ 0 & 1 \end{pmatrix}, & k \in \mathcal{D}_4 \cap D_3,
 	\\
M, & \text{otherwise},
\end{cases}$$
satisfies the RH problem
\begin{align}\label{RHMcal}
\begin{cases}
\mathcal{M}(x, t, \cdot) \in I + \dot{E}^2(\C \setminus \Gamma_\epsilon),\\
\mathcal{M}_+(x,t,k) = \mathcal{M}_-(x, t, k) \mathcal{J}(x, t, k) \quad \text{for a.e.} \ k \in \Gamma_\epsilon,
\end{cases}
\end{align}
where $\mathcal{J}$ is given by the same equation (\ref{Jdef}) as $J$ except that the $D_j$ are replaced by $\mathcal{D}_j$.
Using that $a(k)$, $d(k)$, and $A(k)$ are nonzero in $\bar{\C}_+$, $\bar{\mathcal{D}}_2$, and $\bar{\mathcal{D}}_3$, respectively, we infer that
$$\bigg(\frac{d(k)}{\overline{A(\bar{k})}}\bigg)^{\pm 1}  \in E^\infty(\mathcal{D}_2), \qquad
a(k)^{\pm 1} \in 1 + (\dot{E}^2 \cap E^\infty)(\C_+).$$
Consequently, the functions
\begin{align}\nonumber
& G_1(t,k) =  \begin{pmatrix} a(k) & 0 \\ 0 & \frac{1}{a(k)} \end{pmatrix},
\qquad G_2(t,k) = \begin{pmatrix} \frac{d(k)}{\overline{A(\bar{k})}} & - b(k) e^{-2i\theta_2t} \\ 0 & \frac{\overline{A(\bar{k})}}{d(k)} \end{pmatrix},
	\\ \label{Gjdef}
& G_3(t,k) = \begin{pmatrix} \frac{A(k)}{\overline{d(\bar{k})}} & 0 \\ \overline{b(\bar{k})} e^{2i\theta_2t} & \frac{\overline{d(\bar{k})}}{A(k)} \end{pmatrix},
\qquad
G_4(t,k) = \begin{pmatrix} \frac{1}{\overline{a(\bar{k})}} & 0 \\ 0 & \overline{a(\bar{k})} \end{pmatrix},
\end{align}
satisfy $G_j(t,\cdot) \in I + (\dot{E}^2 \cap E^\infty)(\mathcal{D}_j)$ for $j = 1,\dots, 4$.
Using the expression (\ref{Jdef}) for $J$ and the global relation (\ref{GR}), long but straightforward computations show that the function $M^{(t)}(t,k)$ defined by
\begin{align}\label{Mtdef}
M^{(t)}(t,k) = \mathcal{M}(0,t,k)G_j(t,k), \qquad k \in \mathcal{D}_j, \quad j = 1,\dots,4,
\end{align}
also satisfies the RH problem (\ref{mtRH}).
Hence $M^{(t)} = m^{(t)}$ by uniqueness. 

Equation (\ref{Mtdef}) implies
$$\ntlim_{\underset{\im k > 0}{k \to 0}} [M^{(t)}(t,k)]_1 = \frac{d(0)}{A(0)} \ntlim_{k \to 0} [M(0,t,k)]_1, \qquad t\geq 0,$$
and so, by the definition (\ref{ulim2}) of $u(x,t)$,
$$u(0,t) = 2 \ntlim_{\underset{\im k > 0}{k \to 0}} \arg\big[M_{11}^{(t)}(t,k) + i M_{21}^{(t)}(t,k)\big] \quad  \text{(mod $2\pi$)}, \qquad t \geq 0.$$
Comparing this with (\ref{gjlima}), we obtain that $u(0,t)$ equals $g_0(t)$ up to an integer multiple of $2\pi$ for $t > 0$. By Lemma \ref{boundarylemma} and the compatibility at $x = t = 0$, we have $u(0,0) = u_0(0) = g_0(0)$. Since $u(0,t)$ and $g_0(t)$ are continuous for $t \geq 0$, we deduce that $u(0,t) = g_0(t)$ for $t \geq 0$.
Comparison of the expression (\ref{uxutrecover}) for $u_x + u_t$ with (\ref{gjlimb}) then gives $u_x(0,t) = g_1(t)$ for $t \geq 0$.
\end{proof}

\section{Proof of Theorem \ref{asymptoticsth}: Overview}\label{overviewsec}
The proof of Theorem \ref{asymptoticsth} utilizes nonlinear steepest descent techniques first developed for the mKdV equation on the line by Deift and Zhou \cite{DZ1993}, and later applied to the sine-Gordon equation (\ref{sg}) on the line by Cheng, Venakides, and Zhou \cite{CVZ1999}. In Sections \ref{sectorIsec}-\ref{sectorIVsec}, we present detailed derivations of the four asymptotic formulas in (\ref{uasymptotics}). In this section, we give a brief overview of the key elements that enter the proofs.

Assume $r_1:\R \to \C$ and $h:\bar{D}_2 \to \C$ are two arbitrary functions satisfying Assumption \ref{r1hassumption2} and define the jump matrix $J(x,t,k)$ by (\ref{Jdef}). According to Theorem \ref{existenceth}, the associated RH problem (\ref{RHM}) determined by $J$ has a unique solution $M(x,t,k)$. Our goal is to find the behavior as $(x,t) \to \infty$ of the function $u(x,t)$  defined in terms of $M$  by (\ref{ulim}). We will accomplish this by applying various transformations to the RH problem (\ref{RHM}) for $M$.

\subsection{Sector I}
We first consider Sector I, which is the easiest to deal with. 
The jump matrix (\ref{Jdef}) involves the exponentials $e^{\pm 2i\theta} = e^{\pm x\tilde{\Phi}}$,
where $\tilde{\Phi} := \tilde{\Phi}(\zeta, k)$ is defined by
$$\tilde{\Phi}(\zeta, k) 
= \frac{i}{2}\bigg(k - \frac{1}{k}\bigg) + \frac{i}{2}\bigg(k + \frac{1}{k}\bigg)\frac{1}{\zeta}.$$
Since $\zeta = x/t \geq 1$ in Sector I, we have $\re \tilde{\Phi} \lessgtr 0$ for $k \in \C_\pm$. Thus, as $x \to\infty$, the exponentials $e^{x\tilde{\Phi}}$ and  $e^{-x\tilde{\Phi}}$ are small for $k$ in the upper and lower half-planes, respectively. This suggests deforming the contour of the RH problem (\ref{RHM}) in such a way that all exponentials of the form $e^{x\tilde{\Phi}}$ ($e^{-x\tilde{\Phi}}$) are brought into the upper (lower) half-plane. We accomplish this in two steps. 
First, we introduce $M^{(1)}$ by deforming the jump across the unit circle involving the spectral function $h(k)$, so that it passes through the origin, see Figure \ref{Gamma1I.pdf}. Then, after introducing an analytic approximation of $r_1(k)$, we define $m(x,t,k)$ by deforming (the analytic part of) the jump of $M^{(1)}$ across $\R$ into the upper and lower half-planes in an appropriate way, see Figure \ref{SigmaI.pdf}. The jump matrix $v$ of the RH problem for $m$ is exponentially close to the identity matrix except near $k = 0$. Since $r(k)$ vanishes to all orders at $k = 0$, the outcome is that $m(x,t,k)$, and hence also $u(x,t)$, is of $O(x^{-N})$ for each $N \geq 1$ as $x \to \infty$.

\subsection{Sector II}
In this case, we suppose $\zeta \in [0, 1)$ and write $e^{\pm 2i\theta} = e^{\pm t\Phi}$,
where $\Phi := \Phi(\zeta, k)$ is defined by
$$\Phi(\zeta, k) 
= \frac{i}{2}\bigg(k - \frac{1}{k}\bigg)\zeta + \frac{i}{2}\bigg(k + \frac{1}{k}\bigg)
= \frac{i(k^2 + k_0^2)}{k(1+ k_0^2)}$$
with
\begin{align}\label{k0def}
k_0 := k_0(\zeta) = \sqrt{\frac{1-\zeta}{1+\zeta}} \in (0,1].
\end{align}
The exponential  $e^{\pm t\Phi}$ is small for large $t$ if $\re \Phi \lessgtr 0$. The function $\re \Phi$ changes sign across $\R$ and across the circle $|k| = k_0$, see Figure \ref{Gamma1II.pdf}. This suggests performing the following transformations. First, we define $M^{(1)}$ by deforming the jump across the unit circle so that it passes through the two critical points $\pm k_0$. 
Then, after introducing an analytic approximation\footnote{As opposed to the analytic approximation employed in Sector I (which is a good approximation of $r_1(k)$ for large $x$), this approximation must be good for large $t$.}
of $r_1(k)$, we define $m(x,t,k)$ by deforming (the analytic part of) the jump of $M^{(1)}$ across $(-\infty, -k_0) \cup (k_0, \infty)$ into the upper and lower half-planes in an appropriate way, see Figure \ref{SigmaII.pdf}. 
Letting $w := v - I$ denote the departure of the jump matrix $v$ of $m$ from the identity matrix, it can be shown that $(1+|k|^{-1})w$ is $O(k_0^N)$ on the interval $(-k_0,k_0)$ and $O(k_0^N + t^{-N})$ on the diagonal lines emanating from $\pm k_0$. On the remaining part of the contour, $(1+|k|^{-1})w$ is either exponentially small or $O(t^{-N})$. This leads to the asymptotics in Sector II.

\subsection{Sector III}
The analysis of Sector III, which is the most complicated, begins in the same way as for Sector II. However, in Sector III, the critical points $\pm k_0$ may not approach the origin as $t \to \infty$. Therefore, we need to treat the jump across $(-k_0, k_0)$ in a different way. Thus, after introducing $M^{(1)}$ as in Sector II, we introduce $M^{(2)}$ by $M^{(2)} = M^{(1)}\delta^{-\sigma_3}$, where the complex-valued function $\delta(\zeta,k)$ is analytic in $\C \setminus [-k_0,k_0]$. The jump of $\delta$ across $[-k_0,k_0]$ is chosen so that the jump  $J^{(2)}$ of $M^{(2)}$ across the same interval can be appropriately factorized. Then, after introducing analytic approximations of $h(k)$, $r_1(k)$, and $r_2(k) := r^*(k)/(1+|r(k)|^2)$, we define $M^{(3)}(x,t,k)$ by appropriately deforming the analytic parts of the jumps, see Figure \ref{Gamma3III.pdf}. The jump of the resulting RH problem is close to the identity, except across two small crosses centered at the critical points $\pm k_0$. By constructing an approximate solution $M^{app}$ whose jumps across these small crosses are close to those of $M^{(3)}$, and then considering the matrix quotient $m = M^{(3)} (M^{app} )^{-1}$, we arrive at the asymptotics in Sector III.

\subsection{Sector IV}
As in Sector III, the main contributions in Sector IV come from two small crosses centered on the critical points $\pm k_0$. However, in Sector IV, since $\pm k_0$ are close to the points $\pm 1$ at which $r(k)$ vanishes, these contributions are suppressed. The smaller $r(k)$ is near $\pm 1$, the more suppressed they are. In other words, the vanishing behavior of $r(k)$ at $k = \pm 1$ is tied to decay of $u(x,t)$  near the boundary. 

\begin{remark}\upshape
We know from Theorem \ref{existenceth} that $u$ is a $C^2$ solution of (\ref{sg}). Actually, using the stronger assumptions on $r_1(k)$ and $h(k)$ of Assumption \ref{r1hassumption2}, it follows easily from the proof of Theorem \ref{existenceth} that $u \in C^\infty([0,\infty) \times [0,\infty), \R)$. 
\end{remark}

\section{Proof of Theorem \ref{asymptoticsth}: Asymptotics in Sector I}\label{sectorIsec}
Let $\mathcal{I} = (1,\infty)$ and suppose $\zeta \in \mathcal{I}$.
Fix an integer $N \geq 1$.

\subsection{Transformations of the RH problem}
Let $\{V_j\}_1^2$ denote the open subsets
\begin{align*}
& V_1 = \{u e^{i\theta} \, | \, 0 < u < \infty, \theta \in (0,\pi/4) \cup (3\pi/4, \pi)\},
	\\
& V_2 = \{u e^{i\theta} \, | \, 0 < u < \infty, \theta \in (-\pi,-3\pi/4) \cup (-\pi/4, 0)\},
\end{align*}
and let $U_j := V_j \cap \{|k| < 1\}$, $j = 1,2$, denote the parts of these sets that lie in the unit disk, see Figure \ref{VjsI.pdf}. 

\begin{figure}
\begin{center}
\begin{overpic}[width=.45\textwidth]{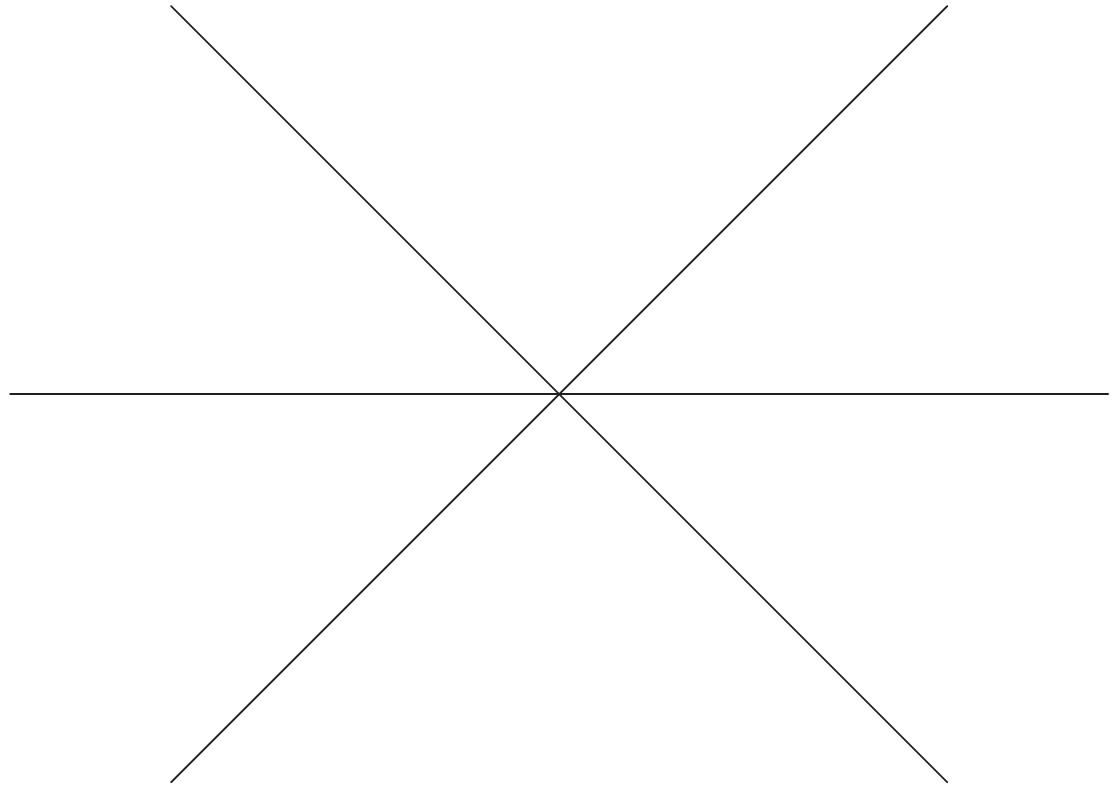}
      \put(75,45){\small $V_1$}
      \put(22,45){\small $V_1$}
        \put(75,23){\small $V_2$}
      \put(22,23){\small $V_2$}
      \put(48.7,30){\small $0$}
    \end{overpic}
    \qquad
    \begin{overpic}[width=.45\textwidth]{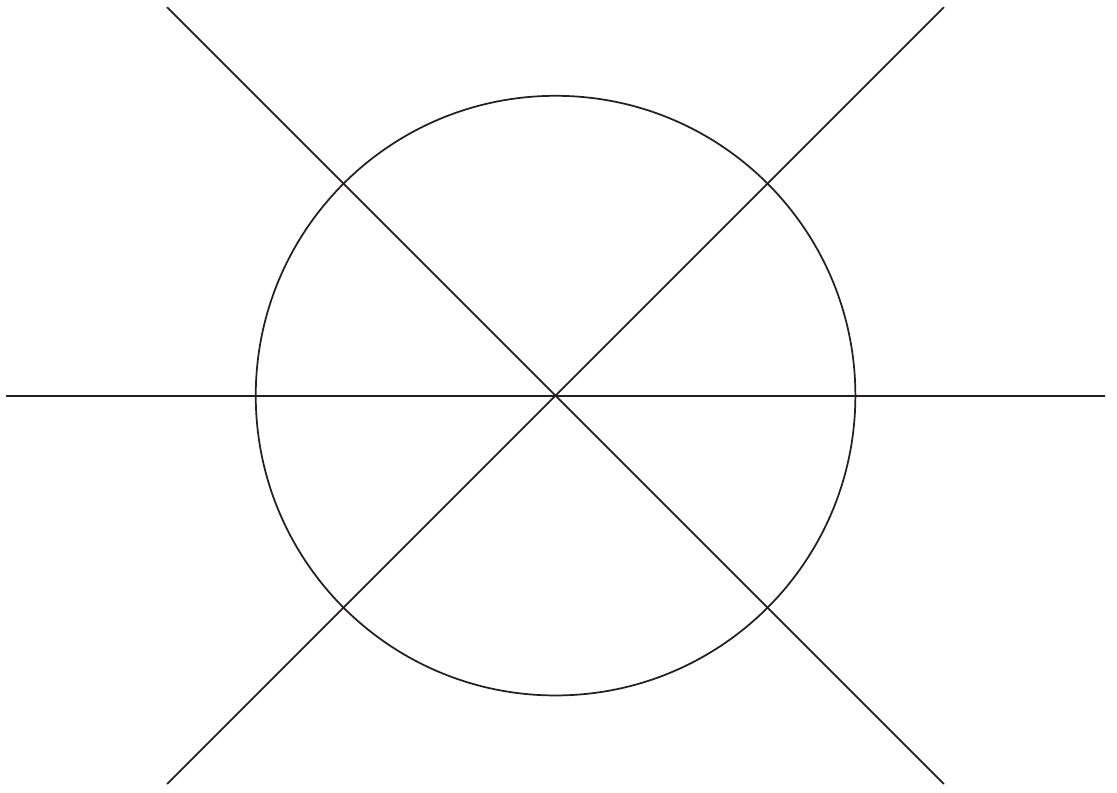}
      \put(63,41){\small $U_1$}
      \put(30,41){\small $U_1$}
        \put(63,27){\small $U_2$}
      \put(30,27){\small $U_2$}
      \put(48.3,30){\small $0$}
      \put(78,30){\small $1$}
      \put(15,30){\small $-1$}
    \end{overpic}\bigskip
     \begin{figuretext}\label{VjsI.pdf}
       The open subsets $\{V_j, U_j\}_1^2$ of the complex $k$-plane relevant for Sector I. 
     \end{figuretext}
     \end{center}
\end{figure}

\subsubsection{First transformation}\label{firsttransformationsubsecI}
Define $M^{(1)}(x,t,k)$ by 
\begin{align}\label{M1defI}
M^{(1)}(x,t,k) = M(x,t,k) \times \begin{cases} 
 \begin{pmatrix} 1 & 0 \\ -h(k) e^{x\tilde{\Phi}(\zeta, k)} & 1 \end{pmatrix}, & k \in U_1, 
 	\\
 \begin{pmatrix} 1 & h^*(k) e^{-x\tilde{\Phi}(\zeta, k)} \\ 0 & 1 \end{pmatrix}, & k \in U_2, 
 	\\
I, & \text{elsewhere}.
\end{cases}
\end{align}
The functions $h e^{x\tilde{\Phi}}$ and $h^*e^{-x\tilde{\Phi}}$ are bounded and analytic in $U_1$ and $U_2$, respectively. Hence $M$ satisfies the RH problem (\ref{RHM}) iff $M^{(1)}$ 
satisfies the RH problem 
\begin{align}\label{RHMj}
\begin{cases}
M^{(j)}(x, t, \cdot) \in I + \dot{E}^2(\C \setminus \Gamma^{(j)}),\\
M^{(j)}_+(x,t,k) = M^{(j)}_-(x, t, k) J^{(j)}(x, t, k) \quad \text{for a.e.} \ k \in \Gamma^{(j)},
\end{cases}
\end{align}
for $j = 1$, where the contour $\Gamma^{(1)}$ is displayed in Figure \ref{Gamma1I.pdf} and the jump matrix $J^{(1)}$ is given by
\begin{align}\nonumber
&J_1^{(1)} = \begin{pmatrix} 1 & 0 \\  -h e^{x\tilde{\Phi}} & 1 \end{pmatrix}, 
\qquad  J_3^{(1)} = \begin{pmatrix} 1 & - h^* e^{-x\tilde{\Phi}} \\ 0 & 1 \end{pmatrix}, 
\qquad
J_4^{(1)} = \begin{pmatrix} 1 + r_1 r_1^* & r_1^* e^{-x\tilde{\Phi}} \\
 r_1e^{x\tilde{\Phi}} & 1 \end{pmatrix}.
\end{align}
Here and below $J_i^{(1)}$ denotes the restriction of $J^{(1)}$ to the subcontour labeled by $i$ in Figure \ref{Gamma1I.pdf}.

\begin{figure}
\begin{center}
\begin{overpic}[width=.65\textwidth]{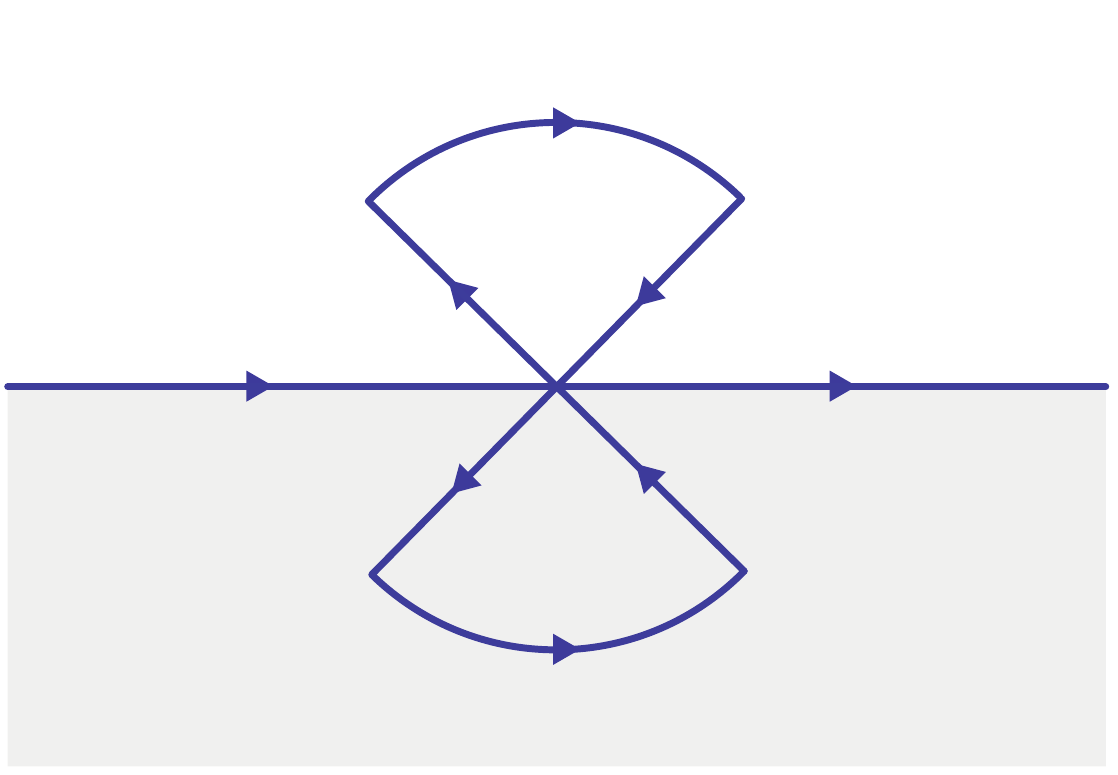}
      \put(102,34){\small $\Gamma^{(1)}$}
      \put(76,60){\small $\re \tilde{\Phi} < 0$}
      \put(76,10){\small $\re \tilde{\Phi} > 0$}
      \put(50,60){\small $1$}
      \put(38,41){\small $1$}
      \put(60,41){\small $1$}
      \put(50,13){\small $3$}
      \put(38,26){\small $3$}
      \put(60,26){\small $3$}
      \put(74,37){\small $4$}
      \put(22,37){\small $4$}
      \put(49,30.5){\small $0$}
    \end{overpic}
     \begin{figuretext}\label{Gamma1I.pdf}
       The contour $\Gamma^{(1)}$ in the complex $k$-plane relevant for Sector I.
     \end{figuretext}
     \end{center}
\end{figure}

\subsubsection{Second transformation} 
In order to deform the jump along the real axis, we need to introduce an analytic approximation of $r_1$.

\begin{lemma}[Analytic approximation of $r_1$ for $1 < \zeta < \infty$]\label{decompositionlemmaI}
There exists a decomposition 
\begin{align*}
 r_1(k) = r_{1,a}(x, t, k) + r_{1,r}(x, t, k), \qquad k \in \R \setminus \{0\}, 
\end{align*}
where the functions $r_{1,a}$ and $r_{1,r}$ have the following properties:
\begin{enumerate}[$(a)$]
\item For each $\zeta \in \mathcal{I}$ and each $x > 0$, $r_{1,a}(x, t, k)$ is defined and continuous for $k \in \bar{V}_1$ and analytic for $k \in V_1$.

\item The function $r_{1,a}$ satisfies
\begin{align}\label{rjaestimatesI}
\begin{cases} 
\bigg|r_{1, a}(x, t, k) - \sum_{j=0}^{N} \frac{r_1^{(j)}(0)k^j}{j!}\bigg| \leq C |k|^{N+1} e^{\frac{x}{4}|\re \tilde{\Phi}(\zeta,k)|}, 
	\\  
|r_{1, a}(x, t, k)| \leq \frac{C}{1 + |k|} e^{\frac{x}{4}|\re \tilde{\Phi}(\zeta,k)|},
\end{cases}   k \in \bar{V}_1, \ \zeta \in \mathcal{I}, \ x > 0,
\end{align}
where the constant $C$ is independent of $x, t, k$.

\item The $L^1$ and $L^\infty$ norms on $\R$ of the function $k \mapsto (1+|k|^{-1}) r_{1,r}(x, t, k)$ are $O(x^{-N-\frac{1}{2}})$ as $x \to \infty$ uniformly with respect to $\zeta \in \mathcal{I}$.

\item $r_{1,a}$ obeys the symmetries 
\begin{align}\label{rsymmetriesI}
r_{1,a}(\zeta, t, k) = \overline{r_{1,a}(\zeta, t, -\bar{k})}, \quad
r_{1,r}(\zeta, t, k) = \overline{r_{1,r}(\zeta, t, -\bar{k})}.
\end{align}
\end{enumerate}
\end{lemma}
\begin{proof}
Let $V_1 = V_1^+ \cup V_1^-$, where $V_1^+$ and $V_1^-$ denote the parts of $V_1$ in the right and left half-planes, respectively. We will derive the decomposition of $r_1$ in $V_1^+$ and then use the symmetry (\ref{rsymmetriesI}) to extend it to $V_1^-$. 
Since $r_1 \in C^{\infty}(\R)$, we have
\begin{align*}
 r_1^{(n)}(k) = \frac{d^n}{dk^n}\bigg(\sum_{j=0}^{2N+2} \frac{r_1^{(j)}(0)}{j!} k^j\bigg) + O(k^{2N+3-n}), \qquad k \to 0, \  n = 0,1,\dots, 2N+3.
\end{align*}
Let $r_{1,j}$ be the coefficients in (\ref{r1expansion2}) describing the behavior of $r_1(k)$ as $k \to \infty$. Define the rational function $f_0(k)$ by
$$f_0(k) = \sum_{j=1}^{2N+5} \frac{a_j}{(k + i)^j},$$
where the complex constants $\{a_j\}_1^{2N+5}$ are chosen so that $f_0(k)$ coincides with $r_1(k)$ to order $2N+2$ as $k \to 0$ and to second order as $k \to \infty$, i.e.,
\begin{align}\label{linearconditionsI}
f_0(k) =
\begin{cases}
\sum_{j=0}^{2N+2} \frac{r_1^{(j)}(0)}{j!} k^j + O(k^{2N+3}), & k \to 0,
	\\
\sum_{j=1}^{2} r_{1,j} k^{-j} + O(k^{-3}), & k \to \infty.
\end{cases}
\end{align}
The coefficients $a_j$ exist and are unique, because (\ref{linearconditionsI}) imposes $2N+5$ independent linear conditions on the $a_j$. 
The function $f(k)$ defined by $f(k) = r_1(k) - f_0(k)$ satisfies
\begin{align}\label{fcoincideI}
 \frac{d^n f}{dk^n} (k) =
\begin{cases}
 O(k^{2N + 3 - n}), \quad& k \to 0, 
	\\
O(k^{-3}), &  k \to \infty, 
 \end{cases}
 \  n = 0,1,\dots, 2N+3.
\end{align}

The decomposition of $r_1(k)$ can now be derived as follows.
For each  $\zeta \in \mathcal{I}$, the map $k \mapsto \phi = \phi(\zeta, k)$ where
\begin{align*}
  \phi = -i\tilde{\Phi}(\zeta, k) = \frac{1}{2}\bigg(k - \frac{1}{k}\bigg) + \frac{1}{2\zeta}\bigg(k + \frac{1}{k}\bigg)
\end{align*}  
is an increasing bijection  $(0, \infty) \to (-\infty,\infty)$, so we may define a function $F(\zeta, \phi)$ by
\begin{align}\label{Fdef2I}
F(\zeta, \phi) = \frac{(k+i)^{N+3}}{k^{N+1}} f(k), \qquad \zeta \in \mathcal{I}, \ \phi \in \R.
\end{align}
For each $\zeta \in \mathcal{I}$, $F(\zeta, \phi)$ is a smooth function of $\phi \in \R$ and 
\begin{align*}
\frac{\partial^n F}{\partial \phi^n}(\zeta, \phi) = \bigg(\frac{1}{\partial \phi/\partial k} \frac{\partial }{\partial k}\bigg)^n 
\bigg[\frac{(k+i)^{N+3}}{k^{N+1}} f(k)\bigg], \qquad \phi \in \R,
\end{align*}
where
$$\frac{\partial \phi}{\partial k}(\zeta, k) = \frac{\zeta - 1 + k^2(\zeta + 1)}{2\zeta k^2}.$$
Using (\ref{fcoincideI}), we see that $\frac{\partial^n F}{\partial \phi^n} = O(k^{N+2-n}) = O(k)$ as $k \to 0$ and 
$\frac{\partial^n F}{\partial \phi^n} = O(k^{-1})$ as $k \to \infty$ uniformly with respect to $\zeta \in \mathcal{I}$ for $n = 0, 1, \dots, N+1$. 
Hence, for $n = 0, 1, \dots, N+1$,
$$\bigg\|\frac{\partial^n F}{\partial \phi^n}(\zeta, \cdot)\bigg\|_{L^2(\R)}^2
= \int_\R \bigg|\frac{\partial^n F}{\partial \phi^n}(\zeta, \phi)\bigg|^2 d\phi
= \int_0^\infty \bigg|\frac{\partial^n F}{\partial \phi^n}(\zeta, \phi)\bigg|^2 \frac{\partial \phi}{\partial k} dk < C, \qquad \zeta \in \mathcal{I}.$$
Thus $F(\zeta, \cdot)$ belongs to the Sobolev space $H^{N+1}(\R)$ and satisfies
\begin{align}\label{FSobolevboundI}
\sup_{\zeta \in \mathcal{I}} \|F(\zeta, \cdot)\|_{H^{N+1}(\R)} < \infty.
\end{align}
We conclude that the Fourier transform $\hat{F}(\zeta, s)$ defined by
\begin{align}\label{FourierdefI}
\hat{F}(\zeta, s) = \frac{1}{2\pi} \int_{\R} F(\zeta, \phi) e^{-i\phi s} d\phi,
\end{align}
satisfies 
\begin{align}\label{FFhatI}
F(\zeta, \phi) =  \int_{\R} \hat{F}(\zeta, s) e^{i\phi s} ds
\end{align}
and $\sup_{\zeta \in \mathcal{I}} \|s^{N+1} \hat{F}(\zeta, s)\|_{L^2(\R)} < \infty$.
By (\ref{Fdef2I}) and (\ref{FFhatI}), we have
$$ \frac{k^{N+1}}{(k+i)^{N+3}}\int_{\R} \hat{F}(\zeta, s) e^{s\tilde{\Phi}(\zeta,k)} ds 
= f(k), \qquad  k > 0, \ \zeta \in \mathcal{I}.$$
Hence we can write 
$$f(k) = f_a(x, t, k) + f_r(x, t, k), \qquad \zeta \in \mathcal{I}, \ x > 0, \ k >0,$$
where the functions $f_a$ and $f_r$ are defined by
\begin{align*}
& f_a(x,t,k) = \frac{k^{N+1}}{(k+i)^{N+3}}\int_{-\frac{x}{4}}^\infty \hat{F}(\zeta,s) e^{s\tilde{\Phi}(\zeta,k)} ds, \qquad \zeta \in \mathcal{I}, \  x > 0, \  k \in \bar{V}_1^+,  
	\\
& f_r(x,t,k) = \frac{k^{N+1}}{(k+i)^{N+3}}\int_{-\infty}^{-\frac{x}{4}} \hat{F}(\zeta,s) e^{s\tilde{\Phi}(\zeta,k)} ds,\qquad \zeta \in \mathcal{I}, \  x > 0, \   k > 0.
\end{align*}
The function $f_a(x, t, \cdot)$ is continuous in $\bar{V}_1^+$ and analytic in $V_1^+$. 
Moreover,
\begin{align*}\nonumber
 |f_a(x, t, k)| 
&\leq \frac{|k|^{N+1}}{|k + i|^{N+3}} \|\hat{F}(\zeta,\cdot)\|_{L^1(\R)}  \sup_{s \geq -\frac{x}{4}} e^{s \re \tilde{\Phi}(\zeta,k)}
\leq \frac{C|k|^{N+1}}{|k + i|^{N+3}}  e^{\frac{x}{4} |\re \tilde{\Phi}(\zeta,k)|} 
	\\ 
& \hspace{5cm} \zeta \in \mathcal{I}, \  x > 0, \  k \in \bar{V}_1^+,
\end{align*}
and
\begin{align*}\nonumber
|f_r(x, t, k)| & \leq \frac{|k|^{N+1}}{|k + i|^{N+3}}  \int_{-\infty}^{-\frac{x}{4}} s^{N+1} |\hat{F}(\zeta,s)| s^{-N-1} ds
	\\\nonumber
& \leq \frac{C|k|^{N+1}}{|k + i|^{N+3}} \| s^{N+1} \hat{F}(\zeta,s)\|_{L^2(\R)} \sqrt{\int_{-\infty}^{-\frac{x}{4}} s^{-2N-2} ds}  
 	\\ \nonumber
&  \leq \frac{C|k|^{N+1}}{|k + i|^{N+3}}x^{-N - 1/2}, \qquad \zeta \in \mathcal{I}, \  x > 0, \  k > 0.
\end{align*}
Hence the $L^1$ and $L^\infty$ norms of $(1 + |k|^{-1})f_r$ on $(0, \infty)$ are $O(x^{-N - 1/2})$ uniformly with respect to $\zeta \in \mathcal{I}$. Letting
\begin{align*}
& r_{1,a}(x, t, k) = f_0(\zeta, k) + f_a(x, t, k), \qquad k \in \bar{V}_1^+,
	\\
& r_{1,r}(x, t, k) = f_r(x, t, k), \qquad k > 0.
\end{align*}
we find a decomposition of $r_1$ for $k > 0$ with the properties listed in the statement of the lemma. We use the symmetry (\ref{rsymmetriesI}) to extend this decomposition to $k < 0$.
\end{proof}

We introduce $m(x,t,k)$ by
$$m(x,t,k) = M^{(1)}(x,t,k)H(x,t,k),$$
where the sectionally analytic function $H$ is defined by
\begin{align*}
H = \begin{cases} 
\begin{pmatrix} 1  & 0 \\ -r_{1,a} e^{x\tilde{\Phi}} & 1 \end{pmatrix}, & k \in V_1,
	\\
\begin{pmatrix} 1  & r_{1,a}^* e^{-x\tilde{\Phi}} \\ 0& 1 \end{pmatrix}, & k \in V_2,
	\\
I, & \text{elsewhere}.	
\end{cases}
\end{align*}

\begin{figure}
\begin{center}
\begin{overpic}[width=.65\textwidth]{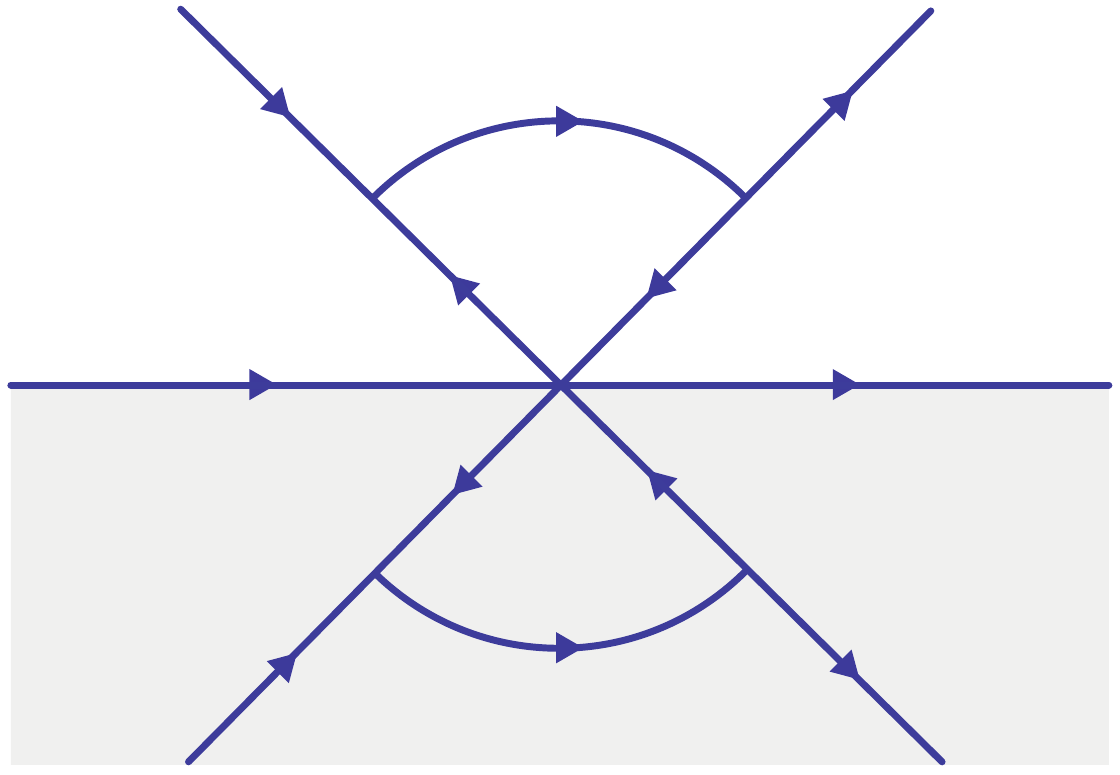}
      \put(101.5,33.7){\small $\Sigma$}
      \put(44,66){\small $\re \tilde{\Phi} < 0$}
      \put(44,5){\small $\re \tilde{\Phi} > 0$}
       \put(38,41.5){\small $1$}
      \put(60.8,41.5){\small $1$}
      \put(38,26){\small $3$}
      \put(61,26){\small $3$}
      \put(22,37){\small $4$}
      \put(75,37){\small $4$}
     \put(77,10){\small $5$}
      \put(22,10){\small $5$}
     \put(77,58){\small $6$}
      \put(21,58){\small $6$}
       \put(50,54){\small $7$}
      \put(50,13){\small $8$}
    \end{overpic}
     \begin{figuretext}\label{SigmaI.pdf}
       The contour $\Sigma$ in the complex $k$-plane with the region $\re \tilde{\Phi} > 0$ shaded. 
     \end{figuretext}
     \end{center}
\end{figure}
By Lemma \ref{decompositionlemmaI}, we have  
$$H(x,t,\cdot)^{\pm1} \in I + (\dot{E}^2 \cap E^\infty)(\C \setminus \Sigma),$$
where $\Sigma \subset \C$ denotes the contour displayed in Figure \ref{SigmaI.pdf}.
It follows that $M$ satisfies the RH problem (\ref{RHM}) iff $m$ 
satisfies the RH problem 
\begin{align}\label{RHmSigma}
\begin{cases}
m(x, t, \cdot) \in I + \dot{E}^2(\C \setminus \Sigma),\\
m_+(x,t,k) = m_-(x, t, k) v(x, t, k) \quad \text{for a.e.} \ k \in \Sigma,
\end{cases}
\end{align} 
where the jump matrix $v$ is given by
\begin{align*}\nonumber
&v_1 = \begin{pmatrix} 1  & 0 \\ -(r_{1,a}+h) e^{x\tilde{\Phi}} & 1 \end{pmatrix},
	\\\nonumber
&v_3 =  \begin{pmatrix} 1  & -(r_{1,a}^* + h^*) e^{-x\tilde{\Phi}} \\ 0& 1 \end{pmatrix},
&&
v_4 = \begin{pmatrix} 1+ |r_{1,r}|^2  & r_{1,r}^* e^{-x\tilde{\Phi}} \\ r_{1,r}e^{x\tilde{\Phi}} & 1 \end{pmatrix},
	\\\nonumber
&v_5 = \begin{pmatrix} 1  & r_{1,a}^* e^{-x\tilde{\Phi}} \\ 0& 1 \end{pmatrix},
&&
v_6 = \begin{pmatrix} 1  & 0 \\ r_{1,a} e^{x\tilde{\Phi}} & 1 \end{pmatrix},
	\\
&v_7 = \begin{pmatrix} 1 & 0 \\  - h e^{x\tilde{\Phi}} & 1 \end{pmatrix},
&&
v_8 = \begin{pmatrix} 1 & - h^* e^{-x\tilde{\Phi}} \\ 0 & 1 \end{pmatrix}.
\end{align*}

\begin{lemma}\label{w2lemmaI}
The function $w := v - I$ satisfies
$$\|(1+|k|^{-1})w\|_{(L^1 \cap L^\infty)(\Sigma)} \leq Cx^{-N},$$
uniformly for $\zeta \in \mathcal{I}$ and $x > 1$.
\end{lemma}
\begin{proof}
Let $\Sigma_j$ denote the part of $\Sigma$ labeled by $j$ in Figure \ref{SigmaI.pdf}.
The $L^1$  and $L^\infty$ norms of $(1+|k|^{-1})w$ on $\cup_{j=5}^8 \Sigma_j$ are $O(e^{-cx})$ because $|\re \tilde{\Phi}(\zeta, k)| > c > 0$ is uniformly bounded away from zero on these subcontours. Moreover,  $\|(1+|k|^{-1}) w\|_{(L^1 \cap L^\infty)(\R)}  \leq Cx^{-N}$ by Lemma \ref{decompositionlemmaI}.
By symmetry, it is therefore enough to show that   
$$\|(1+|k|^{-1})w\|_{(L^1 \cap L^\infty)(l_1)} \leq Cx^{-N},$$
where $l_1 = \{ue^{\frac{\pi i}{4}} \, | \, 0 < u < 1\}$ denotes the part of $\Sigma_1$ that lies in the right half-plane. 

Since $r(k) = r_1(k) + h(k)$ vanishes to all orders at $k = 0$, the estimate (\ref{rjaestimatesI}) gives
\begin{align} \nonumber
|r_{1,a}(x,t,k) + h(k)| & = \bigg|r_{1,a}(x,t,k) - \sum_{j=0}^{N} \frac{r_1^{(j)}(0)k^j}{j!} + h(k) - \sum_{j=0}^{N} \frac{h^{(j)}(0)k^j}{j!}\bigg|
	\\ \nonumber
& \leq C |k|^{N+1} e^{\frac{x}{4}|\re \tilde{\Phi}(\zeta,k)|} + C|k|^{N+1}
	\\ \label{r1ahboundI}
& \leq C |k|^{N+1} e^{\frac{x}{4}|\re \tilde{\Phi}(\zeta,k)|}, \qquad \zeta \in \mathcal{I}, \ k \in l_1.
\end{align}
For $k \in l_1$, only the $(21)$ entry of $w$ is nonzero and, using (\ref{rjaestimatesI}) and (\ref{r1ahboundI}),
\begin{align*}
|(w(x,t,k))_{21}| & = |(r_{1,a}(x,t,k) + h(k)) e^{x\tilde{\Phi}(\zeta, k)}|
 \leq C |r_{1,a}(x,t,k) + h(k)| e^{-x|\re \tilde{\Phi}|}
	\\
& \leq C|k|^{N+1} e^{-\frac{3x}{4}|\re \tilde{\Phi}|}, \qquad \zeta \in \mathcal{I}, \ k \in l_1.
\end{align*}
Now
$$\re \tilde{\Phi}(\zeta, ue^{\frac{\pi i}{4}}) = -\frac{\zeta - 1 +(\zeta +1) u^2}{2 \sqrt{2} \zeta  u}, \qquad \zeta \in \mathcal{I}, \ u \in \R.$$
Hence
$$|\re \tilde{\Phi}(\zeta, ue^{\frac{\pi i}{4}})| = \frac{\zeta - 1}{2 \sqrt{2} \zeta u}
+ \frac{(\zeta +1) u}{2 \sqrt{2} \zeta }, \qquad \zeta \in \mathcal{I}, \ u \in (0,1).$$
We can now estimate the $L^\infty$  norm on $l_1$:
\begin{align}\nonumber
& \|(1+|k|^{-1})(w(x,t,k))_{21}\|_{L^\infty(l_1)} 
 \leq 2\|k^{-1}(w(x,t,k))_{21}\|_{L^\infty(l_1)} 
	\\
&\leq C \sup_{u \in (0,1)} u^N e^{-cx\frac{\zeta - 1}{\zeta}  - cx\frac{(\zeta +1)u}{\zeta}}
 \leq C e^{-cx\frac{\zeta - 1}{\zeta}} \Big(\frac{\zeta}{\zeta +1}\Big)^N x^{-N}
\leq C x^{-N}, \qquad \zeta \in \mathcal{I}.
\end{align}
Since $l_1$ has finite length, the estimate of the $L^1$ norm follows immediately. This completes the proof of the lemma. 
\end{proof}

By Lemma \ref{w2lemmaI}, we have
\begin{align}\label{CwnormI}
\|\mathcal{C}_{w}\|_{\mathcal{B}(L^2(\Sigma))} \leq C \|w\|_{L^\infty(\Sigma)}  
\leq Cx^{-N}, \qquad x > 2, \ \zeta \in \mathcal{I},
\end{align}
where the operator $\mathcal{C}_{w}: L^2(\Sigma) + L^\infty(\Sigma) \to L^2(\Sigma)$ is defined by
$\mathcal{C}_{w}f = \mathcal{C}_-(f w)$.
Thus there exists an $X > 1$ such that whenever $\zeta \in \mathcal{I}$ and $x > X$, the operator $I - \mathcal{C}_{w(\zeta, t, \cdot)} \in \mathcal{B}(L^2(\Sigma))$ is invertible and the RH problem (\ref{RHmSigma}) has a unique solution $m \in I + \dot{E}^2(\hat{\C} \setminus \Sigma)$ given by
\begin{align}\label{mrepresentationI}
m(x, t, k) = I + \mathcal{C}(\mu w) = I + \frac{1}{2\pi i}\int_{\Sigma} \mu(x, t, s) w(x, t, s) \frac{ds}{s - k},
\end{align}
where $\mu(x, t, k) \in I + L^2(\Sigma)$ is defined by (\ref{mudef}).
Equation (\ref{mrepresentationI}) implies that the function $\hat{M}(x,t)$ defined in (\ref{m0def}) admits the representation
\begin{align}\label{hatmrepI}
\hat{M}(x,t) - I
= \ntlim_{k\to 0} m(x,t,k)  - I
= \frac{1}{2\pi i}\int_{\Sigma} (\mu w)(x, t, k) \frac{dk}{k}.
\end{align}
By Lemma \ref{w2lemmaI} and (\ref{CwnormI}),
\begin{align}\label{muestimateI}
\|\mu(x,t,\cdot) - I\|_{L^2(\Sigma)} \leq  
\frac{C\|w\|_{L^2(\Sigma)}}{1 - \|\mathcal{C}_{w}\|_{\mathcal{B}(L^2(\Sigma))}}
\leq Cx^{-N}, \qquad x > X, \ \zeta \in \mathcal{I}.
\end{align}
Using Lemma \ref{w2lemmaI} and (\ref{muestimateI}), we can estimate (\ref{hatmrepI}) as follows as $x \to \infty$:
\begin{align}\nonumber
\hat{M}(x,t) - I
& = \int_{\Sigma} w(x,t,k) \frac{dk}{k}
+
\int_{\Sigma} (\mu(x,t,k) - I) w(x,t,k) \frac{dk}{k}
	\\ \label{hatMsectorI}
& = O(\|k^{-1}w\|_{L^1(\Sigma)})
+ O(\|\mu - I\|_{L^2(\Sigma)}\|k^{-1}w\|_{L^2(\Sigma)})
 = O(x^{-N}),
\end{align}
uniformly with respect to $\zeta \in \mathcal{I}$.
Substituting these asymptotics into the definition (\ref{ulim}) of $u(x,t)$, we obtain
\begin{align*}
  u(x,t) & = 2 \arg\big(1 + O(x^{-N})\big) = O(x^{-N}), \qquad x \to \infty,
\end{align*}
again uniformly with respect to $\zeta \in \mathcal{I}$. This shows that the estimate (\ref{uasymptoticsI}) holds uniformly for $\zeta > 1$; since $u$ is continuous it must hold also for $\zeta = 1$, completing the proof of the asymptotics in Sector I.

\begin{remark}\label{MlambdajIremark}\upshape
For the purposes of Section \ref{solitonsec}, we note that if $\lambda_j$ is a point in $\C \setminus \Gamma$, then by modifying the contour deformations slightly if necessary, we may assume that $\dist(\lambda_j, \Sigma) > 0$. Indeed, the contour $\Sigma$  can for example be replaced by a contour such as the one shown in Figure \ref{SigmaIdeformed.pdf}.
The same arguments that gave (\ref{hatMsectorI}) then imply
\begin{align}\label{MlambdajI}
  M(x,t,\lambda_j) 
= I + O(x^{-N})
\end{align}
uniformly for $\zeta \in (1, \infty)$ as $x \to \infty$. Equation (\ref{MlambdajI}) will be important for the proof of Theorem \ref{solitonasymptoticsth}.
\end{remark}

\begin{figure}
\begin{center}
\begin{overpic}[width=.65\textwidth]{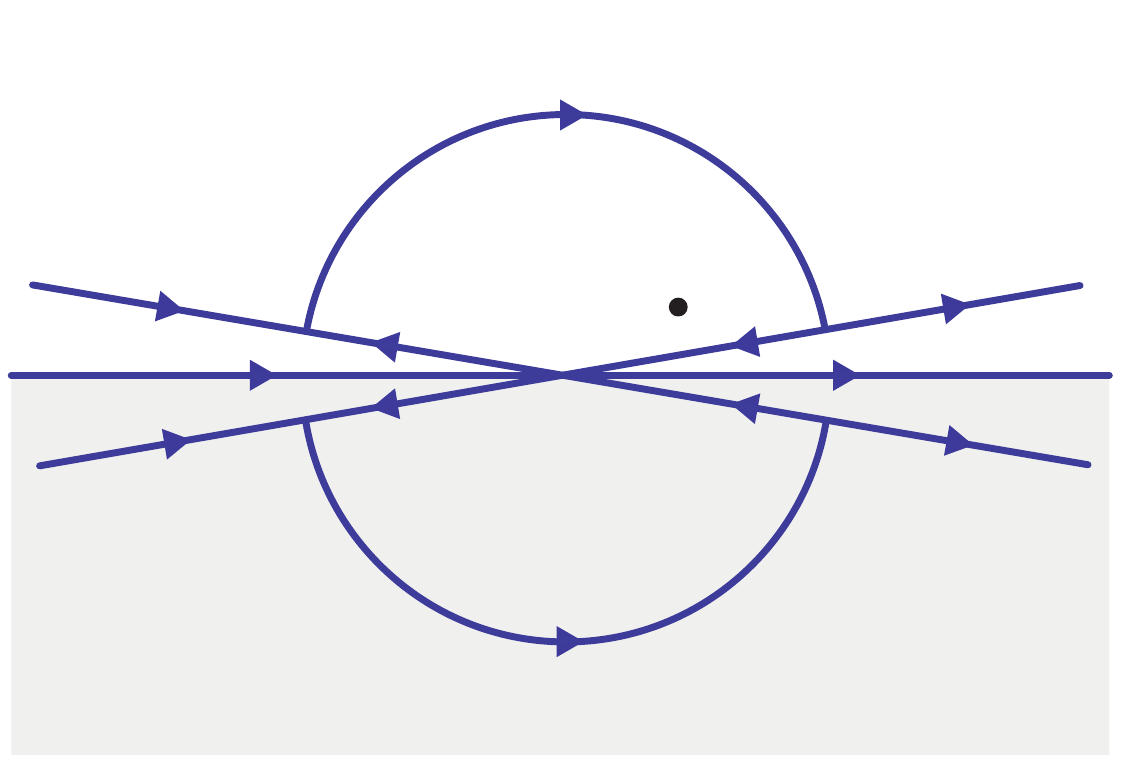}
      \put(101.5,33.7){\small $\Sigma$}
             \put(76,60){\small $\re \tilde{\Phi} < 0$}
      \put(76,10){\small $\re \tilde{\Phi} > 0$}
      \put(59.4,43.5){\small $\lambda_j$}
    \end{overpic}
     \begin{figuretext}\label{SigmaIdeformed.pdf}
       Example of a contour $\Sigma$ which avoids the point $\lambda_j \in \C \setminus \Gamma$.   
          \end{figuretext}
     \end{center}
\end{figure}

\section{Proof of Theorem \ref{asymptoticsth}: Asymptotics in Sector II}\label{sectorIIsec}
Let $\mathcal{I} = [1/2,1)$ and suppose $\zeta \in \mathcal{I}$. 
Let $N \geq 1$ be an integer. 

\subsection{Transformations of the RH problem}

\subsubsection{First transformation}\label{firsttransformationsubsecII}
The first transformation is the same as in Sector I, except that the sets $\{U_j, V_j\}_1^2$ now depend on $\zeta$ and are defined as in Figure \ref{VjsII.pdf}. Defining $M^{(1)}(x,t,k)$ by (\ref{M1defI}), we find that $M$ satisfies the RH problem (\ref{RHM}) iff $M^{(1)}$ 
satisfies the RH problem (\ref{RHMj}) with $j = 1$, where the contour $\Gamma^{(1)}$ is displayed in Figure \ref{Gamma1II.pdf} and the jump matrix $J^{(1)}$ is given by 
\begin{align}\nonumber
&J_1^{(1)} = \begin{pmatrix} 1 & 0 \\  -h e^{t\Phi} & 1 \end{pmatrix}, 
\qquad
J_2^{(1)} = \begin{pmatrix} 1 & - r^* e^{-t\Phi} \\
- re^{t\Phi}& 1 + r r^* \end{pmatrix}, 
	\\ \label{J1def}
& J_3^{(1)} = \begin{pmatrix} 1 & - h^* e^{-t\Phi} \\ 0 & 1 \end{pmatrix}, 
\qquad
J_4^{(1)} = \begin{pmatrix} 1 + r_1 r_1^* & r_1^* e^{-t\Phi} \\
 r_1e^{t\Phi} & 1 \end{pmatrix}.
\end{align}

\begin{figure}
\begin{center}
\begin{overpic}[width=.45\textwidth]{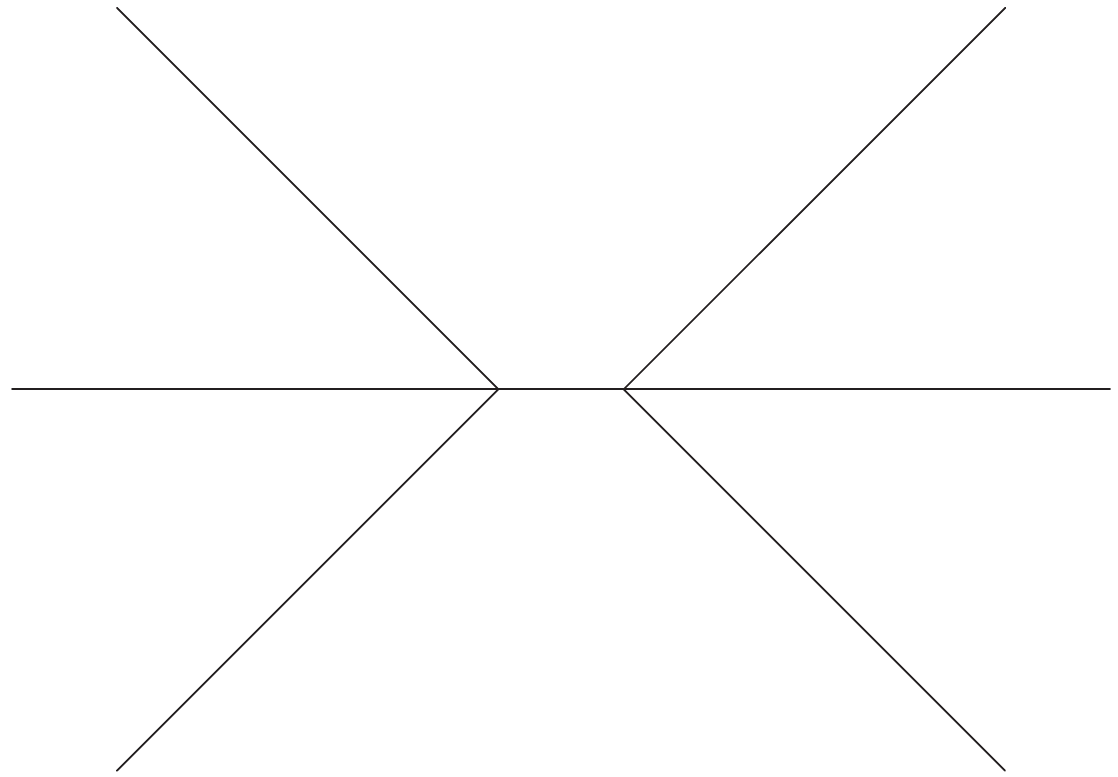}
      \put(80,45){\small $V_1$}
      \put(17,45){\small $V_1$}
        \put(80,23){\small $V_2$}
      \put(17,23){\small $V_2$}
      \put(41,29){\small $-k_0$}
      \put(54,29){\small $k_0$}
    \end{overpic}
    \qquad
    \begin{overpic}[width=.45\textwidth]{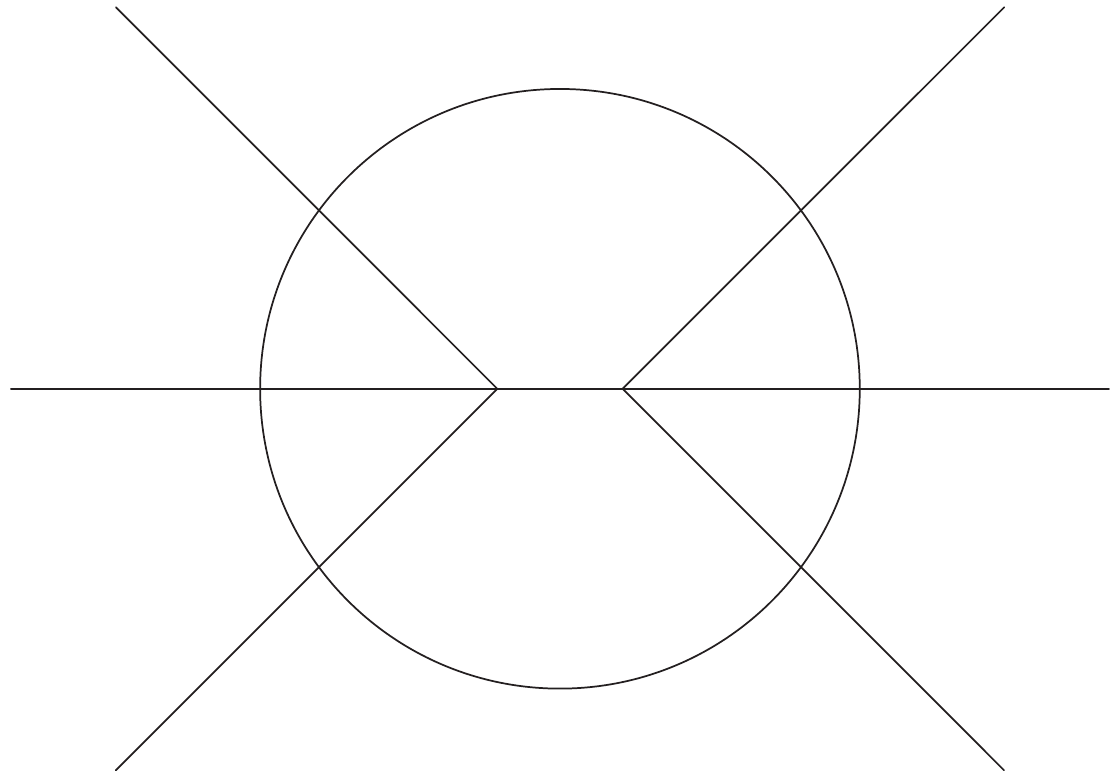}
      \put(66,39){\small $U_1$}
      \put(28,39){\small $U_1$}
        \put(66,28){\small $U_2$}
      \put(28,28){\small $U_2$}
      \put(41,29){\small $-k_0$}
      \put(54,29){\small $k_0$}
      \put(78,30){\small $1$}
      \put(15,30){\small $-1$}
    \end{overpic}\bigskip
     \begin{figuretext}\label{VjsII.pdf}
       The open subsets $\{V_j, U_j\}_1^2$ of the complex $k$-plane relevant for Sector II. 
     \end{figuretext}
     \end{center}
\end{figure}

\begin{figure}
\begin{center}
\begin{overpic}[width=.65\textwidth]{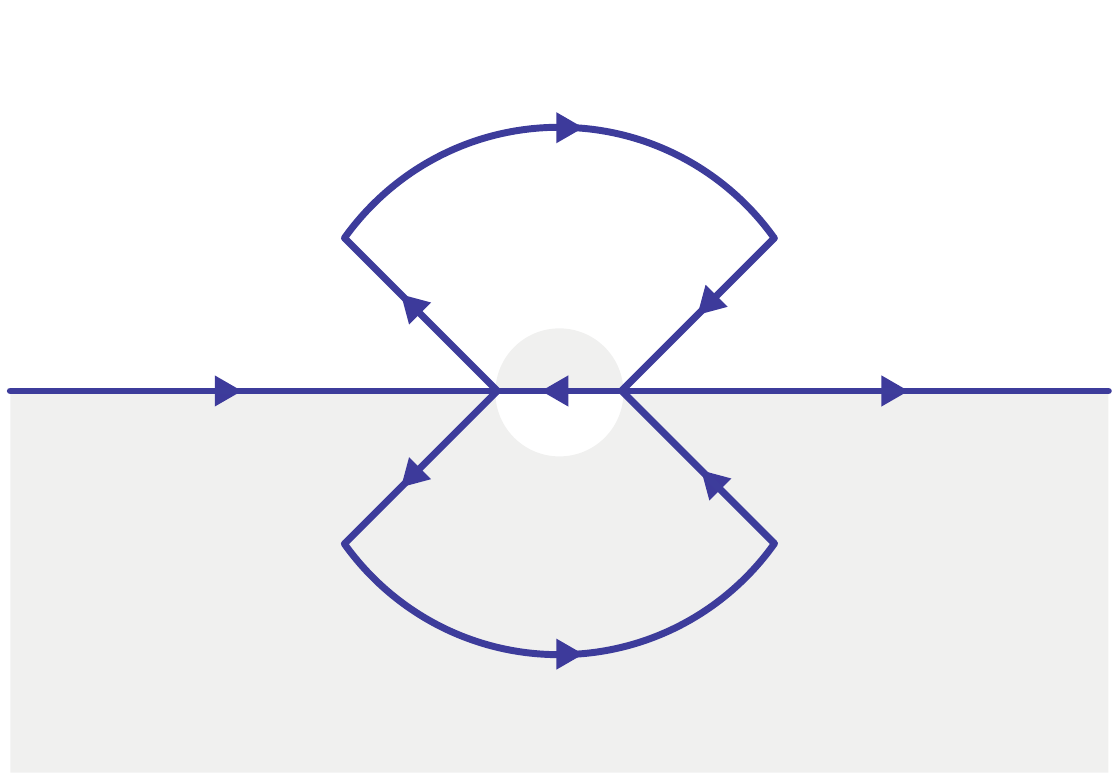}
      \put(102,34){\small $\Gamma^{(1)}$}
      \put(76,60){\small $\re \Phi < 0$}
      \put(76,10){\small $\re \Phi > 0$}
      \put(49,61){\small $1$}
      \put(31,41){\small $1$}
      \put(67,41){\small $1$}
      \put(49,37){\small $2$}
      \put(49,13.5){\small $3$}
      \put(31,26.5){\small $3$}
      \put(67,26){\small $3$}
      \put(86,37){\small $4$}
      \put(13,37){\small $4$}
      \put(43,31.5){\small $-k_0$}
      \put(53.5,31.5){\small $k_0$}
    \end{overpic}
     \begin{figuretext}\label{Gamma1II.pdf}
       The contour $\Gamma^{(1)}$ in the complex $k$-plane relevant for Sector II. 
     \end{figuretext}
     \end{center}
\end{figure}

\subsubsection{Second transformation} 
The second transformation consists of deforming the jump along the real axis. To this end, we need an analytic approximation of $r_1$.

\begin{lemma}[Analytic approximation of $r_1$ for $0 \leq \zeta < 1$]\label{decompositionlemmaII}
There exists a decomposition 
\begin{align*}
 r_1(k) = r_{1,a}(x, t, k) + r_{1,r}(x, t, k), \qquad k \in (-\infty, k_0] \cup [k_0, \infty), 
\end{align*}
where the functions $r_{1,a}$ and $r_{1,r}$ have the following properties:
\begin{enumerate}[$(a)$]
\item For each $\zeta \in \mathcal{I}$ and each $t > 0$, $r_{1,a}(x, t, k)$ is defined and continuous for $k \in \bar{V}_1$ and analytic for $k \in V_1$.

\item The function $r_{1,a}$ satisfies
\begin{align}\label{rjaestimatesII}
\begin{cases} 
\big|r_{1, a}(x, t, k) - \sum_{j=0}^{N} \frac{r_1^{(j)}(k_0)(k-k_0)^j}{j!}\big| \leq C |k - k_0|^{N+1} e^{\frac{t}{4}|\re \Phi(\zeta,k)|}, 
	\\  
|r_{1, a}(x, t, k)| \leq \frac{C}{1 + |k|} e^{\frac{t}{4}|\re \Phi(\zeta,k)|},
\end{cases}   k \in \bar{V}_1, \ \zeta \in \mathcal{I}, \ t > 0.
\end{align}

\item The $L^1$ and $L^\infty$ norms on $(-\infty,-k_0) \cup (k_0, \infty)$ of the function $k \mapsto (1+|k|^{-1}) r_{1,r}(x, t, k)$ are $O(t^{-N})$ as $t \to \infty$ uniformly with respect to $\zeta \in \mathcal{I}$.

\item $r_{1,a}$ and $r_{1,r}$ obey the symmetries (\ref{rsymmetriesI}).

\end{enumerate}
\end{lemma}
\begin{proof}
The proof is similar to that of Lemma \ref{decompositionlemmaI}.
\end{proof}

We introduce $m(x,t,k)$ by
$$m(x,t,k) = M^{(1)}(x,t,k)H(x,t,k),$$
where the sectionally analytic function $H$ is defined by
\begin{align*}
H = \begin{cases} 
\begin{pmatrix} 1  & 0 \\ -r_{1,a} e^{t\Phi} & 1 \end{pmatrix}, & k \in V_1,
	\\
\begin{pmatrix} 1  & r_{1,a}^* e^{-t\Phi} \\ 0& 1 \end{pmatrix}, & k \in V_2,
	\\
I, & \text{elsewhere}.	
\end{cases}
\end{align*}
\begin{figure}
\begin{center}
\begin{overpic}[width=.65\textwidth]{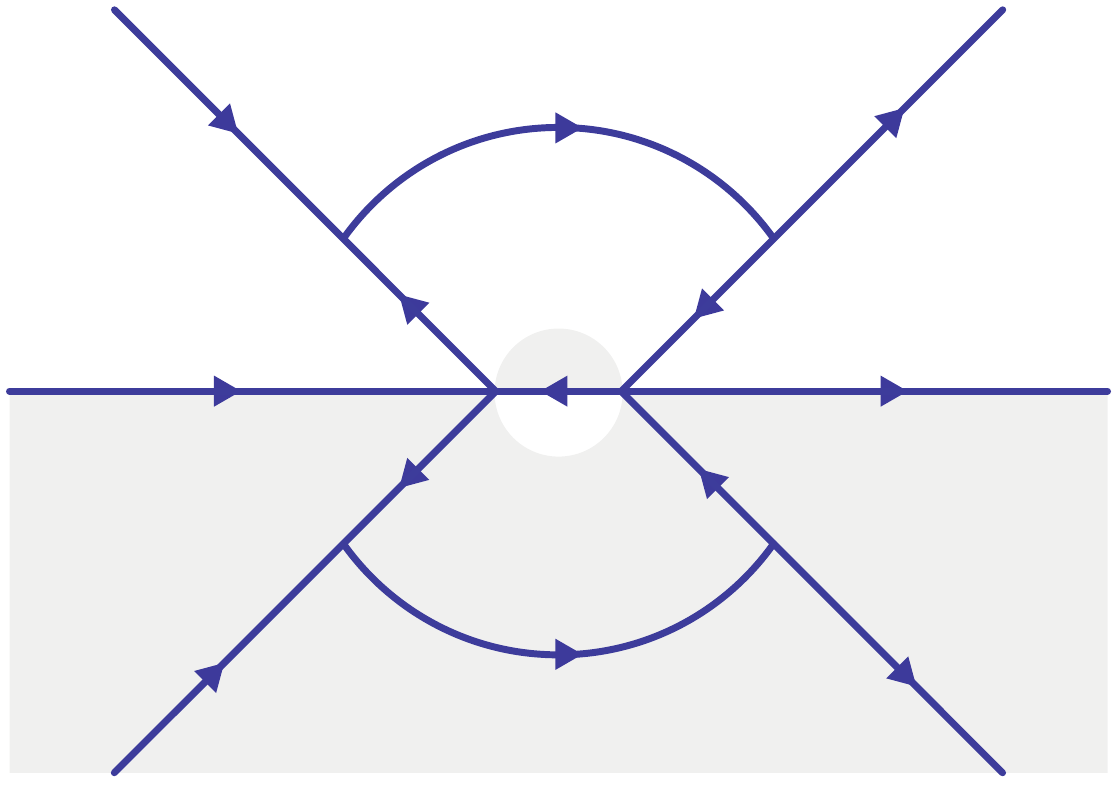}
      \put(102,34){\small $\Sigma$}
      \put(44,66){\small $\re \Phi < 0$}
      \put(44,5){\small $\re \Phi > 0$}
       \put(33.5,40.5){\small $1$}
      \put(65,40.5){\small $1$}
      \put(49,37){\small $2$}
      \put(33.5,27.5){\small $3$}
      \put(66,27){\small $3$}
      \put(19,37){\small $4$}
      \put(79,37){\small $4$}
     \put(82.5,10){\small $5$}
      \put(15,10){\small $5$}
     \put(81.5,57){\small $6$}
      \put(16,57.5){\small $6$}
       \put(50,54.5){\small $7$}
      \put(50,13.5){\small $8$}
    \end{overpic}
     \begin{figuretext}\label{SigmaII.pdf}
       The contour $\Sigma$ in the complex $k$-plane relevant for Sector II with the region $\re \Phi > 0$ shaded. 
     \end{figuretext}
     \end{center}
\end{figure}
By Lemma \ref{decompositionlemmaII}, we have  
$$H(x,t,\cdot)^{\pm1} \in I + (\dot{E}^2 \cap E^\infty)(\C \setminus \Sigma),$$
where $\Sigma \subset \C$ denotes the contour displayed in Figure \ref{SigmaII.pdf}.
It follows that $M$ satisfies the RH problem (\ref{RHM}) iff $m$ 
satisfies the RH problem (\ref{RHmSigma}), where the jump matrix $v$ is given by 
\begin{align*}\nonumber
&v_1 = \begin{pmatrix} 1  & 0 \\ -(r_{1,a}+h) e^{t\Phi} & 1 \end{pmatrix},
&&
v_2 = \begin{pmatrix} 1 & - r^* e^{-t\Phi} \\
- re^{t\Phi}& 1 + r r^* \end{pmatrix}, 
	\\\nonumber
&v_3 =  \begin{pmatrix} 1  & -(r_{1,a}^* + h^*) e^{-t\Phi} \\ 0& 1 \end{pmatrix},
&&
v_4 = \begin{pmatrix} 1+ |r_{1,r}|^2  & r_{1,r}^* e^{-t\Phi} \\ r_{1,r}e^{t\Phi} & 1 \end{pmatrix},
	\\\nonumber
&v_5 = \begin{pmatrix} 1  & r_{1,a}^* e^{-t\Phi} \\ 0& 1 \end{pmatrix},
&&
v_6 = \begin{pmatrix} 1  & 0 \\ r_{1,a} e^{t\Phi} & 1 \end{pmatrix},
	\\
&v_7 = \begin{pmatrix} 1 & 0 \\  - h e^{t\Phi} & 1 \end{pmatrix},
&&
v_8 = \begin{pmatrix} 1 & - h^* e^{-t\Phi} \\ 0 & 1 \end{pmatrix}.
\end{align*}

\begin{lemma}\label{w2lemmaII}
The function $w = v - I$ satisfies
\begin{align}\label{kinvwL1Linfty}
\|(1+|k|^{-1})w\|_{(L^1 \cap L^\infty)(\Sigma)} \leq C(t^{-N} + k_0^N),
\end{align}
uniformly for $\zeta \in \mathcal{I}$ and $t > 1$.
\end{lemma}
\begin{proof}
The $L^1$  and $L^\infty$ norms of $(1+|k|^{-1})w$ on $\cup_{j=5}^8 \Sigma_j$ are $O(e^{-ct})$ because $|\re \Phi(\zeta, k)| > c > 0$ is uniformly bounded away from zero on these subcontours. By Lemma \ref{decompositionlemmaII}, $\|(1+|k|^{-1})w\|_{(L^1 \cap L^\infty)(\Sigma_4)}  \leq Ct^{-N}$. 
Moreover, for $k \in \Sigma_2$, we have $\re \Phi = 0$ and hence
$$|w(x,t,k)| \leq C |r(k)|, \qquad k \in \Sigma_2.$$
Since $r(k)$ vanishes to all orders at $k = 0$, this gives $\|(1+|k|^{-1})w\|_{(L^1 \cap L^\infty)(\Sigma_2)} \leq C k_0^N$.

It remains to show that the contributions from $\Sigma_1$ and $\Sigma_3$ also satisfy (\ref{kinvwL1Linfty}). We give the proof for the part of $\Sigma_1$ that lies in the right half-plane; similar arguments apply to the other parts of $\Sigma_1 \cup \Sigma_3$.

Let $l_1 = \Sigma_1 \cap \{\re k > 0\}$. 
Recalling that $r = r_1 + h$ on $[-1,1]$ and that $r(k)$ vanishes to all orders at $k = 0$, the first inequality in (\ref{rjaestimatesII}) gives
\begin{align*} \nonumber
 |r_{1,a}(x,t,k) + h(k)| = &\; \bigg|r_{1,a}(x,t,k) - \sum_{j=0}^{N} \frac{r_1^{(j)}(k_0)(k-k_0)^j}{j!} 
	\\\nonumber
 &+ h(k) - \sum_{j=0}^{N} \frac{h^{(j)}(k_0)(k-k_0)^j}{j!}
 + \sum_{j=0}^{N} \frac{r^{(j)}(k_0) (k-k_0)^j}{j!} \bigg|
	\\ \nonumber
\leq &\; C |k - k_0|^{N+1} e^{\frac{t}{4}|\re \Phi(\zeta,k)|} 
+ C|k-k_0|^{N+1} + C|k_0|^{N+1}
	\\ 
\leq &\; C |k|^{N+1} e^{\frac{t}{4}|\re \Phi(\zeta,k)|}, \qquad \zeta \in \mathcal{I}, \ k \in l_1.
\end{align*}
Hence
\begin{align*}
|w(x,t,k) | = &\; |r_{1,a}(x,t,k) + h(k)|e^{t\re \Phi(\zeta, k)}
\leq  C |k|^{N+1} e^{-\frac{3t}{4}|\re \Phi(\zeta,k)|}, \qquad k \in l_1.
\end{align*}
But
\begin{align}\label{rePhionL1}
\re \Phi(\zeta, k_0 + ue^{\frac{\pi i}{4}}) = \frac{-u^2(2 k_0 + \sqrt{2} u)}{2 (k_0^2+1)(k_0^2+\sqrt{2} k_0 u+u^2)}, \qquad u \in \R,
\end{align}
which means that
\begin{align*}
|\re \Phi(\zeta, k_0 + ue^{\frac{\pi i}{4}})| \geq \frac{u^2(k_0 + u)}{4(k_0^2+\sqrt{2} k_0 u+u^2)}
\geq \frac{k_0 + u}{16}, \qquad \zeta \in \mathcal{I}, \ k_0 \leq u < \infty.
\end{align*}
Hence
$$|w(x,t,k)| \leq \begin{cases} C|k|^{N+1} e^{-\frac{3t}{4}\frac{k_0 + |k-k_0|}{16}}, & |k-k_0| \geq k_0, \\
C|k|^{N+1}, & |k - k_0| < k_0,
\end{cases} \quad k \in l_1,$$
so
\begin{align*}
& \|(1+|k|^{-1})w(x,t,\cdot)\|_{L^\infty(l_1)}
 \leq 2\|k^{-1}w(x,t,\cdot)\|_{L^\infty(l_1)}
	\\
& \leq C \sup_{0 \leq u \leq k_0}(k_0+u)^N
+ C \sup_{k_0 \leq u < \infty} (k_0+u)^N e^{-\frac{3t}{4}\frac{k_0 + u}{16}} 
 \leq C(k_0^N + t^{-N}),
\end{align*}
which finishes the proof.
\end{proof}

The asymptotics in Sector II follows from Lemma \ref{w2lemmaII} in the same way that the asymptotics in Sector I followed from Lemma \ref{w2lemmaI}. This completes the proof of the asymptotic formula (\ref{uasymptoticsII}) for $u(x,t)$ for $0 \leq \zeta < 1$; the case $\zeta = 1$ follows from the continuity of $u(x,t)$.

\begin{remark}\label{MlambdajIIremark}\upshape
For the purposes of Section \ref{solitonsec}, we note that if $\lambda_j$ is a point in $\C \setminus \Gamma$, then we can ensure that $\inf_{\zeta \in \mathcal{I}} \dist(\lambda_j, \Sigma(\zeta)) > 0$ by modifying the contour deformations slightly if necessary (cf. Remark \ref{MlambdajIremark} and Figure \ref{SigmaIdeformed.pdf}).
The above estimates then imply, for each fixed $\lambda_j \in \C \setminus \Gamma$
\begin{align}\label{MlambdajII}
  M(x,t,\lambda_j) 
= I + O(k_0^N + t^{-N})
\end{align}
uniformly for $\zeta \in [1/2, 1)$ as $t \to \infty$. 
\end{remark}

\section{Proof of Theorem \ref{asymptoticsth}: Asymptotics in Sector III}\label{sectorIIIsec}
Let $\mathcal{I} = [0,1)$ and suppose $\zeta \in \mathcal{I}$. Recall that the jump matrix $J$ admits two real critical points located at $\pm k_0$, where $k_0 := k_0(\zeta) \in (0,1]$ is defined by (\ref{k0def}).

\subsection{Transformations of the RH problem}
Let $V_j := V_j(\zeta)$, $j = 1, \dots, 6$, denote the open subsets of $\C$ displayed in Figure \ref{VjsIII.pdf}. We define $U_j := U_j(\zeta)$ for $j = 1,2$ by $U_1 = V_1 \cap \{|k| < 1\}$ and $U_2 = V_4 \cap \{|k| < 1\}$.

\subsubsection{First transformation}\label{firsttransformationsubsecIII}
The first transformation is the same as in sectors I and II, which means that we define $M^{(1)}(x,t,k)$ by (\ref{M1defI}).
Since $h e^{t\Phi}$ and $h^* e^{-t\Phi}$ are bounded and analytic functions of $k \in U_1$ and $k \in U_2$, respectively, we infer that $M$ satisfies the RH problem (\ref{RHM}) iff $M^{(1)}$ 
satisfies the RH problem (\ref{RHMj}) with $j = 1$, where $\Gamma^{(1)}$  is shown in Figure \ref{Gamma1III.pdf} and the jump matrix $J^{(1)}$ is given by (\ref{J1def}).

\begin{figure}
\begin{center}
\begin{overpic}[width=.65\textwidth]{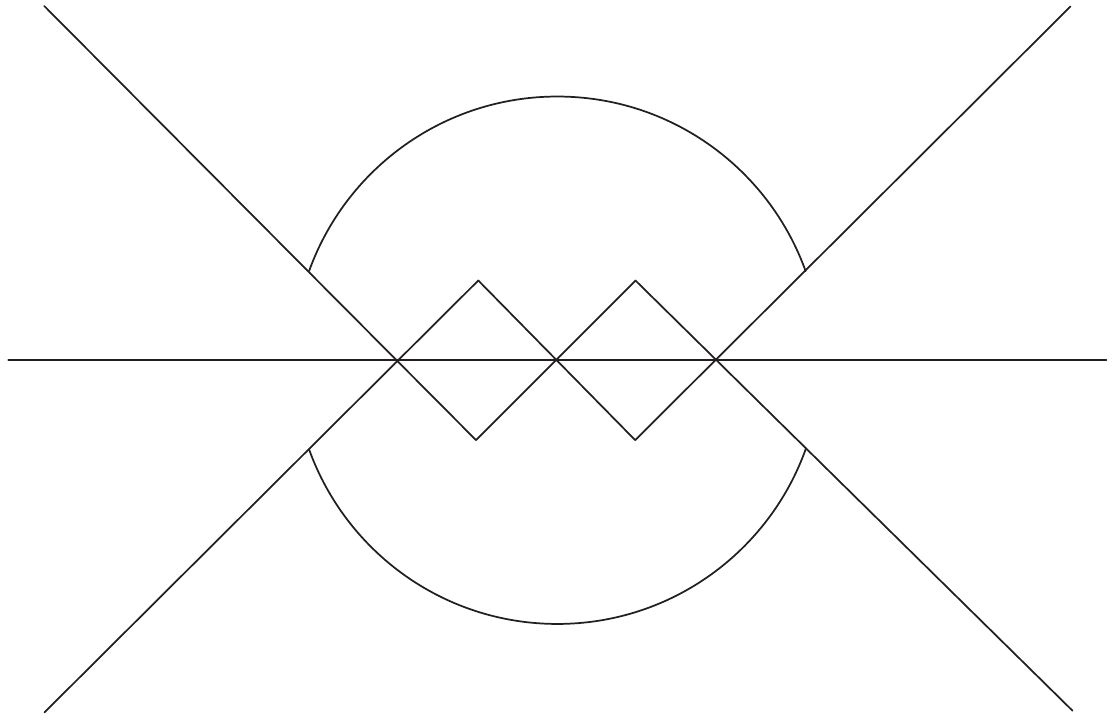}
      \put(88,41){\small $V_1$}
      \put(12,41){\small $V_1$}
      \put(41,34){\small $V_2$}
      \put(55,34){\small $V_2$}
      \put(41,28.2){\small $V_3$}
      \put(55,28.2){\small $V_3$}
      \put(88,21){\small $V_4$}
      \put(12,21){\small $V_4$}
      \put(48,61){\small $V_5$}
      \put(48,1){\small $V_6$}
      \put(32.2,27){\small $-k_0$}
      \put(63.2,27.4){\small $k_0$}
    \end{overpic}\bigskip
     \begin{figuretext}\label{VjsIII.pdf}
       The open subsets $V_j := V_j(\zeta)$, $j = 1,\dots, 6$, of the complex $k$-plane relevant for Sector III. 
     \end{figuretext}
     \end{center}
\end{figure}

\begin{figure}
\begin{center}
\begin{overpic}[width=.65\textwidth]{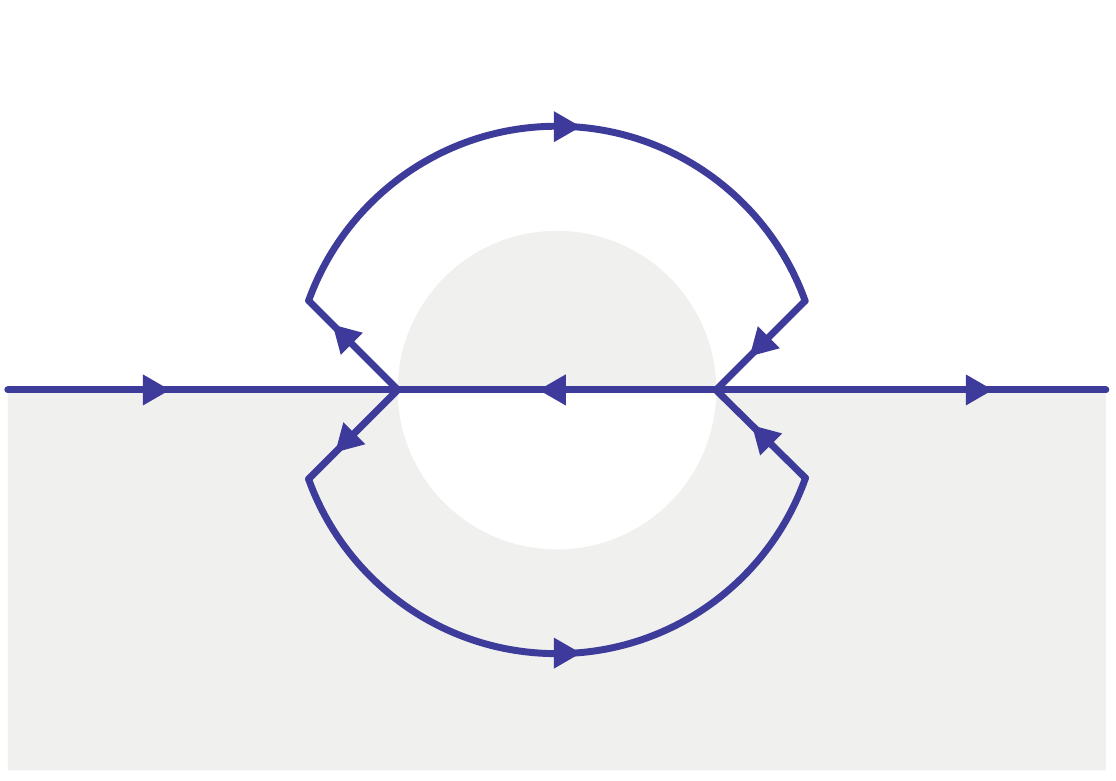}
      \put(102,34){\small $\Gamma^{(1)}$}
      \put(76,60){\small $\re \Phi < 0$}
      \put(76,10){\small $\re \Phi > 0$}
      \put(49,61){\small $1$}
      \put(31,41){\small $1$}
      \put(67,41){\small $1$}
      \put(49,37){\small $2$}
      \put(49,13.5){\small $3$}
      \put(31,26.5){\small $3$}
      \put(67,26){\small $3$}
      \put(86,37){\small $4$}
      \put(13,37){\small $4$}
      \put(34,31){\small $-k_0$}
      \put(61,31){\small $k_0$}
    \end{overpic}
     \begin{figuretext}\label{Gamma1III.pdf}
       The contour $\Gamma^{(1)}$ in the complex $k$-plane. The region where $\re \Phi > 0$ is shaded. 
     \end{figuretext}
     \end{center}
\end{figure}

\subsubsection{Second transformation} 
The jump matrix $v$ does not admit an appropriate triangular factorization for $k \in (-k_0, k_0)$. Hence we introduce $M^{(2)}$ by
$$M^{(2)}(x,t,k) = M^{(1)}(x,t,k) \delta(\zeta, k)^{-\sigma_3},$$
where the complex-valued function $\delta(\zeta, k)$ is defined by
\begin{align}\label{deltadef}
\delta(\zeta, k) = e^{\frac{1}{2\pi i}\int_{-k_0}^{k_0} \frac{\ln(1 + |r(s)|^2)}{s-k} ds}, \qquad k \in \C \setminus [-k_0, k_0].
\end{align}

\begin{lemma}\label{deltalemmaIII}
For each $\zeta \in \mathcal{I}$, the function $\delta(\zeta, k)$ has the following properties:
\begin{enumerate}[$(a)$]

\item $\delta(\zeta, k)$ and $\delta(\zeta, k)^{-1}$ are bounded and analytic functions of $k \in \C \setminus [-k_0, k_0]$ with continuous boundary values on $(-k_0, k_0)$.

\item $\delta(\zeta, k) = \overline{\delta(\zeta, \bar{k})}^{-1} = \delta(\zeta, -k)^{-1}$ for $k \in \C \setminus [-k_0, k_0]$.

\item Across the subcontour $(-k_0, k_0)$ of $\Gamma^{(1)}$ oriented to the left as in Figure \ref{Gamma1III.pdf}, 
$\delta$ satisfies the jump condition 
$$\delta_+(\zeta, k) = \frac{\delta_-(\zeta, k)}{1 + |r(k)|^2}, \qquad k \in (-k_0, k_0).$$

\item $\delta$ satisfies the asymptotic formulas
$$\delta(\zeta, k) = \begin{cases} 1 + O(k^{-1}) & \text{uniformly as $k \to \infty$}, \\
1 + O(k) & \text{uniformly as $k \to 0$}.
\end{cases}$$
\end{enumerate}
\end{lemma}
\begin{proof}
The proof is standard. 
\end{proof}

Lemma \ref{deltalemmaIII} implies that
$$\delta(\zeta,\cdot)^{\sigma_3} \in I + (\dot{E}^2 \cap E^\infty)(\C \setminus [-k_0, k_0]), \qquad \zeta \in \mathcal{I}.$$
Hence $M$ satisfies the RH problem (\ref{RHM}) iff $M^{(2)}$ 
satisfies the RH problem (\ref{RHMj}) with $j = 2$, where $\Gamma^{(2)} = \Gamma^{(1)}$ and the jump matrix $J^{(2)}$ is given by $J^{(2)} =  \delta_-^{\sigma_3} J^{(1)}  \delta_+^{-\sigma_3}$, that is,
\begin{align*}\nonumber
&J_1^{(2)} = \begin{pmatrix} 1 & 0 \\  - \delta^{-2}h e^{t\Phi} & 1 \end{pmatrix}, 
\qquad
J_2^{(2)} = \begin{pmatrix} 1 + r r^* & - \delta_-^2 \frac{r^*}{1 + rr^*} e^{-t\Phi} \\
- \delta_+^{-2} \frac{r}{1 + rr^*} e^{t\Phi}& 1 \end{pmatrix}, 
	\\ 
& J_3^{(2)} = \begin{pmatrix} 1 & - \delta^2 h^* e^{-t\Phi} \\ 0 & 1 \end{pmatrix}, 
\qquad
J_4^{(2)} = \begin{pmatrix} 1 + r_1 r_1^* & \delta^2 r_1^* e^{-t\Phi} \\
\delta^{-2} r_1e^{t\Phi} & 1 \end{pmatrix}.
\end{align*}
If we define $r_2(k)$ by 
\begin{align*}
  r_2(k) = \frac{r^*(k)}{1+r(k)r^*(k)}, 
\end{align*}
we can write the jumps across the real axis as 
\begin{align}\nonumber
&J_2^{(2)} = \begin{pmatrix} 1  & - \delta_-^2 r_2 e^{-t\Phi} \\ 0 & 1 \end{pmatrix}
\begin{pmatrix} 1 & 0 \\- \delta_+^{-2} r_2^* e^{t\Phi}& 1 \end{pmatrix}, 
	\\
& J_4^{(2)} = \begin{pmatrix} 1  & \delta^2 r_1^* e^{-t\Phi} \\ 0& 1 \end{pmatrix}
\begin{pmatrix} 1  & 0 \\ \delta^{-2} r_1e^{t\Phi} & 1 \end{pmatrix}.
\end{align}

\subsubsection{Third transformation}
Before we can deform the contour, we need to decompose each of the functions $h, r_1, r_2$ into an analytic part and a small remainder. 

\begin{lemma}[Analytic approximation of $h$]\label{hdecompositionlemmaIII}
For any integer $N \geq 1$, there exists a decomposition
\begin{align*}
h(k) = h_{a}(t, k) + h_{r}(t, k), \qquad t > 0, \ k \in \bar{D}_1 \cap \bar{D}_2,  
\end{align*}
such that the functions $h_{a}$ and $h_{r}$ have the following properties:
\begin{enumerate}[$(a)$]
\item For each $t > 0$, $h_{a}(t, k)$ is defined and continuous for $k \in \bar{D}_1$ and analytic for $k \in D_1$.

\item The function $h_a$ satisfies
\begin{align}\label{haestimateIII}
\begin{cases} 
|h_{a}(t, k) - \sum_{j=0}^{N} \frac{h^{(j)}(1)(k-1)^j}{j!}| \leq C|k-1|^{N+1} e^{\frac{t}{4} |\re \Phi(\zeta,k)|},
	\\
|h_{a}(t, k)| \leq 
\frac{C}{1 + |k|^2} e^{\frac{t}{4} |\re \Phi(\zeta,k )|},	
\end{cases}  \; k \in \bar{D}_1, \ \zeta \in \mathcal{I}, \ t > 0.
\end{align}

\item The $L^1$ and $L^\infty$ norms of the function $h_{r}(t, \cdot)$ on the semicircle $\bar{D}_1 \cap \bar{D}_2$ are $O(t^{-N-\frac{1}{2}})$ as $t \to \infty$.

\item $h_a(t, k) = \overline{h_a(t, -\bar{k})}$ and $h_r(t, k) = \overline{h_r(t, -\bar{k})}$.

\end{enumerate}
\end{lemma}
\begin{proof}
Recall that $h(k) \in C^{\infty}(\bar{D}_2)$. Let
$$f_0(k) = \sum_{j=2}^{6N+9} \frac{a_j}{k^j},$$
where the complex constants $\{a_j\}_{j=2}^{6N+9}$ are chosen so that $f_0(k)$ coincides with $h(k)$ to order $3N+3$ as $k \to \pm 1$, i.e.,
\begin{align}\label{hlinearconditionsIII}
f_0(k) = \begin{cases} 
\sum_{j=0}^{3N+3} \frac{h^{(j)}(1)}{j!} (k-1)^j + O((k-1)^{3N+4}), & k \to 1,
	\\
\sum_{j=0}^{3N+3} \frac{h^{(j)}(-1)}{j!} (k+1)^j + O((k+1)^{3N+4}), & k \to -1.
\end{cases}
\end{align}
Equation (\ref{hlinearconditionsIII}) imposes $6N+8$ independent linear conditions on the $a_j$; hence the coefficients $a_j$ exist and are unique.
The symmetry $h(k) = \overline{h(-\bar{k})}$ implies that $h^{(j)}(1) = (-1)^j \overline{h^{(j)}(-1)}$ for any $j$, which in turn implies that $a_j \in \R$ for $j$ even and $a_j \in i\R$ for $j$ odd; thus $f_0(k) = \overline{f_0(-\bar{k})}$.
Letting $f = h - f_0$, we have
\begin{align}\label{hcoincide}
 f^{(n)}(k) =
\begin{cases}
 O((k-1)^{3N+4 - n}), & k \to 1, 
	\\
 O((k+1)^{3N+4 - n}), & k \to -1, 
 \end{cases}
 \qquad  k \in \bar{D}_1 \cap \bar{D}_2, \ n = 0,1,2.
\end{align}

The decomposition of $h(k)$ can now be derived as follows.
The map $k \mapsto \phi = \phi(k)$ defined by
$$\phi(k) = \frac{1}{2}\Big(k+\frac{1}{k}\Big)$$
is a bijection $\bar{D}_1 \cap \bar{D}_2 \to [-1,1]$, because $\phi(e^{i\theta}) = \cos \theta$. Hence we can define $F:\R \to \C$ by
\begin{align}\label{Fphi}
F(\phi) = \begin{cases} \frac{k^{2N+4}}{(k^2 - 1)^{N+1}} f(k), & \phi \in [-1,1], \\
0, & \phi \in (-\infty,-1) \cup (1, \infty), 
\end{cases}
\quad \zeta \in \mathcal{I},
\end{align}
The function $F(\phi)$ is smooth for $\phi \in \R \setminus \{-1,1\}$ and 
$$F^{(n)}(\phi) = \bigg(\frac{2k^2}{k^2-1} \frac{\partial }{\partial k}\bigg)^n \bigg(\frac{k^{2N+4}}{(k^2-1)^{N+1}}f(k)\bigg), \qquad \phi \in (-1,1), \ n = 0,1, \dots, N+1.$$ 
Using (\ref{hcoincide}) it follows that $F \in C^{N+1}(\R)$. In particular, $F$ belongs to the Sobolev space $H^{N+1}(\R)$.
Hence $\|s^{N+1} \hat{F}(s)\|_{L^2(\R)} < \infty$ where $\hat{F}$ denotes the Fourier transform of $F$:
\begin{align}\label{FphihatF}
\hat{F}(s) = \frac{1}{2\pi} \int_{\R} F(\phi) e^{-i\phi s} d\phi, \qquad F(\phi) =  \int_{\R} \hat{F}(s) e^{i\phi s} ds.
\end{align}
Equations (\ref{Fphi}) and (\ref{FphihatF}) imply
$$\frac{(k^2-1)^{N+1}}{k^{2N+4}}\int_{\R} \hat{F}(s) e^{\frac{i}{2}(k+\frac{1}{k})s} ds 
= f(k), \qquad  k \in \bar{D}_1 \cap \bar{D}_2.$$
Writing
$$f(k) = f_a(t, k) + f_r(t, k), \qquad t > 0, \  k \in \bar{D}_1 \cap \bar{D}_2,$$
where the functions $f_a$ and $f_r$ are defined by
\begin{align*}
& f_a(t,k) = \frac{(k^2-1)^{N+1}}{k^{2N+4}}\int_{-\frac{t}{4}}^\infty \hat{F}(s) e^{\frac{i}{2}(k+\frac{1}{k})s} ds, \qquad t > 0, \ k \in \bar{D}_1,  
	\\
& f_r(t,k) = \frac{(k^2-1)^{N+1}}{k^{2N+4}}\int_{-\infty}^{-\frac{t}{4}} \hat{F}(s) e^{\frac{i}{2}(k+\frac{1}{k})s} ds,\qquad t > 0, \   k \in \bar{D}_1 \cap \bar{D}_2,
\end{align*}
we infer that $f_a(t, \cdot)$ is continuous in $\bar{D}_1$ and analytic in $D_1$. 
Furthermore, since $|\re \frac{i}{2}(k+k^{-1})| \leq |\re \Phi(\zeta, k)|$ for all $k \in \bar{D}_1$ and $\zeta \in \mathcal{I}$, we find
\begin{align*}\nonumber
 |f_a(t, k)| 
&\leq \frac{|k^2-1|^{N+1}}{|k|^{2N+4}}\|\hat{F}\|_{L^1(\R)}  \sup_{s \geq -\frac{t}{4}} e^{s \re \frac{i}{2}(k+k^{-1})}
\leq C\frac{|k^2-1|^{N+1}}{|k|^{2N+4}} e^{\frac{t}{4} |\re \frac{i}{2}(k+k^{-1})|} 
	\\ 
 &\leq C\frac{|k^2-1|^{N+1}}{|k|^{2N+4}} e^{\frac{t}{4} |\re \Phi(\zeta, k)|}, \qquad t > 0, \ k \in \bar{D}_1,
\end{align*}
and
\begin{align*}\nonumber
|f_r(t, k)| & \leq \frac{|k^2-1|^{N+1}}{|k|^{2N+4}} \int_{-\infty}^{-\frac{t}{4}} s^{N+1} |\hat{F}(s)| s^{-N-1} ds
 \leq C \| s^{N+1} \hat{F}(s)\|_{L^2(\R)} \sqrt{\int_{-\infty}^{-\frac{t}{4}} s^{-2N-2} ds}  
 	\\ 
&  \leq C t^{-N-\frac{1}{2}}, \qquad t > 0, \ k \in \bar{D}_1 \cap \bar{D}_2.
\end{align*}
Hence the $L^1$ and $L^\infty$ norms of $f_r$ on $\bar{D}_1 \cap \bar{D}_2$ are $O(t^{-N-\frac{1}{2}})$. 
The symmetry $f(k) = \overline{f(-\bar{k})}$ implies $F(\phi) = \overline{F(-\phi)}$, so $\hat{F}(s)$ is real valued, which leads to the symmetries $f_a(t, k) = \overline{f_a(t, -\bar{k})}$ and $f_r(t, k) = \overline{f_r(t, -\bar{k})}$.
Letting
\begin{align*}
& h_{a}(t, k) = f_0(k) + f_a(t, k), \qquad t > 0, \ k \in \bar{D}_1,
	\\
& h_{r}(t, k) = f_r(t, k), \qquad t > 0, \ k \in \bar{D}_1 \cap \bar{D}_2,
\end{align*}
we find a decomposition of $h$ with the desired properties.
\end{proof}

\begin{remark}\upshape
The analytic approximation of $h(k)$ is only needed for small $\zeta$ to ensure that the asymptotic formula is valid uniformly as $k_0$ approaches $1$. 
\end{remark}

\begin{lemma}[Analytic approximation of $r_1$ and $r_2$ for $0 \leq \zeta < 1$]\label{decompositionlemmaIII}
For any integer $N \geq 1$, there exist decompositions
\begin{align*}
& r_1(k) = r_{1,a}(x, t, k) + r_{1,r}(x, t, k), \qquad k \in (-\infty, k_0] \cup [k_0, \infty), 
	\\
& r_2(k) = r_{2,a}(x, t, k) + r_{2,r}(x, t, k), \qquad k \in [-k_0, k_0], 
\end{align*}
such that the functions $\{r_{j,a}, r_{j,r}\}_{j=1}^2$ have the following properties:
\begin{enumerate}[$(a)$]
\item For each $\zeta \in \mathcal{I}$ and each $t > 0$, $r_{j,a}(x, t, k)$ is defined and continuous for $k \in \bar{V}_j$ and analytic for $k \in V_j$, $j = 1,2$.

\item The functions $r_{1,a}$ and $r_{2,a}$ satisfy, uniformly for $\zeta \in \mathcal{I}$ and $t > 0$,
\begin{subequations}\label{rjaestimatesIII}
\begin{align}\label{rjaestimatesIIIa}
\bigg|r_{j, a}(x, t, k) - \sum_{n=0}^{N} \frac{r_j^{(n)}(k_0)(k-k_0)^n}{n!} \bigg| \leq C |k - k_0|^{N+1} e^{\frac{t}{4}|\re \Phi(\zeta,k)|}, \qquad k \in \bar{V}_j, \ j = 1, 2, 
\end{align}
and
\begin{align}\label{rjaestimatesIIIb}
\begin{cases}  
|r_{1, a}(x, t, k)| \leq \frac{C}{1 + |k|} e^{\frac{t}{4}|\re \Phi(\zeta,k)|}, & k \in \bar{V}_1,
	\\  
|r_{2, a}(x, t, k)| \leq C|k|^2 e^{\frac{t}{4}|\re \Phi(\zeta,k)|}, \quad & k \in \bar{V}_2.
\end{cases}     
\end{align}
\end{subequations}

\item The $L^1$ and $L^\infty$ norms on $(-\infty,-k_0) \cup (k_0, \infty)$ of the function $k \mapsto (1+|k|^{-1}) r_{1,r}(x, t, k)$ are $O(t^{-N-\frac{1}{2}})$ as $t \to \infty$ uniformly with respect to $\zeta \in \mathcal{I}$.

\item The $L^1$ and $L^\infty$ norms on $(-k_0, k_0)$ of the function $k \mapsto (1+|k|^{-1})r_{2,r}(x, t, k)$ are $O(t^{-N-\frac{1}{2}})$ as $t \to \infty$ uniformly with respect to $\zeta \in \mathcal{I}$.

\item The following symmetries are valid:
\begin{align}\label{rsymmetriesIII}
r_{j,a}(\zeta, t, k) = \overline{r_{j,a}(\zeta, t, -\bar{k})}, \quad
r_{j,r}(\zeta, t, k) = \overline{r_{j,r}(\zeta, t, -\bar{k})}, \qquad j = 1, 2.
\end{align}

\end{enumerate}
\end{lemma}
\begin{proof}
The decomposition of $r_1$ follows immediately from Lemma \ref{decompositionlemmaII} and the decomposition of $r_2$ can be derived in a similar way. 
\end{proof}

Applying Lemma \ref{hdecompositionlemmaIII} and Lemma \ref{decompositionlemmaIII} with $N = 1$, we obtain decompositions of $h,r_1,r_2$ (the lemmas are stated for an arbitrary $N \geq 1$, because this will be needed in Sector IV).
We introduce $M^{(3)}(x,t,k)$ by
\begin{align}\label{M3defIII}
M^{(3)}(x,t,k) = M^{(2)}(x,t,k)H(x,t,k),
\end{align}
where the sectionally analytic function $H$ is defined by
\begin{align*}
H = \begin{cases} 
\begin{pmatrix} 1  & 0 \\ -\delta^{-2} r_{1,a} e^{t\Phi} & 1 \end{pmatrix}, & k \in V_1,
	\\
\begin{pmatrix} 1  & - \delta^2 r_{2,a} e^{-t\Phi} \\ 0 & 1 \end{pmatrix}, & k \in V_2,
	\\
\begin{pmatrix} 1 & 0 \\ \delta^{-2} r_{2,a}^* e^{t\Phi}& 1 \end{pmatrix}, & k \in V_3,
	\\
\begin{pmatrix} 1  & \delta^2 r_{1,a}^* e^{-t\Phi} \\ 0& 1 \end{pmatrix}, & k \in V_4,
	\\
\begin{pmatrix} 1 & 0 \\  \delta^{-2}h_a e^{t\Phi} & 1 \end{pmatrix}, & k \in V_5,
	\\
\begin{pmatrix} 1 & -\delta^2 h_a^* e^{-t\Phi} \\ 0 & 1 \end{pmatrix}, & k \in V_6,
	\\
I, & \text{elsewhere}.	
\end{cases}
\end{align*}
\begin{figure}
\begin{center}
\begin{overpic}[width=.65\textwidth]{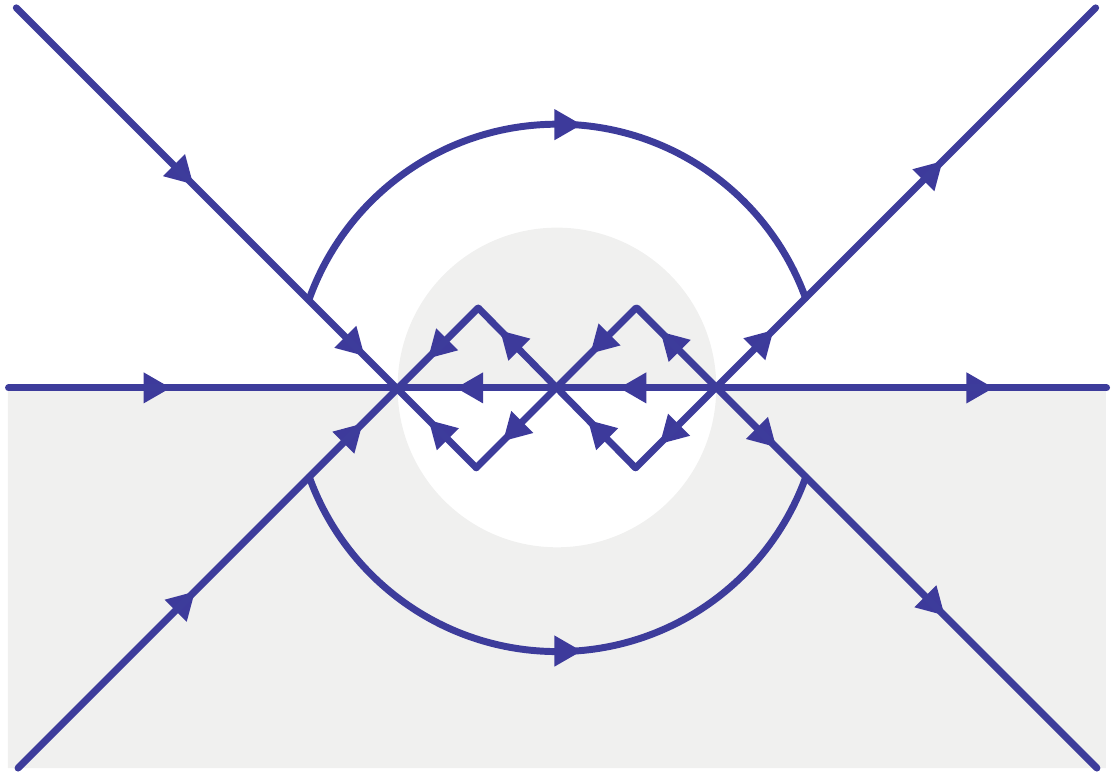}
      \put(102,34){\small $\Gamma^{(3)}$}
      \put(66,62){\small $\re \Phi < 0$}
      \put(66,8){\small $\re \Phi > 0$}
      \put(66.5,40.3){\small $1$}
      \put(60,40){\small $2$}
      \put(60,27){\small $3$}
      \put(66.5,27){\small $4$}
       \put(32.6,39){\small $1$}
      \put(36.8,39){\small $2$}
      \put(36.8,28){\small $3$}
      \put(32.5,28){\small $4$}
     \put(85,15.5){\small $5$}
      \put(12.5,15.5){\small $5$}
     \put(85,52){\small $6$}
      \put(12.5,52){\small $6$}
       \put(50,60.5){\small $7$}
      \put(50,13){\small $8$}
      \put(42,36.5){\small $9$}
      \put(56.4,36.4){\small $9$}
      \put(12,37){\small $10$}
      \put(86,37){\small $10$}
    \end{overpic}
     \begin{figuretext}\label{Gamma3III.pdf}
       The contour $\Gamma^{(3)}$ in the complex $k$-plane with the region $\re \Phi > 0$ shaded. 
     \end{figuretext}
     \end{center}
\end{figure}
By Lemma \ref{deltalemmaIII}, Lemma \ref{hdecompositionlemmaIII}, and Lemma \ref{decompositionlemmaIII}, we have  
$$H(x,t,\cdot)^{\pm1} \in I + (\dot{E}^2 \cap E^\infty)(\C \setminus \Gamma^{(3)}),$$
where $\Gamma^{(3)} \subset \C$ denotes the contour displayed in Figure \ref{Gamma3III.pdf}.
It follows that $M$ satisfies the RH problem (\ref{RHM}) iff $M^{(3)}$ 
satisfies the RH problem (\ref{RHMj}) with $j = 3$, where the jump matrix $J^{(3)}$ is given by 
\begin{align}\nonumber
&J_1^{(3)} = \begin{pmatrix} 1  & 0 \\ \delta^{-2} (r_{1,a}+h) e^{t\Phi} & 1 \end{pmatrix},
&&
J_2^{(3)} = \begin{pmatrix} 1  & - \delta^2 r_{2,a} e^{-t\Phi} \\ 0 & 1 \end{pmatrix}, 
	\\\nonumber
&J_3^{(3)} = \begin{pmatrix} 1 & 0 \\ -\delta^{-2} r_{2,a}^* e^{t\Phi}& 1 \end{pmatrix},
&&
J_4^{(3)} = \begin{pmatrix} 1  & \delta^2 (r_{1,a}^* + h^*) e^{-t\Phi} \\ 0& 1 \end{pmatrix},
	\\\nonumber
&J_5^{(3)} = \begin{pmatrix} 1  & \delta^2 (r_{1,a}^* + h_a^*) e^{-t\Phi} \\ 0& 1 \end{pmatrix},
&&
J_6^{(3)} = \begin{pmatrix} 1  & 0 \\ \delta^{-2} (r_{1,a}+h_a) e^{t\Phi} & 1 \end{pmatrix},
	\\\nonumber
&J_7^{(3)} = \begin{pmatrix} 1 & 0 \\  - \delta^{-2}h_r e^{t\Phi} & 1 \end{pmatrix},
&&
J_8^{(3)} = \begin{pmatrix} 1 & - \delta^2 h_r^* e^{-t\Phi} \\ 0 & 1 \end{pmatrix},
	\\ \nonumber
& J_9^{(3)} = \begin{pmatrix} 1  & - \delta_-^2 r_{2,r} e^{-t\Phi} \\ 0 & 1 \end{pmatrix}
\begin{pmatrix} 1 & 0 \\- \delta_+^{-2} r_{2,r}^* e^{t\Phi}& 1 \end{pmatrix}, 
&&
J_{10}^{(3)} = \begin{pmatrix} 1+ |r_{1,r}|^2  & \delta^2 r_{1,r}^* e^{-t\Phi} \\ \delta^{-2} r_{1,r}e^{t\Phi} & 1 \end{pmatrix}.
\end{align}

\subsection{Local model}
The RH problem for $M^{(3)}$ has the property that the matrix $J^{(3)} - I$ decays to zero as $t \to \infty$ everywhere except near $\pm k_0$. This means that we only have to consider small neighborhoods of $\pm k_0$ when computing the long-time asymptotics of $M^{(3)}$. In this section, we find a local solution $m^{k_0}$ which approximates $M^{(3)}$ near $k_0$. 

\subsection{Exact solution on the cross}
Let $X = X_1 \cup \cdots \cup X_4 \subset \C$ be the cross defined by
\begin{align} \nonumber
&X_1 = \bigl\{se^{\frac{i\pi}{4}}\, \big| \, 0 \leq s < \infty\bigr\}, && 
X_2 = \bigl\{se^{\frac{3i\pi}{4}}\, \big| \, 0 \leq s < \infty\bigr\},  
	\\ \label{XdefIII}
&X_3 = \bigl\{se^{-\frac{3i\pi}{4}}\, \big| \, 0 \leq s < \infty\bigr\}, && 
X_4 = \bigl\{se^{-\frac{i\pi}{4}}\, \big| \, 0 \leq s < \infty\bigr\},
\end{align}
and oriented away from the origin, see Figure \ref{X.pdf}.

\begin{figure}
\begin{center}
 \begin{overpic}[width=.4\textwidth]{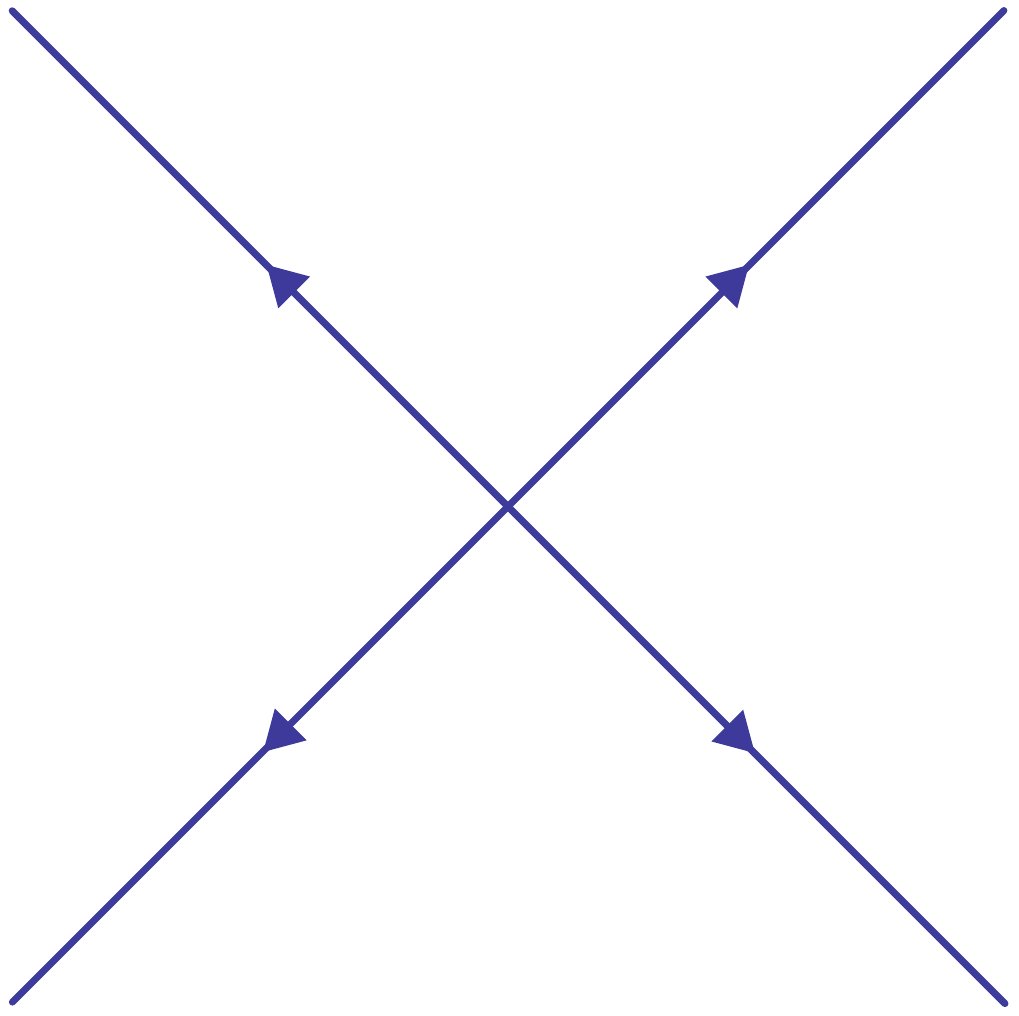}
      \put(73,68){\small $X_1$}
      \put(19,68){\small $X_2$}
      \put(17,27){\small $X_3$}
      \put(75,27){\small $X_4$}
      \put(48,43){$0$}
    \end{overpic}
     \begin{figuretext}\label{X.pdf}
        The contour $X = X_1 \cup X_2 \cup X_3 \cup X_4$.
     \end{figuretext}
     \end{center}
\end{figure}

\begin{lemma}[Exact solution on the cross]\label{XlemmaIII}
Define the function $\nu:\C \to (0,\infty)$ by 
$\nu(q) = \frac{1}{2\pi} \ln(1 + |q|^2)$ and define the jump matrix $v^X(q, z)$ for $z \in X$ by
\begin{align}\label{vXdefIII} 
v^X(q, z) = \begin{cases}
\begin{pmatrix} 1 & 0	\\
  q z^{2i\nu(q)} e^{\frac{iz^2}{2}}	& 1 \end{pmatrix}, &   z \in X_1, 
  	\\
\begin{pmatrix} 1 & -\frac{\bar{q}}{1 + |q|^2} z^{-2i\nu(q)}e^{-\frac{iz^2}{2}}	\\
0 & 1  \end{pmatrix}, &  z \in X_2, 
	\\
\begin{pmatrix} 1 &0 \\
- \frac{q}{1 + |q|^2}z^{2i\nu(q)} e^{\frac{iz^2}{2}}	& 1 \end{pmatrix}, &  z \in X_3,
	\\
 \begin{pmatrix} 1	& \bar{q} z^{-2i\nu(q)}e^{-\frac{iz^2}{2}}	\\
0	& 1 \end{pmatrix}, &  z \in X_4.
\end{cases}
\end{align}
Then, for each $q \in \C$, the RH problem 
\begin{align*}
\begin{cases} m^X(q, \cdot) \in I + \dot{E}^2(\C \setminus X), 
	\\
m_+^X(q, z) =  m_-^X(q, z) v^X(q, z) \quad \text{for a.e.} \ z \in X, 
\end{cases} 
\end{align*}
has a unique solution $m^X(q, z)$. This solution satisfies
\begin{align}\label{mXasymptoticsIII}
  m^X(q, z) = I - \frac{i}{z}\begin{pmatrix} 0 & \beta^X(q) \\ \overline{\beta^X(q)} & 0 \end{pmatrix} + O\biggl(\frac{q}{z^2}\biggr), \qquad z \to \infty,  \ q \in \C, 
\end{align}  
where the error term is uniform with respect to $\arg z \in [0, 2\pi]$ and $q$ in compact subsets of $\C$, and the function $\beta^X(q)$ is defined by
\begin{align*}
\beta^X(q) = \sqrt{\nu(q)} e^{i\left(\frac{\pi}{4} - \arg q - \arg \Gamma(i\nu(q)\right)}, \qquad q \in \C.
\end{align*}
Moreover, for each compact subset $K$ of $\C$, 
\begin{align}\label{mXqboundIII}
\sup_{q \in K} \sup_{z \in \C \setminus X} |m^X(q, z)| < \infty, \qquad
\sup_{q \in K} \sup_{z \in \C \setminus X} \frac{|m^X(q, z)- I|}{|q|} < \infty.
\end{align}
\end{lemma}
\begin{proof}
The proof relies on deriving an explicit formula for the solution $m^X$ in terms of parabolic cylinder functions \cite{I1981}. The lemma is standard except possibly for the presence of $q$ in the error term in (\ref{mXasymptoticsIII}) and for the second estimate in (\ref{mXqboundIII}), which can be derived by considering the explicit formula for $m^X$ in the limit $q \to 0$.
\end{proof}

\subsection{Local model near $k_0$}
Define $\epsilon := \epsilon(\zeta)$ by $\epsilon = k_0/2$. Let $D_\epsilon(k_0)$ denote the open disk of radius $\epsilon$ centered at $k_0$.
In order to relate $M^{(3)}$ to the solution $m^X$ of Lemma \ref{XlemmaIII}, we make a local change of variables for $k$ near $k_0$ and introduce the new variable $z := z(\zeta, k)$ by 
$$z = \sqrt{t}(k -k_0)\psi(\zeta, k), \qquad \text{where} \quad \psi(\zeta, k) = \sqrt{\frac{2}{k(1+k_0^2)}}.$$
The function $\psi$ is analytic for $k \in D_\epsilon(k_0)$; we fix the branch of $\psi(\zeta, k)$ by requiring that $\re \psi(\zeta, k) > 0$ for $k \in D_\epsilon(k_0)$.
For each $\zeta \in \mathcal{I}$, the map $k \mapsto z$ is a biholomorphism from $D_\epsilon(k_0)$ to a neighborhood of the origin. 
We note that $z$ satisfies 
$$\frac{iz^2}{2} = \frac{it(k-k_0)^2}{k(1+k_0^2)} = t(\Phi(\zeta, k) - \Phi(\zeta, k_0)).$$ 

Let $\mathcal{X}^\epsilon := \mathcal{X}^\epsilon(\zeta)$ denote the following small cross centered at $k_0$ (see Figure \ref{calX.pdf}):
$$\mathcal{X}^\epsilon = \bigcup_{j=1}^4 \mathcal{X}_j^\epsilon \quad \text{where} \quad \mathcal{X}_j^\epsilon = \Gamma_j^{(3)} \cap D_\epsilon(k_0),$$
where $\Gamma_j^{(3)}$ is the subcontour of $\Gamma^{(3)}$ labeled by $3$ in Figure \ref{Gamma3III.pdf}. Let $-\mathcal{X}^\epsilon := \bigcup_{j=1}^4 \Gamma_j^{(3)} \cap D_\epsilon(-k_0)$ denote the analogous cross centered at $-k_0$.
Deforming the contours $\Gamma_j^{(3)}$, $j = 1, \dots, 4$, slightly, we may assume that $\mathcal{X}_j^\epsilon$ is mapped into $X_j$ under the map $k \mapsto z(\zeta, k)$, and that $\Gamma^{(3)}$  is invariant (up to orientation) under the involution $k \mapsto -\bar{k}$. Thanks to the symmetry $z(\zeta, k) = \overline{z(\zeta, \bar{k})}$, we may also assume that $\Gamma^{(3)}$ remains invariant under $k \mapsto \bar{k}$.

\begin{figure}
\begin{center}
 \begin{overpic}[width=.7\textwidth]{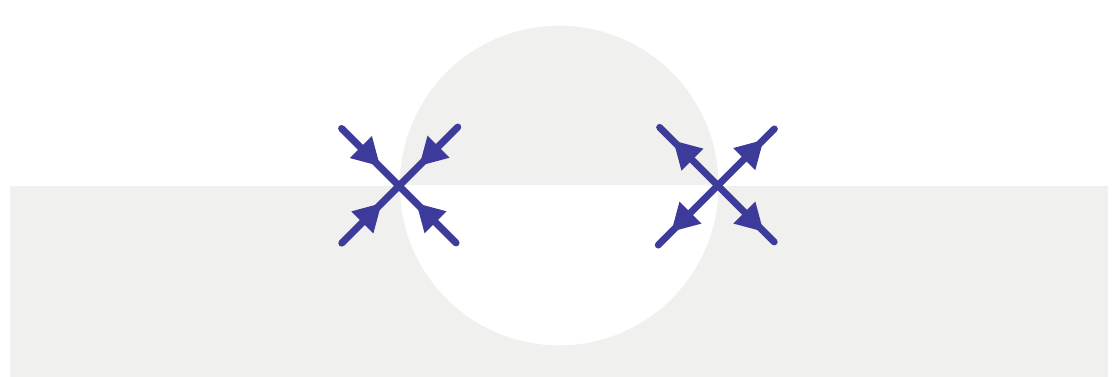}
      \put(63,23){\small $\mathcal{X}^\epsilon$}
      \put(32.4,23){\small $-\mathcal{X}^\epsilon$}
    \end{overpic}
     \begin{figuretext}\label{calX.pdf}
        The small crosses $\mathcal{X}^\epsilon$ and $-\mathcal{X}^\epsilon$ centered at $k_0$ and $-k_0$, respectively.
     \end{figuretext}
     \end{center}
\end{figure}

We next consider the behavior of $\delta(\zeta, k)$ as $k$ approaches $k_0$.
Integration by parts gives
\begin{align*}
 \int_{-k_0}^{k_0}  & \frac{\ln(1 + |r(s)|^2)}{s-k} ds
= \ln (k-s)\ln(1 + |r(s)|^2)\big|_{s=-k_0}^{k_0} 
	\\
& - \int_{-k_0}^{k_0} \ln (k-s) d\ln(1 + |r(s)|^2), \qquad \zeta \in \mathcal{I}, \ k \in \C \setminus (-\infty, k_0],
\end{align*}
where the branch cut for the logarithm runs along $(-\infty, 0]$. It follows that
\begin{align*}
  \delta(\zeta, k) = e^{-i \nu \ln (\frac{k - k_0}{k+k_0}) + \chi(\zeta, k)}, \qquad \zeta \in \mathcal{I}, \ k \in \C \setminus (-\infty, k_0],
\end{align*}
where  $\nu := \nu(\zeta) \geq 0$ is defined by
\begin{align}\label{nudefIII}
\nu(\zeta) = \frac{1}{2\pi} \ln(1 + |r(k_0)|^2), \qquad \zeta \in \mathcal{I},
\end{align}
and the function $\chi(\zeta, k)$ is given by
\begin{align}\label{chidefIII}
\chi(\zeta, k) = -\frac{1}{2\pi i} \int_{-k_0}^{k_0} \ln (k-s) d\ln(1 + |r(s)|^2), \qquad \zeta \in \mathcal{I}, \ k \in \C \setminus (-\infty, k_0].
\end{align}
Hence, for each $\zeta \in \mathcal{I}$, we can write $\delta$ as
$$\delta(\zeta, k)  = z^{-i\nu} \delta_0(\zeta, t) \delta_1(\zeta, k), \qquad k \in D_\epsilon(k_0)\setminus (-\infty, k_0],$$
where the functions $\delta_0(\zeta, t)$ and $\delta_1(\zeta, k)$ are defined by
\begin{align*}
& \delta_0(\zeta, t) = t^{\frac{i\nu}{2}} e^{i\nu\ln(2k_0\psi(\zeta, k_0))} e^{\chi(\zeta, k_0)}, \qquad t > 0,
	\\ 
&\delta_1(\zeta, k) = e^{i\nu\ln(\frac{(k+k_0)\psi(\zeta, k)}{2k_0\psi(\zeta, k_0)})}e^{\chi(\zeta, k) - \chi(\zeta, k_0)}, \qquad k \in D_\epsilon(k_0)\setminus (-\infty, k_0].
\end{align*}

Define $\tilde{m}(x,t,z)$ by
$$\tilde{m}(x,t,z(\zeta, k)) = M^{(3)}(x,t,k)e^{-\frac{t\Phi(\zeta, k_0)\sigma_3}{2}}\delta_0(\zeta,t)^{\sigma_3}, \qquad k \in D_\epsilon(k_0) \setminus \Gamma^{(3)}.$$
Then $\tilde{m}$ is a sectionally analytic function which satisfies $\tilde{m}_+ = \tilde{m}_- \tilde{v}$ for $k \in \mathcal{X}^\epsilon$, where the jump matrix 
$$\tilde{v} = \delta_0(\zeta, t)^{-\hat{\sigma}_3}e^{\frac{t\Phi(\zeta, k_0)}{2}\hat{\sigma}_3} J^{(3)}$$ 
is given by
\begin{align}\nonumber
&\tilde{v}(x,t,z) = \begin{cases}
\begin{pmatrix} 1 & 0 \\ z^{2i\nu} \delta_1^{-2} (r_{1,a} + h_a) e^{\frac{iz^2}{2}}& 1 \end{pmatrix}, & k \in \mathcal{X}_1^\epsilon \cap D_1,
	\\
\begin{pmatrix} 1 & 0 \\ z^{2i\nu} \delta_1^{-2} (r_{1,a} + h) e^{\frac{iz^2}{2}}& 1 \end{pmatrix}, & k \in \mathcal{X}_1^\epsilon \cap D_2,
	\\
\begin{pmatrix} 1 & -z^{-2i\nu} \delta_1^2 r_{2,a} e^{-\frac{iz^2}{2}} \\ 0 & 1 \end{pmatrix}, & k \in \mathcal{X}_2^\epsilon,
	\\
\begin{pmatrix} 1 & 0 \\ -z^{2i\nu} \delta_1^{-2} r_{2,a}^* e^{\frac{iz^2}{2}}& 1 \end{pmatrix}, & k \in \mathcal{X}_3^\epsilon,
	\\
\begin{pmatrix} 1 & z^{-2i\nu} \delta_1^2 (r_{1,a}^* + h_a^*) e^{-\frac{iz^2}{2}} \\ 0 & 1 \end{pmatrix}, & k \in \mathcal{X}_4^\epsilon \cap D_4,
	\\
\begin{pmatrix} 1 & z^{-2i\nu} \delta_1^2 (r_{1,a}^* + h^*) e^{-\frac{iz^2}{2}} \\ 0 & 1 \end{pmatrix}, & k \in \mathcal{X}_4^\epsilon \cap D_3.
\end{cases}
\end{align}
The sets $\mathcal{X}_1^\epsilon \cap D_1$ and $\mathcal{X}_4^\epsilon \cap D_4$ are empty for small $k_0$.

Define the complex-valued function $q := q(\zeta)$ by
\begin{align}\label{qdefIII}
  q = r(k_0), \qquad \zeta \in \mathcal{I}.
\end{align}
For any fixed $z \in X$, we have $k(\zeta, z) \to k_0$ as $t \to \infty$; hence we expect the following in this limit:
$$\delta_1(k) \to 1, \quad r_{1,a}(k) + h(k) \to q, \quad  r_{2,a}(k) \to \frac{\bar{q}}{1 + |q|^2}.$$
This suggests that $\tilde{v}$ tends to the jump matrix $v^{X}$ defined in (\ref{vXdefIII}) for large $t$, i.e., that the jumps of $M^{(3)}$ for $k$ near $k_0$ approach those of the function $m^X \delta_0^{-\sigma_3} e^{\frac{t\Phi(\zeta,k_0)\sigma_3}{2}}$  as $t \to \infty$. 
This suggests that we approximate $M^{(3)}$ in the neighborhood $D_\epsilon(k_0)$ of $k_0$ by a $2 \times 2$-matrix valued function $m^{k_0}$ of the form
\begin{align}\label{mk0defIII}
m^{k_0}(x,t,k) = Y(\zeta,t,k) m^X(q(\zeta),z(\zeta, k)) \delta_0(\zeta, t)^{-\sigma_3}e^{\frac{t\Phi(\zeta, k_0)\sigma_3}{2}},
\end{align}
where $Y(\zeta,t,k)$ is a function which is analytic for $k \in D_\epsilon(k_0)$.
To ensure that $m^{k_0}$ is a good approximation of $M^{(3)}$ for large $t$, we choose $Y$ so that $m^{k_0} \to I$ on $\partial D_\epsilon(k_0)$ as $t \to \infty$. 
Hence we choose
$$Y(\zeta,t,k) \equiv
Y(\zeta,t) = e^{-\frac{t\Phi(\zeta, k_0)\sigma_3}{2}}\delta_0(\zeta, t)^{\sigma_3}.$$

Let $\tau := \tau(\zeta, t) = k_0 t$. 

\begin{lemma}
For each  $\zeta \in \mathcal{I}$ and $t > 0$, the function $m^{k_0}(x,t,k)$ defined in (\ref{mk0defIII}) is an analytic function of $k \in D_\epsilon(k_0) \setminus \mathcal{X}^\epsilon$. Moreover,
\begin{align}\label{mk0boundIII}
|m^{k_0}(x,t,k) - I| \leq C|q| \leq C, \qquad \zeta \in\mathcal{I}, \ t > 2, \ k \in D_\epsilon(k_0) \setminus \mathcal{X}^\epsilon.
\end{align}
Across $\mathcal{X}^\epsilon$, $m^{k_0}$ obeys the jump condition $m_+^{k_0} =  m_-^{k_0} v^{k_0}$, where the jump matrix $v^{k_0}$ satisfies, for each $1 \leq p \leq \infty$,
\begin{align}\label{J3vk0estimateIII}
 \|J^{(3)} - v^{k_0}\|_{L^p(\mathcal{X}^\epsilon)} \leq  C\tau^{-\frac{1}{2} + \frac{1}{2p}} t^{-\frac{1}{p}} (1 + k_0 \ln{t}),
\qquad \zeta \in \mathcal{I}, \ t >2.
\end{align}	
Furthermore, as $t \to \infty$,
\begin{align}\label{mk0LinftyestimateIII}
\|m^{k_0}(x,t,\cdot)^{-1} - I\|_{L^\infty(\partial D_\epsilon(k_0))} = O(q \tau^{-1/2}), 
\end{align}
and
\begin{align}\label{mk0circleIII}
\frac{1}{2\pi i}\int_{\partial D_\epsilon(k_0)}(m^{k_0}(x,t,k)^{-1} - I) \frac{dk}{k}
= -\frac{Y(\zeta,t) m_1^X(\zeta) Y(\zeta,t)^{-1}}{\sqrt{t} \psi(\zeta, k_0)k_0} + O(q \tau^{-1}),
\end{align}	
uniformly with respect to $\zeta \in \mathcal{I}$, where $m_1^X(\zeta)$ is defined by
\begin{align}\label{m1XdefIII}
m_1^X(\zeta) = -i\begin{pmatrix} 0 & \beta^X(q(\zeta)) \\ \overline{\beta^X(q(\zeta))} & 0 \end{pmatrix}.
\end{align}
\end{lemma}
\begin{proof}
The analyticity of $m^{k_0}$ follows directly from the definition. Since
$$m^{k_0}(x,t,k) = Y(\zeta,t)m^X(q, z)Y(\zeta,t)^{-1} 
= e^{-\frac{t\Phi(\zeta, k_0)\hat{\sigma}_3}{2}} \delta_0(\zeta, t)^{\hat{\sigma}_3} m^X(q, z),$$
where $|e^{-\frac{t\Phi(\zeta, k_0)}{2}}| = |\delta_0(\zeta, t)| = 1$, the estimate (\ref{mk0boundIII}) is a consequence of (\ref{mXqboundIII}).

We next establish (\ref{J3vk0estimateIII}). Standard estimates show that
$$|\chi(\zeta, k) - \chi(\zeta, k_0)| \leq  C |k - k_0| ( 1+ |\ln|k-k_0||), \qquad  \zeta \in \mathcal{I}, \ k \in \mathcal{X}^\epsilon,$$
and
\begin{align*}
|e^{i\nu\ln(\frac{(k+k_0)\psi(\zeta, k)}{2k_0\psi(\zeta, k_0)})} - 1|
 \leq C\bigg|\ln\bigg(\frac{(k+k_0)\psi(\zeta, k)}{2k_0\psi(\zeta, k_0)}\bigg)\bigg|
\leq Ck_0^{-1}|k-k_0|, \qquad \zeta \in \mathcal{I}, \ k \in \mathcal{X}^\epsilon.
\end{align*}
This yields
\begin{align}\nonumber
|\delta_1(\zeta, k) - 1| 
& \leq  |e^{i\nu\ln(\frac{(k+k_0)\psi(\zeta, k)}{2k_0\psi(\zeta, k_0)})} - 1| |e^{\chi(\zeta, k) - \chi(\zeta, k_0)}| + |e^{\chi(\zeta, k) - \chi(\zeta, k_0)} -1|
	\\ \label{delta1estimateIII}
& \leq  C |k - k_0| (k_0^{-1} + |\ln|k-k_0||), \qquad \zeta \in \mathcal{I}, 
\ k \in \mathcal{X}^\epsilon.
\end{align}
On the other hand, equations (\ref{haestimateIII}) and (\ref{rjaestimatesIII}) imply that the following estimates hold for all $\zeta \in \mathcal{I}$ and $t > 2$:
\begin{align}\nonumber
& \big|r_{1,a}(x,t,k) + h_a(t,k) - q\big| \leq C|k - k_0|e^{\frac{t}{4} |\re \Phi(\zeta,k )|}, && k \in \mathcal{X}_1^\epsilon \cap D_1,
 	\\\nonumber
& \big|r_{1,a}(x,t,k) + h(k) - q\big| \leq C|k - k_0|e^{\frac{t}{4} |\re \Phi(\zeta,k )|}, && k \in \mathcal{X}_1^\epsilon \cap D_2,
 	\\\nonumber
& \bigg|r_{2,a}(x,t,k) - \frac{\bar{q}}{1 + |q|^2}\bigg| \leq C|k - k_0|e^{\frac{t}{4} |\re \Phi(\zeta,k )|},&& k \in \mathcal{X}_2^\epsilon,
  	\\\nonumber
& \bigg| \overline{r_{2,a}(x,t,\bar{k})} - \frac{q}{1 + |q|^2}\bigg| \leq C|k - k_0|e^{\frac{t}{4} |\re \Phi(\zeta,k )|},&& k \in \mathcal{X}_3^\epsilon,
 	\\ \nonumber
&  \big|\overline{r_{1,a}(x,t,\bar{k})} + \overline{h_a(t,\bar{k})} - \bar{q}\big| \leq C|k - k_0|e^{\frac{t}{4} |\re \Phi(\zeta,k )|}, 
&& k \in \mathcal{X}_4^\epsilon \cap D_4,
 	\\ \label{rqestimateIII}
&  \big| \overline{r_{1,a}(x,t,\bar{k})} + \overline{h(\bar{k})} - \bar{q}\big| \leq C|k - k_0|e^{\frac{t}{4} |\re \Phi(\zeta,k )|}, 
&& k \in \mathcal{X}_4^\epsilon \cap D_3.
\end{align}
Indeed, for $k \in \mathcal{X}_1^\epsilon \cap D_1$, the estimates (\ref{haestimateIII}) and (\ref{rjaestimatesIII}) give
\begin{align*}
 \big|&r_{1,a}(x,t,k) + h_a(t,k) - q\big| 
= \big|r_{1,a}(x,t,k) + h_a(t,k) - r_1(k_0) - h(k_0)\big| 
	\\
 \leq & \; |r_{1,a}(x,t,k)- r_1(k_0)| + |h_a(t,k) - h(1)| + |h(1) - h(k_0)|
	\\
\leq &\; C |k - k_0|e^{\frac{t}{4} |\re \Phi(\zeta,k )|}
 + C|k-1|e^{\frac{t}{4} |\re \Phi(\zeta,k )|} + C|k-k_0|
	\\	 
\leq &\; C|k - k_0|e^{\frac{t}{4} |\re \Phi(\zeta,k )|},
\end{align*}
which proves the first estimate in (\ref{rqestimateIII}); the proofs of the other estimates are similar.

Since
\begin{align*}
\tilde{v} - v^X
= 
 \begin{cases}
\begin{pmatrix} 0 & 0 \\ (\delta_1^{-2} (r_{1,a} + h_a) - q )z^{2i\nu}e^{\frac{iz^2}{2}} & 0 \end{pmatrix}, & k \in \mathcal{X}^\epsilon_1 \cap D_1,
	\\
\begin{pmatrix} 0 & 0 \\ (\delta_1^{-2} (r_{1,a} + h)  - q)z^{2i\nu}e^{\frac{iz^2}{2}} & 0 \end{pmatrix}, & k \in \mathcal{X}^\epsilon_1 \cap D_2,
	\\
\begin{pmatrix} 0 & (-\delta_1^2 r_{2,a} + \frac{\bar{q}}{1 + |q|^2})z^{-2i\nu} e^{-\frac{iz^2}{2}} \\ 0 & 0 \end{pmatrix}, & k \in \mathcal{X}^\epsilon_2,
	\\
\begin{pmatrix} 0 & 0 \\ -(\delta_1^{-2} r_{2,a}^* - \frac{q}{1 + |q|^2})z^{2i\nu} e^{\frac{iz^2}{2}} & 0 \end{pmatrix}, & k \in \mathcal{X}^\epsilon_3,
	\\
\begin{pmatrix} 0 & (\delta_1^2 (r_{1,a}^* + h_a^*) - \bar{q}) z^{-2i\nu}e^{-\frac{iz^2}{2}}\\ 0 & 0 \end{pmatrix}, & k \in \mathcal{X}^\epsilon_4 \cap D_4,
	\\
\begin{pmatrix} 0 & (\delta_1^2 (r_{1,a}^* + h^*) - \bar{q}) z^{-2i\nu}e^{-\frac{iz^2}{2}}\\ 0 & 0 \end{pmatrix}, & k \in \mathcal{X}^\epsilon_4 \cap D_3,
\end{cases}
\end{align*}
equations (\ref{delta1estimateIII}) and (\ref{rqestimateIII}) imply 
\begin{align}\label{vtildevXIII}
|\tilde{v} - v^X| \leq C|k - k_0| (k_0^{-1} + |\ln|k-k_0||)e^{-\frac{3t|k - k_0|^2}{16k_0}}, \qquad k \in \mathcal{X}^\epsilon.
\end{align}
Indeed, for $k \in \mathcal{X}^\epsilon$, we have $|\psi(\zeta, k)| \geq |k|^{-1/2} \geq (2k_0)^{-1/2}$ and so
$$t \re \Phi(\zeta, k) = \re\Big(\frac{iz^2}{2}\Big) = - \frac{|z|^2}{2} = -\frac{t}{2}|k - k_0|^2|\psi(\zeta, k)|^2
\leq -\frac{t|k - k_0|^2}{4k_0}, \qquad k \in \mathcal{X}^\epsilon_1,$$
and hence
\begin{align*}
 |\tilde{v} - v^X|
\leq &\; |\delta_1^{-2} - 1| \big|(r_{1,a} + h_a)z^{2i\nu}e^{\frac{iz^2}{2}}\big|
+ \big|(r_{1,a} + h_a - q)z^{2i\nu}e^{\frac{iz^2}{2}}\big|
	\\
\leq&\; C |\delta_1^{-2} - 1| e^{\frac{t}{4} |\re \Phi(\zeta,k )|} e^{\re \frac{iz^2}{2}}
+ C|k - k_0|e^{\frac{t}{4} |\re \Phi(\zeta,k )|} e^{\re \frac{iz^2}{2}}
	\\
\leq &\; C|k - k_0| (k_0^{-1} + |\ln|k-k_0||)e^{-\frac{3}{4}|\re \frac{iz^2}{2}|}
	\\
\leq &\; C|k - k_0| (k_0^{-1} + |\ln|k-k_0||)e^{-\frac{3t|k - k_0|^2}{16k_0}}, \qquad k \in  \mathcal{X}^\epsilon_1 \cap D_1,
\end{align*}
which gives (\ref{vtildevXIII}) for $k \in \mathcal{X}_1^\epsilon \cap D_1$; the proof is similar for the other parts of $\mathcal{X}^\epsilon$.

Since
$$J^{(3)} - v^{k_0} =  e^{-\frac{t\Phi(\zeta, k_0)}{2}\hat{\sigma}_3} \delta_0^{\hat{\sigma}_3} (\tilde{v} - v^X), \qquad k \in \mathcal{X}^\epsilon,$$
equation (\ref{vtildevXIII}) implies
$$|J^{(3)} - v^{k_0}| \leq C|k - k_0| (k_0^{-1} + |\ln|k-k_0||)e^{-\frac{3t|k - k_0|^2}{16k_0}}, \qquad k \in \mathcal{X}^\epsilon.$$
Thus
$$ \|J^{(3)} - v^{k_0}\|_{L^1(\mathcal{X}^\epsilon)}
\leq C\int_0^\epsilon u (k_0^{-1}+ |\ln u|)e^{-\frac{3tu^2}{16k_0}} du
\leq C t^{-1}(1 + k_0 \ln{t}), \quad \zeta \in \mathcal{I}, \ t > 2,$$
and
$$ \|J^{(3)} - v^{k_0}\|_{L^\infty(\mathcal{X}^\epsilon)}
\leq C \sup_{0 \leq u \leq \epsilon} u (k_0^{-1}+ |\ln u|)e^{-\frac{3tu^2}{16k_0}} 
\leq  C \tau^{-1/2}(1 + k_0 \ln{t}), \quad \zeta \in \mathcal{I}, \  t > 2,$$
which thanks to the general inequality $\|f\|_{L^p} \leq \|f\|_{L^\infty}^{1 - 1/p} \|f\|_{L^1}^{1/p}$ gives (\ref{J3vk0estimateIII}).

The variable $z = \sqrt{t}(k - k_0)\psi(\zeta, k)$ goes to infinity as $t\to \infty$ if $k \in \partial D_\epsilon(k_0)$. In fact, for $k \in \partial D_\epsilon(k_0)$, we have $|z| \geq C\tau^{1/2}$, because $|\psi(\zeta, k)| \geq (2k_0)^{-1/2}$.
Thus equation (\ref{mXasymptoticsIII}) yields
\begin{align*}
 & m^X(q, z(\zeta, k)) = I + \frac{m_1^X(\zeta)}{\sqrt{t}(k - k_0)\psi(\zeta, k)} + O(q\tau^{-1}), 
\qquad  t \to \infty,  \ k \in \partial D_\epsilon(k_0),
 \end{align*}  
uniformly with respect to $k \in \partial D_\epsilon(k_0)$, where $m_1^X$ is given by (\ref{m1XdefIII}).
Since $m^{k_0} = Y m^X Y^{-1}$ and $|Y| \leq C$, this shows that
\begin{align}\label{mk0m1XIII}
 & (m^{k_0})^{-1} - I = - \frac{Ym_1^X(\zeta) Y^{-1}}{\sqrt{t}(k - k_0)\psi(\zeta, k)} + O(q \tau^{-1}), \qquad  t \to \infty,
 \end{align}  
uniformly with respect to $\zeta \in \mathcal{I}$ and $k \in \partial D_\epsilon(k_0)$. Using again that $|\psi(\zeta, k)| \geq (2k_0)^{-1/2}$ on $\partial D_\epsilon(k_0)$ and the bound $|m_1^X| \leq C |q|$, this proves (\ref{mk0LinftyestimateIII}). 
Finally, equation (\ref{mk0circleIII})  follows from (\ref{mk0m1XIII}) and Cauchy's formula.
\end{proof}

\subsection{Final steps}
Define the approximate solution $M^{app}$ by
$$M^{app}(x,t,k) = \begin{cases} m^{k_0}(x,t,k), & k \in D_\epsilon(k_0), \vspace{.1cm} \\
\overline{m^{k_0}(x,t,-\bar{k})}, & k \in D_\epsilon(-k_0), \\
I, & \text{elsewhere}.
\end{cases}$$
The function $m(x,t,k)$ defined by
$$m = M^{(3)} (M^{app} )^{-1}$$
satisfies the RH problem (\ref{RHmSigma}), where the contour $\Sigma = \Sigma(\zeta) = \Gamma^{(3)} \cup \partial D_\epsilon(k_0) \cup \partial D_\epsilon(-k_0)$ is displayed in Figure \ref{SigmaIII.pdf} and the jump matrix $v$ is given by 
$$v = \begin{cases}
J^{(3)}, & k \in \Sigma \setminus (\overline{D_\epsilon(k_0)} \cup \overline{D_\epsilon(-k_0)}), 
	\\
(m^{k_0})^{-1}, & k \in \partial D_\epsilon(k_0), 
	\\
m_-^{k_0} J^{(3)}(m_+^{k_0})^{-1}, & k \in \Sigma \cap D_\epsilon(k_0),
\end{cases}
$$
and is extended to the remainder of $\Sigma$ via the symmetry
$$v(x,t,k) = \overline{v(x, t, -\bar{k})}, \qquad k \in \Sigma.$$

\begin{figure}
\begin{center}
 \begin{overpic}[width=.6\textwidth]{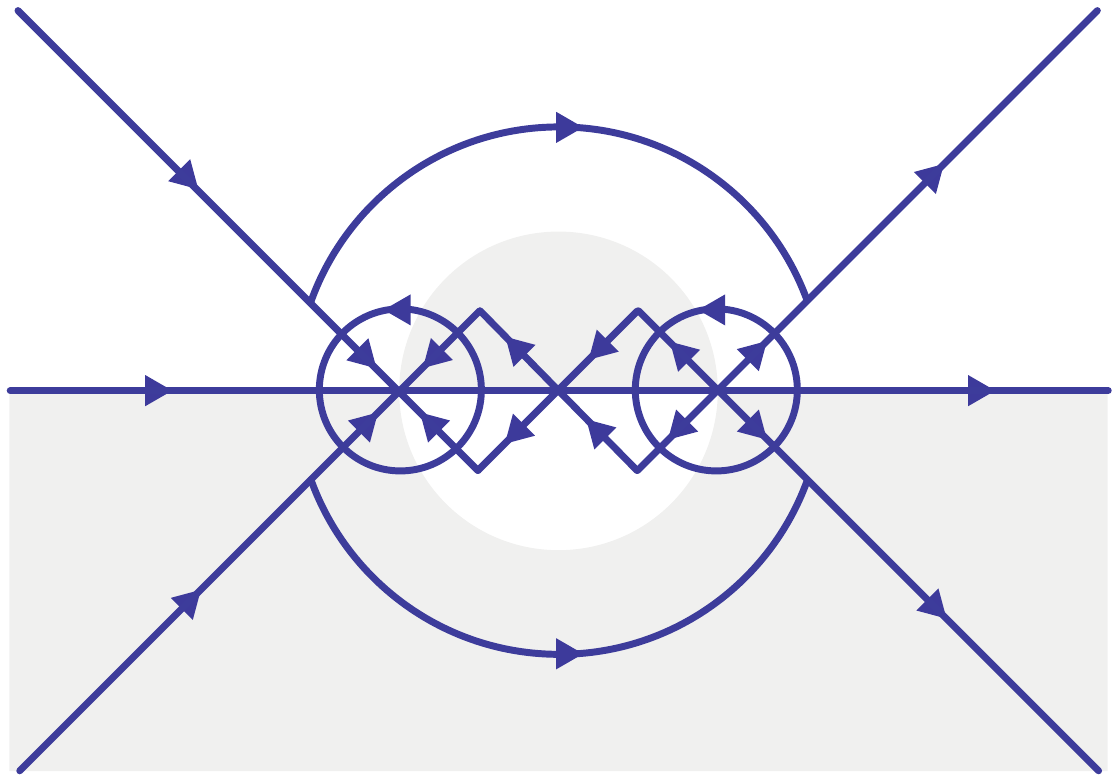}
      \put(102,33.5){$\Sigma$}
      \put(43,64){\small $\re \Phi < 0$}
      \put(43,5){\small $\re \Phi > 0$}
    \end{overpic}
    \vspace{.5cm}
     \begin{figuretext}\label{SigmaIII.pdf}
        The contour  $\Sigma$ in the case of Sector III.
     \end{figuretext}
     \end{center}
\end{figure}

We divide $\Sigma$ into six pieces as follows:
$$\Sigma = \mathcal{X}^\epsilon \cup (-\mathcal{X}^\epsilon) \cup \partial D_\epsilon(k_0) \cup \partial D_\epsilon(-k_0) \cup \Sigma' \cup \Sigma'',$$
where (see Figures \ref{Sigma1primeIII.pdf} and \ref{Sigma2primeIII.pdf})
\begin{align}\label{Sigmaprimedef}
\Sigma' := \Gamma^{(3)} \cap (\R \cup \{|k| = 1\})
\quad \text{and} \quad
\Sigma'' := \bigcup_{j=1}^6 \Gamma_j^{(3)} \setminus (D_\epsilon(k_0) \cup D_\epsilon(-k_0)).
\end{align}

\begin{lemma}\label{wlemmaIII}
Let $w = v - I$. For each $1 \leq p \leq \infty$, the following estimates hold uniformly for $t > 2$ and $\zeta \in \mathcal{I}$:
\begin{subequations}\label{westimateIII}
\begin{align}\label{westimateIIIa}
& \|w\|_{L^p(\partial D_\epsilon(k_0))} \leq C|q| k_0^{1/p}\tau^{-\frac{1}{2}},
	\\\label{westimateIIIb}
& \|w\|_{L^p(\mathcal{X}^\epsilon)} \leq C\tau^{-\frac{1}{2} + \frac{1}{2p}} t^{-\frac{1}{p}} (1 + k_0 \ln{t}).
	\\ \label{westimateIIIc}
& \|(1+|k|^{-1})w\|_{L^p(\Sigma')} \leq C t^{-3/2},  
	\\\label{westimateIIId}
& \|(1+|k|^{-1})w\|_{L^p(\Sigma'')} \leq C(t^{-1/p} + k_0)e^{-c\tau},	
\end{align}
\end{subequations}
\end{lemma}
\begin{proof}
The estimate (\ref{westimateIIIa}) follows from (\ref{mk0LinftyestimateIII}).
For $k \in \mathcal{X}^\epsilon$, we have
$$w = m_-^{k_0} (J^{(3)} - v^{k_0}) (m_+^{k_0})^{-1},$$
so equations (\ref{mk0boundIII}) and (\ref{J3vk0estimateIII}) yield (\ref{westimateIIIb}).

On $\Sigma'$, the jump matrix $J^{(3)}$ involves the small remainders $h_r$, $r_{1,r}$, $r_{2,r}$, so the estimate (\ref{westimateIIIc}) holds as a consequence of Lemma \ref{hdecompositionlemmaIII}, Lemma \ref{decompositionlemmaIII}, and (for the part of $\Sigma'$ that lies in the disks $D_\epsilon(\pm k_0)$) the boundedness (\ref{mk0boundIII}) of $m^{k_0}$.

\begin{figure}
\begin{center}
\begin{overpic}[width=.6\textwidth]{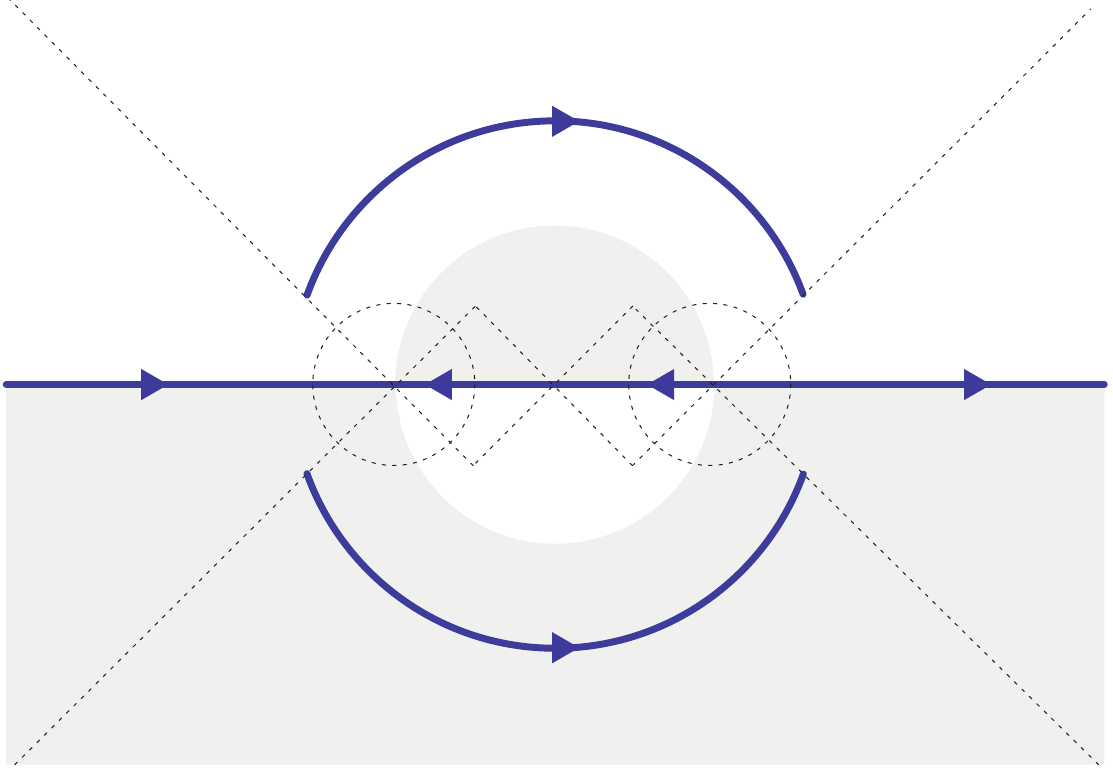}
      \put(102,33.5){\small $\Sigma'$}
    \end{overpic}
     \begin{figuretext}\label{Sigma1primeIII.pdf}
       The subcontour $\Sigma'$ of $\Sigma$.
      \end{figuretext}
     \end{center}
\end{figure}

It only remains to prove (\ref{westimateIIId}). 
We define two subsets $l_1$ and $l_2$ of $\Sigma''$ as follows (see Figure \ref{Sigma2primeIII.pdf}): Let $l_1$ denote the part of $(\Gamma_1^{(3)} \cup \Gamma_6^{(3)}) \setminus D_\epsilon(k_0)$ that lies in the right half-plane, and let $l_2$ denote the part of $\Gamma_2^{(3)} \setminus D_\epsilon(k_0)$ that lies in the right half-plane. The contour $\Sigma''$ consists (up to orientation) of the curves $l_1$ and $l_2$ together with their images under the involutions $k \mapsto \bar{k}$ and $k\mapsto -\bar{k}$. Hence, by symmetry, (\ref{westimateIIId}) will follow once we prove the following two estimates for $1 \leq p \leq \infty$:
\begin{subequations}
\begin{align} \label{l1boundIII}
\|(1+|k|^{-1})w(x,t,k)\|_{L^p(l_1)} & \leq Ct^{-1/p}e^{-c\tau},
	\\ \label{l2boundIII}
\|(1+|k|^{-1})w(x,t,k)\|_{L^p(l_2)} & \leq Ck_0e^{-c\tau}.
\end{align} 
\end{subequations}

\begin{figure}
\begin{center}
    \begin{overpic}[width=.6\textwidth]{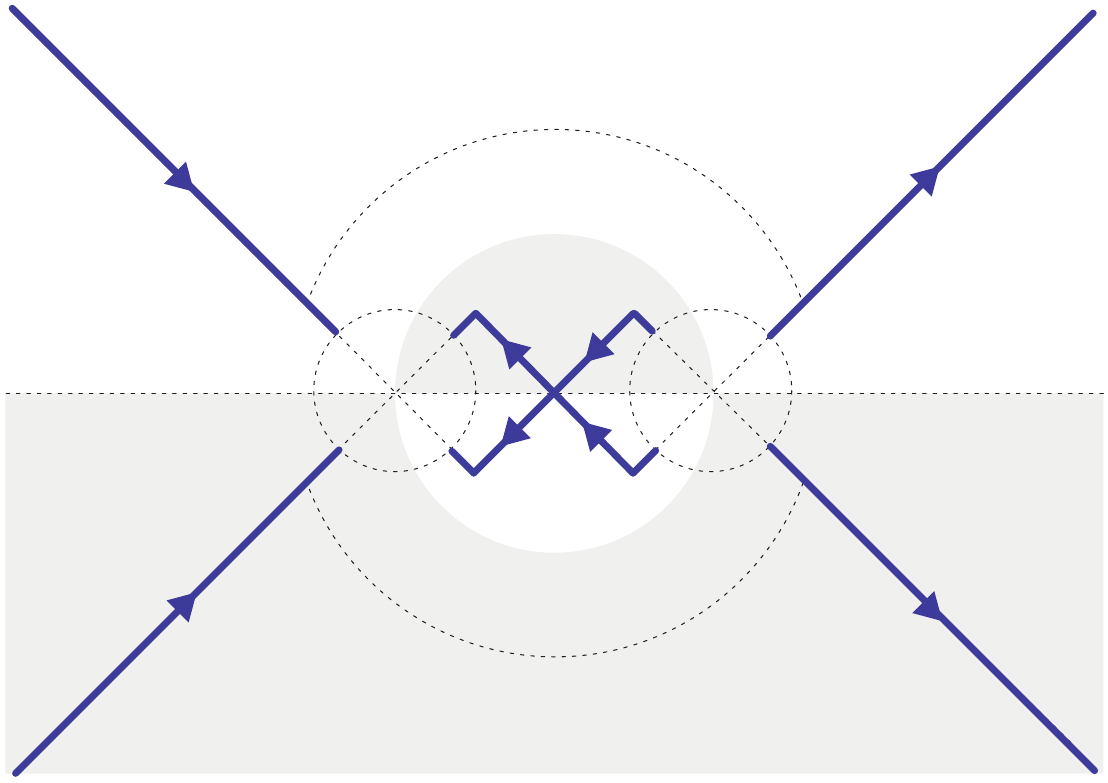}
      \put(85.5,52){\small $l_1$}
      \put(52,42){\small $l_2$}
    \end{overpic}
     \begin{figuretext}\label{Sigma2primeIII.pdf}
       The subcontour $\Sigma''$ of $\Sigma$ consists of the curves $l_1$ and $l_2$ and their images under reflection in the real and imaginary axes. 
     \end{figuretext}
     \end{center}
\end{figure}

In view of (\ref{rePhionL1}), we have
\begin{align}\label{rePhik0uIII}
|\re \Phi(\zeta, k_0 + ue^{\frac{\pi i}{4}})| \geq \frac{u^2(k_0 + u)}{4(k_0^2+\sqrt{2} k_0 u+u^2)}
\geq \frac{k_0 + u}{32}, \qquad \zeta \in \mathcal{I}, \ \epsilon \leq u.
\end{align}
The curve $l_1$  is a small deformation of the ray $\{k_0 + ue^{\frac{\pi i}{4}} \, | \, \epsilon \leq u < \infty\}$. Hence (\ref{rePhik0uIII}) implies that there exists a $c> 0$ such that
\begin{align}\label{rePhionl1III}
|\re \Phi(\zeta, k)| \geq c|k|, \qquad \zeta \in \mathcal{I}, \ k \in l_1.
\end{align}
On the other hand, since $r(0) = r_1(0) + h(0) = 0$, the estimates (\ref{haestimateIII}) and (\ref{rjaestimatesIII}) give
\begin{align} \nonumber
|r_{1,a}(x,t,k) + h(k)| & = |r_{1,a}(x,t,k) - r_1(k_0) + r_1(k_0) - r_1(0) + h(k) - h(0)|
	\\ \nonumber
& \leq C |k - k_0| e^{\frac{t}{4}|\re \Phi(\zeta,k)|} + C |k_0| + C|k|
	\\ \label{r1ahboundIII}
& \leq C|k|e^{\frac{t}{4}|\re \Phi(\zeta,k)|}, \qquad \zeta \in \mathcal{I}, \ k \in l_1 \cap D_2.
\end{align}
For $k \in l_1$, only the $(21)$ entry of $w$ is nonzero and, using (\ref{haestimateIII}), (\ref{rjaestimatesIII}), (\ref{rePhionl1III}), and (\ref{r1ahboundIII}),
\begin{align*}
|(w(x,t,k))_{21}| & = |\delta(\zeta, k)^{-2}(r_{1,a}(x,t,k) + \tilde{h}(t, k)) e^{t\Phi(\zeta, k)}|
	\\
& \leq C |r_{1,a}(x,t,k) + \tilde{h}(t, k)| e^{-t|\re \Phi|}
	\\
& \leq C|k| e^{-\frac{3t}{4}|\re \Phi|}
 \leq C|k| e^{-ct|k|}, \qquad \zeta \in \mathcal{I}, \ k \in l_1,
\end{align*}
where $\tilde{h}(t, k) = h(k)$ for $k \in l_1 \cap D_2$ and $\tilde{h}(t, k) = h_a(t,k)$ for $k \in l_1 \cap D_1$. The estimate (\ref{l1boundIII}) follows. 

We finally prove (\ref{l2boundIII}). The same kind of argument that led to (\ref{rePhionl1III}) shows that
$$|\re \Phi(\zeta, k)| \geq ck_0, \qquad \zeta \in \mathcal{I}, \ k \in l_2.$$
For $k \in l_2$, only the $(12)$ entry of $w$ is nonzero and, using (\ref{rjaestimatesIII}),
\begin{align*}
|(w(x,t,k))_{12}| & = |-\delta(\zeta, k)^2 r_{2,a}(x,t,k) e^{-t\Phi(\zeta, k)}|
\leq C |r_{2,a}| e^{-t|\re \Phi|}
	\\
& \leq C|k|^2e^{-\frac{3t}{4}|\re \Phi|}
\leq C|k|^2 e^{-ck_0 t}, \qquad k \in l_2.
\end{align*}
This gives (\ref{l2boundIII}).
\end{proof}

The estimates in Lemma \ref{wlemmaIII} show that
\begin{align}\label{hatwLinftyIII}
\begin{cases}
\|w\|_{L^2(\Sigma)} \leq C F_2(\zeta, t),
	\\
\|w\|_{L^\infty(\Sigma)} \leq  C F_\infty(\zeta,t),
\end{cases}	 \qquad t > 2, \ \zeta \in \mathcal{I},
\end{align}
where
\begin{align*}
& F_2(\zeta,t) := |q| t^{-\frac{1}{2}} + \tau^{-\frac{1}{4}} t^{-\frac{1}{2}} (1 + k_0 \ln{t}) + t^{-\frac{3}{2}} + (t^{-\frac{1}{2}} + k_0)e^{-c\tau},
	\\
& F_\infty(\zeta,t) := |q|\tau^{-\frac{1}{2}} + C\tau^{-\frac{1}{2}} (1 + k_0 \ln{t}) + t^{-\frac{3}{2}} + e^{-c\tau}.
\end{align*}

By (\ref{hatwLinftyIII}), we have
\begin{align}\label{CwnormIII}
\|\mathcal{C}_{w}\|_{\mathcal{B}(L^2(\Sigma))} \leq C \|w\|_{L^\infty(\Sigma)}  
\leq F_\infty(\zeta, t), \qquad t > 2, \ \zeta \in \mathcal{I},
\end{align}
where the operator $\mathcal{C}_{w}: L^2(\Sigma) + L^\infty(\Sigma) \to L^2(\Sigma)$ is defined by
$\mathcal{C}_{w}f = \mathcal{C}_-(f w)$.
In particular, since $F_\infty \to 0$ as $\tau \to \infty$, there exists a $T > 2$ such that $\|\mathcal{C}_{w}\|_{\mathcal{B}(L^2(\Sigma))} < 1/2$ for all $\tau > T$. For $\tau > T$, the operator $I - \mathcal{C}_{w(\zeta, t, \cdot)} \in \mathcal{B}(L^2(\Sigma))$ is invertible and the RH problem (\ref{RHmSigma}) has a unique solution $m \in I + \dot{E}^2(\hat{\C} \setminus \Sigma)$ given by (\ref{mrepresentationI}) where $\mu(x, t, k) \in I + L^2(\Sigma)$ is defined by (\ref{mudef}).
By (\ref{hatwLinftyIII}) and (\ref{CwnormIII}),
\begin{align}\label{muestimateIII}
\|\mu(x,t,\cdot) - I\|_{L^2(\Sigma)} 
\leq  \frac{C\|w\|_{L^2(\Sigma)}}{1 - \|\mathcal{C}_{w}\|_{\mathcal{B}(L^2(\Sigma))}}\leq C F_2(\zeta, t), \qquad \tau > T, \ \zeta \in \mathcal{I}.
\end{align}

\subsection{Asymptotics of $m$}\label{asymptoticsofmsubsecIII}
We see from (\ref{mrepresentationI}) and Lemma \ref{deltalemmaIII} that the function $\hat{M}(x,t)$ defined in (\ref{m0def}) satisfies
\begin{align}\label{m0limIII}
\hat{M}(x,t) 
 =\ntlim_{k\to 0} M(x,t,k)  = \ntlim_{k\to 0} m(x,t,k)\delta(\zeta, k)^{\sigma_3}  
 = I + \frac{1}{2\pi i}\int_{\Sigma} \mu(x, t, k) w(x, t, k) \frac{dk}{k}.
\end{align}
From now on until the end of Section \ref{sectorIIIsec}, all equations involving error terms of the form $O(\cdot)$ will be valid uniformly for all $(x,t) \in [0,\infty) \times [0,\infty)$ such that $\zeta \in \mathcal{I}$ and $\tau > T$.

By (\ref{mk0circleIII}), (\ref{westimateIIIa}), and (\ref{muestimateIII}), the contribution from $\partial D_\epsilon(k_0)$ to the right-hand side of (\ref{m0limIII}) is
\begin{align*}
&\; \frac{1}{2\pi i} \int_{\partial D_\epsilon(k_0)} w(x,t,k) \frac{dk}{k}	
+  \frac{1}{2\pi i} \int_{\partial D_\epsilon(k_0)} (\mu(x,t,k) - I) w(x,t,k) \frac{dk}{k}
	\\
= &\; \frac{1}{2\pi i}\int_{\partial D_\epsilon(k_0)}((m^{k_0})^{-1} - I) \frac{dk}{k}
+ O(k_0^{-1}\|\mu - I\|_{L^2(\partial D_\epsilon(k_0))}\|w\|_{L^2(\partial D_\epsilon(k_0))})
	\\
= & \; -\frac{Y(\zeta,t) m_1^X(\zeta) Y(\zeta,t)^{-1}}{\sqrt{t} \psi(\zeta, k_0)k_0} 
+ O(q\tau^{-1})
+ O(k_0^{-1} F_2(\zeta, t) q t^{-1/2}).
\end{align*}
The symmetry $w(x,t,k) = \overline{w(x,t,-\bar{k})}$ implies that the contribution from $\partial D_\epsilon(-k_0)$ to the right-hand side of (\ref{m0limIII}) is the complex conjugate of the contribution from $\partial D_\epsilon(k_0)$.

By (\ref{westimateIIIb}) and (\ref{muestimateIII}), the contribution from $\mathcal{X}^\epsilon$ to the right-hand side of (\ref{m0limIII}) is
\begin{align*}
O\big(k_0^{-1}\|w\|_{L^1(\mathcal{X}^\epsilon)}
& + k_0^{-1}\|\mu - I\|_{L^2(\mathcal{X}^\epsilon)}\|w\|_{L^2(\mathcal{X}^\epsilon)}\big)
	\\
&= O\big((\tau^{-1} + F_2(\zeta,t) k_0^{-1/2} \tau^{-3/4})(1 + k_0\ln{t})\big).
\end{align*}

By (\ref{westimateIIIc}) and (\ref{muestimateIII}),
\begin{align*}
\int_{\Sigma'} \mu(x,t,k) w(x,t,k) \frac{dk}{k}
& = \int_{\Sigma'} w(x,t,k) \frac{dk}{k}
+
\int_{\Sigma'} (\mu(x,t,k) - I) w(x,t,k) \frac{dk}{k}
	\\
& = O(\|k^{-1}w\|_{L^1(\Sigma')})
+ O(\|\mu - I\|_{L^2(\Sigma')}\|k^{-1}w\|_{L^2(\Sigma')})
	\\
& = O(t^{-3/2}),
\end{align*}
showing that the contribution from $\Sigma'$ to the integral in (\ref{m0limIII}) is $O(t^{-3/2})$.
Similarly, by (\ref{westimateIIId}) and (\ref{muestimateIII}), the contribution from $\Sigma''$ is
\begin{align*}
\int_{\Sigma''} \mu(x,t,k) w(x,t,k) \frac{dk}{k}
& = O(\|k^{-1}w\|_{L^1(\Sigma'')})
+ O(\|\mu - I\|_{L^2(\Sigma'')}\|k^{-1}w\|_{L^2(\Sigma'')})
	\\
& = O((t^{-1/2} + k_0)e^{-c\tau}).
\end{align*}

Collecting the above contributions, we find from (\ref{m0limIII}) that
\begin{align*}\nonumber
\hat{M}(x,t) - I
=& -2\re\bigg(\frac{Y(\zeta,t) m_1^X(\zeta) Y(\zeta,t)^{-1}}{\sqrt{t} \psi(\zeta, k_0)k_0} \bigg)
+ O(q\tau^{-1}) + O(k_0^{-1} F_2(\zeta, t) q t^{-1/2}) 
	\\  \nonumber
&+ O\big((\tau^{-1} + F_2(\zeta,t) k_0^{-1/2} \tau^{-3/4})(1 + k_0\ln{t})\big)
 	\\ 
&  +O(t^{-3/2})
+O((t^{-1/2} + k_0)e^{-c\tau}).
\end{align*}
Using that $q = r(k_0) = O(k_0^N)$ for any  $N \geq 1$, we can simplify the error terms as follows:
\begin{align}\label{m0lim2III}
\hat{M}(x,t) - I
= -2\re\bigg(\frac{Y(\zeta,t) m_1^X(\zeta) Y(\zeta,t)^{-1}}{\sqrt{t} \psi(\zeta, k_0)k_0} \bigg)
+O\bigg(\frac{1}{\tau} + \frac{\ln{t}}{t}\bigg).
\end{align}

\subsection{Asymptotics of $u$}
Substituting the asymptotics (\ref{m0lim2III}) into the definition (\ref{ulim}) of $u(x,t)$, we obtain
\begin{align*}
  u(x,t) & = 2 \arg\bigg(1 - \frac{2i \re(Y(\zeta,t) m_1^X(\zeta) Y(\zeta,t)^{-1})_{21}}{\sqrt{t} \psi(\zeta, k_0)k_0} + O\bigg(\frac{1}{\tau} + \frac{\ln{t}}{t}\bigg)\bigg).
\end{align*}
Using that
$$(Y(\zeta,t) m_1^X(\zeta) Y(\zeta,t)^{-1})_{21}
= -i \sqrt{\nu} e^{i\alpha(\zeta,t)}.$$
where $\alpha(\zeta,t)$ is defined by (\ref{alphadef}),  we find
\begin{align*}
  u(x,t) & = 2 \arg\bigg(1 - i\sqrt{\frac{2(1 + k_0^2)\nu}{\tau}}   \sin \alpha+ O\bigg(\frac{1}{\tau} + \frac{\ln{t}}{t}\bigg)\bigg)	
	\\
 & = -2 \arctan\bigg(\sqrt{\frac{2(1 + k_0^2)\nu}{\tau}}   \sin \alpha+ O\bigg(\frac{1}{\tau} + \frac{\ln{t}}{t}\bigg)\bigg)	
  	\\
& = -2 \sqrt{\frac{2(1 + k_0^2)\nu}{\tau}}   \sin \alpha+ O\bigg(\frac{1}{\tau} + \frac{\ln{t}}{t}\bigg) \quad \text{(mod $4\pi$)}.	
\end{align*}
This completes the proof of the asymptotic formula (\ref{uasymptoticsIII}).

\begin{remark}\label{MlambdajIIIremark}\upshape
For the purposes of Section \ref{solitonsec}, we note that if $\lambda_j$ is a point in $\C \setminus \Gamma$ such that 
\begin{align}\label{distlambdajSigma}
\inf_{0 \leq \zeta < 1} \dist(\lambda_j, \Sigma(\zeta)) > 0,
\end{align}
then the following analog of (\ref{m0lim2III}) is valid uniformly for $\zeta \in [0,1)$ as $t \to \infty$:
\begin{align}\label{mlambdajIII}
m(x,t,\lambda_j) 
= I 
+ \frac{i}{\sqrt{t}}\bigg\{
\frac{1}{(k_0 - \lambda_j)}\begin{pmatrix} 0 & \bar{A} \\ A & 0 \end{pmatrix}
- \frac{1}{(k_0 + \lambda_j)}\begin{pmatrix} 0 & A \\ \bar{A} & 0 \end{pmatrix}\bigg\}
+O\bigg(\frac{1}{\tau} + \frac{\ln{t}}{t}\bigg),
\end{align}
where $A := A(x,t)$ is given by
\begin{align}\label{Adef}
A(x,t) = \frac{i (Y(\zeta,t) m_1^X(\zeta) Y(\zeta,t)^{-1})_{21}}{\psi(\zeta, k_0)}
= \sqrt{\frac{k_0(1 + k_0^2)\nu}{2}} e^{i\alpha(\zeta,t)}.
\end{align}
Indeed, by (\ref{mrepresentationI}),
\begin{align*}
m(x, t, \lambda_j) = I + \frac{1}{2\pi i}\int_{\Sigma} (\mu w)(x, t, s) \frac{ds}{s - \lambda_j}.
\end{align*}
Hence, using the expansion 
\begin{align}
\frac{1}{2\pi i}\int_{\partial D_\epsilon(k_0)}(m^{k_0}(x,t,k)^{-1} - I) \frac{dk}{k-\lambda_j}
= -\frac{Y(\zeta,t) m_1^X(\zeta) Y(\zeta,t)^{-1}}{\sqrt{t} \psi(\zeta, k_0)(k_0 - \lambda_j)} + O(q t^{-1}),
\end{align}	
which is an easy analog of (\ref{mk0circleIII}), the same type of arguments that gave (\ref{m0lim2III}) yield 
\begin{align*}
m(x,t,\lambda_j) 
= I - \frac{Y(\zeta,t) m_1^X(\zeta) Y(\zeta,t)^{-1}}{\sqrt{t} \psi(\zeta, k_0)(k_0 - \lambda_j)} 
- \overline{\bigg(\frac{Y(\zeta,t) m_1^X(\zeta) Y(\zeta,t)^{-1}}{\sqrt{t} \psi(\zeta, k_0)(k_0 + \bar{\lambda}_j)} \bigg) }
+O\bigg(\frac{1}{\tau} + \frac{\ln{t}}{t}\bigg)
\end{align*}
uniformly for $\zeta \in [0,1)$ as $t \to \infty$.
Simplification gives (\ref{mlambdajIII}).

Given $\lambda_j \in \C \setminus \Gamma$, we may ensure that (\ref{distlambdajSigma}) holds by defining $\epsilon = ck_0$ with $c > 0$ sufficiently small and by modifying the contour deformations if necessary (cf. Remarks \ref{MlambdajIremark} and \ref{MlambdajIIremark}). Thus, recalling the relationship between $M$ and $m$, (\ref{mlambdajIII}) implies, for each fixed $\lambda_j \in \C \setminus \Gamma$, 
\begin{align}\label{MlambdajIII}
  M(x,t,\lambda_j) 
= \bigg\{I + \frac{i}{\sqrt{t}}\bigg\{
\frac{1}{(k_0 - \lambda_j)}\begin{pmatrix} 0 & \bar{A} \\ A & 0 \end{pmatrix}
- \frac{1}{(k_0 + \lambda_j)}\begin{pmatrix} 0 & A \\ \bar{A} & 0 \end{pmatrix}
\bigg\}
+ O\bigg(\frac{1}{\tau} + \frac{\ln{t}}{t}\bigg) \bigg\}\delta(\zeta, \lambda_j)^{\sigma_3}
\end{align}
uniformly for $\zeta \in [0,1)$ as $t \to \infty$.
Equation (\ref{MlambdajIII}) will be important for the proof of Theorem \ref{solitonasymptoticsth}.
\end{remark}

\section{Proof of Theorem \ref{asymptoticsth}: Asymptotics in Sector IV}\label{sectorIVsec}
Let $\mathcal{I} = [0,1/2]$ and suppose $\zeta \in \mathcal{I}$.
Assume $r(k)$ vanishes to order $Z \geq 0$ at $k = 1$, i.e., $r^{(j)}(1) = 0$ for $j =0, \dots, Z$.

\subsection{Transformations of the RH problem}
We transform the RH problem in the same way as in Sector III, but this time there is no need for a local approximation near $\pm k_0$. Thus we set $m := M^{(3)}$, where $M^{(3)}$ is the function defined in (\ref{M3defIII}), and the analytic approximations of $h,r_1,r_2$ have been defined by applying Lemma \ref{hdecompositionlemmaIII} and Lemma \ref{decompositionlemmaIII} with some large integer $N \gg Z$. Then $m$ satisfies the RH problem (\ref{RHmSigma}) with $v := J^{(3)}$ and $\Sigma := \Gamma^{(3)}$, see Figure \ref{SigmaIV.pdf}.

\begin{lemma}\label{w2lemmaIV}
The function $w = v - I$ satisfies
\begin{align*}
& \|(1+|k|^{-1})w\|_{L^1(\Sigma)} \leq C(t^{-\frac{Z+2}{2}} + \zeta^{Z+2}),
	\\
& \|(1+|k|^{-1})w\|_{L^\infty(\Sigma)} \leq C(t^{-\frac{Z+1}{2}} + \zeta^{Z+1}),	
\end{align*}
uniformly for $\zeta \in \mathcal{I}$ and $t > 1$.
\end{lemma}
\begin{proof}
We have $\Sigma = \mathcal{X}^\epsilon \cup (-\mathcal{X}^\epsilon) \cup \Sigma' \cup \Sigma''$, where $\Sigma'$ and $\Sigma''$ are defined in (\ref{Sigmaprimedef}).
The $L^1$  and $L^\infty$ norms of $(1+|k|^{-1})w$ on $\Sigma'$ are $O(t^{-N-\frac{1}{2}})$ by Lemma \ref{decompositionlemmaIII} and are $O(e^{-ct})$ on $\Sigma''$ by (\ref{westimateIIId}).

\begin{figure}
\begin{center}
\begin{overpic}[width=.65\textwidth]{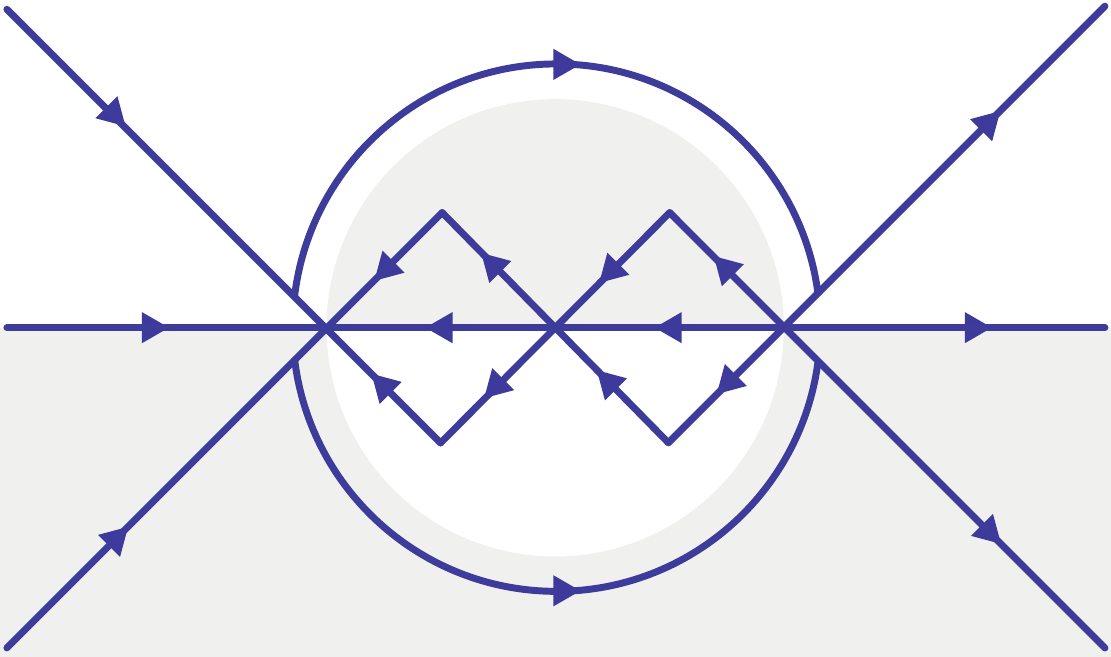}
      \put(102,28.5){\small $\Sigma$}
    \end{overpic}
     \begin{figuretext}\label{SigmaIV.pdf}
       The contour $\Sigma$ in the case of Sector IV.
     \end{figuretext}
     \end{center}
\end{figure}

\begin{figure}
\begin{center}
\begin{overpic}[width=.65\textwidth]{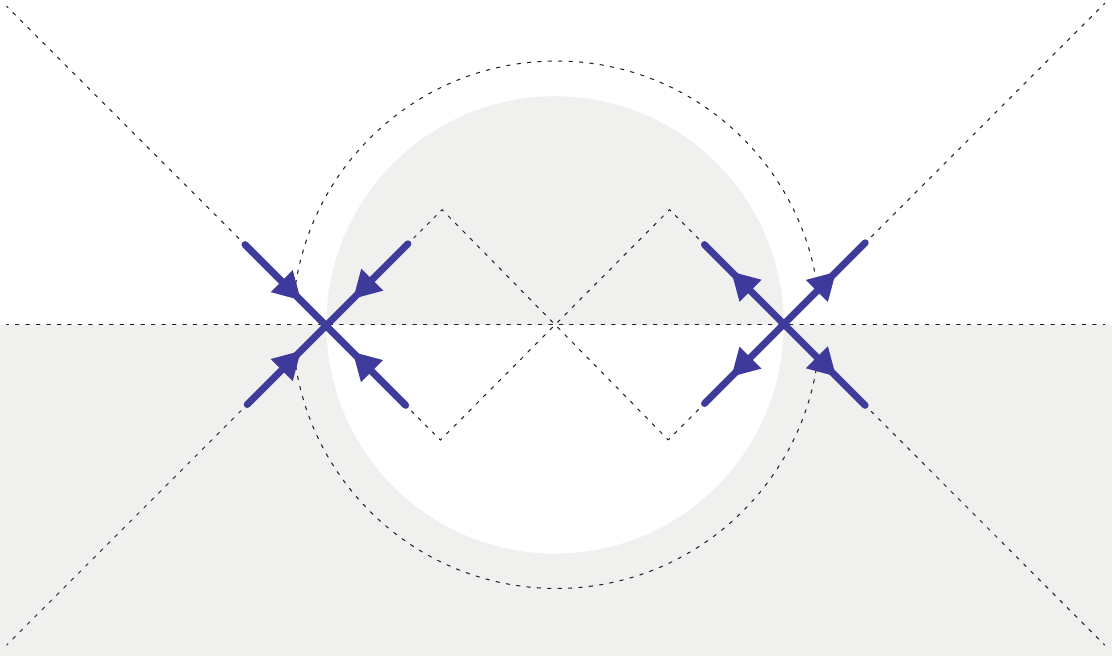}
      \put(75.5,32){\small $\mathcal{X}_1^\epsilon$}
      \put(61.2,32){\small $\mathcal{X}_2^\epsilon$}
     \put(61.2,25.7){\small $\mathcal{X}_3^\epsilon$}
      \put(75.5,25.7){\small $\mathcal{X}_4^\epsilon$}
          \end{overpic}
     \begin{figuretext}\label{SigmaprimeIV.pdf}
             The small crosses $\mathcal{X}^\epsilon = \cup_{j=1}^4 \mathcal{X}_j^{\epsilon}$ and $-\mathcal{X}^\epsilon$ centered at $k_0$ and $-k_0$, respectively.
     \end{figuretext}
     \end{center}
\end{figure}

It remains to consider the contributions from $\pm \mathcal{X}^\epsilon$, see Figure \ref{SigmaprimeIV.pdf}. By symmetry, we only have to consider $\mathcal{X}^\epsilon$. 
Since $r_1(k)$ is smooth on $\R$, we have, uniformly for $k \in \mathcal{X}^\epsilon$ and $\zeta \in \mathcal{I}$,
\begin{align}\nonumber
\sum_{j=0}^{Z} & \frac{r_1^{(j)}(1) (k-1)^j}{j!}
 =
\sum_{j=0}^{Z}\sum_{l=0}^{2Z}\bigg( \frac{r_1^{(j+l)}(k_0)(1-k_0)^l}{l!} + O\big((1-k_0)^{2Z+1}\big)\bigg)\frac{(k-1)^j}{j!}
	\\\nonumber
& = 
\sum_{j=0}^{Z}\sum_{s=j}^{2Z+j} \frac{r_1^{(s)}(k_0)(1-k_0)^{s-j}}{(s-j)!}\frac{(k-1)^j}{j!} 
+ O\big((1-k_0)^{2Z+1}\big)
	\\\nonumber
& = 
\bigg(\sum_{s=0}^{Z}\sum_{j=0}^{s} +\sum_{s=Z+1}^{2Z}\sum_{j=0}^{Z} \bigg) \begin{pmatrix} s \\ j \end{pmatrix}\frac{r_1^{(s)}(k_0)(1-k_0)^{s-j}(k-1)^j }{s!}  
+ O\big((1-k_0)^{Z+1}\big)
	\\\label{sumdifference}
& = \sum_{s=0}^{Z} \frac{r_1^{(s)}(k_0)(k-k_0)^{s}}{s!}  
+ \sum_{j=0}^Z O\big((1-k_0)^{Z+1-j}(k-1)^j\big).
\end{align}
For $k \in \mathcal{X}_1^\epsilon$ with $|k| \geq 1$, we have $|1-k_0| \leq C|k-1| \leq C|k-k_0|$.
Hence, recalling that $r = r_1 + h$ on $[-1,1]$ and $r(k)$ vanishes to order $Z$ at $k = 1$, the estimates (\ref{haestimateIII}) and (\ref{rjaestimatesIIIa}) together with (\ref{sumdifference}) imply
\begin{align} \nonumber
 |r_{1,a}(x,t,k) + h_a(k)| = &\; \bigg|r_{1,a}(x,t,k) - \sum_{j=0}^{Z} \frac{r_1^{(j)}(k_0)(k-k_0)^j}{j!} 
 + h_a(k) - \sum_{j=0}^{Z} \frac{h^{(j)}(1)(k-1)^j}{j!}
	\\\nonumber
 & + \sum_{j=0}^{Z} \frac{r_1^{(j)}(k_0) (k-k_0)^j}{j!} - \sum_{j=0}^{Z} \frac{r_1^{(j)}(1)(k-1)^j}{j!} \bigg|
	\\ \nonumber
\leq &\; C |k - k_0|^{Z+1} e^{\frac{t}{4}|\re \Phi(\zeta,k)|} 
+ C |k - 1|^{Z+1} e^{\frac{t}{4}|\re \Phi(\zeta,k)|} 
+ C|k-k_0|^{Z+1}
	\\ \label{X1farestimate}
\leq &\; C |k-k_0|^{Z+1} e^{\frac{t}{4}|\re \Phi(\zeta,k)|}, \qquad \zeta \in \mathcal{I}, \ k \in \mathcal{X}_1^\epsilon \cap \{|k| \geq 1\}.
\end{align}
For $k \in \mathcal{X}_1^\epsilon$ with $|k| \leq 1$, we have $|k-k_0|  \leq C|k-1| \leq C|1-k_0|$. Hence, using (\ref{sumdifference}) with $r_1$ replaced by $r$, we obtain in a similar way that
\begin{align} \nonumber
 |r_{1,a}(x,t,k) + h(k)| = &\; \bigg|r_{1,a}(x,t,k) - \sum_{j=0}^{Z} \frac{r_1^{(j)}(k_0)(k-k_0)^j}{j!} 
 + h(k) - \sum_{j=0}^{Z} \frac{h^{(j)}(k_0)(k-k_0)^j}{j!}
	\\\nonumber
 & + \sum_{j=0}^{Z} \frac{r^{(j)}(k_0) (k-k_0)^j}{j!} - \sum_{j=0}^{Z}  \frac{r^{(j)}(1) (k-1)^j}{j!}\bigg|
	\\ \nonumber
\leq &\; C |k - k_0|^{Z+1} e^{\frac{t}{4}|\re \Phi(\zeta,k)|} 
+ C|k-k_0|^{Z+1} + C |1-k_0|^{Z+1}
	\\ \label{X1nearestimate}
\leq &\; C |k-k_0|^{Z+1} e^{\frac{t}{4}|\re \Phi(\zeta,k)|}
+ C |1-k_0|^{Z+1}, \qquad \zeta \in \mathcal{I}, \ k \in \mathcal{X}_1^\epsilon \cap \{|k| \leq 1\}.
\end{align}
Similar estimates apply to $\mathcal{X}_2^\epsilon$. For $k \in \mathcal{X}_2^\epsilon$ with $|k-k_0| \geq |1-k_0|$, we have $|1-k_0|  \leq C|k-1| \leq C|k-k_0|$, whereas for 
$k \in \mathcal{X}_2^\epsilon$ with $ |k-k_0| \leq |1-k_0|$, we have $|k-k_0|  \leq C|k-1| \leq C|1 - k_0|$. 
Noting that $r_2 = r^*/(1+|r|^2)$  also vanishes to order $Z$ at $k = 1$, this yields
\begin{align} \nonumber
 |r_{2,a}(x,t,k)| = &\; \bigg|r_{2,a}(x,t,k) - \sum_{j=0}^{Z} \frac{r_2^{(j)}(k_0)(k-k_0)^j}{j!} 
+ \sum_{j=0}^{Z} \frac{r_2^{(j)}(k_0) (k-k_0)^j}{j!} - \sum_{j=0}^{Z}  \frac{r_2^{(j)}(1) (k-1)^j}{j!}\bigg|
	\\ \label{X2farestimate}
\leq &\; C |k-k_0|^{Z+1} e^{\frac{t}{4}|\re \Phi(\zeta,k)|}, \qquad \zeta \in \mathcal{I}, \ k \in \mathcal{X}_2^\epsilon \cap \{|k-k_0| \geq |1-k_0|\},
\end{align}
and, similarly,
\begin{align} \nonumber
& |r_{2,a}(x,t,k)| 
\leq C |k-k_0|^{Z+1} e^{\frac{t}{4}|\re \Phi(\zeta,k)|} + C |1-k_0|^{Z+1}, 
	\\\label{X2nearestimate}
&\hspace{5cm} \zeta \in \mathcal{I}, \ k \in \mathcal{X}_2^\epsilon \cap \{|k-k_0| \leq |1-k_0|\}.
\end{align}

Combining the inequalities (\ref{X1farestimate})-(\ref{X2nearestimate}) (and extending them to $\mathcal{X}_3^\epsilon \cup \mathcal{X}_4^\epsilon$ by symmetry), we arrive at the following estimate for the jump matrix for $\zeta \in \mathcal{I}$, $t > 1$, and $k \in \mathcal{X}^\epsilon$:
\begin{align*}
|w(x,t,k) | \leq  \begin{cases}
C |k-k_0|^{Z+1} e^{-\frac{3t}{4}|\re \Phi(\zeta,k)|}, & |k-k_0| \geq 2(1-k_0), \\
C(|k-k_0|^{Z+1} + |1-k_0|^{Z+1})e^{-\frac{3t}{4}|\re \Phi(\zeta,k)|}, & |k-k_0| \leq 2(1-k_0).
\end{cases} 
\end{align*}
Easy estimates show that (cf. (\ref{rePhionL1}))
$$|\re \Phi(\zeta, k)| \geq \frac{|k-k_0|^2}{8}, \qquad \zeta \in \mathcal{I}, \ k \in \mathcal{X}^\epsilon.$$
Using that $|k|^{-1}$ is uniformly bounded on  $\mathcal{X}^\epsilon$, we arrive at
\begin{align*}
\|(1+|k|^{-1})w\|_{L^1(\mathcal{X}^\epsilon)} 
& \leq \|w\|_{L^1(\mathcal{X}^\epsilon)} 
 \leq C\int_0^{\infty} u^{Z+1} e^{-\frac{t}{8}u^2} du
+ C(1-k_0) |1-k_0|^{Z+1} 
	\\
& \leq C(t^{-\frac{Z+2}{2}} + \zeta^{Z+2})
\end{align*}
and
\begin{align*}
\|(1+|k|^{-1})w\|_{L^\infty(\mathcal{X}^\epsilon)} 
\leq \|w\|_{L^\infty(\mathcal{X}^\epsilon)} 
 \leq C\sup_{0 \leq u \leq \infty} u^{Z+1} e^{-\frac{t}{8}u^2}
+ |1-k_0|^{Z+1} 
\leq C(t^{-\frac{Z+1}{2}} + \zeta^{Z+1})
\end{align*}
for $\zeta \in \mathcal{I}$ and $t > 1$. The lemma follows.
\end{proof}

The asymptotics of $u(x,t)$ in Sector IV follows from Lemma \ref{w2lemmaIV}. Indeed, by (\ref{muestimateI}), we have $\|\mu - I\|_{L^2(\Sigma)} \leq C\|w\|_{L^2(\Sigma)}$, and then Lemma \ref{w2lemmaIV} gives (cf. (\ref{hatMsectorI}))
\begin{align}\nonumber
\hat{M}(x,t) - I \leq C\|k^{-1}w\|_{L^1(\Sigma)})
+ C\|\mu - I\|_{L^2(\Sigma)}\|k^{-1}w\|_{L^2(\Sigma)}
\leq C(t^{-\frac{Z+2}{2}} + \zeta^{Z+2})
\end{align}
uniformly for $\zeta \in \mathcal{I}$ as $t \to \infty$. Recalling (\ref{ulim}), this implies
$u(x,t) = O(t^{-\frac{Z+2}{2}} + \zeta^{Z+2})$  
uniformly in Sector IV. This completes the proof of the asymptotic formulas (\ref{uasymptoticsIV}) and (\ref{uasymptoticsIV2}), and hence also of Theorem \ref{asymptoticsth}.

\begin{remark}\label{MlambdajIVremark}\upshape
If $\lambda_j$ is a point in $\C \setminus \Gamma$, then a slight modification of the above argument shows that (cf. Remark \ref{MlambdajIremark})
\begin{align}\label{MlambdajIV}
  M(x,t,\lambda_j) = I + O(t^{-\frac{Z+2}{2}} + \zeta^{Z+2})
\end{align}
uniformly for $\zeta \in [0,1/2]$ as $t \to \infty$. 
\end{remark}

\section{Solitons}\label{solitonsec}
Until now we have made the assumption that the matrix-valued function $M(x,t,k)$ entering the RH problem (\ref{RHM}) is pole-free, corresponding to the fact that the spectral functions $a(k)$ and $d(k)$ have no zeros (see Assumption \ref{assumption1}). In this section, we extend all the theorems of Section \ref{mainsec} to allow for a finite number of simple poles. This will broaden the scope of these theorems to include solitons and solutions with nonzero topological charge. 
We calculate in detail the asymptotics of the solution and its winding number in the presence of both radiation and solitons. As in the case of the nonlinear Schr\"odinger equation \cite{FIS2005}, it turns out that the zeros of $d(k)$ generate solitons, whereas the zeros of $a(k)$ affect the solution in a more subtle way; see \cite{FI1992} for some early results in this direction. By consistently keeping track of lower-order terms, we are able to also compute the subleading term in the asymptotic expansion. Results on asymptotic solitons for the sine-Gordon equation on the line can be found in \cite{K1989, KK1999}.

\subsection{RH problems with poles}
We begin by introducing a notion of a RH problem with poles. 
As before, let $\Gamma = \R \cup \{|k|= 1\}$. Let $\mathcal{N} \geq 0$ be an integer and let $\{C_j\}_1^\mathcal{N} \subset \C$ be complex constants. 
\begin{assumption}\label{poleassumption}
Suppose $\{\lambda_j\}_1^\mathcal{N} \subset \C_+ \setminus \Gamma$ is a set of distinct points which is invariant under the map $\lambda \mapsto -\bar{\lambda}$. 
\end{assumption}
Assumption \ref{poleassumption} implies that the $\lambda_j$ are either pure imaginary or come in pairs $(\lambda_j, -\bar{\lambda}_j)$ with $\re \lambda_j \neq 0$. 
Let $J:\Gamma \to GL(2, \C)$ be a (to begin with arbitrary) jump matrix. For our present purposes, the following definition will prove convenient.

\begin{definition}\upshape
We say that a sectionally meromorphic $2 \times 2$-matrix valued function $M(k)$ {\it satisfies the RH problem determined by} $\{\Gamma, J(k), \{\lambda_j\}_1^\mathcal{N}, \{C_j\}_1^\mathcal{N}\}$ if 
\begin{enumerate}[$(i)$]
\item $M(k)$ can be written as
\begin{align}\nonumber
M(k) =&\; (kI+B_\mathcal{N})(kI+B_{\mathcal{N}-1})\cdots (kI+B_1)\tilde{M}(k)
	\\ \label{eq-mm}
& \times \begin{pmatrix}
\frac{1}{\prod_{j=1}^{\mathcal{N}}(k-\lambda_j)} & 0 \\
0 & \frac{1}{\prod_{j=1}^{\mathcal{N}}(k-\bar{\lambda}_j)} \end{pmatrix}
\end{align}
for some constant complex $2 \times 2$ matrices $\{B_j\}_1^\mathcal{N}$ and some $2 \times 2$-matrix valued function $\tilde{M} \in I + \dot{E}^2(\C \setminus \Gamma)$.

\item $M_+(k) = M_-(k) J(k)$ for a.e. $k \in \Gamma$.

\item $\det M = 1$.

\item $M(k)$ satisfies the residue conditions
\begin{align}\label{RHmres}
\begin{cases}
  \underset{\lambda_j}{\res} [M(k)]_1= C_j [M(\lambda_j)]_2,
  	\\
 \underset{\bar{\lambda}_j}{\res} [M(k)]_2= -\bar{C}_j [M(\bar{\lambda}_j)]_1,  
 \end{cases} \quad j=1,2,\dots,\mathcal{N}.
 \end{align}
\end{enumerate}
\end{definition}
Condition $(i)$ ensures that $[M]_1$ and $[M]_2$ are analytic functions of $k \in \C \setminus \Gamma$ except for (at most) a finite number of simple poles located at the points $\lambda_j$ and $\bar{\lambda}_j$, respectively. The associated residues must satisfy the relations in \eqref{RHmres}.

\begin{lemma}\label{residueuniquelemma}
The solution  of the RH problem determined by $\{\Gamma, J(k), \{\lambda_j\}_1^\mathcal{N}, \{C_j\}_1^\mathcal{N}\}$ is unique if it exists.
\end{lemma}
\begin{proof}
Let $M$ and $N$ be two solutions. Then the first columns $[M]_1$ and $[N]_1$ have possible simple poles at the $\lambda_j$, while the second columns $[M]_2$ and $[N]_2$ have possible simple poles at the $\bar{\lambda}_j$. Using that $\det N = 1$, we compute
\begin{align*}
 \underset{\lambda_j}{\res}[M N^{-1}]_1 & =  \underset{\lambda_j}{\res}\big( [M]_1 N_{22} -
  [M]_2  N_{21}  \big)
  	\\
&  = C_j [M(\lambda_j)]_2 N_{22}(\lambda_j)
  - [M(\lambda_j)]_2 C_j N_{22}(\lambda_j) = 0,
\end{align*}
and similar computations apply to the second column of $M N^{-1}$ and at $\bar{\lambda}_j$.
It follows that the function $m := M N^{-1}$ is analytic at the $\lambda_j$ and the $\bar{\lambda}_j$.
In particular, $m \in I + \dot{E}^1(\C \setminus \Gamma)$. Since $m$ has no jump across $\Gamma$, we conclude that $m$ vanishes identically.  
\end{proof}

\subsection{Dressing}
The next lemma shows how the solution $M$ of the nonregular RH problem with poles can be obtained from the solution $\tilde{M}$ of a regular RH problem without poles. See \cite{FI1996} for an analogous lemma in the case of the nonlinear Schr\"odinger equation.

Let $r_1:\R \to \C$ and $h:\partial D_2 \to \C$ be two functions satisfying Assumption \ref{r1hassumption}. 
Then the functions $\tilde{r}_1:\R \to \C$ and $\tilde{h}:\bar{D}_2 \to \C$ defined by
\begin{align}\label{tilder1hdef}
\tilde{r}_1(k)= r_1(k) \prod_{j=1}^{\mathcal{N}}\frac{k-\lambda_j}{k-\bar{\lambda}_j},\qquad 
\tilde{h}(k) = h(k) \prod_{j=1}^{\mathcal{N}}\frac{k-\lambda_j}{k-\bar{\lambda}_j}
\end{align}
also fulfill Assumption \ref{r1hassumption}.
From now on $J(x,t,k)$ will denote the jump matrix defined in terms of $r_1$ and $h$ via (\ref{Jdef}). Let $\tilde{J}(x,t,k)$ be the matrix defined by replacing $r_1(k)$ and $h(k)$ with $\tilde{r}_1(k)$ and $\tilde{h}(k)$ in (\ref{Jdef}).
By Theorem \ref{existenceth}, there exists a unique solution $\tilde{M}(x,t,k)$ of the RH problem 
\begin{align}\label{RHMtilde}
\begin{cases}
\tilde{M}(x, t, \cdot) \in I + \dot{E}^2(\C \setminus \Gamma),\\
\tilde{M}_+(x,t,k) = \tilde{M}_-(x, t, k) \tilde{J}(x, t, k) \quad \text{for a.e.} \ k \in \Gamma.
\end{cases}
\end{align}

\begin{lemma}\label{dressinglemma}
Let $\{\lambda_j\}_1^\mathcal{N}$ satisfy Assumption \ref{poleassumption} and suppose $\{C_j(x,t)\}_1^\mathcal{N}$ are complex-valued functions such that
\begin{align}\label{Cjassump}
\begin{cases}
C_j(x,t) \in i\R & \text{if $\lambda_j$ is pure imaginary,}
	\\
C_{j+1}(x,t) = -\overline{C_j(x,t)} & \text{if $(\lambda_j, \lambda_{j+1} = -\bar{\lambda}_j)$ is a pair with $\re \lambda_j \neq 0$}.
\end{cases}
\end{align}

For each $(x,t) \in [0,\infty) \times [0,\infty)$, the RH problem determined by $\{\Gamma$, $J(x,t,\cdot)$, $\{\lambda_j\}_1^\mathcal{N}$, $\{C_j(x,t)\}_1^\mathcal{N}\}$ has a unique solution $M(x,t,k)$. This solution is given by (\ref{eq-mm}), where the $2 \times 2$ matrices $B_j := B_j(x,t)$, $j = 1, \dots, \mathcal{N}$, are determined recursively by the equations
\begin{align}\label{Bjrecursive}
  \begin{cases}
  (\lambda_jI+B_j(x,t))\tilde{M}_{j-1}(x,t,\lambda_j)
  \begin{psmallmatrix} 1 \\ -d_j(x,t) \end{psmallmatrix} = 0,
  	\vspace{.1cm}\\ 
  (\bar{\lambda}_jI +B_j(x,t))\tilde{M}_{j-1}(x,t,\bar{\lambda}_j)
  \begin{psmallmatrix} \overline{d_j(x,t)} \\ 1 \end{psmallmatrix} = 0,
  \end{cases} \quad j = 1, 2, \dots, \mathcal{N},
\end{align}
with the assignments
\begin{align}\nonumber
& \tilde{M}_0(x,t,k) := \tilde{M}(x,t,k),
	\\\nonumber
  &\tilde{M}_{j}(x,t,k) := (kI+B_j)\tilde{M}_{j-1}(x,t,k), \quad j=1,2,\dots, \mathcal{N}, 
  	\\\label{djdef}
  &d_j(x,t) := C_j(x,t) \frac{\prod_{l=1,l\neq j}^{\mathcal{N}}(\lambda_j-\lambda_l)}{\prod_{l=1}^{\mathcal{N}}(\lambda_j-\bar{\lambda}_l)}, \qquad j = 1,2, \dots, \mathcal{N}.
\end{align}
Furthermore, $M$ obeys the symmetries
\begin{align}\label{Msymm1}
& M(x,t,k) = \sigma_2 \overline{M(x,t,\bar{k})} \sigma_2, \qquad k \in \C \setminus \Gamma,
	\\ \label{Msymm2}
& M(x,t,k) = \sigma_2 M(x,t,-k) \sigma_2, \qquad k \in \C \setminus \Gamma.
\end{align}
\end{lemma}
\begin{proof}
We first use induction to show that the $B_j$ are well-defined and that the following hold:
\begin{align}\label{tildeMjsymm}
& \begin{cases} \det\tilde{M}_{i-1}(x,t,k) = \prod_{l = 1}^{i-1} (k - \lambda_l)(k - \bar{\lambda}_l), \vspace{.1cm} \\ 
\tilde{M}_{i-1}(x,t,k) = \sigma_2\overline{\tilde{M}_{i-1}(x,t,\bar{k})}\sigma_2,
\end{cases} \quad i = 1, \dots, \mathcal{N}+1,
	\\ \label{Bjsymm}
& B_i(x,t) = \sigma_2\overline{B_i(x,t)}\sigma_2, \qquad i = 1, \dots, \mathcal{N}.
\end{align}

We know from Section \ref{existenceproofsec} (see Lemma \ref{existencemlemma} and (\ref{msymm})) that (\ref{tildeMjsymm}) holds for $i = 1$.
Fix $1 \leq j \leq \mathcal{N}$ and suppose (\ref{tildeMjsymm}) holds for $i = j$. The linear algebraic system (\ref{Bjrecursive}) can be solved uniquely for $B_j$ iff the matrix 
$$Q_j(x,t) := \bigg(\tilde{M}_{j-1}(x,t,\lambda_j)
  \begin{psmallmatrix} 1 \\ -d_j(x,t) \end{psmallmatrix}, \tilde{M}_{j-1}(x,t,\bar{\lambda}_j)
  \begin{psmallmatrix} \overline{d_j(x,t)} \\ 1 \end{psmallmatrix} \bigg)$$
has nonvanishing determinant. 
The relation $\tilde{M}_{j-1}(x,t,\lambda_j) = \sigma_2 \overline{\tilde{M}_{j-1}(x,t,\bar{\lambda}_j)} \sigma_2$ implies $Q_j(x,t) = \sigma_2\overline{Q_j(x,t)}\sigma_2$, and hence $\det Q_j = |(Q_j)_{11}|^2 + |(Q_j)_{12}|^2$ vanishes only if $Q_j = 0$. But in fact $Q_j \neq 0$, because (\ref{tildeMjsymm}) with $i = j$ shows that $\det \tilde{M}_{j-1}(x,t,\lambda_j)$ and $\det \tilde{M}_{j-1}(x,t,\bar{\lambda}_j)$ are nonzero.
This shows that $B_j$ is well-defined. Taking the complex conjugate of (\ref{Bjrecursive}) and multiplying by $\sigma_2$ from the left, we see that $\sigma_2\overline{B_j(x,t)}\sigma_2$ satisfies the same equations as $B_j(x,t)$. By uniqueness, we conclude that (\ref{Bjsymm}) holds for $i = j$. 
On the other hand, by (\ref{Bjrecursive}), the second order polynomial $\det(kI + B_j)$ has zeros at $\lambda_j$ and $\bar{\lambda}_j$, which gives $\det(kI + B_j) = (k - \lambda_j)(k - \bar{\lambda}_j)$. It follows that (\ref{tildeMjsymm}) holds for $i = j +1$. By induction, this proves (\ref{tildeMjsymm}) and (\ref{Bjsymm}) for all values of $i$.

In view of (\ref{Jdef}) and (\ref{tilder1hdef}), the computation
$$M_-^{-1}M_+ = \begin{pmatrix}
\prod_{j=1}^{\mathcal{N}}(k-\lambda_j) & 0 \\
0 & \prod_{j=1}^{\mathcal{N}}(k-\bar{\lambda}_j) \end{pmatrix} \tilde{J} \begin{pmatrix}
\frac{1}{\prod_{j=1}^{\mathcal{N}}(k-\lambda_j)} & 0 \\
0 & \frac{1}{\prod_{j=1}^{\mathcal{N}}(k-\bar{\lambda}_j)} \end{pmatrix} = J
$$
shows that $M$ satisfies the correct jump condition across $\Gamma$.  The condition $\det M = 1$ follows from (\ref{tildeMjsymm}) with $i = \mathcal{N}+1$. 
Finally, for each $j = 1, \dots, \mathcal{N}$, we have
\begin{align*}
\underset{\lambda_j}{\res} [M(x,t,k)]_1
& = \frac{(\lambda_jI + B_\mathcal{N}) \cdots (\lambda_jI + B_1)}{\prod_{l=1,l\neq j}^{\mathcal{N}}(\lambda_j-\lambda_l)}[\tilde{M}(x,t,\lambda_j)]_1
	\\
& = \frac{(\lambda_jI + B_\mathcal{N}) \cdots (\lambda_jI + B_{j})}{\prod_{l=1,l\neq j}^{\mathcal{N}}(\lambda_j-\lambda_l)}[\tilde{M}_{j-1}(x,t,\lambda_j)]_1.
\end{align*}
By (\ref{Bjrecursive}) and the definition (\ref{djdef}) of $d_j$, this yields
\begin{align*}
\underset{\lambda_j}{\res} [M(x,t,k)]_1
& = \frac{d_j (\lambda_jI + B_\mathcal{N}) \cdots (\lambda_jI + B_{j})}{\prod_{l=1,l\neq j}^{\mathcal{N}}(\lambda_j-\lambda_l)}[\tilde{M}_{j-1}(x,t,\lambda_j)]_2
	\\
&= C_j  \frac{(\lambda_jI + B_\mathcal{N}) \cdots (\lambda_jI + B_j)}{\prod_{l=1}^{\mathcal{N}}(\lambda_j-\bar{\lambda}_l)} [\tilde{M}_{j-1}(x,t,\lambda_j)]_2
= C_j  [M(x,t,\lambda_j)]_2,
\end{align*}
showing that $M$ satisfies the first of the residue conditions in (\ref{RHmres}); the proof of the second condition is analogous. 
The symmetry in (\ref{Msymm1}) follows from (\ref{eq-mm}) and (\ref{Bjsymm}). 

It remains to prove (\ref{Msymm2}). 
Note that (\ref{Cjassump}) implies
\begin{enumerate}[$(i)$]
\item $d_j(x,t) \in \R$ if $\lambda_j$ is pure imaginary. 
\item $d_{j+1}(x,t) = \overline{d_j(x,t)}$ if $(\lambda_j, \lambda_{j+1} = -\bar{\lambda}_j)$ is a pair with $\re \lambda_j \neq 0$.
\end{enumerate}
We know from Section \ref{existenceproofsec} (see (\ref{msymm})) that $\tilde{M} = \tilde{M}_0$ obeys (\ref{Msymm2}). Seeking a proof by induction, we suppose $\tilde{M}_{j-1}(x,t,k) = \pm \sigma_2 \tilde{M}_{j-1}(x,t,-k) \sigma_2$ for some $1 \leq j \leq \mathcal{N}$.

If $\lambda_j \in i\R$, then multiplying the equations in (\ref{Bjrecursive}) from the left by $\sigma_2$ and using $(i)$, we infer that $-\sigma_2B_j\sigma_2$ satisfies the same equations as $B_j$. Thus $B_j = -\sigma_2B_j\sigma_2$, which gives $\tilde{M}_{j}(x,t,k) = \mp \sigma_2 \tilde{M}_{j}(x,t,-k) \sigma_2$.

If $(\lambda_j, \lambda_{j+1} = -\bar{\lambda}_j)$ is a pair with $\re \lambda_j \neq 0$, we instead argue as follows. The uniqueness of 
Lemma \ref{residueuniquelemma} implies that the recursive construction of $M$ via (\ref{Bjrecursive})-(\ref{djdef}) leads to the same result regardless of the ordering of the $\lambda_j$. Let $\{B_j'\}_1^\mathcal{N}$ be the matrices defined by the recursive formula (\ref{Bjrecursive}) when the points $\lambda_j$ and $\lambda_{j+1}$ are  interchanged. Multiplying (\ref{Bjrecursive}) from the left by $\sigma_2$ and using $(ii)$, we deduce that $B_j' = -\sigma_2 B_j \sigma_2$ and $B_{j+1}' = -\sigma_2 B_{j+1} \sigma_2$. Since the result is independent of the ordering of the $\lambda_j$, we obtain
\begin{align}\label{BBsymm}
(kI + B_{j+1})(kI + B_j) = (kI + B_{j+1}')(kI + B_j') = (kI - \sigma_2B_{j+1}\sigma_2)(kI - \sigma_2B_j\sigma_2),
\end{align}
which shows that $\tilde{M}_{j+1}(x,t,k) = \pm \sigma_2 \tilde{M}_{j+1}(x,t,-k) \sigma_2$.

Combining the above two cases, induction yields $\tilde{M}_\mathcal{N}(x,t,k) = (-1)^\mathcal{N} \sigma_2 \tilde{M}_\mathcal{N}(x,t,-k) \sigma_2$. Taking the rightmost factor in (\ref{eq-mm}) into account, the symmetry (\ref{Msymm2}) follows.
\end{proof}

\subsection{Construction of solutions}
We can now prove generalizations of Theorems \ref{existenceth} and \ref{existenceth2} which allow for a finite number of simple poles.

\begin{theorem}[Construction of quarter-plane solutions]\label{solitonexistenceth}
Let $r_1:\R \to \C$ and $h:\partial D_2 \to \C$ be functions satisfying Assumption \ref{r1hassumption} and define $J(x,t,k)$ by (\ref{Jdef}). Let $\{\lambda_j\}_1^\mathcal{N}$ satisfy Assumption \ref{poleassumption} and let $\{c_j\}_1^\mathcal{N}$ be a set of complex numbers such that
\begin{align}\label{cjassump}
\begin{cases} 
c_j \in i\R & \text{if $\lambda_j$ is pure imaginary,}
	\\
c_{j+1} = -\overline{c_j} & \text{if $(\lambda_j, \lambda_{j+1} = -\bar{\lambda}_j)$ is a pair with $\re \lambda_j \neq 0$.}
\end{cases}
\end{align}

Then the RH problem determined by $\{\Gamma, J(x,t,\cdot), \{\lambda_j\}_1^\mathcal{N}, \{c_j e^{2i\theta(x,t,\lambda_j)}\}_1^\mathcal{N}\}$ has a unique solution $M(x,t,k)$ for each $(x,t) \in [0,\infty) \times [0, \infty)$. Moreover, the nontangential limit (\ref{m0def}) exists for each $(x,t) \in [0,\infty) \times [0, \infty)$ and the function $u(x,t)$ defined by (\ref{ulim}) belongs to $C^2([0,\infty) \times [0, \infty), \R)$ and satisfies the sine-Gordon equation \eqref{sg} for $x \geq 0$ and $t \geq 0$. 
\end{theorem}
\begin{proof}
By Lemma \ref{dressinglemma}, the solution $M(x,t,k)$ exists and is uniquely defined by (\ref{eq-mm}) where $\tilde{M}(x,t,k)$ is the solution of (\ref{RHMtilde}) and the $B_j$ are determined by (\ref{Bjrecursive}). 
Defining $m(x,t,k)$  by (\ref{mxtkdef}), we see that there are functions $\{F_j(x,t), G_j(x,t)\}_{-1}^1$ such that equations (\ref{mxmtFFF}) hold, because the residue conditions (\ref{RHmres}) imply that $(m_x - i\theta_1m\sigma_3)m^{-1}$ and $(m_t - i\theta_2m\sigma_3)m^{-1}$ are regular at the points $\{\lambda_j, \bar{\lambda}_j\}_1^\mathcal{N}$.
By (\ref{Msymm1}) and (\ref{Msymm2}), $m$ obeys the symmetries in (\ref{msymm}). 
Thus we can proceed as in the proof of Lemma \ref{laxlemma} to deduce that $m$ satisfies the Lax pair equations (\ref{mlax}) with $u(x,t)$  defined in terms of $M$ via (\ref{m0def})-(\ref{ulim}). The rest of the proof follows as in Section \ref{existenceproofsec}.
\end{proof}

\begin{assumption}\label{solitonassumption1}
Assume the following conditions are satisfied:
\begin{itemize}
\item The global relation (\ref{GR}) holds.

\item $d(k)$ is nonzero in $\bar{D}_2$ except for $\Lambda$ simple zeros $\{\lambda_j \}_{1}^{\Lambda}$ in $D_2$.
\item $a(k)$ is nonzero in $\bar{D}_1 \cup \bar{D}_2$ except for $n_1+n_2$ simple zeros $\{k_j \}_{j=1}^{n_1+n_2}$, where $\{k_j\}_{1}^{n_1} \subset D_1$ and $\{k_j\}_{n_1+1}^{n_2} \subset D_2$.
\item None of the zeros of $a(k)$ in $D_2$ coincides with a zero of $d(k)$.

\end{itemize}
\end{assumption}

For convenience of notation, we write $\mathcal{N} := \Lambda + n_1$ and set $\lambda_{\Lambda + j} := k_j$ for $j = 1, \dots, n_1$.

\begin{theorem}[Construction of quarter-plane solutions with given initial and boundary values]\label{solitonexistenceth2}

Let $u_0, u_1, g_0, g_1$ be functions satisfying (\ref{ujgjassump}) for $n = 1$, $m = 2$, and some integers $N_x, N_t \in \Z$. Suppose $u_0, u_1, g_0, g_1$ are compatible with equation (\ref{sg}) to third order at $x=t=0$. Define the spectral functions $a(k)$, $b(k)$, $A(k)$, $B(k)$ by (\ref{abABdef}) and suppose Assumption \ref{solitonassumption1} holds. 

Then the spectral functions $r_1(k)$  and $h(k)$ defined by (\ref{r1hdef}) satisfy Assumption \ref{r1hassumption}. 
Moreover, if $M(x,t,k)$ is the unique solution of the RH problem determined by $\{\Gamma$, $J(x,t,\cdot)$, $\{\lambda_j\}_1^\mathcal{N}$, $\{c_j e^{2i\theta(x,t,\lambda_j)}\}_1^\mathcal{N}\}$, where
\begin{align}\label{cjdef}
c_j=\begin{cases}
  -\frac{\overline{B(\bar{\lambda}_{j})}}{a(\lambda_{j})\dot d(\lambda_{j})}, \quad & j=1,\dots, \Lambda, \\
 \frac{1}{\dot a(\lambda_{j})b(\lambda_{j})}, \quad & j=\Lambda+1,\dots,\mathcal{N},
\end{cases}
\end{align}
then the associated solution $u(x,t)$ of sine-Gordon defined by (\ref{ulim2}) has, for an appropriate choice of the integer $j \in \Z$, initial and boundary values given by $u_0, u_1, g_0, g_1$.
\end{theorem}
\begin{proof}
The proof proceeds along the same lines as in Section \ref{existence2proofsec}. However, because $a(k)$ and $d(k)$ may now have zeros, the RH problem (\ref{mxRH}) for $m^{(x)}$ must be supplemented with the residue conditions
\begin{align}\label{RHmxres}
\begin{cases}
 \underset{k_j}{\res} [m^{(x)}(x,k)]_1= \frac{e^{2i\theta_1(k_j)x}}{\dot{a}(k_j)b(k_j)} [m^{(x)}(x,k_j)]_2,  & j=1,2,\dots,n_1 + n_2,
  	\\
 \underset{\bar{k}_j}{\res} [m^{(x)}(x,k)]_2= -\frac{e^{-2i\theta_1(\bar{k}_j)x}}{\overline{\dot{a}(k_j)b(k_j)}} [m^{(x)}(x,\bar{k}_j)]_1,  \quad & j=1,2,\dots,n_1 + n_2.
 \end{cases}
\end{align}
Now the definition (\ref{originalMdef}) of $M$ shows that $M$ obeys the residue conditions 
\begin{align}\label{RHmres2}
\begin{cases}
  \underset{\lambda_j}{\res} [M(x,t,k)]_1= c_j e^{2i\theta(x,t,\lambda_j)} [M(x,t,\lambda_j)]_2,
  	\\
 \underset{\bar{\lambda}_j}{\res} [M(x,t,k)]_2= -\bar{c}_j e^{-2i\theta(x,t,\bar{\lambda}_j)} [M(x,t,\bar{\lambda}_j)]_1,  
 \end{cases} \quad j=1,2,\dots,\mathcal{N}.
 \end{align}
Thus, noting that the zeros of $a$ and $d$ give rise to simple poles for $h$ (see (\ref{r1hdef})), algebraic manipulations show that the function $M^{(x)}$ defined in (\ref{Mxdef}) also obeys the residue conditions (\ref{RHmxres}). By uniqueness (see Lemma \ref{residueuniquelemma}), we conclude that $M^{(x)} = m^{(x)}$.

In a similar way, the  RH problem (\ref{mtRH}) for $m^{(t)}$ must be supplemented with the residue conditions
\begin{align}\label{RHmtres}
 \begin{cases}
  \underset{k_j}{\res} [m^{(t)}(t,k)]_2 = \frac{b(k_j)}{\dot{a}(k_j)} e^{-2i\theta_2(k_j)t} [m^{(t)}(t,k_j)]_1,  & k_j \in \mathcal{D}_1,
  	\\
 \underset{\bar{k}_j}{\res} [m^{(t)}(t,k)]_1 = - \frac{\overline{b(k_j)}}{\overline{\dot{a}(k_j)}} e^{2i\theta_2(\bar{k}_j)t} [m^{(t)}(t,\bar{k}_j)]_2,  \quad& \bar{k}_j \in \mathcal{D}_4.
 \end{cases}
\end{align}
By (\ref{cdexpansionsd}), we have $d(0) = (-1)^{N_x - N_t}$. Thus, shrinking $\mathcal{D}_2$ if necessary, we may assume that $d(k)$  is nonzero in $\mathcal{D}_2$. Tedious but straightforward algebraic manipulations now show that the function $M^{(t)}$ defined in (\ref{Mtdef}) obeys the same residue conditions as $m^{(t)}$. Thus, by uniqueness, $M^{(t)} = m^{(t)}$.
Given that $M^{(x)} = m^{(x)}$ and $M^{(t)} = m^{(t)}$, the rest of the proof follows as in Section \ref{existence2proofsec}.
\end{proof}

\subsection{Example: Kink/antikink}
The sine-Gordon one-soliton is obtained by applying Theorem \ref{solitonexistenceth} with $r_1 = h = 0$ and a single pole $\lambda_1$. For $\lambda_1 \in i\R_+$ and $c_1 \in i\R$, the function $d_1(x,t) =\frac{c_1}{2\lambda_1}e^{2i\theta(x,t,\lambda_1)}$ is real-valued and a computation gives
\begin{align}\label{onesolitonB1mhat}
B_1 = \frac{\lambda_1}{1+ d_1^2}\begin{pmatrix}
  d_1^2-1 & 2d_1 \\ 2d_1 & 1-d_1^2 \end{pmatrix},
 \qquad
 \hat{M}(x,t) = \frac{1}{1+ d_1^2}\begin{pmatrix}
  1-d_1^2 & 2d_1 \\ -2d_1 & 1-d_1^2 \end{pmatrix}.
\end{align}
This yields the one-soliton
\begin{align}\label{onesoliton}
u_{1\text{-sol}}(x,t) = 2 \arg(1-d_1^2 - 2id_1)
= 4\arg(1 - id_1) 
= 4 \arctan\bigg(-\frac{\im c_1}{2|\lambda_1|} e^{-\gamma(x - vt)}\bigg),
\end{align}
where the real parameters $\gamma$ and $v$ are given by
\begin{align}\label{gammavdef}
\gamma = \frac{\im \lambda_1}{2}(1 + |\lambda_1|^{-2}) > 0, \qquad v = \frac{1-|\lambda_1|^2}{1+|\lambda_1|^2},
\end{align}
and we have chosen the branch of arg so that $u_{1\text{-sol}}(x,t) \to 0$ as $x\to \infty$.
For $\im c_1 > 0$, $u_{1\text{-sol}}$ is a kink changing value from  $-2\pi$ to $0$ as $x$ goes from $-\infty$ to $\infty$; for $\im c_1 < 0$ it is an antikink changing value from $2\pi$ to $0$.

\subsection{Example: Breather}
For $r_1 = h = 0$ and a pair of poles $(\lambda_1, \lambda_2 = -\bar{\lambda}_1)$ with $\re \lambda_1 > 0$, the two-soliton solution generated by Theorem \ref{solitonexistenceth} is referred to as a breather.
In this case, using that $\tilde{M} = I$, $\lambda_2 = -\bar{\lambda}_1$, and $c_2 = -\bar{c}_1$, we find 
$$d_1(x,t) = \overline{d_2(x,t)} = \frac{c_1\re\lambda_1}{2i\lambda_1\im\lambda_1}e^{2i\theta(x,t,\lambda_1)}$$
and a computation yields
\begin{align*}
u_{\text{breather}}(x,t) = 2\arg\bigg(\frac{\Theta}{\bar{\Theta}}\bigg) = 4 \arg \Theta,
\end{align*}
where $\Theta := \Theta(x,t)$ is given by
\begin{align*}
\Theta(x,t) & = (1-id_1)(1 + i\bar{d}_1)\lambda_1 + (1 + id_1)(1 - i\bar{d}_1)\bar{\lambda}_1
	\\
& = 2\big(|d_1|^2 + 1\big) \re\lambda_1  + 4 i \im d_1 \im \lambda_1.
\end{align*}
That is, choosing the branch of arg so that $u_{\text{breather}}(x,t) \to 0$ as $x\to \pm \infty$,
\begin{align}\label{solitonbreather}
u_{\text{breather}}(x,t) = 4\arctan\bigg\{\frac{2 \im d_1 \im \lambda_1}{(1 + |d_1|^2) \re\lambda_1}\bigg\}.
\end{align}
The breather travels at the speed $v$ given in (\ref{gammavdef}) because
\begin{align}\label{thetajexp}
e^{2i\theta(x,t,\lambda_1)} = e^{-\gamma(x -vt)} e^{i\tilde{\gamma}(x - t/v)},
\end{align}
and so
\begin{align*}
& |d_1| =\frac{|c_1| \re\lambda_1}{2|\lambda_1| \im\lambda_1} e^{-\gamma(x -vt)},
	\\
& \im d_1 = \frac{\big[\im(c_1\bar{\lambda}_1)\sin(\tilde{\gamma}(x - \frac{t}{v}))
- \re(c_1\bar{\lambda}_1) \cos(\tilde{\gamma}(x - \frac{t}{v})) \big]\re\lambda_1}{2|\lambda_1|^2 \im \lambda_1}  e^{-\gamma(x -vt)},
\end{align*}	
where $\gamma, v \in \R$ are given by (\ref{gammavdef}) and $\tilde{\gamma} := \frac{\re \lambda_1}{2}(1 - |\lambda_1|^{-2}) \in \R$.

\subsection{Asymptotics}
In this subsection, we prove Theorems \ref{solitonasymptoticsth} and \ref{solitonasymptoticsth2} which generalize Theorems \ref{asymptoticsth} and \ref{asymptoticsth2} of Section \ref{mainsec} to the case when solitons are present.

\begin{assumption}\label{poleassumption2}
Suppose that $\{\lambda_j\}_1^\mathcal{N} \subset \C_+ \setminus \Gamma$ is a set of distinct points such that:
\begin{itemize}
\item $\{\lambda_j\}_1^\Lambda \subset D_2$ and $\{\lambda_j\}_{\Lambda+1}^\mathcal{N} \subset D_1$. 

\item $\{\lambda_j\}_1^\mathcal{N}$ is invariant under the map $\lambda \mapsto -\bar{\lambda}$. 

\item The ordering is such that if $(\lambda, -\bar{\lambda})$ is a pair in $\{\lambda_j\}_1^\mathcal{N}$ with $\re \lambda > 0$, then $(\lambda, -\bar{\lambda}) = (\lambda_j, \lambda_{j+1})$ for some $j$.

\item $|\lambda_1| \leq |\lambda_2| \leq \cdots \leq |\lambda_\mathcal{N}|$ where strict inequality $|\lambda_j| < |\lambda_{j+1}|$ holds except when $\lambda_{j+1} = -\bar{\lambda}_j$. 

\end{itemize}
\end{assumption}

Let $v_j := \frac{1-|\lambda_j|^2}{1+|\lambda_j|^2}$ be the speed corresponding to $\lambda_j$ according to (\ref{gammavdef}).  
For $|\lambda_j| < 1$, the speed $v_j$ is positive, meaning that the soliton travels to the right, whereas for $|\lambda_j| > 1$ it is negative, meaning that the soliton travels to the left. This suggests that, for the half-line problem, only those $\lambda_j$ that satisfy $|\lambda_j| < 1$ (i.e., $\{\lambda_j\}_1^\Lambda$) give rise to solitons asymptotically, see Figure \ref{poles.pdf}. 

Assumption \ref{poleassumption2} implies that the poles $\lambda_j$ are either pure imaginary or come in pairs $(\lambda_j, \lambda_{j+1} = -\bar{\lambda}_j)$ where $\re \lambda_j > 0$.
According to the next theorem, each pure imaginary pole $\lambda_j$ with $|\lambda_j| < 1$ gives rise to a kink/antikink, whereas each pair $(\lambda_j, -\bar{\lambda}_j)$ with $|\lambda_j| < 1$ gives rise to a breather. The speeds of these solitons satisfy
$$0 < v_\Lambda \leq \cdots \leq v_2 \leq v_1 < 1$$
with strict inequality $v_{j+1} < v_j$ except when $\lambda_{j+1} = -\bar{\lambda}_j$. 

\begin{figure}
\begin{center}
\bigskip\bigskip
\hspace{-.6cm}
 \begin{overpic}[width=.525\textwidth]{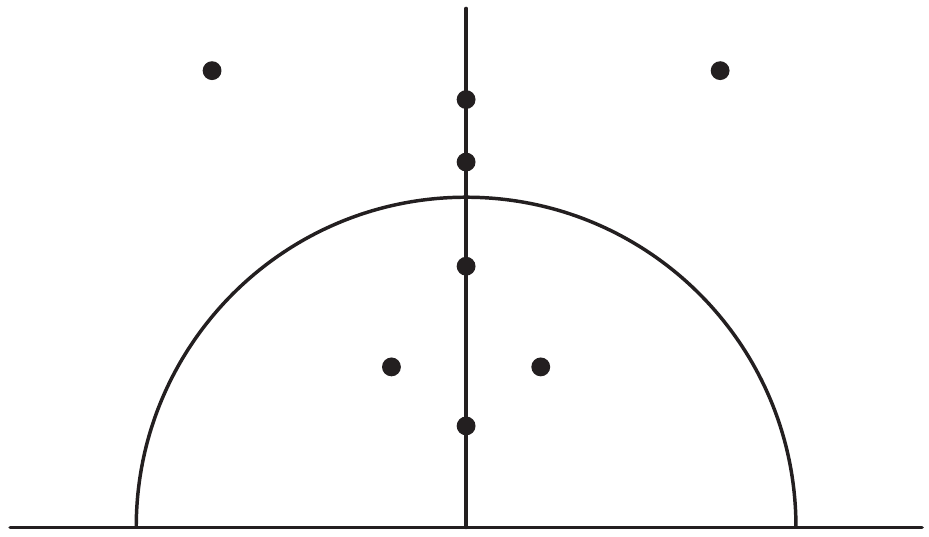}
      \put(51.8,10){\small $\lambda_1$}
      \put(60,17){\small $\lambda_2$}
      \put(35.5,17){\small $\lambda_3$}
      \put(51.8,27.5){\small $\lambda_4$}
      \put(51.8,39){\small $\lambda_5$}
      \put(51.8,46){\small $\lambda_6$}
      \put(79,48.5){\small $\lambda_7$}
      \put(16.5,48.5){\small $\lambda_8$}
      \put(101.5,-.5){\small $\re k$}
      \put(84.7,-3.6){\small $1$}
      \put(11,-3.6){\small $-1$}
   \end{overpic} \qquad\qquad
 \begin{overpic}[width=.325\textwidth]{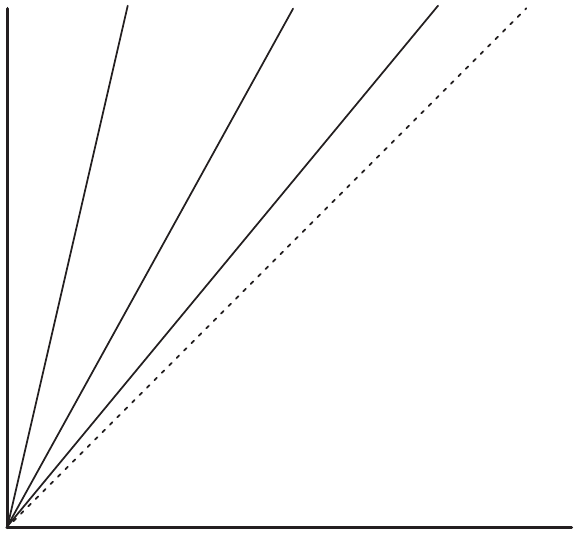}
      \put(100,0){ $x$}
      \put(-3,96){ $t$}
       \put(14,80){\makebox(0,0){\rotatebox{77}{{\small $\zeta = v_4$}}}}
       \put(35,74){\makebox(0,0){\rotatebox{61}{{\small $\zeta = v_2 = v_3$}}}}
       \put(62,83){\makebox(0,0){\rotatebox{51}{{\small $\zeta = v_1$}}}}
       \put(88,82){\makebox(0,0){\rotatebox{45}{{\small $\zeta = 1$}}}}
   \end{overpic}
     \begin{figuretext}\label{poles.pdf}
        The left figure displays a possible distribution of the $\lambda_j$. The associated sine-Gordon solution is dominated asymptotically by solitons traveling in the three directions shown on the right. 
        The line $\zeta = 1$ is also shown (dashed). 
The pure imaginary poles $\lambda_1$ and  $\lambda_4$ generate kinks/antikinks traveling with speeds $v_1$ and  $v_4$, respectively. The pair $(\lambda_2, \lambda_3 = -\bar{\lambda}_2)$ generates a breather traveling with speed $v_2 = v_3$. Only those $\lambda_j$ with $|\lambda_j| \leq 1$ generate solitons. 
     \end{figuretext}
     \end{center}
\end{figure}

\begin{theorem}[Asymptotics of quarter-plane solutions]\label{solitonasymptoticsth}
Let $\{\lambda_j\}_1^\mathcal{N} \subset \C_+ \setminus \Gamma$ satisfy Assumption \ref{poleassumption2} and let $\{c_j\}_1^\mathcal{N}$ be nonzero complex numbers satisfying (\ref{cjassump}). Suppose $r_1:\R \to \C$ and $h:\bar{D}_2 \to \C$ satisfy Assumption \ref{r1hassumption2} except that $h(k)$ may have a simple pole at each point in the set $\{\lambda_j\}_1^\Lambda \subset D_2$.
Let $u(x,t)$ be the sine-Gordon quarter-plane solution corresponding to $r_1, h, \lambda_j, c_j$ via Theorem \ref{solitonexistenceth}. 
  
Then $u \in C^\infty([0,\infty) \times [0,\infty), \R)$ and there exists a choice of the branch of $\arg$ in (\ref{ulim}) such that $u(x,0) \to 0$ as $x \to \infty$. 
For this choice of branch, $u(x,t)$ satisfies the asymptotic formulas (\ref{uasymptoticsI}) and (\ref{uasymptoticsII}) Sectors I and II. In Sector IV, $u(x,t)$ satisfies
\begin{align}\label{uasymptoticsIV3}
\text{\upshape Sector IV:} \quad  u(x,t) = -2\pi \sum_{i=1}^{\Lambda} \sgn(\im c_i) + O\big(\zeta^{Z+2} + t^{-\frac{Z}{2}-1}\big), \qquad 0 \leq \zeta \leq 1/2, \ t > 1,
\end{align}
whenever $r(k)$ vanishes to order $Z \geq 0$ at $k = 1$.
In Sector III the asymptotics of $u(x,t)$ is given as follows: For each $1 \leq j \leq \Lambda$ and each small enough $\epsilon > 0$, the asymptotics in the narrow sector $\zeta \in (v_{j} - \epsilon, v_{j} + \epsilon)$ centered on the line $\zeta = v_j$ is given uniformly by
\begin{subequations}\label{solitonuasymptotics}
\begin{align}\nonumber
u(x,t) = &\;  u_{const}(j) + u_{sol}(x,t; j)  + \frac{u_{rad}^{(1)}(x,t) + u_{rad}^{(2)}(x,t)}{\sqrt{t}}
	\\ \label{solitonuasymptoticsa}
& + O\bigg(\frac{\ln{t}}{t}\bigg), \qquad \zeta \in (v_{j} - \epsilon, v_{j} + \epsilon), \ t > 2,
\end{align}
and the asymptotics away from these narrow sectors is given uniformly by
\begin{align}\nonumber
& u(x,t) = u_{const}(j) + \frac{u_{rad}^{(1)}(x,t) + u_{rad}^{(2)}(x,t)}{\sqrt{t}} + O\bigg(\frac{1}{k_0 t} + \frac{\ln{t}}{t}\bigg), 
	\\ \label{solitonuasymptoticsb}
&   \zeta \in [v_j + \epsilon, v_{j-1} - \epsilon], \ t > 2, \ j = 1, \dots, \Lambda+1, \ (v_{\Lambda +1} + \epsilon \equiv 0, v_0 - \epsilon \equiv 1),
\end{align}
\end{subequations}
where
\begin{itemize}

\item $u_{const}(j) \in 2\pi \Z$ is the constant background in the sector\footnote{The sector $v_j < \zeta < v_{j-1}$ is empty if $\lambda_{j} = -\bar{\lambda}_{j-1}$.} $v_j < \zeta < v_{j-1}$ generated by the solitons with speeds $\{v_i\}_1^{j-1}$:
\begin{align}\label{uconstdef}
u_{const}(j) = -2\pi \sum_{i=1}^{j-1} \sgn(\im c_i), \qquad  j = 1, \dots, \Lambda+1.
\end{align}

\item If $\lambda_j \in i\R$, then $u_{sol}(x,t;j)$ is the kink/antikink of speed $v_j$ given by
\begin{align}\label{usoljdef}
u_{sol}(x,t;j) = -4\arctan d_j'(x,t),
\end{align}
where $\gamma_j = \frac{1}{2}(|\lambda_j| + |\lambda_j|^{-1})$ and
\begin{align}\label{djprimedef}
d_j'(x,t) := \frac{e^{-\gamma_j(x - v_jt)}e^{-\frac{|\lambda_j|}{\pi}\int_{-k_0}^{k_0}\frac{\ln(1+|r(s)|^2)}{s^2+|\lambda_j|^2} ds}\im c_j}{2|\lambda_j|} 
\prod_{l=1}^{j-1} \bigg|\frac{\lambda_j-\lambda_l}{\lambda_j-\bar{\lambda}_l}\bigg|^2.
\end{align}

\item If $\lambda_{j+1} = -\bar{\lambda}_j$, then $u_{sol}(x,t; j) \equiv u_{sol}(x,t; j+1)$ is the breather of speed $v_j = v_{j+1}$  given by
\begin{align}\label{usoljdef2}
u_{sol}(x,t;j) = 4\arctan\bigg\{\frac{2\im d_j'' \, \im \lambda_j}{(1 + |d_j''|^2)\re \lambda_j}  \bigg\},\end{align}
where
\begin{align}\label{djdoubleprimedef}
d_j''(x,t) := \frac{c_j e^{2i\theta(x,t,\lambda_j)} e^{-\frac{1}{\pi i }\int_{-k_0}^{k_0}\frac{\ln (1+|r(s)|^2)}{s - \lambda_j}ds} \re\lambda_j}{2i\lambda_j \im \lambda_j} 
\prod_{l=1}^{j-1} \bigg(\frac{\lambda_j-\lambda_l}{\lambda_j-\bar{\lambda}_l}\bigg)^2.
\end{align}

\item $u_{rad}^{(1)}(x,t)$ is (up to a factor of $(-1)^\mathcal{N}$) the asymptotic radiation contribution of Theorem \ref{asymptoticsth} determined by $\tilde{r}(k) := \tilde{r}_1(k) + \tilde{h}(k)$, i.e.,
\begin{align}\label{uradraddef}
u_{rad}^{(1)}(x,t) = - 2 (-1)^\mathcal{N} \sqrt{\frac{2(1 + k_0^2)\tilde{\nu}}{k_0}}   \sin \tilde{\alpha},
\end{align}
where $\tilde{\nu}$ and $\tilde{\alpha}$ are defined by replacing $r(k)$ with $\tilde{r}(k)$ in the definition (\ref{alphadef}) of $\nu$ and $\alpha$.

\item $u_{rad}^{(2)}(x,t)$ is the asymptotic radiation contribution generated by $\{\lambda_j\}_1^{\mathcal{N}}$ as follows: 
If $\zeta \in [v_j + \epsilon, v_{j-1} - \epsilon]$ for some $j = 1, \dots, \Lambda$ (that is, if $(x,y)$ lies away from the narrow asymptotic soliton sectors), then
\begin{align}\nonumber
u_{rad}^{(2)}(x,t)
= & \; \sum_{\substack{\lambda_l \in i\R \\ 1 \leq l \leq j-1}} 4(-1)^{\mathcal{N} - l} r_{\mathcal{N} + l}^{(l-1)}  
 +\sum_{\substack{\lambda_{l} \in i\R \\ j \leq l \leq \mathcal{N}}} 4(-1)^{\mathcal{N} - l}
r_{l}^{(l-1)}
	\\\label{outsideusolraddef}
& - \sum_{\substack{\re \lambda_l > 0 \\ 1 \leq l \leq j-1}}
8(-1)^{\mathcal{N}-l} \frac{\im \lambda_l}{\re \lambda_l}\im r_{\mathcal{N} + l}^{(l-1)}
 - \sum_{\substack{\re \lambda_{l} > 0 \\ j \leq l \leq \mathcal{N}}} 8 (-1)^{\mathcal{N}-l}\frac{\im \lambda_{l}}{\re \lambda_{l}} \im r_{l}^{(l-1)}, 
\end{align}
where the functions $\{r_m^{(l)}(x,t), r_{\mathcal{N} + m}^{(l)}(x,t)\}$, $l = 0, \dots, \mathcal{N}-1$, $m = l+1, \dots, \mathcal{N}$, are defined for $l = 0$ by 
\begin{subequations}\label{rrecursive}
\begin{align}\label{rinitial}
& \begin{cases}
r_m^{(0)} = \frac{i\tilde{A}}{k_0 -\lambda_m} - \frac{i\bar{\tilde{A}}}{k_0 + \lambda_m},
	\\
r_{\mathcal{N} + m}^{(0)} = \frac{i\tilde{A}}{k_0 - \bar{\lambda}_m} - \frac{i\bar{\tilde{A}}}{k_0 + \bar{\lambda}_m},
\end{cases}  m = 1, \dots, \mathcal{N}; \quad
 \tilde{A}(x,t) := \sqrt{\frac{k_0(1 + k_0^2)\tilde{\nu}}{2}} e^{i\tilde{\alpha}(x,t)},
\end{align}
and for $l = 1, \dots, \mathcal{N}$ via the recursive formulas
\begin{align}\label{rrecursive1}
& \begin{cases}
r_m^{(l)} = \frac{\lambda_m - \lambda_l}{\lambda_m - \bar{\lambda}_l} r_m^{(l-1)} + \frac{\lambda_l - \bar{\lambda}_l}{\lambda_m - \bar{\lambda}_l}r_{\mathcal{N} + l}^{(l-1)},
	\\
r_{\mathcal{N} + m}^{(l)} = \frac{\bar{\lambda}_m - \lambda_l}{\bar{\lambda}_m - \bar{\lambda}_l} r_{\mathcal{N} + m}^{(l-1)}
+ \frac{\lambda_l - \bar{\lambda}_l}{\bar{\lambda}_m - \bar{\lambda}_l} r_{\mathcal{N} + l}^{(l-1)},
\end{cases} \ l = 1, \dots, j-1,
	\\ \label{rrecursive2}
& \begin{cases}
r_m^{(l)} = \frac{\lambda_m - \bar{\lambda}_{l}}{\lambda_m - \lambda_{l}}r_m^{(l-1)}
+ \frac{\bar{\lambda}_{l}-\lambda_{l}}{\lambda_m - \lambda_{l}}r_{l}^{(l-1)},
	\\
r_{\mathcal{N} + m}^{(l)} = 
\frac{\bar{\lambda}_m - \bar{\lambda}_{l}}{\bar{\lambda}_m - \lambda_{l}} r_{\mathcal{N} + m}^{(l-1)}
+ \frac{\bar{\lambda}_{l}-\lambda_{l}}{\bar{\lambda}_m - \lambda_{l}} r_{l}^{(l-1)},
\end{cases}  l = j, \dots, \mathcal{N}.
\end{align}
\end{subequations}

If $\zeta \in (v_{j} - \epsilon, v_{j} + \epsilon)$ for some $j = 1, \dots, \Lambda$ with $\lambda_j \in i\R$ (that is, if $(x,t)$ lies in the narrow asymptotic sector centered on an asymptotic kink/antikink), then
\begin{align}\nonumber
u_{rad}^{(2)}(x,t)
= & \; 4\frac{r^{(\mathcal{N}-1)}_j + r^{(\mathcal{N}-1)}_{\mathcal{N} + j} (d_j')^2}{1 + (d_j')^2}
+ \sum_{\substack{\lambda_l \in i\R \\ 1 \leq l \leq j-1}} 4(-1)^{\mathcal{N} - l} r_{\mathcal{N} + l}^{(l-1)}  
 +\sum_{\substack{\lambda_{l} \in i\R \\ j +1 \leq l \leq \mathcal{N}}} 4(-1)^{\mathcal{N} - l +1}
r_{l}^{(l-2)}
	\\\label{kinkusolraddef}
& - \sum_{\substack{\re \lambda_l > 0 \\ 1 \leq l \leq j-1}}
8(-1)^{\mathcal{N}-l} \frac{\im \lambda_l}{\re \lambda_l}\im r_{\mathcal{N} + l}^{(l-1)}
+ \sum_{\substack{\re \lambda_{l} > 0 \\ j+1 \leq l \leq \mathcal{N}}} 8 (-1)^{\mathcal{N}-l+1}\frac{\im \lambda_{l}}{\re \lambda_{l}} \im r_{l}^{(l-2)}, 
\end{align}
where the functions $\{r_m^{(l)}(x,t), r_{\mathcal{N} + m}^{(l)}(x,t)\}$, $l = 0, \dots, \mathcal{N}-1$, are defined recursively by the same formulas (\ref{rrecursive}) except that (\ref{rrecursive2}) is replaced by
\begin{align}\label{rrecursive2kink}
& \begin{cases}
r_m^{(l)} = \frac{\lambda_m - \bar{\lambda}_{l+1}}{\lambda_m - \lambda_{l+1}}r_m^{(l-1)}
+ \frac{\bar{\lambda}_{l+1}-\lambda_{l+1}}{\lambda_m - \lambda_{l+1}}r_{l+1}^{(l-1)},
	\\
r_{\mathcal{N} + m}^{(l)} = 
\frac{\bar{\lambda}_m - \bar{\lambda}_{l+1}}{\bar{\lambda}_m - \lambda_{l+1}} r_{\mathcal{N} + m}^{(l-1)}
+ \frac{\bar{\lambda}_{l+1}-\lambda_{l+1}}{\bar{\lambda}_m - \lambda_{l+1}} r_{l+1}^{(l-1)},
\end{cases}  l = j, \dots, \mathcal{N}-1,\ m \in \{ l+2, \dots, \mathcal{N}\} \cup \{j\}.
\end{align}

If $\zeta \in (v_{j} - \epsilon, v_{j} + \epsilon)$ for some $j = 1, \dots, \Lambda$ with $\re \lambda_j > 0$ (that is, if $(x,t)$ lies in the narrow asymptotic sector centered on an asymptotic breather), then
\begin{align}\nonumber
u_{rad}^{(2)}(x,t)
= & 
- \frac{8 (\im\lambda_j)(\re\lambda_j)  \im\big[(1 + (\overline{d_j''})^2) r_j^{(\mathcal{N}-2)} - (1 + (d_j'')^2)(\overline{d_j''})^2 r_{\mathcal{N}+j}^{(\mathcal{N}-2)}\big]}{4(\im d_j'')^2 (\im \lambda_j)^2 + (1 + |d_j''|^2)^2 (\re \lambda_j)^2}
	\\ \nonumber
& + \sum_{\substack{\lambda_l \in i\R \\ 1 \leq l \leq j-1}} 4(-1)^{\mathcal{N} - l} r_{\mathcal{N} + l}^{(l-1)}  
 +\sum_{\substack{\lambda_{l} \in i\R \\ j +2 \leq l \leq \mathcal{N}}} 4(-1)^{\mathcal{N} - l}
r_{l}^{(l-3)}
	\\ \label{breatherusolraddef}
&	- \sum_{\substack{\re \lambda_l > 0 \\ 1 \leq l \leq j-1}}
8(-1)^{\mathcal{N}-l} \frac{\im \lambda_l}{\re \lambda_l}\im r_{\mathcal{N} + l}^{(l-1)}
 - \sum_{\substack{\re \lambda_{l} > 0 \\ j+2 \leq l \leq \mathcal{N}}} 8 (-1)^{\mathcal{N}-l}\frac{\im \lambda_{l}}{\re \lambda_{l}} \im r_{l}^{(l-3)},
\end{align}
where the functions $\{r_m^{(l)}(x,t), r_{\mathcal{N} + m}^{(l)}(x,t)\}$, $l = 0, \dots, \mathcal{N}-2$,  are defined recursively by the same formulas (\ref{rrecursive}) except that (\ref{rrecursive2}) is replaced by
\begin{align}\label{rrecursive2breather}
\begin{cases}
r_m^{(l)} = \frac{\lambda_m - \bar{\lambda}_{l+2}}{\lambda_m - \lambda_{l+2}}r_m^{(l-1)}
+ \frac{\bar{\lambda}_{l+2}-\lambda_{l+2}}{\lambda_m - \lambda_{l+2}}r_{l+2}^{(l-1)},
	\\
r_{\mathcal{N} + m}^{(l)} = 
\frac{\bar{\lambda}_m - \bar{\lambda}_{l+2}}{\bar{\lambda}_m - \lambda_{l+2}} r_{\mathcal{N} + m}^{(l-1)}
+ \frac{\bar{\lambda}_{l+2}-\lambda_{l+2}}{\bar{\lambda}_m - \lambda_{l+2}} r_{l+2}^{(l-1)},
\end{cases} l = j, \dots, \mathcal{N}-2, \ m \in \{l+3, \dots, \mathcal{N}\} \cup \{j,j+1\}.
\end{align}
\end{itemize}
\end{theorem}
\begin{proof}
Let $\hat{M}(x,t) := \ntlim_{k \to 0} M(x,t,k)$ and $\hat{\tilde{M}}(x,t) := \ntlim_{k \to 0} \tilde{M}(x,t,k)$. 
The functions $\tilde{r}_1:\R \to \C$ and $\tilde{h}:\bar{D}_2 \to \C$ defined in terms of $r_1$ and $h$ by (\ref{tilder1hdef}) satisfy Assumption \ref{r1hassumption2} (and hence also Assumption \ref{r1hassumption}). Thus, by Theorem \ref{existenceth}, the pure radiation solution 
$$\tilde{u} := 2 \arg(\hat{\tilde{M}}_{11} + i\hat{\tilde{M}}_{21})$$
generated by $\tilde{r}_1$ and $\tilde{h}$ is well-defined, and, by Theorem \ref{asymptoticsth}, $\tilde{u}(x,t)$ satisfies the asymptotic formulas (\ref{uasymptotics}) with $\nu$ and $\alpha$ replaced by $\tilde{\nu}$ and $\tilde{\alpha}$, respectively. 

Let $B_j$ denote the matrices obtained by applying the recursive procedure of Lemma \ref{dressinglemma} to the sequence $\{\lambda_j\}_1^\mathcal{N}$ with $C_j(x,t) = c_j e^{2i\theta(x,t,\lambda_j)}$ and $\tilde{M}_0 = \tilde{M}$.
Evaluating the expression (\ref{eq-mm}) for $M$ as $k \to 0$, we find
\begin{align}\label{hatmsoliton}
\hat{M}(x,t) = i^\mathcal{N}\mathcal{B}(x,t)\hat{\tilde{M}}(x,t)\begin{pmatrix}
\prod_{j=1}^{\mathcal{N}}\lambda_j^{-1} & 0 \\
0 & \prod_{j=1}^{\mathcal{N}} \bar{\lambda}_j^{-1} \end{pmatrix},
\end{align}
where $\mathcal{B} := \mathcal{B}(x,t)$ is defined by $\mathcal{B} = i^\mathcal{N} B_\mathcal{N}B_{\mathcal{N}-1} \cdots B_1$. 
Since  $\prod_{j=1}^\mathcal{N} \lambda_j = i^\mathcal{N}\prod_{j=1}^\mathcal{N} |\lambda_j|$, substitution of (\ref{hatmsoliton}) into (\ref{ulim}) yields
\begin{align*}
u(x,t) & = 2 \arg\Big((\mathcal{B} \hat{\tilde{M}})_{11} + i (\mathcal{B} \hat{\tilde{M}})_{21}\Big)
= 2 \arg\Big((\mathcal{B}_{11} + i \mathcal{B}_{21})\hat{\tilde{M}}_{11} 
+ (\mathcal{B}_{12} + i \mathcal{B}_{22}) \hat{\tilde{M}}_{21}\Big).
\end{align*}
The proof of Lemma \ref{dressinglemma} shows that $B_j = -\sigma_2 B_j \sigma_2$ if $\lambda_j \in i\R$ and (evaluating (\ref{BBsymm}) at $k = 0$) that $B_{j+1}B_j = \sigma_2B_{j+1}B_j\sigma_2$ if $(\lambda_j, \lambda_{j+1} = -\bar{\lambda}_j)$ is a pair with $\re \lambda_j > 0$. Hence $\mathcal{B} = (-1)^\mathcal{N} \sigma_2 \mathcal{B} \sigma_2$, that is, $\mathcal{B}_{22} = (-1)^\mathcal{N} \mathcal{B}_{11}$ and $\mathcal{B}_{12} = -(-1)^\mathcal{N} \mathcal{B}_{21}$. Thus
\begin{align*}
u(x,t) & = 2 \arg\Big((\mathcal{B}_{11}+i \mathcal{B}_{21} )(\hat{\tilde{M}}_{11} + i(-1)^\mathcal{N}\hat{\tilde{M}}_{21})\Big)
	\\
& = 2 \arg(\mathcal{B}_{11}+i \mathcal{B}_{21}) + (-1)^\mathcal{N} \tilde{u}(x,t) \quad \text{(mod $4\pi$)}.
\end{align*}
In general, if  $U,V$ are $2 \times 2$  matrices such that $U_{22} = \alpha U_{11}$ and $U_{12} = -\alpha U_{21}$ for some constant $\alpha \in \C$, then simple algebra gives
$$(UV)_{11} \pm i (UV)_{21} = (U_{11} \pm i U_{21})(V_{11} \pm \alpha i V_{21}).$$
Repeated use of this identity gives
\begin{align}\nonumber
u(x,t) = &\; \sum_{\lambda_j \in i\R} 2 \arg\Big((iB_j)_{11}+i (-1)^{\mathcal{N} - j} (iB_j)_{21}\Big) 
	\\\label{uargsum}
& + \sum_{\re \lambda_j > 0} 2 \arg\Big(-(B_{j+1}B_j)_{11}+i (-1)^{\mathcal{N} - j}(B_{j+1}B_j)_{21}\Big) 
+ (-1)^\mathcal{N} \tilde{u}(x,t) \quad \text{(mod $4\pi$)}.
\end{align}
Since $B_j = \sigma_2\overline{B_j}\sigma_2$ for all $j$ by (\ref{Bjsymm}), the entries of the matrices $iB_j$ and $B_{j+1}B_j$ are real for $\lambda_j \in i\R$ and $\re \lambda_j > 0$, respectively.

For each $j$, equation (\ref{MlambdajI}) shows that $\tilde{M}(x,t,\lambda_j) = I + O(x^{-N})$ uniformly as $(x,t) \to \infty$ in Sector I, and (\ref{MlambdajII}) shows that 
$\tilde{M}(x,t,\lambda_j) = I + O(k_0^N + t^{-N})$ uniformly as $(x,t) \to \infty$ in Sector II.
Moreover, (\ref{thetajexp}) shows that the factor $e^{2i\theta(x,t,\lambda_j)}$, and hence also the factor $d_j(x,t)$, is exponentially large as $(x,t) \to\infty$ in Sector I or Sector II. 
It follows that $B_j(x,t)$ asymptotes to $-\diag(\lambda_j, \bar{\lambda}_j)$ with an error of order $O(x^{-N})$ ($O(k_0^N + t^{-N})$) as $(x,t) \to \infty$ in Sector I (Sector II) for each $j = 1, \dots, \mathcal{N}$. Using this in equation (\ref{uargsum}), we conclude that $u(x,t)$ satisfies the asymptotic formulas (\ref{uasymptoticsI}) and (\ref{uasymptoticsII}) in Sectors I and II. 

We next consider Sector III. We will determine the asymptotics in this sector by establishing a sequence of six lemmas. The first two lemmas deal with the asymptotics away from the narrow sectors centered on the solitons.  

\begin{lemma}\label{outsideBjlemma}
Let $\{B_l\}_1^\mathcal{N}$ be the matrices obtained by applying the procedure of Lemma \ref{dressinglemma} with $\tilde{M}_0 = \tilde{M}$ and the $\lambda_l$ in the standard order $\lambda_1, \dots, \lambda_{\mathcal{N}}$.
For each $j = 1, \dots, \Lambda$, the following asymptotic formulas are valid as $t \to \infty$, uniformly with respect 
$\zeta \in [v_j + \epsilon, v_{j-1} - \epsilon]$:
\begin{align}\label{outsideBl}
B_l = \begin{cases}
- \begin{psmallmatrix} \bar{\lambda}_l & 0 \\ 0 & \lambda_l \end{psmallmatrix}
 + \frac{\lambda_l - \bar{\lambda}_l}{\sqrt{t}} 
 \begin{psmallmatrix} 0 & 
 \overline{r_{\mathcal{N} + l}^{(l-1)}} \\
  r_{\mathcal{N} + l}^{(l-1)}  & 0
 \end{psmallmatrix} 
+ O\big(\frac{1}{k_0 t} + \frac{\ln{t}}{t}\big), & l = 1, \dots, j-1, 
	\\
- \begin{psmallmatrix} \lambda_{l} & 0 \\ 0 & \bar{\lambda}_{l} \end{psmallmatrix}
 + \frac{\bar{\lambda}_{l}- \lambda_{l}}{\sqrt{t}} \begin{psmallmatrix} 0 & \overline{r_{l}^{(l-1)}} \\
 r_{l}^{(l-1)}& 0
 \end{psmallmatrix} 
+ O\big(\frac{1}{k_0 t} + \frac{\ln{t}}{t}\big), & l = j, \dots, \mathcal{N},
\end{cases}
\end{align}
where the functions $\{r_m^{(l)}(x,t), r_{\mathcal{N} + m}^{(l)}(x,t)\}$ are given by (\ref{rrecursive}).
\end{lemma} 
\proofbegin
We will prove that\footnote{Here and elsewhere in the proof the expansion is valid as $t \to \infty$ uniformly with respect 
$\zeta \in [v_j + \epsilon, v_{j-1} - \epsilon]$.}
\begin{align}\label{tildeMiexpansion}
\tilde{M}_i(x,t,\lambda_m)
= \bigg\{\begin{psmallmatrix} a_m^{(i)} & 0 \\ 0 & \overline{a_{\mathcal{N} + m}^{(i)}} \end{psmallmatrix} +
  \frac{1}{\sqrt{t}}\begin{psmallmatrix} 0 & -\overline{b_{\mathcal{N} + m}^{(i)}} \\ b_m^{(i)} & 0 \end{psmallmatrix}
 + O\bigg(\frac{1}{k_0 t} + \frac{\ln t}{t}\bigg)\bigg\}\delta(\zeta, \lambda_m)^{\sigma_3}
 \end{align}
for $i = 0, \dots, \mathcal{N}$ and $m = 1, \dots, \mathcal{N}$, where $\{b_m^{(i)}(x,t)\}$ are complex-valued functions and the constants $\{a_m^{(i)}\} \subset \C$ are given by
\begin{align}\label{amlexpressions1}
&\begin{cases}
a_m^{(i)} = \prod_{s = 1}^i (\lambda_m - \bar{\lambda}_s), \\
a_{\mathcal{N} + m}^{(i)} = \prod_{s = 1}^i (\bar{\lambda}_m - \bar{\lambda}_s), 
\end{cases} \quad i = 1, \dots, j-1,
	\\ \label{amlexpressions2}
&\begin{cases}
a_m^{(i)} = \big(\prod_{s = 1}^{j-1} (\lambda_m - \bar{\lambda}_s)\big)
\big(\prod_{s = j}^{i} (\lambda_m - \lambda_{s})\big),
	\\
a_{\mathcal{N} + m}^{(i)} = \big(\prod_{s = 1}^{j-1} (\bar{\lambda}_m - \bar{\lambda}_s)\big)\big(\prod_{s = j}^{i} (\bar{\lambda}_m - \lambda_{s})\big), 
\end{cases} \quad i = j, \dots, \mathcal{N}.
\end{align}
The symmetries $\tilde{M}_{i}(x,t,k) = \sigma_2\overline{\tilde{M}_{i}(x,t,\bar{k})}\sigma_2$ and $\delta(\zeta, k) = \overline{\delta(\zeta, \bar{k})}^{-1}$ show that (\ref{tildeMiexpansion}) is equivalent to
\begin{align}\label{tildeMiexpansionbar}
\tilde{M}_i(x,t,\bar{\lambda}_m)
= \bigg\{\begin{psmallmatrix} a_{\mathcal{N} + m}^{(i)} & 0 \\ 0 & \overline{a_m^{(i)}} \end{psmallmatrix} +
  \frac{1}{\sqrt{t}}\begin{psmallmatrix} 0 & -\overline{b_m^{(i)}} \\ b_{\mathcal{N} + m}^{(i)} & 0 \end{psmallmatrix}
 + O\bigg(\frac{1}{k_0 t} + \frac{\ln t}{t}\bigg)\bigg\}\delta(\zeta, \bar{\lambda}_m)^{\sigma_3}.
 \end{align}
According to (\ref{MlambdajIII}), we have, for each $m = 1, \dots, \mathcal{N}$,
\begin{align}\label{tildeMlambdaj}
  \tilde{M}(x,t,\lambda_m) 
= \bigg\{I + \frac{F(x,t,\lambda_m)}{\sqrt{t}}
+ O\bigg(\frac{1}{k_0 t} + \frac{\ln{t}}{t}\bigg)\bigg\} \delta(\zeta, \lambda_m)^{\sigma_3},
\end{align}
where
$$F(x,t,\lambda_m)
:= \frac{i}{(k_0 - \lambda_m)}\begin{pmatrix} 0 & \bar{\tilde{A}} \\ \tilde{A} & 0 \end{pmatrix}
- \frac{i}{(k_0 + \lambda_m)}\begin{pmatrix} 0 & \tilde{A} \\ \bar{\tilde{A}} & 0 \end{pmatrix}$$
with $\tilde{A}(x,t)$ given by (\ref{rinitial}). 
Since $\tilde{M}_0 = \tilde{M}$, (\ref{tildeMlambdaj}) shows that (\ref{tildeMiexpansion}) holds for $i = 0$ with $a_m^{(0)} = a_{\mathcal{N} + m}^{(0)} = 1$ and
\begin{align*}
&b_m^{(0)}(x,t) = F_{21}(x,t,\lambda_m), \qquad 
b_{\mathcal{N} + m}^{(0)}(x,t) = -\overline{F_{12}(x,t,\lambda_m)}, \qquad m = 1, \dots, \mathcal{N}, 
\end{align*}
Defining 
$$r_m^{(l)}(x,t) := \frac{b_m^{(l)}(x,t)}{a_m^{(l)}}, \qquad 
r_{\mathcal{N} + m}^{(l)}(x,t) := \frac{b_{\mathcal{N} + m}^{(l)}(x,t)}{a_{\mathcal{N} + m}^{(l)}},$$ 
we see that the initial conditions (\ref{rinitial}) hold.

Seeking a proof by induction, we fix $1 \leq l \leq j-1$ and assume (\ref{tildeMiexpansion}) and (\ref{amlexpressions1}) (and hence also (\ref{tildeMiexpansionbar})) holds for $i = l - 1$. Then (\ref{Bjrecursive}) gives
\begin{align}\label{Blsystem1}
  \begin{cases}
  (\lambda_lI+B_l)
  \bigg\{\begin{psmallmatrix} a_l^{(l-1)} & 0 \\ 0 & \overline{a_{\mathcal{N} +l}^{(l-1)}} \end{psmallmatrix} +
  \frac{1}{\sqrt{t}}
  \begin{psmallmatrix} 0 & -\overline{b_{\mathcal{N} + l}^{(l-1)}} \\ b_l^{(l-1)} & 0 \end{psmallmatrix}
   + O\big(\frac{1}{k_0 t} + \frac{\ln t}{t}\big)\bigg\}\delta(\zeta, \lambda_l)^{\sigma_3}
  \begin{psmallmatrix} 1 \\ -d_l \end{psmallmatrix} = 0,
  	\vspace{.1cm}\\ 
  (\bar{\lambda}_lI +B_l)  
 \bigg\{\begin{psmallmatrix} a_{\mathcal{N} + l}^{(l-1)} & 0 \\ 0 & \overline{a_l^{(l-1)}} \end{psmallmatrix} +
  \frac{1}{\sqrt{t}}\begin{psmallmatrix} 0 & -\overline{b_l^{(l-1)}} \\ b_{\mathcal{N} + l}^{(l-1)} & 0 \end{psmallmatrix}
 + O\big(\frac{1}{k_0 t} + \frac{\ln t}{t}\big)\bigg\}\delta(\zeta, \bar{\lambda}_l)^{\sigma_3}
  \begin{psmallmatrix} \overline{d_l} \\ 1 \end{psmallmatrix} = 0,
  \end{cases}
\end{align}
where
$$d_{l}(x,t) = c_{l} e^{2i\theta(x,t,\lambda_{l})} \frac{\prod_{s=1,s\neq l}^{\mathcal{N}}(\lambda_{l} -\lambda_s)}{\prod_{s=1}^{\mathcal{N}}(\lambda_{l}-\bar{\lambda}_s)}.$$
For $m$ such that $v_m \geq v_{j}$ ($v_m \leq v_{j-1}$), the factor $e^{2i\theta(x,t,\lambda_m)}$, and hence also the factor $d_m(x,t)$, is exponentially large (small) as $t\to\infty$ uniformly for $\zeta \in [v_{j} + \epsilon, v_{j-1} - \epsilon]$ (see (\ref{thetajexp})).
Since $l \leq j - 1$, $d_l(x,t)$ is exponentially large as $t \to \infty$; hence (\ref{Blsystem1}) yields
\begin{align}\label{}
  \begin{cases}
  (\lambda_lI+B_l)
 \bigg\{\begin{psmallmatrix} a_l^{(l-1)} & 0 \\ 0 & \overline{a_{\mathcal{N} +l}^{(l-1)}} \end{psmallmatrix} +
  \frac{1}{\sqrt{t}}
  \begin{psmallmatrix} 0 & -\overline{b_{\mathcal{N} + l}^{(l-1)}} \\ b_l^{(l-1)} & 0 \end{psmallmatrix}
   + O\big(\frac{1}{k_0 t} + \frac{\ln t}{t}\big)\bigg\}
  \begin{psmallmatrix} 0 \\ 1 \end{psmallmatrix} = O(e^{-ct}),
  	\vspace{.1cm}\\ 
  (\bar{\lambda}_lI +B_l)  
   \bigg\{\begin{psmallmatrix} a_{\mathcal{N} + l}^{(l-1)} & 0 \\ 0 & \overline{a_l^{(l-1)}} \end{psmallmatrix} +
  \frac{1}{\sqrt{t}}\begin{psmallmatrix} 0 & -\overline{b_l^{(l-1)}} \\ b_{\mathcal{N} + l}^{(l-1)} & 0 \end{psmallmatrix}
 + O\big(\frac{1}{k_0 t} + \frac{\ln t}{t}\big)\bigg\}
  \begin{psmallmatrix} 1 \\ 0 \end{psmallmatrix} = O(e^{-ct}),
  \end{cases}
\end{align}
that is,
\begin{align*}
  \begin{cases}
  (\lambda_lI+B_l)
  \bigg\{\begin{psmallmatrix} 0 \\ \overline{a_{\mathcal{N} +l}^{(l-1)}}  \end{psmallmatrix} +
  \frac{1}{\sqrt{t}}\begin{psmallmatrix} -\overline{b_{\mathcal{N} + l}^{(l-1)}} \\ 0 \end{psmallmatrix}
 + O\big(\frac{1}{k_0 t} + \frac{\ln t}{t}\big)\bigg\} = O(e^{-ct}),
  	\vspace{.1cm}\\ 
  (\bar{\lambda}_lI +B_l)  
  \bigg\{\begin{psmallmatrix} a_{\mathcal{N} + l}^{(l-1)}  \\ 0 \end{psmallmatrix} +
  \frac{1}{\sqrt{t}}\begin{psmallmatrix} 0 \\ b_{\mathcal{N} + l}^{(l-1)} \end{psmallmatrix}
 + O\big(\frac{1}{k_0 t} + \frac{\ln t}{t}\big)\bigg\} = O(e^{-ct}).
  \end{cases}
\end{align*}
We deduce from this system that $B_l$ admits the expansion in the first line of (\ref{outsideBl}).
The recursive definition $\tilde{M}_{l} = (kI+B_l)\tilde{M}_{l-1}$ then gives
\begin{align*}
\tilde{M}_l(x,t,\lambda_m)
= & \; \bigg\{\begin{psmallmatrix} \lambda_m - \bar{\lambda}_l & 0 \\ 0 & \lambda_m - \lambda_l \end{psmallmatrix}
 + \frac{\lambda_l - \bar{\lambda}_l}{\sqrt{t}} \begin{psmallmatrix} 0 & 
 \overline{r_{\mathcal{N} + l}^{(l-1)}} \\
  r_{\mathcal{N} + l}^{(l-1)}  & 0
 \end{psmallmatrix} 
+ O\bigg(\frac{1}{k_0 t} + \frac{\ln t}{t}\bigg)\bigg\}
	\\
&\times \bigg\{\begin{psmallmatrix} a_m^{(l-1)} & 0 \\ 0 & \overline{a_{\mathcal{N} + m}^{(l-1)}} \end{psmallmatrix} +
  \frac{1}{\sqrt{t}}\begin{psmallmatrix} 0 & -\overline{b_{\mathcal{N} + m}^{(l-1)}} \\ b_m^{(l-1)} & 0 \end{psmallmatrix}
 + O\bigg(\frac{1}{k_0 t} + \frac{\ln t}{t}\bigg)\bigg\}\delta(\zeta, \lambda_m)^{\sigma_3},	
\end{align*}
which shows that (\ref{tildeMiexpansion}) holds also for $i = l$ with 
\begin{align*}
& a_m^{(l)} = (\lambda_m - \bar{\lambda}_l)a_m^{(l-1)}, &&
 b_m^{(l)} = (\lambda_m - \lambda_l)b_m^{(l-1)} +  (\lambda_l - \bar{\lambda}_l) r_{\mathcal{N} + l}^{(l-1)}a_m^{(l-1)},
	\\
&a_{\mathcal{N} + m}^{(l)} = (\bar{\lambda}_m - \bar{\lambda}_l)a_{\mathcal{N} + m}^{(l-1)},
&&
 b_{\mathcal{N} + m}^{(l)} = (\bar{\lambda}_m - \lambda_l)b_{\mathcal{N} + m}^{(l-1)} +  (\lambda_l - \bar{\lambda}_l) r_{\mathcal{N} + l}^{(l-1)}a_{\mathcal{N} + m}^{(l-1)}.
\end{align*}
By induction, we see that (\ref{tildeMiexpansion}) and (\ref{amlexpressions1}) hold for $i = 0, \dots, j-1$ and that (\ref{outsideBl}) holds for $l = 1, \dots, j-1$. 
Moreover, the recursive formulas (\ref{rrecursive1}) are satisfied.

Now fix $j \leq l \leq \mathcal{N}$ and assume (\ref{tildeMiexpansion})-(\ref{amlexpressions2}) (and hence also (\ref{tildeMiexpansionbar})) hold for $i = l - 1$. Equation (\ref{Bjrecursive}) gives
\begin{align}\label{Blsystem2}
  \begin{cases}
  (\lambda_{l}I+B_l)
  \bigg\{\begin{psmallmatrix} a_{l}^{(l-1)} & 0 \\ 0 & \overline{a_{\mathcal{N} +l}^{(l-1)}} \end{psmallmatrix} +
  \frac{1}{\sqrt{t}}
  \begin{psmallmatrix} 0 & -\overline{b_{\mathcal{N} + l}^{(l-1)}} \\ b_{l}^{(l-1)} & 0 \end{psmallmatrix}
   + O\big(\frac{1}{k_0 t} + \frac{\ln t}{t}\big)\bigg\}\delta(\zeta, \lambda_{l})^{\sigma_3}
  \begin{psmallmatrix} 1 \\ -d_l \end{psmallmatrix} = 0,
  	\vspace{.1cm}\\ 
  (\bar{\lambda}_{l}I +B_l)  
 \bigg\{\begin{psmallmatrix} a_{\mathcal{N} + l}^{(l-1)} & 0 \\ 0 & \overline{a_{l}^{(l-1)}} \end{psmallmatrix} +
  \frac{1}{\sqrt{t}}\begin{psmallmatrix} 0 & -\overline{b_{l}^{(l-1)}} \\ b_{\mathcal{N} + l}^{(l-1)} & 0 \end{psmallmatrix}
 + O\big(\frac{1}{k_0 t} + \frac{\ln t}{t}\big)\bigg\}\delta(\zeta, \bar{\lambda}_{l})^{\sigma_3}
  \begin{psmallmatrix} \overline{d_l} \\ 1 \end{psmallmatrix} = 0.
  \end{cases}
\end{align}
Since $l \geq j$, $d_l(x,t)$ is exponentially small as $t \to \infty$, so (\ref{Blsystem2}) yields
\begin{align}\label{}
  \begin{cases}
  (\lambda_{l}I+B_l)
  \bigg\{\begin{psmallmatrix} a_{l}^{(l-1)} \\ 0  \end{psmallmatrix} +
  \frac{1}{\sqrt{t}}
  \begin{psmallmatrix} 0 \\ b_{l}^{(l-1)}  \end{psmallmatrix}
 + O\big(\frac{1}{k_0 t} + \frac{\ln t}{t}\big)\bigg\} = O(e^{-ct}),
  	\vspace{.1cm}\\ 
  (\bar{\lambda}_{l}I +B_l)  
  \bigg\{\begin{psmallmatrix} 0 \\ \overline{a_{l}^{(l-1)}}  \end{psmallmatrix} +
  \frac{1}{\sqrt{t}}\begin{psmallmatrix} -\overline{b_{l}^{(l-1)}}  \\ 0 \end{psmallmatrix}
 + O\big(\frac{1}{k_0 t} + \frac{\ln t}{t}\big)\bigg\} = O(e^{-ct}).
  \end{cases}
\end{align}
We deduce from this system that $B_l$ admits the expansion in the second line of (\ref{outsideBl}).
The recursive definition $\tilde{M}_{l} = (kI+B_l)\tilde{M}_{l-1}$ then gives
\begin{align*}
\tilde{M}_l(x,t,\lambda_m)
= & \; \bigg\{\begin{psmallmatrix} \lambda_m - \lambda_{l} & 0 \\ 0 & \lambda_m - \bar{\lambda}_{l} \end{psmallmatrix}
 + \frac{\bar{\lambda}_{l}- \lambda_{l}}{\sqrt{t}} \begin{psmallmatrix} 0 &
 \overline{r_{l}^{(l-1)}} \\ r_{l}^{(l-1)} & 0
 \end{psmallmatrix} 
+ O\bigg(\frac{1}{k_0 t} + \frac{\ln t}{t}\bigg)\bigg\}
	\\
&\times \bigg\{\begin{psmallmatrix} a_m^{(l-1)} & 0 \\ 0 & \overline{a_{\mathcal{N} + m}^{(l-1)}} \end{psmallmatrix} +
  \frac{1}{\sqrt{t}}\begin{psmallmatrix} 0 & -\overline{b_{\mathcal{N} + m}^{(l-1)}} \\ b_m^{(l-1)} & 0 \end{psmallmatrix}
 + O\bigg(\frac{1}{k_0 t} + \frac{\ln t}{t}\bigg)\bigg\}\delta(\zeta, \lambda_m)^{\sigma_3},	
\end{align*}
which shows that (\ref{tildeMiexpansion}) holds also for $i = l$ with 
\begin{align*}
& a_m^{(l)} = (\lambda_m - \lambda_{l})a_m^{(l-1)}, &&
b_m^{(l)} = (\lambda_m - \bar{\lambda}_{l})b_m^{(l-1)}
+ (\bar{\lambda}_{l}-\lambda_{l}) r_{l}^{(l-1)}a_m^{(l-1)}, 
	\\
&a_{\mathcal{N} + m}^{(l)} = (\bar{\lambda}_m - \lambda_{l})a_{\mathcal{N} + m}^{(l-1)},
&&  b_{\mathcal{N} + m}^{(l)} = 
(\bar{\lambda}_m - \bar{\lambda}_{l})b_{\mathcal{N} + m}^{(l-1)}
+ (\bar{\lambda}_{l}-\lambda_{l}) r_{l}^{(l-1)} a_{\mathcal{N} + m}^{(l-1)}.
\end{align*}
By induction, we see that (\ref{tildeMiexpansion})-(\ref{amlexpressions2}) hold for $i = 1, \dots,\mathcal{N}$ and that (\ref{outsideBl}) holds for $l = 1, \dots,\mathcal{N}$. 
Moreover, the recursive formulas (\ref{rrecursive2}) are satisfied.
\proofendcontinue

We will determine the constant background $u_{const}$ by simply adding up the accumulated topological charge, so it is enough in the following to derive the asymptotics modulo $2\pi$. 

\begin{lemma}[Asymptotics away from the narrow soliton sectors]\label{outsidesectorlemma} 
For $j = 1, \dots, \Lambda$, we have
\begin{align}\nonumber
u(x,t) = \frac{u_{rad}^{(1)}(x,t) + u_{rad}^{(2)}(x,t)}{\sqrt{t}}
+ O\bigg(\frac{1}{k_0t} + \frac{\ln t}{t}\bigg)  \quad \text{\upshape (mod $2\pi$)}
\end{align}
as $t \to \infty$ uniformly for $\zeta \in [v_{j} + \epsilon, v_{j-1} - \epsilon]$.
\end{lemma}
\proofbegin
By Theorem \ref{asymptoticsth}, $\tilde{u}(x,t)$ asymptotes to $-2\sqrt{\frac{2(1 + k_0^2)\tilde{\nu}}{tk_0}}   \sin \tilde{\alpha}$ (mod $4\pi$), with an error of $O(\frac{1}{k_0 t} + \frac{\ln{t}}{t})$ in Sector III. Thus the term $(-1)^\mathcal{N} \tilde{u}(x,t)$ on the right-hand side of (\ref{uargsum}) gives rise to the term $t^{-1/2}u_{rad}^{(1)}$ asymptotically. Let us consider the other terms in (\ref{uargsum}).

For $1 \leq l \leq j-1$ with $\lambda_l \in i\R$, (\ref{outsideBl}) gives
\begin{align*}
2\arg\big((iB_l)_{11}+i (-1)^{\mathcal{N} - l} (iB_l)_{21}\big) 
& = 
2\arg\bigg(1 + 2 i (-1)^{\mathcal{N} - l} \frac{r_{\mathcal{N} + l}^{(l-1)}}{\sqrt{t}} 
+ O\bigg(\frac{1}{k_0t} + \frac{\ln t}{t}\bigg) \bigg) 
	\\
& = 
\frac{4(-1)^{\mathcal{N} - l} r_{\mathcal{N} + l}^{(l-1)}}{\sqrt{t}} + O\bigg(\frac{1}{k_0t} + \frac{\ln t}{t}\bigg) \quad \text{(mod $2\pi$)}.
\end{align*}
Similarly, for $j \leq l \leq \mathcal{N}$ with $\lambda_{l} \in i\R$, (\ref{outsideBl}) gives
\begin{align*}
2\arg((iB_l)_{11}+i (-1)^{\mathcal{N} - l} (iB_l)_{21}) 
= 
\frac{4(-1)^{\mathcal{N} - l}r_{l}^{(l-1)}}{\sqrt{t}}  + O\bigg(\frac{1}{k_0t} + \frac{\ln t}{t}\bigg) \quad \text{(mod $2\pi$)}.
\end{align*}

Assume now that $1 \leq l \leq j-1$ with $\re \lambda_l > 0$. 
The proof of Lemma \ref{dressinglemma} gives
$$\tilde{M}_{l-1}(x,t,k) = (-1)^{l-1}\sigma_2 \tilde{M}_{l-1}(x,t,-k) \sigma_2
= (-1)^{l-1} \overline{\tilde{M}_{l-1}(x,t,-\bar{k})},$$
which for $k = \lambda_l = -\bar{\lambda}_{l+1}$ leads to
\begin{align}\label{roverliner}
r_{l}^{(l-1)} = \overline{r_{l+1}^{(l-1)}}, \qquad r_{\mathcal{N}+l}^{(l-1)} = \overline{r_{\mathcal{N}+l+1}^{(l-1)}}.
\end{align}
Hence, using (\ref{rrecursive1}),
\begin{align*}
r_{\mathcal{N} + l}^{(l-1)} - r_{\mathcal{N} + l + 1}^{(l)}
& = r_{\mathcal{N} + l}^{(l-1)} - \frac{\bar{\lambda}_{l+1} - \lambda_l}{\bar{\lambda}_{l+1} - \bar{\lambda}_l} r_{\mathcal{N} + l+1}^{(l-1)}
- \frac{\lambda_l - \bar{\lambda}_l}{\bar{\lambda}_{l+1} - \bar{\lambda}_l} r_{\mathcal{N} + l}^{(l-1)}
= \frac{2i\lambda_l \im r_{\mathcal{N} + l}^{(l-1)}}{\re \lambda_l}.
\end{align*}
Thus (\ref{outsideBl}) gives
\begin{align*}
& 2 \arg(-(B_{l+1}B_l)_{11}+i (-1)^{\mathcal{N} - l}(B_{l+1}B_l)_{21}) 
	\\
& =
2 \frac{(-1)^{\mathcal{N}-l}(\lambda_l - \bar{\lambda}_l)}{ \sqrt{t} \lambda_l} \big(r_{\mathcal{N} + l}^{(l-1)}
- r_{\mathcal{N} + l + 1}^{(l)}\big)
+ O\bigg(\frac{1}{k_0t} + \frac{\ln t}{t}\bigg)
	\\
& =
- 8 \frac{(-1)^{\mathcal{N}-l} \im \lambda_l}{ \sqrt{t}\re \lambda_l} \im r_{\mathcal{N} + l}^{(l-1)}
+ O\bigg(\frac{1}{k_0t} + \frac{\ln t}{t}\bigg) \quad \text{(mod $2\pi$)}.
\end{align*}
Similarly, for $j \leq l \leq \mathcal{N}-1$ with $\re \lambda_{l} > 0$, (\ref{outsideBl}) gives
\begin{align*}
& 2 \arg(-(B_{l+1}B_l)_{11}+i (-1)^{\mathcal{N} - l}(B_{l+1}B_l)_{21}) 
	\\
& =
2 \frac{(-1)^{\mathcal{N}-l}(\bar{\lambda}_{l} - \lambda_{l})}{ \sqrt{t} \bar{\lambda}_{l}} \big(r_{l}^{(l-1)}
- r_{l + 1}^{(l)}\big)
+ O\bigg(\frac{1}{k_0t} + \frac{\ln t}{t}\bigg)  \quad \text{(mod $2\pi$)}
\end{align*}
where, by (\ref{rrecursive2}) and (\ref{roverliner}),
$$r_{l}^{(l-1)}
- r_{l + 1}^{(l)}
= r_{l}^{(l-1)}
- \frac{\lambda_{l+1} - \bar{\lambda}_{l}}{\lambda_{l+1} - \lambda_{l}}r_{l+1}^{(l-1)}
- \frac{\bar{\lambda}_{l}-\lambda_{l}}{\lambda_{l+1} - \lambda_{l}}r_{l}^{(l-1)}
= \frac{2 i\bar{\lambda}_{l} \im r_{l}^{(l-1)}}{\re \lambda_{l}}.$$
The lemma follows by inserting the above contributions in (\ref{uargsum}).
\proofendcontinue

Until now we have assumed that the recursive procedure of Lemma \ref{dressinglemma} has been applied with the sequence $\{\lambda_j\}_1^{\mathcal{N}}$ arranged in the standard order $\lambda_1, \dots, \lambda_\mathcal{N}$, so that $B_j$ corresponds to the insertion of a pole at $\lambda_j$. If we instead apply the procedure of Lemma \ref{dressinglemma} with the $\lambda_j$ in the order $\lambda_{\sigma(1)}, \dots, \lambda_{\sigma(\mathcal{N})}$, where $\sigma$ is a permutation of $\{1, \dots, \mathcal{N}\}$ such that $\sigma(j+1) = \sigma(j) +1$ if $\re \lambda_j > 0$, then (\ref{uargsum}) takes the form
\begin{align}\nonumber
u(x,t) = &\; \sum_{\lambda_{\sigma(j)} \in i\R} 2 \arg\Big((iB_j)_{11}+i (-1)^{\mathcal{N} - j} (iB_j)_{21}\Big) 
	\\\label{uargsum2}
& + \sum_{\re \lambda_{\sigma(j)} > 0} 2 \arg\Big(-(B_{j+1}B_j)_{11}+i (-1)^{\mathcal{N} - j}(B_{j+1}B_j)_{21}\Big) 
+ (-1)^\mathcal{N} \tilde{u}(x,t) \quad \text{(mod $4\pi$)}.
\end{align}
In this case, the matrix $B_j$ defined in the $j$th recursive step corresponds to the insertion of a pole at $\lambda_{\sigma(j)}$. 

The next two lemmas deal with the asymptotics in a sector containing an asymptotic kink/antikink.

\begin{lemma}\label{kinkBjlemma}
Assume $\lambda_j \in i\R$ for some $j = 1, \dots, \Lambda$. 
Let $\{B_l\}_1^\mathcal{N}$ be the matrices obtained by applying the procedure of Lemma \ref{dressinglemma} with $\tilde{M}_0 = \tilde{M}$ and the $\lambda_l$ in the order 
\begin{align}\label{ordering1}
\lambda_1, \dots, \lambda_{j-1}, \lambda_{j+1}, \dots, \lambda_{\mathcal{N}}, \lambda_j.
\end{align}

Then the following asymptotic formulas are valid as $t \to \infty$, uniformly with respect 
$\zeta \in [v_{j+1} + \epsilon, v_{j-1} - \epsilon]$:
\begin{align}\label{kinkBl}
B_l = \begin{cases}
- \begin{psmallmatrix} \bar{\lambda}_l & 0 \\ 0 & \lambda_l \end{psmallmatrix}
 + \frac{\lambda_l - \bar{\lambda}_l}{\sqrt{t}} 
 \begin{psmallmatrix} 0 & 
 \overline{r_{\mathcal{N} + l}^{(l-1)}} \\
  r_{\mathcal{N} + l}^{(l-1)}  & 0
 \end{psmallmatrix} 
+ O\big(\frac{1}{k_0 t} + \frac{\ln{t}}{t}\big), & l = 1, \dots, j-1, 
	\\
- \begin{psmallmatrix} \lambda_{l+1} & 0 \\ 0 & \bar{\lambda}_{l+1} \end{psmallmatrix}
 + \frac{\bar{\lambda}_{l+1}- \lambda_{l+1}}{\sqrt{t}} \begin{psmallmatrix} 0 & \overline{r_{l+1}^{(l-1)}} \\
 r_{l+1}^{(l-1)}& 0
 \end{psmallmatrix} 
+ O\big(\frac{1}{k_0 t} + \frac{\ln{t}}{t}\big), & l = j, \dots, \mathcal{N} - 1,
	\\
B_{\mathcal{N},0}+ \frac{1}{\sqrt{t}} B_{\mathcal{N},1} + O\big(\frac{1}{k_0 t} + \frac{\ln{t}}{t}\big), & l = \mathcal{N},
\end{cases}
\end{align}
where the functions $\{r_m^{(l)}(x,t), r_{\mathcal{N} + m}^{(l)}(x,t)\}$ are given by (\ref{rrecursive}) and the matrix-valued functions $B_{\mathcal{N},0}(x,t)$ and $B_{\mathcal{N},1}(x,t)$ satisfy
\begin{align}\label{BcalN0}
B_{\mathcal{N},0} = \frac{\lambda_j}{1+ (d_j')^2}\begin{psmallmatrix}
  (d_j')^2-1 & 2d_j' \\ 2d_j' & 1-(d_j')^2 \end{psmallmatrix}
\end{align}  
and
\begin{align}\label{BN1BN0quotient}
\frac{i(B_{\mathcal{N},1})_{11} - (B_{\mathcal{N},1})_{21}}{i(B_{\mathcal{N},0})_{11} - (B_{\mathcal{N},0})_{21}}
= 2i\frac{r^{(\mathcal{N}-1)}_j
+ (d_j')^2r^{(\mathcal{N}-1)}_{\mathcal{N} + j}}{1 + (d_j')^2}.
\end{align}
\end{lemma} 
\proofbegin
For $l = 1, \dots, \mathcal{N}-1$, the proof is analogous to that of Lemma \ref{kinkBjlemma}, except that this time (\ref{amlexpressions2}) is replaced by
\begin{align}\label{amlexpressionskink}
\begin{cases}
a_m^{(i)} = \big(\prod_{s = 1}^{j-1} (\lambda_m - \bar{\lambda}_s)\big)
\big(\prod_{s = j+2}^{i+2} (\lambda_m - \lambda_{s})\big),
	\\
a_{\mathcal{N} + m}^{(i)} = \big(\prod_{s = 1}^{j-1} (\bar{\lambda}_m - \bar{\lambda}_s)\big)\big(\prod_{s = j+2}^{i+2} (\bar{\lambda}_m - \lambda_{s})\big), 
\end{cases} \quad i = j, \dots, \mathcal{N} - 1,
\end{align}
and the recursive equations (\ref{rrecursive2}) are replaced by (\ref{rrecursive2kink}).
Since $\sigma(\mathcal{N}) = j$, (\ref{Bjrecursive}) gives the following system for $B_{\mathcal{N}}(x,t)$:
\begin{align*}
  \begin{cases}
  (\lambda_j I+B_{\mathcal{N}})
  \bigg\{\begin{psmallmatrix} a_j^{(\mathcal{N}-1)} & 0 \\ 0 & \overline{a_{\mathcal{N} + j}^{(\mathcal{N}-1)}} \end{psmallmatrix} +
  \frac{1}{\sqrt{t}}
  \begin{psmallmatrix} 0 & -\overline{b_{\mathcal{N} + j}^{(\mathcal{N}-1)}} \\ b_j^{(\mathcal{N}-1)} & 0 \end{psmallmatrix}
   + O\big(\frac{1}{k_0 t} + \frac{\ln t}{t}\big)\bigg\}\delta(\zeta, \lambda_j)^{\sigma_3}
  \begin{psmallmatrix} 1 \\ -d_{\mathcal{N}} \end{psmallmatrix} = 0,
  	\vspace{.1cm}\\ 
  (\bar{\lambda}_jI +B_{\mathcal{N}})  
 \bigg\{\begin{psmallmatrix} a_{\mathcal{N} + j}^{(\mathcal{N}-1)} & 0 \\ 0 & \overline{a_j^{(\mathcal{N}-1)}} \end{psmallmatrix} +
  \frac{1}{\sqrt{t}}\begin{psmallmatrix} 0 & -\overline{b_j^{(\mathcal{N}-1)}} \\ b_{\mathcal{N} + j}^{(\mathcal{N}-1)} & 0 \end{psmallmatrix}
 + O\big(\frac{1}{k_0 t} + \frac{\ln t}{t}\big)\bigg\}\delta(\zeta, \bar{\lambda}_j)^{\sigma_3}
  \begin{psmallmatrix} \overline{d_{\mathcal{N}}} \\ 1 \end{psmallmatrix} = 0,
  \end{cases}
\end{align*}
where
$$d_{\mathcal{N}}(x,t) = c_j e^{2i\theta(x,t,\lambda_j)} \frac{\prod_{l=1,l\neq j}^{\mathcal{N}}(\lambda_j-\lambda_l)}{\prod_{l=1}^{\mathcal{N}}(\lambda_j-\bar{\lambda}_l)}.$$
Since $\lambda_j \in i\R$, the symmetry properties of $\delta$ (see Lemma \ref{deltalemmaIII}) imply $\delta(\zeta, \lambda_j) = \delta(\zeta, \bar{\lambda}_j)^{-1} \in \R$.
Letting
$$d_j'(x,t) := d_{\mathcal{N}}(x,t)\frac{\overline{a_{\mathcal{N} + j}^{(\mathcal{N}-1)}}}{ a_j^{(\mathcal{N}-1)}}\delta(\zeta, \lambda_j)^{-2},
$$
a computation using (\ref{amlexpressionskink}) shows that $d_j'$ can be written as in (\ref{djprimedef}); in particular, $d_j'$ is real-valued. 
We infer that $B_{\mathcal{N}}$ admits the expansion in (\ref{kinkBl}), where $B_{\mathcal{N},0}$ is determined by the same system as the one-soliton except that $d_j$ is replaced by $d_j'$, i.e.,
\begin{align}\label{BcalN0system}
&  \begin{cases}
  (\lambda_jI+B_{\mathcal{N},0})
   \begin{psmallmatrix} 1 \\ - d_j' \end{psmallmatrix} = 0,
  	\vspace{.1cm}\\ 
  (\bar{\lambda}_jI +B_{\mathcal{N},0})  
  \begin{psmallmatrix} d_j'  \\ 1 \end{psmallmatrix} = 0,
 \end{cases}
\end{align}
and $B_{\mathcal{N},1}$ is determined by
\begin{align}\label{BcalN1system}
 \begin{cases}
  \bigg\{ (\lambda_jI+B_{\mathcal{N},0})
   \begin{psmallmatrix} 0 & -\overline{b_{\mathcal{N} + j}^{(\mathcal{N}-1)}} \\ b_j^{(\mathcal{N}-1)} & 0 \end{psmallmatrix}
 + B_{\mathcal{N},1}\begin{psmallmatrix} a_j^{(\mathcal{N}-1)} & 0 \\ 0 & \overline{a_{\mathcal{N} + j}^{(\mathcal{N}-1)}} \end{psmallmatrix}
   \bigg\}
   \begin{psmallmatrix} 1 \\ -d_{\mathcal{N}}\delta(\zeta, \lambda_j)^{-2} \end{psmallmatrix}
   = 0,
  	\vspace{.1cm}\\ 
 \bigg\{ (\bar{\lambda}_jI+B_{\mathcal{N},0})
 \begin{psmallmatrix} 0 & -\overline{b_j^{(\mathcal{N}-1)}} \\ b_{\mathcal{N} + j}^{(\mathcal{N}-1)} & 0 \end{psmallmatrix}
 + B_{\mathcal{N},1}\begin{psmallmatrix} a_{\mathcal{N} + j}^{(\mathcal{N}-1)} & 0 \\ 0 & \overline{a_j^{(\mathcal{N}-1)}} \end{psmallmatrix}
   \bigg\}
  \begin{psmallmatrix} \overline{d_\mathcal{N}} \delta(\zeta, \bar{\lambda}_j)^2 \\ 1 \end{psmallmatrix} = 0.
  \end{cases}
\end{align}
The system (\ref{BcalN0system}) yields the expression (\ref{BcalN0}) for $B_{\mathcal{N},0}$ (cf. (\ref{onesolitonB1mhat})).
On the other hand, we know from the proof of Lemma \ref{dressinglemma} that
$$\tilde{M}_{\mathcal{N}-1}(x,t,k) = (-1)^{\mathcal{N}-1}\sigma_2 \tilde{M}_{\mathcal{N}-1}(x,t,-k) \sigma_2,$$
that is, the functions $a_i^{(\mathcal{N}-1)},a_{\mathcal{N}+i}^{(\mathcal{N}-1)}, b_i^{(\mathcal{N}-1)},b_{\mathcal{N}+i}^{(\mathcal{N}-1)}$ take pure imaginary (real) values if $\mathcal{N}$ is even (odd). Long but straightforward computations using (\ref{BcalN0}) and (\ref{BcalN1system}) now yield (\ref{BN1BN0quotient}).
\proofendcontinue

\begin{lemma}[Asymptotics near a kink/antikink]\label{kinksectorlemma} 
Assume $\lambda_j \in i\R$ for some $j = 1, \dots, \Lambda$. Then
\begin{align}\nonumber
u(x,t) = & -4 \arctan(d_j') + \frac{u_{rad}^{(1)}(x,t) + u_{rad}^{(2)}(x,t)}{\sqrt{t}}
+ O\bigg(\frac{1}{k_0t} + \frac{\ln t}{t}\bigg)  \quad \text{\upshape (mod $2\pi$)}
\end{align}
as $t \to \infty$ uniformly for $\zeta \in [v_{j+1} + \epsilon, v_{j-1} - \epsilon]$.
\end{lemma}
\proofbegin
Using the ordering in (\ref{ordering1}), equation (\ref{uargsum2}) gives
\begin{align}\nonumber
u(x,t) = &\; 
2 \arg((iB_{\mathcal{N}})_{11}+i (iB_\mathcal{N})_{21}) 
+ \bigg(\sum_{\substack{\lambda_l \in i\R \\ 1 \leq l \leq j-1}} +\sum_{\substack{\lambda_{l+1} \in i\R \\ j \leq l \leq \mathcal{N}-1}}\bigg)2 \arg\big((iB_l)_{11}+i (-1)^{\mathcal{N} - l} (iB_l)_{21}\big) 
	\\\nonumber
& + \bigg(\sum_{\substack{\re \lambda_l > 0 \\ 1 \leq l \leq j-1}} + \sum_{\substack{\re \lambda_{l+1} > 0 \\ j \leq l \leq \mathcal{N}-1}}\bigg)2 \arg\Big(-(B_{l+1}B_l)_{11}+i (-1)^{\mathcal{N} - l}(B_{l+1}B_l)_{21}\Big) 
	\\
& + (-1)^\mathcal{N} \tilde{u}(x,t) \quad \text{(mod $4\pi$)}.
\end{align}
The last term on the right-hand side gives rise to the contribution from $u_{rad}^{(1)}$.
By Lemma \ref{kinkBjlemma}, the first term on the right-hand side satisfies
\begin{align*}
2 \arg((iB_{\mathcal{N}})_{11}& +i (iB_\mathcal{N})_{21}) 
= 2 \arg((iB_{\mathcal{N},0})_{11} - (B_{\mathcal{N},0})_{21})
	\\
& \hspace{2.3cm} + 2\arg\bigg(1 + \frac{i(B_{\mathcal{N},1})_{11} - (B_{\mathcal{N},1})_{21}}{\sqrt{t}(i(B_{\mathcal{N},0})_{11} - (B_{\mathcal{N},0})_{21})} + O\bigg(\frac{1}{k_0t} + \frac{\ln t}{t}\bigg) \bigg) 
	\\
= &\; 4\arg(1 - i d_j') 
+ 2\arg\bigg(1 + \frac{2i}{\sqrt{t}}\frac{r^{(\mathcal{N}-1)}_j 
+ (d_j')^2 r^{(\mathcal{N}-1)}_{\mathcal{N} + j}}{1 + (d_j')^2} + O\bigg(\frac{1}{k_0t} + \frac{\ln t}{t}\bigg) \bigg) 
	\\
= & -4 \arctan(d_j')
+ \frac{4}{\sqrt{t}}\frac{r^{(\mathcal{N}-1)}_j 
+  (d_j')^2 r^{(\mathcal{N}-1)}_{\mathcal{N} + j}}{1 + (d_j')^2}
+ O\bigg(\frac{1}{k_0t} + \frac{\ln t}{t}\bigg) 
\quad \text{(mod $4\pi$)}.
\end{align*}
The remaining terms are handled as in the proof of Lemma \ref{outsidesectorlemma}.
\proofendcontinue

The last two lemmas deal with the asymptotics in a sector containing an asymptotic breather. 
 
 \begin{lemma}\label{breatherBjlemma}
Assume $\re \lambda_j > 0$ for some $j = 1, \dots, \Lambda$. 
Let $\{B_l\}_1^\mathcal{N}$ be the matrices obtained by applying the procedure of Lemma \ref{dressinglemma} with $\tilde{M}_0 = \tilde{M}$ and the $\lambda_l$ in the order 
\begin{align}\label{ordering2}
\lambda_1, \dots, \lambda_{j-1}, \lambda_{j+2}, \dots, \lambda_{\mathcal{N}}, \lambda_j, \lambda_{j+1}.
\end{align}

Then the following asymptotic formulas are valid as $t \to \infty$, uniformly with respect 
$\zeta \in [v_{j+2} + \epsilon, v_{j-1} - \epsilon]$:
\begin{align}\label{breatherBl}
B_l = \begin{cases}
- \begin{psmallmatrix} \bar{\lambda}_l & 0 \\ 0 & \lambda_l \end{psmallmatrix}
 + \frac{\lambda_l - \bar{\lambda}_l}{\sqrt{t}} 
 \begin{psmallmatrix} 0 & 
 \overline{r_{\mathcal{N} + l}^{(l-1)}} \\
  r_{\mathcal{N} + l}^{(l-1)}  & 0
 \end{psmallmatrix} 
+ O\big(\frac{1}{k_0 t} + \frac{\ln{t}}{t}\big), & l = 1, \dots, j-1, 
	\\
- \begin{psmallmatrix} \lambda_{l+2} & 0 \\ 0 & \bar{\lambda}_{l+2} \end{psmallmatrix}
 + \frac{\bar{\lambda}_{l+2}- \lambda_{l+2}}{\sqrt{t}} \begin{psmallmatrix} 0 & \overline{r_{l+2}^{(l-1)}} \\
 r_{l+2}^{(l-1)}& 0
 \end{psmallmatrix} 
+ O\big(\frac{1}{k_0 t} + \frac{\ln{t}}{t}\big), & l = j, \dots, \mathcal{N} - 2,
	\\
B_{l,0}+ \frac{1}{\sqrt{t}} B_{l,1} + O\big(\frac{1}{k_0 t} + \frac{\ln{t}}{t}\big), & l = \mathcal{N}-1, \mathcal{N},
\end{cases}
\end{align}
where the functions $\{r_m^{(l)}(x,t), r_{\mathcal{N} + m}^{(l)}(x,t)\}$ are given by (\ref{rinitial}), (\ref{rrecursive1}), and (\ref{rrecursive2breather}), and the entries of $B_{\mathcal{N},0}(x,t)$ and $B_{\mathcal{N},1}(x,t)$ are such that the functions
\begin{align}\nonumber
& \beta_0(x,t):= 
2 \arg\big(-(B_{\mathcal{N},0}B_{\mathcal{N}-1,0})_{11} - i (B_{\mathcal{N},0}B_{\mathcal{N}-1,0})_{21}\big),
	\\
& \beta_1(x,t) := \frac{(B_{\mathcal{N},1}B_{\mathcal{N}-1,0})_{11} + i (B_{\mathcal{N},1}B_{\mathcal{N}-1,0})_{21}
+ (B_{\mathcal{N},0}B_{\mathcal{N}-1,1})_{11} + i (B_{\mathcal{N},0}B_{\mathcal{N}-1,1})_{21}}{(B_{\mathcal{N},0}B_{\mathcal{N}-1,0})_{11} + i (B_{\mathcal{N},0}B_{\mathcal{N}-1,0})_{21}}
\end{align}
satisfy
\begin{align}\label{beta0expression}
& \beta_0 = 4\arctan\bigg\{\frac{2\im d_j'' \, \im \lambda_j}{(1 + |d_j''|^2)\re \lambda_j}  \bigg\}\quad \text{\upshape (mod $2\pi$)},
	\\  \label{beta1expression}
& \beta_1 = - \frac{4 i (\im\lambda_j)(\re\lambda_j)  \im\big[(1 + (\overline{d_j''})^2) r_j^{(\mathcal{N}-2)} - (1 + (d_j'')^2)(\overline{d_j''})^2 r_{\mathcal{N}+j}^{(\mathcal{N}-2)}\big]}{4(\im d_j'')^2 (\im\lambda_j)^2 + (1 + |d_j''|^2)^2 (\re \lambda_j)^2}.
\end{align}
\end{lemma} 
\proofbegin
For $l = 1, \dots, \mathcal{N}-2$, the proof is analogous to that of Lemma \ref{outsideBjlemma}, except that (\ref{amlexpressions2}) is replaced by
\begin{align}\label{amlexpressionsbreather}
\begin{cases}
a_m^{(i)} = \big(\prod_{s = 1}^{j-1} (\lambda_m - \bar{\lambda}_s)\big)
\big(\prod_{s = j+2}^{i+2} (\lambda_m - \lambda_{s})\big),
	\\
a_{\mathcal{N} + m}^{(i)} = \big(\prod_{s = 1}^{j-1} (\bar{\lambda}_m - \bar{\lambda}_s)\big)\big(\prod_{s = j+2}^{i+2} (\bar{\lambda}_m - \lambda_{s})\big), 
\end{cases} \quad i = j, \dots, \mathcal{N} - 2,
\end{align}
and the recursive equations (\ref{rrecursive2}) are replaced by (\ref{rrecursive2breather}).
Equation (\ref{Bjrecursive}) then gives the following system for $B_{\mathcal{N} -1}(x,t)$ and $B_{\mathcal{N}}(x,t)$:
\begin{align}\label{BNBNm1system}
&  \begin{cases}
  (\lambda_j I+B_{\mathcal{N}-1})
  \tilde{M}_{\mathcal{N}-2}(x,t,\lambda_j)
  \begin{psmallmatrix} 1 \\ -d_{\mathcal{N}-1} \end{psmallmatrix} = 0,
  	\vspace{.1cm}\\ 
  (\bar{\lambda}_jI +B_{\mathcal{N}-1})  \tilde{M}_{\mathcal{N}-2}(x,t,\bar{\lambda}_j)
  \begin{psmallmatrix} \overline{d_{\mathcal{N}-1}} \\ 1 \end{psmallmatrix} = 0,
	\\
  (\lambda_{j+1} I+B_{\mathcal{N}})(\lambda_{j+1} I+B_{\mathcal{N}-1})
 \tilde{M}_{\mathcal{N}-2}(x,t,\lambda_{j+1})
  \begin{psmallmatrix} 1 \\ -d_{\mathcal{N}} \end{psmallmatrix} = 0,
  	\vspace{.1cm} \\ 
  (\bar{\lambda}_{j+1}I +B_{\mathcal{N}})(\bar{\lambda}_{j+1}I +B_{\mathcal{N}-1})  
 \tilde{M}_{\mathcal{N}-2}(x,t, \bar{\lambda}_{j+1})
  \begin{psmallmatrix} \overline{d_{\mathcal{N}}} \\ 1 \end{psmallmatrix} = 0,
  \end{cases}
\end{align}
where
\begin{align*}
\begin{cases}
d_{\mathcal{N}-1}(x,t) = c_j e^{2i\theta(x,t,\lambda_j)} \frac{\prod_{l=1,l\neq j}^{\mathcal{N}}(\lambda_j-\lambda_l)}{\prod_{l=1}^{\mathcal{N}}(\lambda_j-\bar{\lambda}_l)},
	\vspace{.1cm}\\ 
d_{\mathcal{N}}(x,t) = c_{j+1} e^{2i\theta(x,t,\lambda_{j+1})} \frac{\prod_{l=1,l\neq j+1}^{\mathcal{N}}(\lambda_{j+1}-\lambda_l)}{\prod_{l=1}^{\mathcal{N}}(\lambda_{j+1}-\bar{\lambda}_l)}.
\end{cases}
\end{align*}
Since $c_{j+1} = -\bar{c}_j$ and $\theta(x,t,\lambda_{j+1}) = -\overline{\theta(x,t,\lambda_j)}$, we find that 
$\overline{d_{\mathcal{N}}(x,t)} = d_{\mathcal{N}-1}(x,t)$.
Also, by Lemma \ref{deltalemmaIII},
\begin{align}\label{deltasymmbreather}
\delta(\zeta, \lambda_j)
= \overline{\delta(\zeta, \lambda_{j+1})}
= \delta(\zeta, \bar{\lambda}_{j+1})^{-1}
= \overline{\delta(\zeta, \bar{\lambda}_{j})}^{-1}.
\end{align}
On the other hand, the proof of Lemma \ref{dressinglemma} gives
$$\tilde{M}_{\mathcal{N}-2}(x,t,k) = (-1)^{\mathcal{N}}\sigma_2 \tilde{M}_{\mathcal{N}-2}(x,t,-k) \sigma_2
= (-1)^{\mathcal{N}} \overline{\tilde{M}_{\mathcal{N}-2}(x,t,-\bar{k})}.$$
In particular, 
\begin{align}\nonumber
& a_{j}^{(\mathcal{N}-2)} = (-1)^{\mathcal{N}} \overline{a_{j+1}^{(\mathcal{N}-2)}}, \qquad
a_{\mathcal{N}+j}^{(\mathcal{N}-2)} = (-1)^{\mathcal{N}} \overline{a_{\mathcal{N}+j+1}^{(\mathcal{N}-2)}}, 
	\\ \label{aboverlinebreather}
& b_{j}^{(\mathcal{N}-2)} = (-1)^{\mathcal{N}} \overline{b_{j+1}^{(\mathcal{N}-2)}}, \qquad
b_{\mathcal{N}+j}^{(\mathcal{N}-2)} = (-1)^{\mathcal{N}} \overline{b_{\mathcal{N}+j+1}^{(\mathcal{N}-2)}}.
\end{align}
Letting
$$d_j''(x,t) := d_{\mathcal{N}-1}  \frac{\overline{a_{\mathcal{N} + j}^{(\mathcal{N}-2)}} }{a_j^{(\mathcal{N}-2)}} \delta(\zeta, \lambda_j)^{-2},$$
a computation using (\ref{amlexpressionsbreather}) shows that $d_j''$ can be written as in (\ref{djdoubleprimedef}). Furthermore, utilizing (\ref{tildeMiexpansion}), (\ref{deltasymmbreather}), and (\ref{aboverlinebreather}) in (\ref{BNBNm1system}), it follows that $B_{\mathcal{N}}$ admits the expansion in (\ref{breatherBl}), where $B_{\mathcal{N},0}$ is determined by the same system as the breather except that $d_j$ is replaced by $d_j''$, i.e.,
\begin{align}\label{BNm1Bnsystembreather}
 &  \begin{cases}
  (\lambda_j I+B_{\mathcal{N}-1,0})
  \begin{psmallmatrix} 1 \\ - d_j'' \end{psmallmatrix} = 0,
  	\vspace{.1cm}\\ 
  (\bar{\lambda}_jI +B_{\mathcal{N}-1,0})  
  \begin{psmallmatrix} \overline{d_j''} \\ 1 \end{psmallmatrix} = 0,
  \end{cases}
\quad  \begin{cases}
  (\lambda_{j+1} I+B_{\mathcal{N},0})(\lambda_{j+1} I+B_{\mathcal{N}-1,0})
  \begin{psmallmatrix} 1 \\ - \overline{d_j''} \end{psmallmatrix} = 0,
  	\vspace{.1cm} \\ 
  (\bar{\lambda}_{j+1}I +B_{\mathcal{N},0})(\bar{\lambda}_{j+1}I +B_{\mathcal{N}-1,0})  
  \begin{psmallmatrix} d_j'' \\ 1 \end{psmallmatrix} = 0.
  \end{cases}
\end{align}
The solution of (\ref{BNm1Bnsystembreather}) leads to the expression (\ref{beta0expression}) for $\beta_0$ (cf. (\ref{solitonbreather})).
Taking into account also the terms of $O(t^{-1/2})$ in (\ref{BNBNm1system}), long but straightforward computations yield (\ref{beta1expression}).
\proofendcontinue

\begin{lemma}[Asymptotics near breather]\label{breathersectorlemma}
Assume $\re \lambda_j > 0$ for some $j = 1, \dots, \Lambda$. Then
\begin{align*}\nonumber
u(x,t) = &\; 4\arctan\bigg\{\frac{2\im d_j'' \, \im \lambda_j}{(|d_j''|^2 + 1)\re \lambda_j}  \bigg\} 
+ \frac{u_{rad}^{(1)}(x,t) + u_{rad}^{(2)}(x,t)}{\sqrt{t}}
	\\\nonumber
& + O\bigg(\frac{1}{k_0t} + \frac{\ln t}{t}\bigg)  \quad \text{\upshape (mod $2\pi$)}
\end{align*}
as $t \to \infty$ uniformly for $\zeta \in [v_{j+2} + \epsilon, v_{j-1} - \epsilon]$.
\end{lemma}
\proofbegin
Using the ordering in (\ref{ordering2}), equation (\ref{uargsum2}) gives
\begin{align}\nonumber
u(x,t) = &\; 2 \arg\big(-(B_{\mathcal{N}}B_{\mathcal{N}-1})_{11} - i (B_{\mathcal{N}}B_{\mathcal{N}-1})_{21}\big)
	\\\nonumber
& + \bigg(\sum_{\substack{\lambda_l \in i\R \\ 1 \leq l \leq j-1}} +\sum_{\substack{\lambda_{l+2} \in i\R \\ j \leq l \leq \mathcal{N}-2}}\bigg)2 \arg\big((iB_l)_{11}+i (-1)^{\mathcal{N} - l} (iB_l)_{21}\big) 
	\\\nonumber
& + \bigg(\sum_{\substack{\re \lambda_l > 0 \\ 1 \leq l \leq j-1}} 
+ \sum_{\substack{\re \lambda_{l+2} > 0 \\ j \leq l \leq \mathcal{N}-2}}\bigg)2 \arg\big(-(B_{l+1}B_l)_{11}+i (-1)^{\mathcal{N} - l}(B_{l+1}B_l)_{21}\big) 
	\\
& + (-1)^\mathcal{N} \tilde{u}(x,t) \quad \text{(mod $4\pi$)}.
\end{align}
The last term on the right-hand side gives rise to the contribution from $u_{rad}^{(1)}$.
By Lemma \ref{breatherBjlemma}, the first term on the right-hand side satisfies
\begin{align*}
& 2 \arg\big(-(B_{\mathcal{N}}B_{\mathcal{N}-1})_{11} - i (B_{\mathcal{N}}B_{\mathcal{N}-1})_{21}\big)
=
\beta_0
+ 2 \arg\bigg( 1 + \frac{\beta_1}{\sqrt{t}} + O\bigg(\frac{1}{k_0t} + \frac{\ln t}{t}\bigg) \bigg) \quad \text{(mod $2\pi$)}.
\end{align*}
The remaining terms are handled as in the proof of Lemma \ref{outsidesectorlemma}.
Using the expressions of Lemma \ref{breatherBjlemma} for $\beta_0$ and $\beta_1$, the lemma follows. 
\proofendcontinue

If $\lambda_j \in i\R$, then (\ref{usoljdef}) shows that $u_{sol}(x,t;j)$ is a kink (antikink) for $\im c_j > 0$ ($\im c_j < 0$). The breather has zero net topological charge. Accordingly, the topological charge accumulated by the $j-1$ fastest solitons is  $\sum_{i=1}^{j-1} \sgn(\im c_i)$; this yields the expression (\ref{uconstdef}) for $u_{const}(j)$.
The asymptotics in Sector III now follows from Lemma \ref{outsidesectorlemma}, Lemma \ref{kinksectorlemma}, and Lemma \ref{breathersectorlemma}. 
Using (\ref{MlambdajIV}) and the fact that $u_{const}(\Lambda + 1) = -2\pi \sum_{i=1}^{\Lambda} \sgn(\im c_i)$, an argument similar to the one used in Sectors I and II yields the asymptotic formula (\ref{uasymptoticsIV3}) in Sector IV.
\end{proof}

\begin{remark}\upshape
In the case of a single pure imaginary eigenvalue (i.e. $\mathcal{N}=1$ and $\lambda_1 \in i\R$), the leading order term $u_{sol}$ in (\ref{solitonuasymptoticsa}) was derived already in \cite[p. 976]{FI1992}. The identity $\arctan{x} + \arctan x^{-1} = \pi/2$, $x > 0$, shows that formula (\ref{usoljdef}) for $u_{sol}$ reduces to the expression in \cite{FI1992} when $\mathcal{N}=1$.
\end{remark}

\begin{remark}\upshape
The recursive system for $\{r_m^{(l)}, r_{\mathcal{N} + m}^{(l)}\}$ in Theorem \ref{solitonasymptoticsth} can be solved explicitly. For example, introducing new variables $\{s_m^{(l)}, s_{\mathcal{N} + m}^{(l)}\}$ in (\ref{rrecursive1}) via
$$r_m^{(l)} = s_m^{(l)} \prod_{i=1}^l \frac{\lambda_m - \lambda_i}{\lambda_m - \bar{\lambda}_i}, \qquad
r_{\mathcal{N}+m}^{(l)} = s_{\mathcal{N}+m}^{(l)} \prod_{i=1}^l \frac{\bar{\lambda}_m - \lambda_i}{\bar{\lambda}_m - \bar{\lambda}_i},$$
a computation leads to the explicit (but perhaps not very illuminating) expressions
$$\begin{cases}
s_m^{(l)} = r_m^{(0)} + \sum_{s=1}^l \sum_{1 \leq i_1 \leq \cdots \leq i_s \leq l} C_{m,\mathcal{N}+i_s} \bigg(\prod_{u=1}^{s-1} C_{\mathcal{N}+i_{u+1},\mathcal{N} + i_u}\bigg) r_{\mathcal{N}+i_1}^{(0)},
	\\
s_{\mathcal{N}+m}^{(l)} = r_{\mathcal{N}+m}^{(0)} + \sum_{s=1}^l \sum_{1 \leq i_1 \leq \cdots \leq i_s \leq l} C_{\mathcal{N}+m,\mathcal{N}+i_s} \bigg(\prod_{u=1}^{s-1} C_{\mathcal{N}+i_{u+1},\mathcal{N} + i_u}\bigg) r_{\mathcal{N}+i_1}^{(0)}, 
\end{cases}$$
for $l = 1, \dots, j-1$ and $m = l + 1, \dots, \mathcal{N}$, where
$$\begin{cases}
C_{m,\mathcal{N}+l} = \frac{\lambda_l - \bar{\lambda}_l}{\lambda_m - \lambda_l} \prod_{i=1}^{l-1}\frac{(\bar{\lambda}_l - \lambda_i)(\lambda_m - \bar{\lambda}_i)}{(\bar{\lambda}_l - \bar{\lambda}_i)(\lambda_m - \lambda_i)},
	\\
C_{\mathcal{N}+m,\mathcal{N}+l} = \frac{\lambda_l - \bar{\lambda}_l}{\bar{\lambda}_m - \lambda_l} \prod_{i=1}^{l-1}\frac{(\bar{\lambda}_l - \lambda_i)(\bar{\lambda}_m - \bar{\lambda}_i)}{(\bar{\lambda}_l - \bar{\lambda}_i)(\bar{\lambda}_m - \lambda_i)}.
\end{cases}
$$
\end{remark}

\begin{theorem}[Asymptotics of quarter-plane solutions with given initial and boundary values]\label{solitonasymptoticsth2}
Let $u_0, u_1, g_0, g_1$ be functions which satisfy (\ref{ujgjschwartz}) for some integers $N_x, N_t \in \Z$ and which are compatible with equation (\ref{sg}) to all orders at $x=t=0$. 
Define $a(k)$, $b(k)$, $A(k)$, $B(k)$ by (\ref{abABdef}), suppose Assumption \ref{solitonassumption1} holds, and define $J$ by (\ref{Jdef}). Let $M(x,t,k)$ be the solution from Theorem \ref{solitonexistenceth2} of the RH problem determined by $\{\Gamma$, $J(x,t,\cdot)$, $\{\lambda_j\}_1^\mathcal{N}$, $\{c_j e^{2i\theta(x,t,\lambda_j)}\}_1^\mathcal{N}\}$.
 
Then, for an appropriate choice of the integer $j \in \Z$, the function $u(x,t)$ defined by (\ref{ulim2}) is a smooth solution of (\ref{sg}) in the quarter-plane $\{x \geq 0, t \geq 0\}$ which obeys the initial and boundary conditions (\ref{initialboundaryvalues}). In Sectors I, II, and IV, $u(x,t)$ satisfies the asymptotic formulas
\begin{align*}
\text{\upshape Sector I:} \quad & u(x,t) = 2\pi N_x + O\big(x^{-N}\big), \qquad 1 \leq \zeta < \infty, \ x > 1,	
	\\
\text{\upshape Sector II:} \quad & u(x,t) = 2\pi N_x + O\big((1-\zeta)^N + t^{-N}\big), \qquad 0 \leq \zeta \leq 1, \ t > 1,
	\\
\text{\upshape Sector IV:} \quad & u(x,t) = 2\pi N_t + O\big(\zeta^{N} + t^{-N}\big), \qquad 0 \leq \zeta \leq 1/2, \ t > 1,
\end{align*}
for each integer $N \geq 1$. 
In Sector III, $u(x,t)$ satisfies the asymptotic formulas (\ref{solitonuasymptotics}) with $u_{sol}$, $u_{rad}^{(1)}$, $u_{rad}^{(2)}$ given as in Theorem \ref{solitonasymptoticsth}, and $u_{const}(j)$ given by 
$$u_{const}(j) = 2\pi N_x -2\pi \sum_{i=1}^{j-1} \sgn(\im c_i), \qquad  j = 1, \dots, \Lambda+1.$$
The topological charge of the solution is
$$N_x - N_t = \sum_{i=1}^{\Lambda} \sgn(\im c_i).$$
\end{theorem}
\begin{proof}
Lemmas \ref{ablemma} and \ref{ABlemma} together with Assumption \ref{solitonassumption1} imply that $r_1(k)$ and $h(k)$ satisfy Assumption \ref{r1hassumption2} except that $h(k)$ may have a finite number of simple poles in $D_2$ at the zeros of $a(k)$ and $d(k)$. The symmetries of $a,b,A,B$ imply that the constants $\{c_j\}_1^\mathcal{N}$, $\mathcal{N}:= \Lambda + n_1$, defined in (\ref{cjdef}) satisfy (\ref{cjassump}).
The global relation (\ref{GR}) implies that $r(k)$ vanishes to all orders at $k = \pm1$.   
Thus, by Theorem \ref{solitonasymptoticsth}, $u(x,t)$ is a smooth solution which satisfies the stated asymptotic formulas; the term $2\pi N_x$ has to be included in the asymptotic formulas to account for the fact that $u \to 2\pi N_x$ as $x \to \infty$. By Theorem \ref{solitonexistenceth2}, $u(x,t)$ obeys the initial and boundary conditions (\ref{initialboundaryvalues}).
\end{proof}

\bigskip
\noindent
{\bf Acknowledgement}  {\it The authors are grateful to Prof. V. P. Kotlyarov for valuable comments on a first version of the manuscript.
Support is acknowledged from the G\"oran Gustafsson Foundation, the European Research Council, Consolidator Grant No. 682537, the Swedish Research Council, Grant No. 2015-05430, and the National Science Foundation of China, Grant No. 11671095.}

\bibliographystyle{plain}
\bibliography{is}

\end{document}